%% file: epsilon-dpm2.tex
\documentclass[11pt,a4paper,reqno]{amsart}

\usepackage{amsmath,amsfonts,amsthm,epsfig,graphicx}
\usepackage{graphicx,color}

\def\B{\mathcal B}
\def\E{\mathcal{E}}
\def\F{\mathcal F}

\def\H{\mathcal{H}}

\def\N{\mathbb N}

\def\R{\mathbb R}

\def\C{\mathbf{C}}

\def\PHI{\mathbf{\Phi}}
\def\PSI{\mathbf{\Psi}}

\def\La{\Lambda}
\def\L{\Lambda}
\def\Om{\Omega}
\def\S{\Sigma}

\def\cl{{\rm cl}}                            
\def\X{\boldsymbol{\mathcal E}}
\def\ecc{\varepsilon_{\rm crit}}
\def\eh{\varepsilon_{\rm hb}}
\def\el{\varepsilon_{ \rm lip}}
\def\ea{\varepsilon_{ \rm har}}
\def\ec{\varepsilon_{\rm Ca}}
\def\et{\varepsilon_{\rm tilt}}
\def\er{\varepsilon_{\rm reg}}
\def\esc{\varepsilon_{\rm sc}}

\def\a{\alpha}
\def\b{\beta}
\def\g{\gamma}
\def\de{\delta}
\def\e{\varepsilon}
\def\k{\kappa}
\def\l{\lambda}
\def\s{\sigma}

\def\om{\omega}
\def\vphi{\varphi}

\def\Lip{{\rm Lip}}

\def\Sym{\mathbf{Sym}}

\def\tr{{\rm Tr}}
\def\Div{{\rm div}\,}
\def\Id{{\rm Id}\,}
\def\dist{{\rm dist}}

\def\diam{{\rm diam}}
\def\cof{{\rm cof}\,}
\def\spt{{\rm spt}}

\def\weak{\stackrel{*}{\rightharpoonup}}

\def\pa{\partial}

\def\cc{\subset\subset}

\def\p{\mathbf{p}}
\def\q{\mathbf{q}}
\def\C{\mathbf{C}}
\def\D{\mathbf{D}}
\def\K{\mathbf{K}}

\def\exc{\mathbf{exc}}

\def\fl{\mathbf{flat}}

\def\h{\mathbf{h}}

\def\h{\mathbf{h}\,}

\theoremstyle{plain}
\newtheorem{theorem}{Theorem}[section]
\newtheorem{lemma}[theorem]{Lemma}
\newtheorem{corollary}[theorem]{Corollary}
\newtheorem{proposition}[theorem]{Proposition}
\newtheorem*{theorem*}{Theorem}
\newtheorem*{corollary*}{Corollary}

\theoremstyle{definition}
\newtheorem{definition}[theorem]{Definition}

\newtheorem{remark}[theorem]{Remark}

\newtheorem*{notation*}{Notation}

\numberwithin{equation}{section}
\numberwithin{figure}{section}

\title[On the validity of Young's law]{Regularity of free boundaries \\ in anisotropic capillarity problems\\ and the validity of Young's law}

\author{G. De Philippis}
\address{Institute for Applied Mathematics, University of Bonn, Endenicher Allee 60, D-53115 Bonn, Germany}
\email{guido.de.philippis@hcm.uni-bonn.de}

\author{F. Maggi}
\address{Department of Mathematics, The University of Texas at Austin,  2515 Speedway Stop C1200, Austin, Texas 78712-1202, USA}
\email{maggi@math.utexas.edu}

\begin{document}

\begin{abstract}
Local volume-constrained minimizers in anisotropic capillarity problems develop free boundaries on the walls of their containers. We prove the regularity of the free boundary  outside a closed negligible set, showing in particular the validity of Young's law at almost every point of the free boundary. Our regularity results are not specific to capillarity problems, and actually apply to sets of finite perimeter (and thus to codimension one integer rectifiable currents) arising as minimizers in other variational problems with free boundaries.
\end{abstract}

\maketitle


\section{Introduction} \subsection{Young's law in anisotropic capillarity problems} According to the historical introduction to Finn's beautiful monograph \cite{Finn}, Young \cite{young1805} introduced in 1805 the notion of mean curvature of a surface in the study of capillarity phenomena. Mean curvature was reintroduced the following year by Laplace, together with its analytic expression and its linearization (the Laplacian), the latter being recognized as inadequate to describe real liquids in equilibrium. In the same essay \cite{young1805}, Young also formulates the equilibrium condition for the contact angle of a capillarity surface commonly known as {\it Young's law}. These ideas were later reformulated by Gauss \cite{gauss1830} through the principle of virtual work and the introduction of a suitable free energy. Gauss' free energy consists of four terms: a free surface energy, proportional to the area of the surface separating the fluid and the surrounding media (another fluid or gas) in the given solid container, a wetting energy, accounting for the adhesion between the fluid and the walls of the container,
the gravitational energy; and, finally, a Lagrange multiplier taking into account the volume constraint on the region occupied by the liquid. Since then, a huge amount of interdisciplinary literature has been devoted to the study of qualitative and quantitative properties of local minimizers and stationary surfaces of Gauss' free energy.

A modern formulation of Gauss' model, including the case of possibly anisotropic surface tension densities, as well as that of general potential energy terms, and extending the setting of the problem to (the geometrically relevant case of) arbitrary ambient space dimension, is obtained as follows. Given $n\ge 2$, an open set $\Omega$ with Lipschitz boundary in $\R^n$ (the container), and a set $E\subset \Omega$ (the region of occupied by the liquid droplet) with $\pa E\cap \Omega$ a smooth hypersurface, one considers the free energy
\begin{equation}
\label{free energy}
  \F(E)=\int_{\pa E\cap\Omega}\Phi(x,\nu_E)\,d\H^{n-1}+\int_{\pa E\cap \pa \Omega}\s(x)\,d\H^{n-1}+\int_E\,g(x)\,dx\,,
\end{equation}
where $\H^k$ is the $k$-dimensional Hausdorff measure on $\R^n$, $\nu_\Omega$ and $\nu_E$ denote the outer unit normals to $\Omega$ and $E$ respectively. Here $\Phi:\Omega\times \R^{n}\to[0,\infty)$ is  convex and positively one-homogeneous in the second variable and represents  the  (possibly anisotropic) surface tension density, $\s:\pa \Omega\to\R$ is the relative adhesion coefficient between the liquid and the boundary walls of the container and it satisfies
\begin{equation}
\label{sigma constraint}
  -\Phi(x,-\nu_\Omega(x))\le \s(x)\le \Phi(x,\nu_\Omega(x))\,,\qquad \forall x\in\pa \Omega\,,
\end{equation}
and $g:\Omega\to\R$ is a potential energy per unit mass.
The classical capillarity problem is then obtained by taking $n=3$, $\Phi=|\nu|$, and $g(x)=g_0\,\rho\,x_3$, where $\rho$ is the constant density of the fluid and $g_0$ is the gravity of Earth.

If we are interested in global volume-constrained minimizers of $\E$, then we are led to consider the variational problem
\begin{equation}
  \label{variational problem}
  \inf\Big\{\F(E):|E|=m\Big\}\,,\qquad\mbox{$m>0$ fixed}\,,
\end{equation}
where $|E|$ stands for the Lebesgue measure of $E$. Note that if  we set \(\sigma=0\) and \(g=0\) in \eqref{free energy}, then problem \eqref{variational problem} reduces to the relative isoperimetric problem in \(\Omega\) with respect to the $\Phi$-perimeter, also known as the relative Wulff problem in $\Om$. Alternatively, one may consider local volume-constrained minimizers $E$ of $\F$, or even stationary sets $E$ for $\F$ with respect to volume-preserving flows. In all these cases, provided the objects involved are smooth enough, the equilibrium conditions \eqref{euler lagrange equation anisotropic} and \eqref{youngs law anisotropic} below are satisfied. Precisely, if $\pa \Omega$ and $\Phi$ are of class $C^2$, if $g$ and $\s$ are continuous, and if (denoting by $\cl$ topological closure in $\R^n$ and by \(\partial_{\partial \Omega}\) topological boundary in the relative topology of $\pa\Om$)
\begin{figure}
  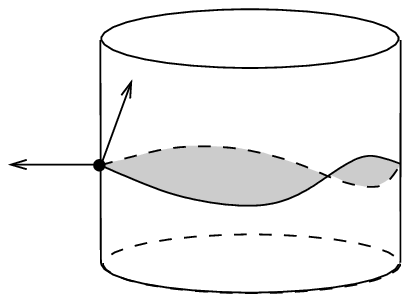\caption{{\small In this picture the region $E$ occupied by the liquid lies at the bottom of the container $\Om$. The adhesion coefficient $\s$ is integrated on the wetted surface $\pa E\cap\pa\Om$, which consists of the bottom of the cylinder plus the later cylindrical surface below the free boundary $M\cap\pa\Om$ of the surface $M=\cl(\Om\cap\pa E)$. In particular, one expects $\pa_{\pa\Om}(\pa E\cap\pa\Om)$ (the topological boundary of the wetted surface relative to the boundary of the container) to coincide with $M\cap\pa\Om$. The angle between $\nu_\Om(x)$ and $\nu_E(x)$ at $x\in M\cap\pa\Om$ is prescribed by $\Phi$ and $\s$ through Young's law \eqref{youngs law anisotropic}.}}\label{fig container}
\end{figure}
the ``capillarity surface''
\[
M=\cl(\pa E\cap \Omega)\,,
\]
is a class $C^2$ hypersurface with boundary \(M\cap \partial \Omega=\partial_{\partial \Omega}(\partial E\cap\pa  \Omega)\), then one has
\begin{align}
\label{euler lagrange equation anisotropic}
 &\Div_M\big[\nabla \Phi(x,\nu_E)\big]+\nabla_x\Phi(x,\nu_E)\cdot\nu_E(x)=-g(x)+{\rm constant}\,,\qquad&&\forall\, x\in M\cap \Omega\,,
  \\
  \label{youngs law anisotropic}
 & \nabla\Phi(x,\nu_E(x))\cdot\nu_\Omega(x)=\s(x)\,,&&\forall x\in M\cap\pa \Omega\,;
\end{align}
see Figure \ref{fig container}. Here, $\nabla_x\Phi$ and $\nabla\Phi$ denote the gradients of $\Phi(x,\nu)$ in the $x$ and $\nu$ variables respectively, while $\Div_M$ denote the tangential divergence with respect to $M$. In the isotropic case $\Phi(x,\nu)=|\nu|$, we thus find
\begin{align}\nonumber
  &\,H_M(x)=-g(x)+{\rm constant}\,,\qquad&&\forall x\in M\cap \Omega\,,
  \\\label{youngs law}
  &\nu_E(x)\cdot\nu_\Omega(x)=\s(x)\,,\qquad&&\forall x\in M\cap\pa \Omega\,,
\end{align}
where $H_M$ is the scalar mean curvature of $M$ with respect to the orientation induced by $\nu_E$. In particular, the equilibrium condition \eqref{youngs law} is Young's law.  Notice also  that by the  convexity and the one-homogeneity of \(\Phi(x,\cdot)\), \eqref{youngs law anisotropic} implies that \eqref{sigma constraint} is a necessary condition in order to have \(M\cap \partial \Omega \ne \emptyset\).

\subsection{Boundary regularity and validity of Young's law} Mathematically speaking, the most elementary setting in which one can prove the existence of such capillarity surfaces is given by the theory of sets of finite perimeter developed by Caccioppoli and De Giorgi. In this framework, one can easily prove the existence of minimizers in \eqref{variational problem} under natural assumptions on $\Omega$, $g$ and $\s$. (In particular, it is easy to see that the constraint \eqref{sigma constraint} on the adhesion coefficient is, in general, a necessary condition to ensure the existence of minimizers; see \cite[Section 19.1]{maggiBOOK} for various examples and remarks.)

When writing $\F(E)$ for $E$ a set of finite perimeter one has to replace the topological boundary $\pa E$ of $E$ (that in the case of a generic set of finite perimeter could have positive {\it volume}!) with its {\it reduced boundary} $\pa^*E$. (See section \ref{section definitions} for the definition.) It is important to take into account that $\pa^*E$ is, in general, just a generalized hypersurface in the sense of Geometric Measure Theory, that is, $\pa^*E$ is just a countable union of compact subsets of  $C^1$-hypersurfaces. Correspondingly, the actual ``capillarity surfaces'' $M$ one proves the existence of take the form
\[
M=\cl(\pa^*E\cap \Omega)\,.
\]
In other words, existence theory forces one to consider extremely rough hypersurfaces. Addressing the regularity issue is thus a fundamental task in order to understand the physical significance of the model itself and the validity of the equilibrium conditions \eqref{euler lagrange equation anisotropic} and \eqref{youngs law anisotropic}, and, indeed, the problem has been considered by several authors. We now review the known results on this problem, that are mainly concerned with the case when $E$ is a local volume-constrained minimizer of $\F$.

Interior regularity, that is, the regularity of $M\cap \Omega$, can be addressed in the framework developed by De Giorgi \cite{DeGiorgiREG}, Reifenberg \cite{reifenberg1,reifenberg2,reifenberg3}, and Almgren \cite{Almgren68}. Precisely, if we assume that $\Phi$ is a smooth, uniformly elliptic integrand (see Definition \ref{def:integrand} below), and that $g$ is a smooth function, then, by combining results from \cite{Almgren76,schoensimonalmgren,bombieri}, one can see that
\[
M\cap \Omega=M_{{\rm reg}}^{{\rm int}}\cup M_{{\rm sing}}^{{\rm int}}\,,
\]
where $M_{{\rm reg}}^{{\rm int}}$ is a smooth hypersurface, relatively open into $M\cap \Omega$, and  where the singular set $M_{{\rm sing}}^{{\rm int}}$ is  relatively closed, with $\H^{n-3}(M_{{\rm sing}}^{{\rm int}})=0$. Moreover, in the isotropic case \(\Phi=|\nu|\), $M_{{\rm sing}}^{{\rm int}}$ is discrete if $n=8$ and satisfies ${\rm dim}(M_{{\rm sing}}^{{\rm int}})\le n-8$ if $n\ge 9$, where ${\rm dim}$ stands for Hausdorff dimension.
In particular, interior regularity ensures that the Euler--Lagrange equation \eqref{euler lagrange equation anisotropic} holds true in classical sense at every $x\in M_{{\rm reg}}^{{\rm int}}$. The picture for what concerns  the regularity of the free-boundary \(M\cap \pa \Omega\), and thus validity of Young's law \eqref{youngs law anisotropic}, is however much more incomplete.

When \(\Omega\) is the half-space \(\{x_n>0\}\), \(\Phi=|\nu|\), $\s$ is constant and $g=g(x_n)$ (this is the so-called sessile liquid drop problem when $g$ is the gravity potential), then one can deduce the regularity of the free boundary by combining the interior regularity theory with the  symmetry properties of minimizers, see \cite{GonzalezREGOLARITAGOCCIA}. Although this kind of analysis was recently carried out in the anisotropic setting under suitable symmetry assumptions on $\Phi=\Phi(\nu)$, see \cite{baer}, it is clear that the approach itself is intrinsically limited to the case when $\Om$ is an half-space, \(\sigma\) is a constant, and \(g\) is a function of the vertical variable \(x_n\) only.

Again in the case of the sessile liquid drops, Caffarelli and Friedman in \cite{caffarellifriedman85} (see also \cite{caffarellimellet}) study the regularity of the free boundary regularity when $2\le n\le 7$ and $\s$ is possibly non-constant and takes values in $(-1,0)$. The non-positivity of $\s$, in combination with global minimality, implies that  $E$  is the subgraph of a function $u:\R^{n-1}\to[0,\infty)$. Since \(\sigma\ne 0\), they can show that \(u\) is globally Lipschitz, and thus exploit the regularity theory for free boundaries of uniformly elliptic problems  developed in \cite{altcaffarelli,altcaffarellifriedman}. Note that it is (the a-posteriori validity of) Young's law $-\nu_E(x)\cdot e_n=\sigma(x)$ itself to show how the assumption $\s\ne 0$ is essential to this method: indeed, at a boundary point where $\s=0$  one cannot certainly expect $u$ to be Lipschitz regular. We also point out the the proof in \cite{caffarellifriedman85} highly relies on the analyticity of the minimizers in the interior, that is actually the reason for the restriction \(2\le n\le 7\) on the ambient space dimension, and a further obstruction to the extension to anisotropic problems.

In the case of generic containers \(\Omega\) we are only aware of a sharp result by Taylor \cite{taylor77} in dimension $n=3$. Taylor fully addresses  three-dimensional isotropic case \(\Phi=|\nu|\) as a byproduct of the methods she developed in the study of Plateau's laws \cite{taylor76}. Her result is fully satisfactory for what concerns local minimizers of Gauss' energy in physical space, but it does not extend to anisotropic surface energies (as it is based on monotonicity formulas and epiperimetric inequalities). Moreover, even in the isotropic case, her arguments seem to be somehow limited to the case $n=3$ (although, of course, this is not really a limitation in the study of the capillarity problem).

The case \(\Phi=|\nu|\) and $\s\equiv0$ in arbitrary dimension is covered by the works of  Gr\"uter and Jost \cite{gruterjost} and of Gr\"uter \cite{gruter,gruter2,gruter3}.
These results apply for instance to the regularity of free boundaries of minimizers of  relative isoperimetric problems and  of mass minimizing current in relative homology classes. Part of the theory also extends to case of stationary varifolds of arbitrary codimension, \cite{gruterjost}. The key idea here is to take advantage of the condition \(\s\equiv0\), together with the isotropy of the area functional, in order to apply the interior regularity theory after a local ``reflection'' of the minimizer across \(\pa \Omega\).


\subsection{Main results} Our main result, Theorem \ref{thm main}, is a general regularity theorem for free boundaries of local minimizers of anisotropic surface energies. One can deduce from Theorem \ref{thm main} a regularity result for anisotropic capillarity surfaces, that works without artificial restrictions on the dimension or the geometry of the container, and that -- in the anisotropic case -- appears to be new even in dimension $n=3$, see Theorem \ref{thm capillari} below. Let us  premise the following two definitions to the statements of these results:

\begin{definition}\label{def:integrand}
  [Elliptic integrands] Given an open set $\Om\subset\R^n$, one says that $\Phi$ is an {\it elliptic integrand on $\Om$} if $\Phi:\cl(\Om)\times\R^n\to[0,\infty]$ is lower semicontinuous, with $\Phi(x,\cdot)$ convex and positively one-homogeneous, i.e \(\Phi(x, t\nu)=t\,\Phi(x,\nu)\) for every \(t\ge 0\). If $\Phi$ is an elliptic integrand on $\Om$ and $E$ is a set of locally finite perimeter in $\Om$, then we set
  \[
  \PHI(E;G)=\int_{G\cap\pa^*E}\Phi(x,\nu_E(x))\,d\H^{n-1}(x)\in[0,\infty]\,,
  \]
  for every Borel set $G\subset \Om$. Given $\lambda\ge 1$ and $\ell\ge0$, one says that $\Phi$ is a {\it regular elliptic integrand on $\Om$ with ellipticity constant $\l$ and Lipschitz constant $\ell$}, and write
  \[
  \Phi\in\X(\Om,\lambda,\ell)\,,
  \]
  if $\Phi$ is an elliptic integrand on $\Om$, with $\Phi(x,\cdot)\in C^{2,1}(\mathbf S^{n-1})$ for every $x\in\cl(\Om)$, and if the following properties hold true for every $x\,,y\in \cl(\Om)$, $\nu,\nu'\in \mathbf S^{n-1}$, and $e\in\R^n$:
  \begin{gather}\label{Phi 1}
  \frac1\lambda\le\Phi(x,\nu)\le\lambda\,,
  \\
  \label{Phi x ell}
  |\Phi(x,\nu)-\Phi(y,\nu)|+|\nabla\Phi(x,\nu)-\nabla\Phi(y,\nu)|\le \ell\,|x-y|\,,
  \\
  \label{Phi nabla 1}
  |\nabla\Phi(x,\nu)|+\|\nabla^2\Phi(x,\nu)\|+\frac{\|\nabla^2\Phi(x,\nu)-\nabla^2\Phi(x,\nu')\|}{|\nu-\nu'|}\le\lambda\,,
  \\
   \label{elliptic}
  \nabla^2\Phi(x,\nu)e\cdot e\ge\frac{\big|e-(e\cdot\nu )\nu\big|^2}{\lambda}\,,
  \end{gather}
 where $\nabla\Phi$ and $\nabla^2\Phi$ stand for the gradient and Hessian of $\Phi$ with respect to the $\nu$-variable. Finally, any \(\Phi\in \X_*(\l)=\X(\R^n,\l, 0)\) is said a {\it regular autonomous elliptic integrand}.
\end{definition}


We now state our main regularity result concerning capillarity problems.

\begin{theorem}\label{thm capillari}
  If $\Omega$ is an open bounded set with \(C^{1,1}\) boundary in $\R^n$, $\Phi$ is a regular elliptic integrand on \(\Omega\), $g\in L^\infty (\Omega)$, and $\s\in \Lip(\pa \Omega)$ satisfies
  \begin{equation}\label{sigma constraint stretto}
  -\Phi(x,-\nu_\Omega(x))<\s(x) <\Phi(x,\nu_\Omega(x))\,,\qquad \forall x\in\pa \Omega\,,
  \end{equation}
  then there exists a minimizer $E$ in \eqref{variational problem} such that  $E$ is equivalent to an open set and its trace \(\pa E\cap \pa \Omega\) is a set of finite perimeter in \(\pa \Omega\). Moreover if  \(M=\cl (\pa E \cap \Omega)\) then
\[
\pa_{\pa \Omega} (\pa E\cap \pa \Omega)=M\cap \pa \Omega\,,
\]
and   there exists   a closed set \(\Sigma\subset M\), with \(\H^{n-2}(\Sigma)=0\) such that  \(M\setminus \Sigma\) is a \(C^{1,1/2}\) hypersurface with boundary. In particular, Young's law \eqref{youngs law anisotropic} holds true at every \(x\in (M \cap \pa \Omega)\setminus \Sigma\).
\end{theorem}

\begin{remark}
 As proved in \cite{schoensimonalmgren} one has a better estimate on the singular set in the interior of $\Omega$, namely $\H^{n-3}(\S\cap\Omega)=0$.
\end{remark}

\begin{corollary}[Isotropic case]\label{corollario capillari}
  Under the assumptions of Theorem \ref{thm capillari}, let $\Phi(x,\nu)=|\nu|$ for every $x\in\Omega$ and $\nu\in \R^{n}$. Then $\S\cap\pa\Om=\emptyset$ if $n=3$, $\S\cap\pa\Om$ is a discrete set if $n=4$, and $\H^s(\S\cap\pa\Om)=0$ for every $s>n-4$ if $n\ge 5$.
\end{corollary}

\begin{remark}
  By the case $n=3$ of Corollary \ref{corollario capillari} we obtain an alternative proof of Taylor's theorem \cite{taylor77}. Notice also that, under the assumptions of Corollary \ref{corollario capillari}, classical regularity for local minimizers of the perimeter gives that  $\S\cap\Om=\emptyset$ if $n\le 7$, $\S\cap\Om$ discrete if $n=8$, and $\H^s(\S\cap\Om)=0$ for every $s>n-8$ if $n\ge 9$.
\end{remark}

\begin{remark}
  Higher regularity of \(\cl(\pa  E\cap \Omega)\setminus \Sigma\) is obtained by combining Theorem \ref{thm capillari} with elliptic regularity theory for non-parametric solutions of \eqref{euler lagrange equation anisotropic} and \eqref{youngs law anisotropic}.
\end{remark}

\begin{remark}
 The
 \begin{figure}
   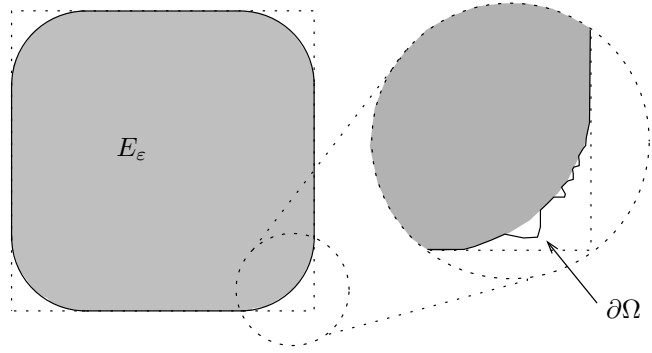\caption{{\small When the strict upper bound in \eqref{sigma constraint stretto} is an equality the conclusions of Theorem \ref{thm main} can possibly fail.}}\label{fig cheeger}
 \end{figure}
 strict inequality in   \eqref{sigma constraint stretto} is somehow necessary. Indeed, according to \eqref{youngs law anisotropic} it predicts that \(M\) will intersect \(\pa \Omega\) transversally. Moreover, if \eqref{sigma constraint stretto} fails, it is possible to construct examples of minimizers of \eqref{variational problem} which do not satisfy the conclusion of Theorem \ref{thm capillari}. For example, if $Q=(0,1)^2\subset\R^2$ is a unit square and $\e>0$ is small enough, then the open set $E_\e$ depicted in Figure \ref{fig cheeger} is a minimizer in
 \[
 \inf\Big\{P(E):E\subset Q\,,|E|=1-\e\Big\}\,.
 \]
 Let now $\Om$ be an open set with smooth boundary such that $E_\e\subset\Om\subset Q$ and $Q\cap\pa E_\e\cap\pa\Om$ is a Cantor-type set contained in the circular arc $Q\cap\pa E_\e$. Then $E_\e$ is a minimizer in
 \[
 \inf\Big\{P(E):E\subset \Om\,,|E|=1-\e\Big\}\,,
 \]
 but \(E_\e\) does not satisfy the conclusion of Theorem \ref{thm capillari}.
\end{remark}

Theorem \ref{thm capillari} can be actually obtained as corollary of Theorem \ref{thm main} below, which addresses the boundary regularity issue in the class of almost-minimizers introduced in the next definition.

\begin{definition}[Almost-minimizers]\label{definition minimizers}  Let an open set $A$ and an open half-space $H$ in $\R^n$ be given (possibly \(H=\R^n\)), together with constants $r_0\in(0,\infty]$ and $\Lambda\ge0$, a regular elliptic integrand $\Phi$ on $A\cap H$, and a function \(\sigma: A\cap \pa H\to \R\) with
\begin{equation*}
-\Phi(x,-\nu_H)\le \sigma(x)\le \Phi(x,\nu_H)\qquad \forall\, x\in A\cap \pa H\,.
\end{equation*}
A set $E\subset H$ of locally finite perimeter in \(A\) is a {\it $(\La,r_0)$-minimizer of $(\PHI,\s)$ in $(A,H)$}, if
\begin{multline}\label{minimalityintro}
\PHI(E;H\cap W)+\int_{W\cap(\pa^* E\cap \pa H)}\hspace{-0.5cm} \sigma\,\,d\H^{n-1}\le \PHI(F;H\cap W) +\int_{W\cap(\pa^* F\cap \pa H)}\hspace{-0.5cm}\sigma \,d\H^{n-1}+\Lambda\,|E\Delta F|\,,
\end{multline}
whenever $F\subset H$, $E\Delta F\cc W$, and $W\cc A$ is open, with $\diam(W)<2r_0$;
\begin{figure}
    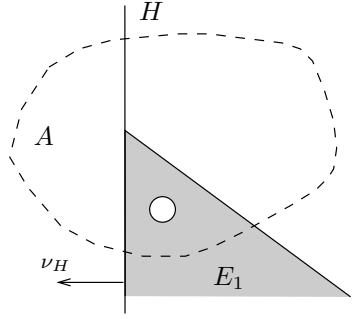\caption{{\small A $(\La,r_0)$-minimizer of $(\PHI,\s)$ in $(A,H)$. Rougly speaking, inside balls $B_{x,r}$ of radius at most $r_0$ that are compactly contained in $A$, and up to a volume-type higher order perturbation, $E$ is a minimizer of $F\mapsto\PHI(F,H)+\int_{\pa F\cap\pa H}\s$ with respect to its own boundary data on $H\cap\pa B_{x,r}$, and with free boundary on $B_{x,r}\cap \pa H$. On balls $B_{x,r}$ that do not intersect $\pa H$, we just have a local almost-minimality condition.}}\label{fig minimi}
  \end{figure}
see Figure \ref{fig minimi}. When $\s=0$, we simply say that $E$ is a {\it $(\La,r_0)$-minimizer of $\PHI$ in $(A,H)$}; when $\s=0$, $\La=0$, and $r_0=+\infty$, then we say that $E$ is a {\it minimizer of $\PHI$ in $(A,H)$}.
\end{definition}

\begin{remark}
  {\rm Note that if $\cl(A)\subset H$ (as it happens, for example, in the limit case $H=\R^n$), then Definition \ref{definition minimizers} reduces to a local almost-minimality notion analogous to the ones considered in \cite{Almgren76,bombieri,tamanini,DuzaarSteffen} and \cite[Section 21]{maggiBOOK}.}
\end{remark}

\begin{theorem}\label{thm main} If $E$ is a $(\La,r_0)$-minimizer of $(\PHI,\s)$ in $(A,H)$ for some \(\sigma\in \Lip( A\cap \pa H)\) with
\begin{equation}\label{sigmalim2}
-\Phi(x,-\nu_H)<\sigma(x)<\Phi(x,\nu_H)\,,\qquad \forall\, x\in A\cap \pa H\,.
\end{equation}
Then there is an open set \(A'\subset A\) with \(A\cap \pa H=A'\cap \pa H\) such that  \(E\) is equivalent to an open set in \(A'\) and  \(\pa E\cap \pa H\) is a set of locally finite perimeter in \(A'\cap\pa H\) (equivalently in \(A\cap \pa H)\). Moreover,    if \(M=\cl(\pa E\cap H)\) then
\[
\pa_{\pa H} (\pa E\cap \pa H)\cap A =\pa_{\pa H} (\pa E\cap \pa H)\cap A' =M\cap \pa H
\]
and there exists a relatively closed set \(\Sigma\subset M\cap \pa H\) such that    \(\H^{n-2}(\Sigma)=0\) and for every \(x\in (M\cap \pa H)\setminus \Sigma\),  \(M\) is a \(C^{1,1/2}\) manifold with boundary in a neighborhood of \(x\) for which
\[
\nabla\Phi(x,\nu_E(x))\cdot \nu_H=\sigma(x)\qquad \forall x\in (M\cap \pa H)\setminus\Sigma\,.
\]
\end{theorem}

Being the class of regular elliptic integrands invariant under \(C^{1,1}\) diffeomorphism (see the discussion in section \ref{section examples}), Theorem \ref{thm main} applies to a wider class of variational problems than just \eqref{variational problem}. For instance, it applies to relative anisotropic isoperimetric problems  in smooth domains, or in Riemannian and Finsler manifolds. Moreover, by arguing as in \cite{gruter2}, Theorem \ref{thm main} can be used to address the regularity of \(\Phi\)-minimizing integer rectifiable   codimension  one currents in relative homology classes \(\mathbf H_{n-1}(N,B)\) where \(N\) is a smooth \(n\)-dimensional manifold and \(B\subset N\) is a smooth \((n-1)\)-dimensional submanifold, see \cite[4.4.1, 5.1.6]{FedererBOOK} for definitions and terminology. 

\subsection{Proof of Theorem \ref{thm main} and organization of the paper} We conclude this introduction with a few comments on our proofs, and with a brief description of the structure of the paper.

The core of the paper consists of sections \ref{section almost-minimizers}--\ref{section singular set}, where we prove Theorem \ref{thm main} in the $\s=0$ case. In section \ref{section almost-minimizers}, after setting our notation and terminology, we prove several basic properties of almost-minimizers to be repeatedly used in subsequent arguments. Sections  \ref{section eps regularity}-\ref{section proof of lemma eps} are devoted to the proof of an ``\(\e\)-regularity theorem'' for almost-minimizers, Theorem \ref{thm epsilon}. This theorem states the existence of an universal constant  \(\overline \e\) with the following property: if around a free boundary point \(x\),  and for some $r>0$ sufficiently small, one has
 \begin{equation}
  \label{regularity criterion}
  \inf_{\nu\in \mathbf S^{n-1}}\,\frac1{r^{n-1}}\int_{H\cap B_{x,r}\cap\pa^*E}\frac{|\nu_E-\nu|^2}2\,d\H^{n-1}\le\overline \e\,,
\end{equation}
then $\cl(\pa E\cap H)$ is a $C^{1,1/2}$-manifold with boundary in a neighborhood of  $x$. Here the consideration of the case \(\sigma=0\), together  with an appropriate choice of coordinates, allows us to ``linearize'' on a Neumann-type elliptic problem for which good estimates are known. (In other words, we develop the appropriate version of De Giorgi's harmonic approximation technique in our setting.) In section \ref{section singular set}, Theorem \ref{thm singular set}, we estimate the size of the set where the \(\e\)-regularity theorem applies by exploiting some ideas introduced by Hardt in  \cite{hardt}. Note that when $x$ is an interior point, De Giorgi's rectifiability theorem  ensures that the set where (the appropriate version) of  \eqref{regularity criterion} holds true at some scale \(r\)  is of full \(\H^{n-1}\)-measure in the boundary of $E$. However, as we expect the free boundary to be \((n-2)\)-dimensional, and thus $\H^{n-1}$-negligible, we cannot deduce the existence of boundary points at which the $\e$-regularity theorem applies by De Giorgi's theorem only. We have instead to rely on ad hoc arguments based on minimality, and
this is exactly the content of section \ref{section singular set}.

In section \ref{section examples} we begin by showing how to reduce the proof of Theorem \ref{thm main} to the case when $\s=0$. This is achieved with the aid of the divergence theorem. Precisely, we show that if $E$ is a \((\Lambda,r_0)\)-minimizer of $(\PHI,\s)$ in $(A,H)$ and $x\in A\cap\pa H$, then $E$ is actually a \((\Lambda_*,r_0)\)-minimizer of $(\PHI_*,0)$ in $(B_{x,r_*},H)$ for suitable constants $\Lambda_*$ and $r_*$, and for a suitable regular elliptic integrand $\Phi_*$. Having assumed strict inequalities in \eqref{sigmalim2} plays a crucial role in showing that the new integrand \(\Phi_*\) is still uniformly elliptic. Another interesting qualitative remark is that our method, even in the isotropic case, requires the consideration of anisotropic functionals in order to reduce to the case that $\s=0$. We finally conclude section \ref{section examples} with the proofs of Theorem \ref{thm main}, Theorem \ref{thm capillari} and Corollary \ref{corollario capillari}.

\medskip

\noindent {\bf Acknowledgement}: We thank Frank Duzaar for pointing out to us Jean Taylor's paper \cite{taylor77} and the lack of a general boundary regularity theorem in higher dimension, thus stimulating the writing of this paper. The work of FM was supported by the NSF Grant DMS-1265910.

\section{Almost-minimizers with free boundaries}\label{section almost-minimizers} In section \ref{section notation} we fix our notation for sets in $\R^n$, while in section \ref{section sofp} we gather the basic facts concerning sets of finite perimeter. In section \ref{section definitions} we discuss some properties of the almost-minimizers introduced in Definition \ref{definition minimizers}, while in section \ref{section young on half} we derive the anisotropic Young's law for half-spaces. Sections \ref{section density estimates} and \ref{section compactness} contain classical density estimates and compactness properties of almost-minimizers. In section \ref{section contact sets} we discuss some general properties of contact sets of almost-minimizers, prove a strong maximum principle, and set a useful normalization convention to be used in the rest of the paper. Finally, in section \ref{section cambio di variabili L}, we study the transformation of almost-minimizers under ``shear-strained'' deformations, a technical device that will be repeatedly applied in the proof of the $\e$-regularity theorem, Theorem \ref{thm epsilon}.

\subsection{Basic notation}\label{section notation} {\it Norms and measures}. We denote by \(v\cdot w\) the scalar product in \(\R^n\) and by   $|v|=(v\cdot v)^{1/2}$ the Euclidean norm. We  set
\[
\|L\|=\sup\{|Lx|:x\in\R^n\,,|x|<1\}\,,
\]
for the operator norm of a linear map $L:\R^n\to\R^n$. We denote by $\H^k$ the $k$-dimensional Hausdorff measure in $\R^n$ and set $\H^n(E)=|E|$ for every $E\subset\R^n$.

\noindent {\it Reference cartesian decomposition}. We denote by
 \[
 \p:\R^n\to\R^{n-1}\quad \textrm{and}\qquad  \q:\R^n\to\R
 \]
the orthogonal projections associated to the Cartesian decomposition of $\R^n$ as $\R^{n-1}\times\R$; correspondingly, $x=(\p x,\q x)$ for every $x\in\R^n$. We set
\[
B=\{x\in\R^n:|x|<1\}\,,\quad\C=\{x\in\R^n:|\p x|<1\,,|\q x|<1\}\,,\quad \D=\{z\in\R^{n-1}:|z|<1\}\,,
\]
so that $\C=\D\times(-1,1)$. Sometimes we will identify \(\D\) with the subset of \(\R^{n}\) given by \(\D\times \{0\}\). Even when doing so, \(\pa \D\) denotes the boundary of \(\D\) relative to \(\R^{n-1}\), i.e. we always have
\[
\pa \D=\Big\{z\in\R^{n-1}:|z|=1\Big\}\,.
\]
Given a vertical half-space \(H=\{x_1>b\}\subset \R^n\) (\(b\in \R\)), again with a slight abuse of notion we will set
\[
\D\cap H=\Big\{z\in\R^{n-1}:|z|<1\ z_1>b\Big\}\,,
\]
as well as
\begin{eqnarray*}
H\cap\pa \D&=&\Big\{z\in\R^{n-1}:|z|=1\ z_1>b\Big\}\,,
\\
\pa (\D\cap H)&=&\Big(H\cap\pa \D\Big)\cup\Big\{z\in\R^{n-1}:|z|\le 1\ z_1=b\Big\}\,.
\end{eqnarray*}

\noindent {\it Scaling maps}. Given $E\subset\R^n$, $x\in\R^n$ and $r>0$, we set
\[
E_{x,r}=x+r\,E\,,\qquad E^{x,r}=\frac{E-x}r\,.
\]
In this way, for every $x\in\R^n$ and $r>0$,
\begin{eqnarray*}
  &&B_{x,r}=\{y\in\R^n:|x-y|<r\}=B(x,r)\,,
  \\
  &&\C_{x,r}=\{y\in\R^n:|\p(y-x)|<r\,,|\q(y-x)|<r\}=\C(x,r)\,,
\end{eqnarray*}
and, similarly, for every $z\in\R^{n-1}$ and $r>0$
\[
\D_{z,r}=\{y\in\R^{n-1}:|y-z|<r\}=\D(z,r)\,.
\]
In case \(x\,,z=0\) we simply write \(B_r\), \(\C_r\) and \(\D_r\).


\smallskip

\noindent {\it Convergence of sets}. Let $A$ be an open set in $\R^n$. Given a sequence of Lebesgue measurable sets $\{E_h\}_{h\in\N}$ in $\R^n$, we say that
\[
\textrm{$E_h\to E$ in $L^1_{{\rm loc}}(A)$  }\qquad \textrm{if $|(E_h\Delta E)\cap K|\to 0$ as $h\to\infty$ for every $K\cc A$}\,.
\]
Given an open half-space $H\subset\R^n$ and a sequence of Borel sets $\{G_h\}_{h\in\N}\subset \pa H$, we say that
\[
\textrm{$G_h \to G$ in $L^1_{\rm loc}(A\cap\pa H)$ }\qquad \textrm{if $\H^{n-1}(K\cap(G_h\Delta G))\to 0$ as $h\to\infty$ for every $K\cc A$.}
\]

\subsection{Sets of finite perimeter}\label{section sofp} Given a Lebesgue measurable set $E\subset\R^n$ and an open set $A\subset\R^n$, we say that $E$ is of locally finite perimeter in $A$ if there exists a $\R^n$-valued Radon measure $\mu_E$ (called the Gauss-Green measure of \(E\)) on $A$  such that
\[
\int_{E}\nabla\vphi(x)\,dx=\int_A\vphi\,d\mu_E\,,\qquad\forall\vphi\in C^1_c(A)\,,
\]
and set $P(E;G)=|\mu_E|(G)$ for the perimeter of $E$ relative to $G\subset A$. (Notice that $\mu_E=-D1_E$, the distributional derivative of $1_E$.) The well-known compactness theorem for sets of finite perimeter states that if $\{E_h\}_{h\in\N}$ is a sequence of sets of locally finite perimeter in $A$ and $\{P(E_h;A_0)\}_{h\in\N}$ is bounded for every $A_0\cc A$, then there exists $E$ of locally finite perimeter in $A$ such that, up to extracting subsequences, $E_h\to E$ in $L^1_{\rm loc}(A)$; see, for instance, \cite[Corollary 12.27]{maggiBOOK}.

\medskip

\noindent {\it Minimal topological boundary.} The support of $\mu_E$ can be characterized by
\begin{equation}
  \label{sptmuE}
  \spt\mu_E=\Big\{x\in A:0<|E\cap B(x,r)|<\om_n\,r^n\,,\forall r>0\Big\}\subset A\cap\pa E\,.
\end{equation}
If $E$ is of locally finite perimeter in $A$ and $|(E\Delta F)\cap A|=0$, then $F$ is of locally finite perimeter in $A$ with $\mu_E=\mu_F$.

\smallskip

\noindent {\it Reduced and essential boundaries}. If $E\subset\R^n$, $t\in[0,1]$,we set
\[
E^{(t)}=\big\{x\in\R^n \textrm{ such that } |E\cap B_{x,r}|=t\,|B_{x,r}|+o(r^n) \textrm{ as $r\to 0^+$}\big\},
\]
The {\it essential boundary} of $E$ is defined as $\pa^{\rm e}E=\R^n\setminus(E^{(0)}\cup E^{(1)})$. If $E$ is of locally finite perimeter in the open set $A$, then the {\it reduced boundary} $\pa^*E\subset A$ of $E$ is the set of those $x\in A$ such that
\[
\nu_E(x)=\lim_{r\to 0^+}\frac{\mu_E(B_{x,r})}{|\mu_E|(B_{x,r})}\,,
\]
exists and belongs to $\mathbf S^{n-1}$. As it turns out,
\[
\pa^*E\subset A\cap\pa^{\rm e}E\subset\spt\mu_E\subset A\cap\pa E\,,\qquad A\cap\cl(\pa^*E)=\spt\mu_E\,,
\]
and each inclusion may be strict.
 Federer's criterion, see for instance  \cite[Theorem 16.2]{maggiBOOK}, ensures that
\begin{equation}
  \label{federer theorem}
  \H^{n-1}((A\cap\pa^{\rm e}E)\setminus\pa^*E)=0\,.
\end{equation}
Moreover,
\begin{equation}\label{eq:partition}
A=_{\H^{n-1}} \big(E^{(0)}\cup E^{(1)}\cup \pa^{\rm e} E\big)\cap A=_{\H^{n-1}} \big(E^{(0)}\cup E^{(1)}\cup \pa^{*} E\big)\cap A\,,
\end{equation}
where the unions are \(\H^{n-1}\) disjoints and we have introduced  the notation  \(G=_{\H^{n-1}} F\) to mean  \(\H^{n-1}(G\Delta F)=0\) (and, similarly,  \(G\subset_{\H^{n-1}} F\) means that \(\H^{n-1}(F\setminus G)=0\)). We finally recall that De Giorgi's rectifiability theorem \cite[Theorem 15.5]{maggiBOOK} asserts that, for every \(x\in \pa^* E\),
\[
E^{x,r}\to \{y\in \R^n: \nu_E(x)\cdot y\le 0\}\qquad\textrm{in \(L^1_{\rm loc} (\R^n)\)}\,,
\]
and that $\mu_E=\nu_E\,\H^{n-1}\llcorner \pa^*E$ on Borel sets compactly contained in $A$ where,  given a Radon measure \(\mu\) and a Borel set \(G\), by  \(\mu\llcorner G\) we mean the measure given by \(\mu\llcorner G(F)=\mu(G\cap F)\). In particular
\begin{equation}
  \label{gauss green on E}
  \int_{E}\nabla\vphi(x)\,dx=\int_{\pa^*E}\vphi\,\nu_E\,d\H^{n-1}\,,\qquad\forall\vphi\in C^1_c(A)\,,
\end{equation}
see for instance  \cite[Section 15]{maggiBOOK}. In particular \(\mu_E(G)=0\) if \( \H^{n-1}(G)=0\).

\medskip

\noindent {\it Gauss-Green measure  and set operations.} It is well-known that, if $E$ and $F$ are of locally finite perimeter in $A$, then  $E\cap F$, $E\cup F$, $E\setminus F$ and $E\Delta F$ are sets of locally finite perimeter in $A$. Since the construction of competitors used in testing minimality inequalities often involves a combination of these set operations, being able to describe the corresponding behavior of Gauss--Green measures turns out to be extremely convenient. Recalling that $\nu_E(x)=\pm\nu_F(x)$ at $\H^{n-1}$-a.e. $x\in\pa^*E\cap\pa^*F$, setting  $\{\nu_E=\nu_F\}$ for the sets of those $x\in\pa^*E\cap\pa^*F$ such that $\nu_E(x)=\nu_F(x)$, and defining similarly $\{\nu_E=-\nu_F\}$, one can prove that
\begin{eqnarray}\label{cap}
  \mu_{E\cap F}&=&\mu_E \llcorner (F^{(1)}\cap\pa^*E)+\mu_F \llcorner(E^{(1)}\cap\pa^*F)+\mu_E \llcorner\{\nu_E=\nu_F\}\,,
  \\\label{cup}
  \mu_{E\cup F}&=&\mu_E \llcorner (F^{(0)}\cap\pa^*E)+\mu_F \llcorner(E^{(0)}\cap\pa^*F)+\mu_E \llcorner\{\nu_E=\nu_F\}\,,
  \\\label{minus}
  \mu_{E\setminus F}&=&
  \mu_E \llcorner (F^{(0)}\cap\pa^*E)-\mu_F \llcorner(E^{(1)}\cap\pa^*F)+\mu_E \llcorner\{\nu_E=-\nu_F\}\,,
\end{eqnarray}
see \cite[Section 16.1]{maggiBOOK}.  Moreover, if $E\subset F$, then
\begin{equation}
  \label{subset}
  \mu_E=\mu_E\llcorner F^{(1)}+\mu_F\llcorner\{\nu_E=\nu_F\}=\mu_E\llcorner F^{(1)}+\mu_F\llcorner(\pa^*E\cap\pa^*F)\,.
\end{equation}

\medskip

\noindent {\it Reduced boundary and bi-Lipschitz transformations.} If $f:\R^n\to\R^n$ is a Lipschitz diffeomorphism with $\det(\nabla f)>0$ on $\R^n$, then by the area formula it follows that  $f(E)$ is a set of locally finite perimeter in $f(A)$, with $f(\pa^*E)=_{\H^{n-1}}\pa^*(f(E))$ and
\[
\nu_{f(E)}(f(x))=\frac{\cof(\nabla f(x))\nu_E(x)}{|\cof(\nabla f(x))\nu_E(x)|}\,,\qquad\mbox{for $\H^{n-1}$-a.e. $x\in\pa^*f(E)$}\,.
\]
(Recall that, if $L:\R^n\to\R^n$ is an invertible linear map, then
\[
\cof L= (\det L)\,(L^{-1})^*\,,
\]
where $L^*$ denotes the adjoint map to $L$.) Moreover, one has
\begin{equation}
\label{change of variables cofattore}
  \int_{f(G\cap\pa^*E)}\Psi(y,\nu_{f(E)}(y))\,d\H^{n-1}(y)=\int_{G\cap\pa^*E}\Psi\Big(f(x),\cof(\nabla f(x))\,\nu_{E}(x)\Big)\,d\H^{n-1}(x)\,,
\end{equation}
for every Borel measurable function $\Psi:A\times \R^{n}\to[0,\infty]$ which is one-homogeneous in the second variable and every $G\subset A$.

  \medskip

  \noindent {\it Traces of sets of finite perimeter.} Let $A$ be an open set in $\R^n$, let $H$ be an open half-space in $\R^n$, and let $E\subset H$ be a set of locally finite perimeter in $A$. Since $1_{E}\in BV(A'\cap H)$ for every open set $A'\cc A$, by \cite[Lemma 2.4, Theorem 2.10]{GiustiMinimalSurfacesBOOK} there exists a Borel set \(\tr_{\partial H}(E)\subset A\cap\partial H\) such that
   \begin{equation}\label{eq:trace}
   \int_E \Div T(x)dx =\int_{H\cap\pa^*E} T\cdot \nu_{E}\,d\mathcal H^{n-1}+\int_{\tr_{\partial H}(E) } T \cdot \nu_{H} \,d\mathcal H^{n-1}\qquad \forall\, T\in C_c^1(A,\R^n)\,,
   \end{equation}
   and with the property that, if $E_t=\{z\in\pa H:(z,t)\in E\cap A\}$ ($t>0$), then
   \begin{equation}\label{eq:traceapprox}
   \lim_{t\to 0^+}\H^{n-1}(K\cap(E_t\Delta\tr_{\pa H}(E)))=0\,,\qquad\mbox{for every $K\cc A$.}
   \end{equation}
   On taking into account that, by \eqref{subset},
   \begin{equation}
    \label{subset of H}
    \mu_E=\nu_E\,\H^{n-1}\llcorner(H\cap\pa^*E)-e_1\,\H^{n-1}\llcorner(\pa^*E\cap\pa H)\,,
   \end{equation}
   by comparing \eqref{gauss green on E}, \eqref{subset of H} and \eqref{eq:trace} we get
   \begin{equation}
   \label{traces equivalenza}
   \tr_{\partial H} (E)=_{\H^{n-1}}\partial ^* E \cap \partial H\,.
   \end{equation}
   We also notice that
   \begin{equation}\label{tracciacomplementare}
   \tr_{\pa H}(H\setminus  E)=_{\H^{n-1}}\pa H\setminus\tr_{\pa H}(E)\,.
   \end{equation}
   Finally, from  \cite[Theorem 2.11]{GiustiMinimalSurfacesBOOK}, we have that if $\{E_h\}_{h\in\N}$ and $E$ are sets of locally finite perimeter in $A$, then
   \begin{equation}
     \label{trace continuity}
     \left\{
     \begin{array}
       {l}
       \mbox{$E_h\to E$ in $L^1_{\rm loc}(A)$}\,,
       \\
       \mbox{$P(E_h;A\cap H)\to P(E;A\cap H)$ as $h\to\infty$}\,,
     \end{array}
     \right .\quad\Rightarrow\quad
     \mbox{$\tr_{\pa H}(E_h)\to\tr_{\pa H}(E)$ in $L^1_{\rm loc}(A\cap\pa H)$.}
   \end{equation}

\subsection{Basic remarks on almost-minimizers}\label{section definitions} Let us recall from Definition \ref{definition minimizers} that if $A$ and $H$ are an open set and an open half-space in $\R^n$, $\Phi$ is an elliptic integrand on $A\cap H$, $r_0\in(0,\infty]$, and $\Lambda\ge0$, then one says that $E$ is a {\it $(\La,r_0)$-minimizer of $\PHI$ in $(A,H)$} provided $E\subset H$, $E$ is a set of locally finite perimeter in $A$, and
\begin{equation}\label{inequality Lambda minimality}
\PHI(E;W\cap H)\le \PHI(F;W\cap H)+\La\,|E\Delta F|
\end{equation}
whenever $F\subset H$ and $E\Delta F\cc W$ for some open set $W\cc A$ with $\diam(W)<2r_0$. A \((0,\infty)\)-minimizer will be is simply called minimizer.

The following two simple  remarks concerning the behavior of  almost minimizers with respect to the scaling and set complement will be frequently used in the sequel:
 \begin{remark}[Minimality and set complement]\label{remark complement}
  {\rm If $E$ is a $(\La,r_0)$-minimizer of $\PHI$ in $(A,H)$, then $H\setminus E$ is a  $(\La,r_0)$-minimizer of $\widetilde{\PHI}$ in $(A,H)$, provided we set
  \[
  \widetilde{\Phi}(x,\nu)=\Phi(x,-\nu)\,.
  \]
  Of course, $\Phi\in\X(A\cap H,\lambda,\ell)$ if and only if $\widetilde\Phi\in\X(A\cap H,\lambda,\ell)$.}
\end{remark}

\begin{remark}[Minimality and scaling]\label{remark blow-up}
  {\rm Given $x\in\cl(A\cap H)$ and $r<r_0$ such that $B_{x,r}\cc A$, one notices that $E$ is a $(\La,r_0)$-minimizer of $\PHI$ in $(A,H)$ if and only if $E^{x,r}$ is a  $(\La\,r,r_0/r)$-minimizer of $\PHI^{x,r}$ in $(A^{x,r},H^{x,r})$, where we have set
  \[
  \Phi^{x,r}(y,\nu)=\Phi(x+r\,y,\nu)\,.
  \]
 Notice that $\Phi\in\X(A\cap H,\lambda,\ell)$ if and only if $\Phi^{x,r}\in\X((A\cap H)^{x,r},\lambda,r\,\ell)$.}
\end{remark}

 It is sometimes convenient to consider sets which satisfy the minimality inequality \eqref{inequality Lambda minimality} only with respect to inner or outer variations. Hence we also give the following definition, with $A$, $H$, $\Phi$, $r_0$ and $\La$ as above.

 \begin{definition}[Sub/superminimizer]
   One says that $E$ is a {\it $(\La,r_0)$-subminimizer of $\PHI$ in $(A,H)$} if $E\subset H$, $E$ is of locally finite perimeter in $A$, and inequality \eqref{inequality Lambda minimality} holds true whenever $F\subset E$ and $E\setminus F\cc W$ for some open set $W\cc A$ with $\diam(W)<2r_0$; and that $E$ is a {\it $(\La,r_0)$-superminimizer of $\PHI$ in $(A,H)$} if inequality \eqref{inequality Lambda minimality} holds true whenever $E\subset F\subset H$ and $F\setminus E\cc W$ for some open set $W\cc A$ with $\diam(W)<2r_0$. In analogy with Definition \ref{definition minimizers}, when \(E\) is a \((0,\infty)\)-sub/superminimizer one simply says that $E$ is a sub/superminimizer.
\end{definition}

\begin{remark}\label{subsupermin} It is clear that a $(\La,r_0)$-minimizer in \((A,H)\) is both a $(\La,r_0)$-superminimizer and a $(\La,r_0)$-subminimizer. The converse is also true. Indeed, using \eqref{cap} and \eqref{cup}, one easily verifies that for every sets \(E,F\subset H\) of locally finite perimeter in \(A\),
\begin{equation}\label{capcup}
\PHI(E\cap F, W\cap H)+\PHI(E\cup F, W\cap H)\le \PHI(E, W\cap H)+\PHI(F, W\cap H),
\end{equation}
wherever \(W\cc A\). Hence, if \(E\) is a both a \((\Lambda,r_0)\)-superminimizer and a \((\Lambda, r_0)\)-subminimizer and \(F\Delta E \cc W\cc A\), comparing \(E\) with \(E\cup F\) and \(E\cap F\) (which are immediately seen to be admissible) and using \eqref{capcup} we obtain
\[
\begin{split}
2\,\PHI(E, W\cap H)&\le \PHI(E\cap F, W\cap H)+\PHI(E\cup F, W\cap H)+\Lambda \big(|E\setminus F|+|F\setminus E|\big)\\
&\le \PHI(E, W\cap H)+\PHI(F, W\cap H)+\Lambda |E\Delta F|.
\end{split}
\]
\end{remark}

The following `` transfer of sub/superminimality property''   will  be useful in section \ref{section singular set}.

\begin{proposition}\label{proposition subsuper}
  Let $A$ and be $H$ be an open set and an open half-space in $\R^n$, let $\Phi\in\X(A\cap H,\l,\ell)$, and let $E_1, E_2\subset H$ be sets of locally finite perimeter in $A$ with
  \begin{equation}\label{contenimento}
  \pa^*E_2\subset_{\H^{n-1}}\pa^*E_1\,.
  \end{equation}
  If $E_1$ is a $(\La,r_0)$-superminimizer of $\PHI$ in $(A,H)$ and $E_1\subset E_2$, then $E_2$ is a $(\La,r_0)$-superminimizer of $\PHI$ in $(A,H)$; see Figure \ref{fig superminimi}.
  \begin{figure}
    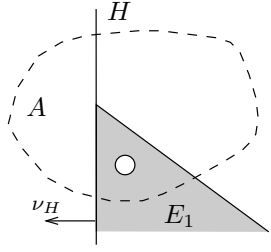\caption{{\small The situation in Proposition \ref{proposition subsuper}: if $E_1$ is a $(\La,r_0)$-superminimizer of $\PHI$ in $(A,H)$ (which is the case for the set $E_1$ in the picture if $\Lambda$ and $r_0$ are large and small enough respectively, and if the angle between the flat part of the boundary of $E_1$ and $\pa H$ is in a suitable range, cf. with Proposition \ref{proposition young}), then the set $E_2$ obtained by adding the interior of the missing disk to $E_1$ (larger set with smaller boundary) is still a superminimizer.}}\label{fig superminimi}
  \end{figure}
  Similarly, if $E_1$ is a $(\La,r_0)$-subminimizer of $\PHI$ in $(A,H)$ and $E_2\subset E_1$, then $E_2$ is a $(\La,r_0)$-subminimizer of $\PHI$ in $(A,H)$.
\end{proposition}

\begin{proof} We give details only in the case of superminimizers, the case of subminimizers being entirely analogous. Let \(F\) be such that  $E_2\subset F\subset H$ and $F\setminus E_2\cc W$ for some open set $W\cc A$ with $\diam(W)<2r_0$. Setting $G^+=G\cap H$ for every $G\subset\R^n$, we want to show that $\PHI(E_2;W^+)\le \PHI(F;W^+)+\La\,|F\setminus E_2|$. By $E_2\subset F$ and \eqref{subset}, this last inequality is equivalent to
  \begin{equation}
  \label{superH1}
  \PHI(E_2;F^{(1)}\cap W^+)\le \PHI(F;E_2^{(0)}\cap W^+)+\La\,|F\setminus E_2|\,.
  \end{equation}
 To prove \eqref{superH1}, we set
  \[
  F_*=(F\setminus E_2)\cup E_1\,,
  \]
  so that $E_1\subset F_*$ with $F_*\setminus E_1=F\setminus E_2\cc W$. By $(\La,r_0)$-superminimality of $E_1$, we have
  \begin{equation}
  \label{tardi}
  \PHI(E_1;W^+)\le \PHI(F_*;W^+)+\La\,|F\setminus E_2|\,.
  \end{equation}
  We now deduce \eqref{superH1} from \eqref{tardi} by repeatedly applying the formulas for Gauss-Green measures under set operations in conjunction with $E_1\subset E_2$ and \eqref{contenimento}. We begin by noticing that, by \eqref{federer theorem}, \eqref{subset} and \eqref{contenimento} we have
  \begin{eqnarray}\label{molto}
  \mu_{E_1}=\mu_{E_2}+\mu_{E_1}\llcorner E_2^{(1)}\,,\qquad\mu_{E_2}=\mu_{E_2}\llcorner \pa^*F+\mu_{E_2}\llcorner F^{(1)}\,.
  \end{eqnarray}
  By \eqref{eq:partition} and \eqref{molto} we find
  \begin{eqnarray}\nonumber
  \PHI(E_1;W^+)&=&\PHI(E_2;W^+)+\PHI(E_1;E_2^{(1)}\cap W^+)
  \\\label{tardi2}
  &=&\PHI(E_2;F^{(1)}\cap W^+)+\PHI(F;\pa^*E_2\cap W^+)+\PHI(E_1;E_2^{(1)}\cap W^+)\,.
  \end{eqnarray}
  Since $\nu_{E_1}=-\nu_{F\setminus E_2}$ $\H^{n-1}$-a.e. on $\pa^*E_1\cap\pa^*(F\setminus E_2)$ (due to the fact that $E_1\subset E_2\subset F$), by applying \eqref{cup} to $F_*$ we find that
  \begin{eqnarray}\label{tardi3}
  \PHI(F_*;W^+)=\PHI(F\setminus E_2;E_1^{(0)}\cap W^+)+\PHI(E_1;(F\setminus E_2)^{(0)}\cap W^+)\,.
  \end{eqnarray}
  We start noticing that
  \begin{equation}
    \label{tardi4}
      \PHI(F\setminus E_2;E_1^{(0)}\cap W^+)=\PHI(F;E_2^{(0)}\cap W^+)\,.
  \end{equation}
  Indeed,  by \eqref{molto}, \eqref{minus} gives $\mu_{F\setminus E_2}=\mu_F\llcorner E_2^{(0)}-\mu_{E_2}\llcorner F^{(1)}$, so that we just need to show that $P(E_2;E_1^{(0)})=0$: but this is obvious, since $E_1\subset E_2$ implies \( E_2^{(0)}\subset E_1^{(0)}\), and thus, by \eqref{contenimento},
  \[
  \H^{n-1}(E_1^{(0)}\cap\pa^*E_2)\le \H^{n-1}(E_1^{(0)}\cap\pa^*E_1)=0\,.
  \]
  This proves \eqref{tardi4}. Next, we notice that $(F\setminus E_2)^{(0)}=_{\H^{n-1}}F^{(0)}\cup E_2^{(1)}\cup(\pa^*F\cap\pa^*E_2)$, with $\H^{n-1}(F^{(0)}\cap\pa^*E_1)=0$ by \(F^{(0)}\subset E_1^{(0)}\) and \eqref{eq:partition}, so that
  \begin{eqnarray}\nonumber
      \PHI(E_1;(F\setminus E_2)^{(0)}\cap W^+)&=&\PHI(E_1;E_2^{(1)}\cap W^+)+\PHI(E_1;\pa^*F\cap\pa^*E_2\cap W^+)
      \\\nonumber
      &=&\PHI(E_1;E_2^{(1)}\cap W^+)+\PHI(F;\pa^*E_1\cap\pa^*E_2\cap W^+)
      \\
      &=&\PHI(E_1;E_2^{(1)}\cap W^+)+\PHI(F;\pa^*E_2\cap W^+)\,,\label{tardi5}
  \end{eqnarray}
  where in the last two identities we have first used that $\nu_F=\nu_{E_1}$ $\H^{n-1}$-a.e. on $\pa^*F\cap\pa^*E_1$, and then \eqref{contenimento}. By combining \eqref{tardi}, \eqref{tardi2}, \eqref{tardi3}, \eqref{tardi4}, and \eqref{tardi5} we thus find \eqref{superH1}.
  \end{proof}

\subsection{Anisotropic Young's law on half-spaces}\label{section young on half} It is well known that, if $A$ is an open set and $\Phi$ is an autonomous elliptic integrand, then
\begin{equation}
  \label{minimality of half-spaces}
  \PHI(\{x\cdot\nu<s\};A)\le\PHI(F;A)\,,
\end{equation}
whenever $\nu\in \mathbf S^{n-1}$, $s\in\R$, $\{x\cdot\nu<s\}=\{x\in\R^n:x\cdot\nu<s\}$ and $\{x\cdot\nu<s\}\Delta F\cc A$. If, in addition, $\Phi\in\X_*(\l)$ for some $\l>0$, then by Taylor's formula, \eqref{Phi nabla 1} and \eqref{elliptic}, one can find positive constants $\k_1$ and $\k_2$ depending on $\l$ only, such that
\begin{equation}
  \label{k1k2}
  \k_1\,\frac{|\nu_1-\nu_2|^2}2\le\Phi(\nu_2)-\Phi(\nu_1)-\nabla\Phi(\nu_1)\cdot(\nu_2-\nu_1)\le \k_2\,\frac{|\nu_1-\nu_2|^2}2\,,
\end{equation}
for every $\nu_1,\nu_2\in \mathbf S^{n-1}$. Correspondingly, one can strengthen \eqref{minimality of half-spaces} into
\begin{equation}
  \label{minimality of half-spaces lambda interna}
  \k_1\int_{A\cap\pa^*F}\frac{|\nu_F-\nu|^2}2\le\PHI(F;A)-\PHI(\{x\cdot\nu<s\};A)
\le\k_2\int_{A\cap\pa^*F}\frac{|\nu_F-\nu|^2}2\,,
\end{equation}
which holds true whenever $\nu\in \mathbf S^{n-1}$, $s\in\R$, and $\{x\cdot\nu<s\}\Delta F\cc A$. The following proposition provides similar assertions when a free boundary condition on a given hyperplane is considered.

\begin{proposition}[Anisotropic Young's law]\label{proposition young}
  Let $H=\{x_1>0\}$, $A$ be an open set, $\Phi\in\X_*(\l)$ for some $\l>0$, $\nu\in \mathbf S^{n-1}\setminus\{\pm e_1\}$ and $c\in\R$ be such that the set
  \[
      E=H\cap\{x\cdot\nu>c\}\,,
  \]
  satisfies $A\cap H\cap\pa E\ne\emptyset$, see Figure \ref{fig young}. Then, \(E\) is a superminimizer of \(\PHI\) in \((A, H)\) if and only if
  \begin{equation}
    \label{young law}
    \nabla\Phi(\nu)\cdot e_1\ge 0\,;
  \end{equation}
  similarly, \(E\) is a subminimizer of \(\PHI\) in \((A, H)\) if and only if  \(\nabla\Phi(\nu)\cdot e_1\le 0\). In particular, \(E\) is a minimizer of \(\PHI\) in \((A, H)\) if and only if \(\nabla \Phi(\nu) \cdot e_1=0\). Moreover, in this last case,
  \begin{eqnarray}\nonumber
  \k_1\int_{A\cap H\cap\pa^*F}\frac{|\nu_F-\nu|^2}2\,d\H^{n-1}&\le&\PHI(F;W\cap H)-\PHI(E;W\cap H)
  \\
  \label{minimality of half-spaces lambda neumann}
  &\le&\k_2\int_{A\cap H\cap\pa^*F}\frac{|\nu_F-\nu|^2}2\,d\H^{n-1}\,,\hspace{0.5cm}
  \end{eqnarray}
  whenever $F\subset H$ with $E\Delta F\cc W\cc  A$. Here, $\k_1$ and $\k_2$ are as in \eqref{k1k2}.
\end{proposition}

\begin{proof}
 {\it Step one}: We prove
  \begin{figure}
    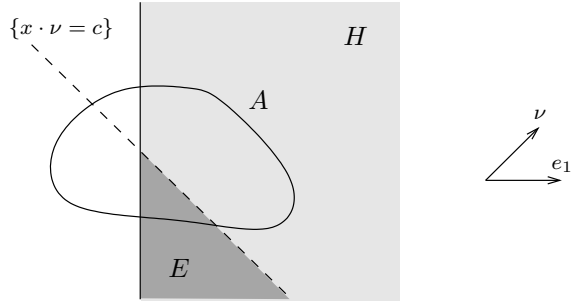\caption{{\small If $H=\{x_1>0\}$, then an half-space with outer unit normal $\nu$ is a superminimizer of $\PHI$ on $(\R^n,H)$ if and only if $\nabla\Phi(\nu)\cdot e_1=0$.}}\label{fig young}
  \end{figure}
 that  \eqref{young law} implies
 \begin{equation}
  \label{minimality of half-spaces neumann}
  \PHI(E;W\cap H)\le\PHI(F;W\cap H)\,,
  \end{equation}
  whenever $E\subset F\subset H$ with $F\setminus E \cc W\cc A$. Indeed, let \(W'\cc W\) be a set with smooth boundary such that \(F\setminus E\cc W'\) and
\begin{equation}\label{lasolita}
\H^{n-1}(\pa W'\cap \pa^* E)=\H^{n-1}(\pa W'\cap \pa^* F)=0.
\end{equation}
By applying the divergence theorem to the constant vector field $\nabla\Phi(\nu)$ on the sets of finite perimeter $E\cap W'$ and \(F\cap W'\), by  taking into account \eqref{lasolita}, and by noticing that \(\nu_E=-e_1\) \(\H^{n-1}\)-a.e. on \(\pa^* E\cap \pa H \) (and that an analogous relation holds true for $F$), we obtain
\begin{equation}\label{divteo}
\begin{split}
\int_{\pa^* E \cap W'\cap H}\nabla \Phi(\nu)\cdot \nu_E\,d\H^{n-1}-(\nabla \Phi(\nu)\cdot e_1) \H^{n-1}(\pa^* E\cap \pa H \cap W')&\\
=-\int_{E^{(1)}\cap \pa W'} \nabla \Phi (\nu)\cdot \nu_{W'}\,d \H^{n-1}&\,,
\\
\int_{\pa^* F \cap W'\cap H}\nabla \Phi(\nu)\cdot \nu_F\,d \H^{n-1}-(\nabla \Phi(\nu)\cdot e_1) \H^{n-1}(\pa^* F\cap \pa H \cap W')&\\
=-\int_{F^{(1)}\cap \pa W'} \nabla \Phi (\nu)\cdot \nu_{W'}\,d \H^{n-1}&\,.
\end{split}
\end{equation}
By \(F\setminus E\cc W'\), we have \(E^{(1)}\cap \partial W'=F^{(1)}\cap \partial W'\); moreover, the inclusions \(E\subset F\subset H\) and the definition of essential boundary imply that $\pa^* E\cap \pa H=_{\H^{n-1}}\pa^e  E\cap \pa H\subset \pa ^eF \cap \pa H=_{\H^{n-1}} \pa ^*F \cap \pa H$: thus, by subtracting the two identities in \eqref{divteo}, we find
\begin{eqnarray}\nonumber
  &&\int_{\pa^* E \cap W'\cap H}\nabla \Phi(\nu)\cdot \nu_E\,d\H^{n-1}-\int_{\pa^* F \cap W'\cap H}\nabla \Phi(\nu)\cdot \nu_F\,d\H^{n-1}
  \\\label{divteo cons}
  &&\hspace{2cm}=
  -(\nabla \Phi(\nu)\cdot e_1) \H^{n-1}\Big((\pa^* F\setminus\pa^*E)\cap \pa H \cap W'\Big)\,.
\end{eqnarray}
Since \(\nu_E=\nu\) on \(\pa^* E \cap H\) and \(\nabla \Phi(\nu)\cdot \nu=\Phi(\nu)\), the first integral on the left-hand side of \eqref{divteo cons} coincides with $\PHI(E;W'\cap H)$. Therefore, \eqref{divteo cons} gives
\begin{multline}\label{resto}
\PHI(E, W\cap H)+\int_{\pa^* F\cap H\cap W} \gamma_F\,d\H^{n-1}
\\
= \PHI(F, W\cap H)-(\nabla \Phi(\nu)\cdot e_1) \H^{n-1}\big((\pa^* F\setminus \pa^* E)\cap \pa H \cap W\big)\,,
\end{multline}
where we have defined $\g_F:\pa^*F\to\R$ by setting
\[
\gamma_F=\Phi(\nu_F)-\nabla\Phi(\nu)\cdot\nu_F=\Phi(\nu_F)-\Phi(\nu)+\nabla\Phi(\nu)\cdot(\nu_F-\nu)\,.
\]
By convexity of \(\Phi\),  \(\gamma_F\ge 0 \) on $\pa^*F$, and thus \eqref{young law} and \eqref{resto} imply \eqref{minimality of half-spaces neumann}. The case of subminimizers is treated analogously, and then the characterization of minimizers follows by Remark \ref{subsupermin}. Moreover, in this last case, by exploiting \eqref{capcup} as in Remark \ref{subsupermin}, and by using  \eqref{resto}, \eqref{cup}, and \eqref{cap}, we obtain
\begin{equation}\label{davinci}
\PHI(F, W\cap H)-\PHI(E, W\cap H)=\int_{\pa^* F\cap H\cap W} \gamma_F\,d\H^{n-1}\,.
\end{equation}
Since, by \eqref{k1k2}, $\k_1\,|\nu_F-\nu|^2\le 2\,\gamma_F(y)\le \k_2\,|\nu_F-\nu|^2$ on $\pa^*F$, we see that \eqref{davinci} implies \eqref{minimality of half-spaces lambda neumann}.

\medskip

\noindent {\it Step two}: We now prove that \eqref{minimality of half-spaces neumann} implies \eqref{young law}. Without loss of generality we shall assume that $0\in A$ and that $c=0$. In particular, by \eqref{minimality of half-spaces neumann}, there exists $r>0$ such that \eqref{minimality of half-spaces neumann} holds true for every $E\subset F\subset H$ with $F\setminus  E\cc B_r$. To exploit this property, we pick $\zeta\in C^1_c(B_r)$, \(\zeta \ge 0\),   $e\in \mathbf S^{n-1}$ with
\[
e\cdot e_1=0\,  \qquad e\cdot \nu \ge 0,
\]
  and we define the maps $f_t(x)=x+t\,T(x)$ for $T=\zeta\,e\in C^1_c(B_r;\R^n)$,  $t\ge 0 $ and $x\in\R^n$. Clearly there exists $\e_0>0$ such that $\{f_t\}_{t\in [0,\e_0)}$ is a one-parameter family of diffeomorphisms on $\R^n$ such that, if we set $F_t=f_t(E)$, then $E \subset F_t\subset H$ with $F_t\setminus  E\cc B_r$. In particular by \eqref{minimality of half-spaces neumann}
\begin{equation}
  \label{first variation}
  0\le \frac{d}{dt}\bigg|_{t=0^+}\PHI(f_t(E);B_r\cap H)\,.
\end{equation}
By \eqref{change of variables cofattore},
\begin{eqnarray*}
  &&\PHI(f_t(E);B_r\cap H)=\int_{B_r\cap\pa^*E}\Phi(\cof(\nabla f_t)\nu_E)\,d\H^{n-1}
  \\
  &=&\int_{B_r\cap\pa^*E}\Phi(\nu_E)+t\,\Big(\Phi(\nu_E)\,\Div\,T-\nabla\Phi(\nu_E)\cdot [(\nabla T)^*\nu_E]\Big)\,d\H^{n-1}+o(t)\,,
\end{eqnarray*}
where we have also used the fact that
\begin{equation}\label{sviluppini}
\begin{aligned}
 \nabla f_t&=\Id+t\,\nabla T\,, &\cof(\nabla f_t)=(J\,f_t)[(\nabla f_t)^{-1}\circ f_t]^*&\,,
 \\
 (\nabla f_t)^{-1}\circ f_t&=\Id-t\,\nabla T+O(t^2)\,,& Jf_t=1+t\,\Div T+O(t^2)&\,.
\end{aligned}
\end{equation}
By \eqref{first variation}, $\nabla T=e\otimes\nabla\zeta$, $\nu_E=\nu$, and $\Phi(\nu)=\nabla\Phi(\nu)\cdot\nu$, we thus find that
\begin{eqnarray}\nonumber
0&\le &\int_{B_r\cap\pa^*E}\Phi(\nu)\,(e\cdot\nabla\zeta)-(e\cdot\nu)\,(\nabla\Phi(\nu)\cdot\nabla\zeta)\,d\H^{n-1}
\\\label{younger}
&=&\Big((\nabla\Phi(\nu)\cdot\nu)\,e-(e\cdot\nu)\nabla\Phi(\nu)\Big)\cdot\int_{B_r\cap\pa^*E} \nabla\zeta\,d\H^{n-1}\,.
\end{eqnarray}
 We now recall that \(B_r\cap\pa^*E\)  is the intersection with $H$ of the $(n-1)$-dimensional disk in $\R^n$ of radius $r>0$, center at the origin, and perpendicular to $\nu$, and that $\zeta=0$ on $\pa B_r$. Therefore, if we denote by $\nu_*$ its unit co-normal vector along \(\{ x\cdot \nu =0\}\cap \pa H\),  then by divergence theorem
\begin{equation}
  \label{younger2}
  \int_{B_r\cap\pa^*E} \nabla\zeta\,d\H^{n-1}=\nu_*\,\int_{\{ x\cdot \nu =0\}\cap\pa H}\zeta\,d\H^{n-2}\,.
\end{equation}
By exploiting \eqref{younger} and \eqref{younger2}, and choosing $\zeta\in C_c(B_r)$ with $\int_{\{ x\cdot \nu =0\}\cap H}\zeta\,d\H^{n-2}> 0$, we find
\begin{equation}
  \label{younger3}
  \Big((\nabla\Phi(\nu)\cdot\nu)\,e-(e\cdot\nu)\nabla\Phi(\nu)\Big)\cdot\nu_*\ge 0\,,\qquad\forall\, e\in e_1^\perp\quad e\cdot \nu\ge 0\,.
\end{equation}
Since $\nu\ne\pm e_1$ we can find $\a$ and $\beta<0$ such that $\nu_*=\a\,\nu+\beta\,e_1$. If we plug this identity into \eqref{younger3}, as $\beta<0$, then we find
\begin{equation*}
 (e\cdot\nu)\,(\nabla\Phi(\nu)\cdot e_1)\ge 0\,,\qquad e\in e_1^\perp\quad e\cdot \nu\ge 0.
\end{equation*}
As \(\nu\ne \pm e_1\), there exists \(e\in e_1^\perp\) with \(\nu\cdot e>0\). Thus $\nabla\Phi(\nu)\cdot e_1\ge 0$, as desired.
\end{proof}

We conclude this section on anisotropic Young's laws with an elementary technical lemma that shall be frequently used in the sequel. Given $\nu\in \mathbf S^{n-1}$ with $|\nu\cdot e_1|<1$, we shall set
\begin{equation}
  \label{bf e1}
  {\bf e_1}(\nu)=\frac{\nu-(\nu\cdot e_1)\,e_1}{\sqrt{1-(\nu\cdot e_1)^2}}\,,
\end{equation}
for the normalized projection of $\nu$ on $e_1^\perp$. In the light of Proposition \ref{proposition young}, the following lemma says that if $\{x\cdot\nu<0\}\cap H$ is close to be a minimizer of $\PHI\in\X_*(\l)$ (in the sense that $\nabla\Phi(\nu)\cdot e_1$ is small), then there exists a minimizer of $\PHI$ of the form $\{x\cdot\nu_0<0\}\cap H$ with $\nu_0$ close to $\nu$, and with the normalized projections of $\nu_0$ and $\nu$ on $e_1^\perp$ being actually equal to each other.

\begin{lemma}\label{lemma 0}
  For every $\lambda\ge 1$, there exist positive constants $\e_0$ and $C_0$, depending on $\lambda$ only, with the following property. If  $\Phi\in\X_*(\lambda)$, $\nu\in \mathbf S^{n-1}$, and
  \begin{equation}\label{pranzozurigo}
  |\nabla\Phi(\nu)\cdot e_1|\le\e_0\,,
  \end{equation}
  then  there exists $\nu_0\in \mathbf S^{n-1}$ such that
  \begin{equation}
    \label{lemma 0 tesi}
      {\bf e_1}(\nu_0)={\bf e_1}(\nu)\,,\qquad \nabla\Phi(\nu_0)\cdot e_1=0\,,\qquad |\nu_0-\nu|\le C_0\,|\nabla\Phi(\nu)\cdot e_1|\,.
  \end{equation}
\end{lemma}

\begin{proof}
  We begin noticing that, for every $e\in \mathbf S^{n-1}$
  \begin{equation}
    \label{elliptic boundary eps}
    |e\cdot e_1|\le \sqrt{1-\frac1{\l^4}}+|\nabla\Phi(e)\cdot e_1|\,\l^3\,.
  \end{equation}
  Indeed, if ${\bf j}$ denotes the projection of $\R^n$ onto $\nabla\Phi(e)^{\perp}$, then by \(\Phi(e)=\nabla \Phi(e)\cdot e\), \eqref{Phi 1} and \eqref{Phi nabla 1},
  \begin{eqnarray*}
    &&1=\frac{|e\cdot\nabla\Phi(e)|^2}{|\nabla\Phi(e)|^2}+|{\bf j}\,e|^2\ge\frac1{\l^4}+|{\bf j}\,e\cdot {\bf j}\,e_1|^2=
    \frac1{\l^4}+|e\cdot{\bf j}\,e_1|^2
    \\
    &&\ge\frac1{\l^4}+\Big(|e\cdot e_1|-\frac{(\nabla\Phi(e)\cdot e)\,(\nabla\Phi(e)\cdot e_1)}{|\nabla\Phi(e)|^2}\Big)^2\,,
  \end{eqnarray*}
  that leads to \eqref{elliptic boundary eps} by $\Phi(e)\le\l$ and $|\nabla\Phi(e)|\ge1/\l$ (this last property follows by one homogeneity and \eqref{Phi 1}). We now notice that $\nu_0\in \mathbf S^{n-1}$ is such that ${\bf e_1}(\nu_0)={\bf e_1}(\nu)$ if and only if $\nu_0=\cos\a_0\,{\bf e_1}(\nu)-\sin\a_0 \,e_1$ for some $|\a_0|<\pi/2$. Let us thus set
   \[
   f(\a)=\nabla\Phi(\cos\a\,{\bf e_1}(\nu)-\sin\a e_1)\cdot e_1\qquad  |\a|<\pi/2.
   \]
  By the  one-homogeneity of \(\Phi\) and by \eqref{Phi 1}, we obtain
  \begin{eqnarray*}
  f(\pi/2)=\nabla\Phi(-e_1)\cdot e_1=-\Phi(-e_1)\le-\frac1\lambda\,,\qquad
  f(-\pi/2)=\nabla\Phi(e_1)\cdot e_1=\Phi(e_1)\ge\frac1\lambda\,,
  \end{eqnarray*}
  so that there exists $\a_0\in(-\pi/2,\pi/2)$ such that $f(\a_0)=0$; correspondingly, $\nu_0$ satisfies the first two identities in \eqref{lemma 0 test}. We now notice that, by \eqref{elliptic boundary eps},  by $\nabla\Phi(\nu_0)\cdot e_1=0$ and by \eqref{pranzozurigo}
   \begin{eqnarray*}
    |\nu_0\cdot e_1|\le\sqrt{1-\frac1{\l^4}}\,,\qquad |\nu\cdot e_1|\le\sqrt{1-\frac1{\l^4}}+\e_0\,\l^3\,.
  \end{eqnarray*}
  Hence, for every $\l\ge 1$ we can find $\eta(\l)\in(0,1)$ and $\e_0=\e_0(\l)$ such that
  \[
  \max\{|\nu_0\cdot e_1|,|\nu\cdot e_1|\}\le 1-\eta.
  \]
   Correspondingly, for some $\tau(\l)<\pi/2$, we find that $|\a_0|\le\tau$ and, if  $\a_1\in(-\pi/2,\pi/2)$ is such that $\nu=\cos\a_1\,{\bf e_1}(\nu)-\sin\a_1 \,e_1$, then $|\a_1|\le\tau$ too. Since, by zero-homogeneity of $\nabla\Phi$, $f(\a)=\nabla\Phi({\bf e_1}(\nu)-\tan\a\,e_1)\cdot e_1$ for every $|\a|<\pi/2$, by \eqref{elliptic} we conclude that
  \begin{eqnarray}\label{fprimo}
    f'(\a)=-\frac{e_1\cdot\nabla^2\Phi({\bf e_1}(\nu)-\tan\a\,e_1)\,e_1}{\cos^2\a}\le-\frac{1}{\lambda\,\cos^2\a\,|{\bf e_1}(\nu)-\tan\a\,e_1|}=-\frac1{\lambda\,\cos\a}\,,
  \end{eqnarray}
  for every $|\a|<\pi/2$. In particular, there exists $\bar\a$ between $\a_0$ and $\a_1$ such that
  \[
  |\nabla\Phi(\nu)\cdot e_1|=|f(\a_1)|=|f(\a_0)-f(\a_1)|\ge \frac{|\a_0-\a_1|}{\lambda\,\cos\bar{\a}}\ge \frac{|\a_0-\a_1|}{\lambda\,\cos\tau}\ge\frac{|\a_0-\a_1|}{C(\lambda)}\,.
  \]
  Since \(|\nu-\nu_0|\le 2 |\alpha_0-\alpha_1|\), the above equation concludes the proof of  the lemma.
\end{proof}

\subsection{Density estimates}\label{section density estimates} Density estimates for almost-minimizers are proved by a classical argument. The only significant difference is that when deducing lower perimeter estimates from upper volume estimates, a whole family of relative isoperimetric inequalities has to be used in place of the sole relative isoperimetric inequality on a ball, see \eqref{gammasigma} below.

\begin{lemma}\label{lemma stime di densita}
  For every $\l\ge 1$ there exist constants \(c_1=c_1(n,\l)\in (0,1)\) and \(C_1=C_1(n,\l)\) with the following property. If $A$ is an open set, $H$ is an open half-space, $\Phi\in\X(A\cap H,\l,\ell)$, and $E$ is a $(\La,r_0)$-minimizer of $\PHI$ in $(A,H)$, then
  \begin{equation}
  \label{stime densita perimetro upper}
  P(E;B_{x,r})=P(E;B_{x,r}\cap H)+P(E;B_{x,r}\cap\pa H)\le C_1\, r^{n-1}\,,
  \end{equation}
  for every $x\in A\cap\cl(H)$ and $r<\min\big\{r_0,\dist(x,\pa A),1/2\l \L\big\}$,   and
  \begin{align}
  \label{stime densita volume lower}
  |E\cap B_{x,r}|&\ge c_1|B_{x,r}\cap H|\,,&&\forall x\in \cl(H)\cap\spt\mu_E\,,
  \\
  \label{stime densita volume upper}
  |E\cap B_{x,r}|&\le (1-c_1)|B_{x,r}\cap H|\,&&\forall x\in \cl(H)\cap\spt\mu_{H\setminus E}\,,
  \\
  \label{stime densita perimetro lower}
  P(E;B_{x,r}\cap H)&\ge c_1\, r^{n-1}\,,&&\forall x\in A\cap\cl(H\cap\spt\mu_E)\,,
  \end{align}
for every $r<\{r_0,\dist(x,\pa A),1/2\l \L\big\}$.
\end{lemma}

Recall that, in our notation, if $E\subset\R^n$ is of locally finite perimeter in the open set $A$, then $\spt\mu_E$ and $\pa^*E$ are automatically defined as subsets of $A$.

\begin{proof} {\it Step one}: Without loss of generality, we shall assume that $H=\{x_1>0\}$, and set $G^+=G\cap H$ for every $G\subset\R^n$. For $\s\in(1/2,1)$, we   consider the relative isoperimetric problem in the truncated ball $B_{t\,e_1}^+$ ($t\ge 0$) with volume fraction $\s$, and set
\begin{equation}\label{gammasigma}
\g(\s)=\inf_{t\ge0}\inf\Big\{\frac{P(F;B_{t\,e_1}^+)}{|F|^{(n-1)/n}}:|F|\le \s|B_{t\,e_1}^+|\Big\}\,.
\end{equation}
Then  $\g(\s)>0$ for every $\s\in(1/2,1)$

\medskip

\noindent {\it Step two}: Given $x\in A$ we set $r_x=\min\{r_0,\dist(x,\pa A),1/2\l \L\}$, and define
 \[
 m_x(r)=|E\cap B_{x,r}|=|E\cap B_{x,r}\cap H|
 \]
  for every $r\in (0,r_x)$, so that $m_x$ is absolutely continuous on $(0,r_x)$ (and strictly positive  if we also have $x\in\spt\,\mu_E$), with $m_x'(r)=\H^{n-1}(E\cap\pa B_{x,r})$ for a.e. $r\in(0,r_x)$. If $x\in A\cap\cl(H)$ and $r\in(0,r_x)$, then the set $F_r=E\setminus B_{x,r}\subset H$ satisfies $E\Delta F_r\cc B_{x,r_*}\cc A$ for some $r_*\in(r,r_x)$.  Since  $r_*<r_0$, the set $F_r$ is admissible in \eqref{inequality Lambda minimality}, which gives
  \[
  \PHI(E;B_{x,r_*}^+)\le\PHI(F_r;B_{x,r_*}^+)+\La\,m_x(r)\,.
  \]
  We combine this last inequality and \eqref{Phi 1} with the remark that, by \eqref{minus},
  \[
  \PHI(F_r;B_{x,r_*}^+)=\PHI(F_r;B_{x,r_*}^+\setminus B_{x,r}^+)+\int_{E\cap\pa B_{x,r}}\Phi(y,\nu_{B_{x,r}}(y))\,d\H^{n-1}(y)\,,\quad\mbox{for a.e. $r>0$}\,,
  \]
  in order to get
  \begin{equation}
  \label{density estimates proof}
  P(E;B_{x,r}^+)\le \l \PHI(E;B_{x,r}^+)\le \l^2\,m_x'(r)+\l \L m_x(r)\,,\qquad\forall x\in A\cap\cl(H)\,,r<r_x\,.
  \end{equation}
  Since $m_x'(r)\le n\om_nr^{n-1}$, $m_x(r)\le \om_n\,r_x\,r^{n-1}$, and $2\l \La\,r_x\le 1/\l$, this proves that
  \begin{equation}
  \label{stime densita perimetro upperx}
  P(E;B_{x,r}^+)\le C\,r^{n-1}\,,\qquad\forall x\in A\cap\cl(H)\,,r<r_x\,,
  \end{equation}
  where $C=C(n,\l)$. Now, by the divergence theorem (see \cite[Proposition 19.22]{maggiBOOK}), we have
  \begin{equation}
  \label{caviglia}
  P(E\cap B_{x,r};\pa H)\le P(E\cap B_{x,r};H)\,,\qquad\forall x\in\R^n\,,r>0\,.
  \end{equation}
  At the same time, by \eqref{cap} one has
  \[
  P(E\cap B_{x,r};H)=P(E;B_{x,r}^+)+m_x'(r) \qquad \textrm{for a.e. $r>0$,}
  \]
  while $P(B_{x,r};\pa H)=0$ gives $P(E\cap B_{x,r};\pa H)=P(E;B_{x,r}\cap\pa H)$. By combining these facts with \eqref{stime densita perimetro upperx} and \eqref{caviglia} we obtain  \eqref{stime densita perimetro upper}. Moreover, for every $x\in \R^n$ and for a.e. $r>0$,
\begin{eqnarray*}
  P(E\cap B_{x,r})&=&P(E\cap B_{x,r};H)+P(E\cap B_{x,r};\pa H)\le 2\,P(E\cap B_{x,r};H)
  \\
  &=&2\,\Big\{P(E; B_{x,r}^+)+\H^{n-1}(E\cap\pa B_{x,r})\Big\}\,.
\end{eqnarray*}
  This last inequality, together with \eqref{density estimates proof} and the isoperimetric inequality, gives
\begin{equation}
  \label{precisi}
  n\om_n^{1/n}m_x(r)^{(n-1)/n}\le  2\,(1+\l^2)\,m_x'(r)+2\,\l\,\L\,m_x(r)\,,
\end{equation}
for every $x\in A\cap\cl(H)$ and for a.e. $r<r_x$. Since  $2\l \La\,r_x\le1$, for every $r<r_x$ we get
\[
2\,\l\,\L\,m_x(r)\le 2\,\l\,\L\, m_x(r_x)^{1/n}\,m_x(r)^{(n-1)/n}\le 2\,\l\,\L\,\omega_n^{1/n} r_x \,m_x(r)^{(n-1)/n}\le \om_n^{1/n}m_x(r)^{(n-1)/n}\,,
\]
so that \eqref{precisi} gives, for every $x\in A\cap\cl(H)$ and for a.e. $r<r_x$,
\begin{equation}
  \label{precisi2}
  (n-1)\om_n^{1/n}\,m_x(r)^{(n-1)/n}\le  2\,(1+\l^2)\,m_x'(r)\,.
\end{equation}
If we now assume that $x\in\spt\,\mu_E$, then $m_x(r)>0$ for every $r>0$, and thus we can divide by $m_x(r)^{(n-1)/n}$ in \eqref{precisi2}. By integrating the resulting differential inequality, we get
\[
|E\cap B_{x,r}|\ge \om_n\,\Big(\frac{n-1}{2n(1+\l^2)}\Big)^n\,r^n\ge c(n,\lambda)|B_{x,r}\cap H| \,,\qquad\forall x\in \cl(H)\cap\spt\mu_E\,,r<r_x\,.
\]
We have thus found a constant $c=c(n,\l)\in(0,1)$ such that \eqref{stime densita volume lower} holds true with $c$ in place of $c_1$ (the final value of $c_1$ will be smaller than this). We prove \eqref{stime densita volume upper} (again, with $c$ in place of $c_1$) by repeating the above argument with $H\setminus E$ in place of $E$; see Remark \ref{remark complement}. Since $H\cap\spt\mu_E=H\cap\spt\mu_{H\setminus E}$, we notice that both \eqref{stime densita volume lower} and \eqref{stime densita volume upper} hold true for every $x\in A\cap\cl( H\cap\spt\mu_E)$ and $r<r_x$: correspondingly, by definition of $\gamma(\s)$, see \eqref{gammasigma}, we find
\begin{eqnarray*}
  P(E;B_{x,r}^+)&=&P(E\cap B_{x,r}^+;B_{x,r}^+)\ge \g(1-c)|E\cap B_{x,r}^+|^{(n-1)/n}
  \\
  &\ge&\g(1-c)\big(|B_{x/r,1}^+|\,c\big)^{(n-1)/n}\,r^{n-1}\,,
\end{eqnarray*}
for every $x\in A\cap\cl( H\cap\spt\mu_E)$ and every $r<r_x$. Since $|B_{x/r,1}^+|\ge |B_{0,1}^+|=\om_n/2$, this proves \eqref{stime densita perimetro lower} with $c_1=\g(1-c)(c\,\om_n/2)^{(n-1)/n}$.
\end{proof}

\subsection{Compactness theorem}\label{section compactness} The density estimates of Lemma \ref{lemma stime di densita} lead to the following  compactness theorem for almost-minimizers.

\begin{theorem}\label{thm compactness minimizers}
  Let $\l\ge1$, $A$ an open set in $\R^n$ and $H$ an open half-space in $\R^n$. For every $h\in\N$, let $\ell_h\ge0$, $\La_h\ge0$, $r_h\in(0,\infty]$ be such that
  \[
      \lim_{h\to\infty}\ell_h=\ell<\infty\,,\qquad\lim_{h\to\infty}r_h=r_0>0\,,\qquad \lim_{h\to\infty}\La_h=\La_0<\infty\,.
  \]
  If, for every $h\in\N$, $\Phi_h\in\X(A\cap H,\l,\ell_h)$ and $E_h$ is a $(\La_h,r_h)$-minimizer of $\PHI_h$ in $(A,H)$, then there exist $\Phi\in\X(A\cap H,\l,\ell)$ and a set $E$ of locally finite perimeter in $A$ such that, up to extracting a subsequence,
  \begin{eqnarray}\label{compactness 1}
   \mbox{$E_h\to E$ in $L^1_{\rm loc}(A)$}\,,
  \end{eqnarray}
  where $E$ is a $(\La,r_0)$-minimizer of $\PHI$ in $(A,H)$, and where
  \begin{equation}
    \label{compactness 0}
    \mu_{E_h}\weak\mu_E\,,\quad|\mu_{E_h}|\llcorner H\weak|\mu_E|\llcorner H\,,\quad |\mu_{E_h}|\weak|\mu_E|\,,\quad\mbox{as Radon measures in $A$}\,.
  \end{equation}
  Moreover,
  \begin{eqnarray}
    \label{compactness 4}
    &&\mbox{$\tr_{\pa H}(E_h)\to \tr_{\pa H}(E)$ in $L^1_{\rm loc}(A\cap\pa H)$}\,,
  \end{eqnarray}
  while
  \begin{equation}
    \label{kuratowski2}
    \mbox{for every $x\in\spt\mu_E$ there exists $x_h\in \spt\,\mu_{E_h}$ such that}\,\,\lim_{h\to\infty}x_h=x\,,
  \end{equation}
  and
  \begin{equation}
    \label{kuratowski1}
    \left\{\begin{array}{l l}
    x_h\in\cl(H\cap\spt\,\mu_{E_h})\,,&\forall h\in\N\,,
    \\
    \lim_{h\to\infty}x_h=x\,,&
    \end{array}\right .\quad\Rightarrow\quad x\in \cl(H\cap\spt\,\mu_E)\,.
  \end{equation}
\end{theorem}


\begin{proof}  {\it Step one}: We assume without loss of generality that $H=\{x_1>0\}$, and set $G^+=G\cap H$ for every $G\subset\R^n$. Up to extracting subsequence, we may assume that
\[
  \frac{r_0}2< r_h<2\,r_0\,,\qquad \La_h<2\L_0\,,\qquad\forall h\in\N\,.
\]
In particular, if $x\in A\cap\cl(H)$ and $s_x=\min\big\{r_0/2,\dist(x,\pa A),1/4\l \La_0\big\}$, then we can apply the upper density estimate \eqref{stime densita perimetro upper} to each set $E_h$ at the point $x$ and at any scale $r<s_x$, to find that
\begin{equation*}
  \label{nothingiseasy 2}
P(E_h;B_{x,r})\le C_1\,r^{n-1}\,,\qquad\forall x\in A\cap\cl(H)\,,r<s_x\,,
\end{equation*}
where $C_1=C_1(n,\l)$. Since $E\subset H$, a simple covering argument gives
\begin{equation*}
  \label{nothingiseasy 3}
  \sup_{h\in\N}P(E_h;A_0)<\infty\,,\qquad\mbox{for every open set $A_0\cc A$.}
\end{equation*}
Hence there exists a set $E$ of locally finite perimeter in $A$ such that, up to extracting subsequences, $E_h\to E$ in $L^1_{\rm loc}(A)$, and
\begin{equation}
  \label{nothingiseasy 4}
  \mu_{E_h}\weak  \mu_E\,,\qquad\mbox{as Radon measures in $A$}\,.
\end{equation}
This proves \eqref{compactness 1} and the first part of \eqref{compactness 0},

\medskip


\noindent {\it Step two}: By the Ascoli-Arzel\'a theorem, there exists $\Phi\in\X(A^+,\l,\ell)$ such that, up to extracting a subsequence, $\Phi_h\to\Phi$ in $C^0(\cl(A^+)\times \mathbf S^{n-1})$ with $\Phi_h(x,\cdot)\to\Phi(x,\cdot)$ in $C^2(\mathbf S^{n-1})$ uniformly on $x\in\cl(A^+)$.
By exploiting the uniform convergence of $\Phi_h$ to $\Phi$ on $\cl(A^+)\times \mathbf S^{n-1}$, together with the lower bound in \eqref{Phi 1}, we see that for every $\e>0$ there exists $h_\e$ such that if $h\ge h_\e$, then
\begin{equation}
  \label{Phiepsh}
(1+\e)\Phi(x,\nu)\ge \Phi_h(x,\nu)\ge(1-\e)\Phi(x,\nu)\,,\qquad\forall (x,\nu)\in \cl(A^+)\times \mathbf S^{n-1}\,.
\end{equation}
By \eqref{nothingiseasy 4} and by Reshetnyak lower semicontinuity theorem \cite[Theorem 2.38]{AFP}, we thus find that, for every open set $U\subset A$,
\begin{equation}\label{caviglia 3}
\liminf_{h\to\infty}\PHI_h(E_h;U)\ge\PHI(E;U)\,.
\end{equation}

\medskip

\noindent {\it Step three}: In order to prove that $E$ is a $(\L,r_0)$-minimizer of $\PHI$ in $(A,H)$, we need to show that
\begin{equation}
  \label{compactness theorem proof 1}
  \PHI(E;W^+)\le \PHI(F;W^+)+\La\,|E\Delta F|\,,
\end{equation}
whenever $W$ is open and $F\subset H$ is such that
\begin{equation}
  \label{comp condo}
  E\Delta F\cc W\cc A\,,\qquad \diam(W)<2r_0\,.
\end{equation}
Indeed, let $F$ and $W$ satisfy \eqref{comp condo}. Clearly we can find an open set $W'$ with
\begin{gather}
  \label{Rstar1}
  E\Delta F\cc W\cc  W'\,,\qquad  \H^{n-1}(\pa W'\cap(\pa^* E\cup\pa^* F))=0\,,
  \\
  \label{Rstar2}
  \lim_{h\to\infty}\H^{n-1}\Big((E\Delta E_h)\cap\pa W'\Big)=0\,,
\end{gather}
and $\diam(W')<2r_0$. Hence, there exists $h_*\in\N$ such that $\diam(W')<2\,r_h$ for every $h\ge h_*$; in particular, we can find an open set $W''\cc A$, such that, if we set
\begin{equation}
  \label{caviglia 1}
  F_h=(F\cap W')\cup (E_h\setminus W')\,,
\end{equation}
then $F_h\subset H$, $E_h\Delta F_h\cc W''\cc  A$, and $\diam(W'')<2r_h$ for every $h\ge h_*$. We can exploit the fact that $E_h$ is a $(\La_h,r_h)$-minimizer of $\PHI$ in $(A,H)$ to find
\[
  \PHI_h(E_h;(W'')^+)\le \PHI_h(F_h;(W'')^+)+\La_h\,|E_h\Delta F_h|\,,
\]
for every $h\ge h_*$. By taking into account \eqref{Rstar1}, \eqref{cap}, \eqref{cup} and \eqref{minus} we find that
\begin{equation}
  \label{compactness theorem proof 2}
  \PHI_h(E_h;(W')^+)\le \PHI_h(F;(W')^+)+\e_h+\La_h\,\Big|(E_h\Delta F)\cap W'\Big|\,,
\end{equation}
for every $h\ge h_*$, where we have applied \eqref{Phi 1} and set
\[
\e_h=\l\,\H^{n-1}\Big((E\Delta E_h)\cap\pa W'\Big)\,,\qquad h\in\N\,.
\]
By letting $h\to\infty$ in \eqref{compactness theorem proof 2}, by \eqref{Rstar2}, and since $E_h\to E$ in $L^1_{\rm loc}(A)$, we conclude
\begin{equation}
  \label{nothingiseasy 5}
  \limsup_{h\to\infty}\PHI_h(E_h;(W')^+)\le \PHI(F;(W')^+)+\La\,|E\Delta F|\,.
\end{equation}
Since $(W')^+\cc A$, by \eqref{caviglia 3} we get $\PHI(E;(W')^+)\le \PHI(F;(W')^+)+\La\,|E\Delta F|$, that is exactly \eqref{compactness theorem proof 1}.

\medskip

\noindent {\it Step four}: If we choose $F=E$ in the argument of step three, then the combination of \eqref{caviglia 3} and \eqref{nothingiseasy 5} gives
\[  \lim_{h\to\infty}\PHI_h(E_h;B^+_{x,r})=\PHI(E;B^+_{x,r})\,,
\]
for every $x\in A\cap\cl(H)$ and for a.e. $r<r_x=\min\{r_0,\dist(x,\pa A)\}$. By taking \eqref{Phiepsh} into account, we thus find that
\[
  \lim_{h\to\infty}\PHI(E_h;B^+_{x,r})=\PHI(E;B^+_{x,r})\,,
\]
for every $x\in A\cap\cl(H)$ and for a.e. $r<r_x$. By \eqref{nothingiseasy 4} and by the strict convexity of $\Phi(x,\cdot)$ (in the sense of \eqref{elliptic}), we can apply a classical result of Reshetnyak, see, e.g. \cite[Theorem 1, section 3.4]{GMSbook2}, to find that
\begin{equation}
  \label{nothingiseasy 6tris}
  \lim_{h\to\infty}P(E_h;B^+_{x,r})=P(E;B^+_{x,r})\,,
\end{equation}
for every $x\in A\cap\cl(H)$ and for a.e. $r<r_x$. By \eqref{nothingiseasy 6tris} , \eqref{trace continuity} and a covering argument we deduce the validity of \eqref{compactness 4}. Let now \(\mu\) be a weak*-cluster point of the family of measures \(|\mu_{E_h}|\llcorner H\). By  \eqref{nothingiseasy 6tris} and by the Lebesgue-Besicovitch differentiation theorem, we find \(\mu=|\mu_{E}|\llcorner H\), hence
\begin{equation}
\label{nothingiseasy 6quater}
|\mu_{E_h}|\llcorner H\weak|\mu_E|\llcorner H\,,\qquad\mbox{as Radon measures in $A$}\,,
\end{equation}
 which proves  the second statement in \eqref{compactness 0}. We finally complete the proof of \eqref{compactness 0}: given a compact set $K\subset A$, by $E\subset H$ we have $|\mu_E|(K)=|\mu_E|\llcorner H(K)+\H^{n-1}(K\cap\pa H\cap\pa^*E)$, where
\[
|\mu_E|\llcorner H(K)\ge\limsup_{h\to\infty}|\mu_{E_h}|\llcorner H(K)\,,\quad \H^{n-1}(K\cap\pa H\cap\pa^*E)=
\lim_{h\to\infty}\H^{n-1}(K\cap\pa H\cap\pa^*E_h)\,,
\]
by \eqref{nothingiseasy 6quater} and  by \eqref{compactness 4} respectively. This shows that $|\mu_E|(K)\ge\limsup_{h\to\infty}|\mu_{E_h}|(K)$ for every compact set $K\subset A$. This last fact, combined with \eqref{nothingiseasy 4}, implies the last statement in  \eqref{compactness 0}, see for instance \cite[Proposition 4.26]{maggiBOOK}.

\medskip

\noindent {\it Step four}: We finally prove \eqref{kuratowski2} and \eqref{kuratowski1}.
The validity of \eqref{kuratowski2} is a standard consequence of \eqref{nothingiseasy 4}: indeed, if \eqref{kuratowski2} fails, then there exists $\e>0$ such that, up to extracting subsequences, $B(x,\e)\cap\spt\mu_{E_h}=\emptyset$ for every $h\in\N$; but then, by \eqref{nothingiseasy 4} we get
 \[
 |\mu_E|(B(x,\e))\le\liminf_{h\to\infty}|\mu_{E_h}|(B(x,\e))=0,
 \]
against $x\in \spt\mu_E$. The validity of \eqref{kuratowski1} follows, again via a standard argument, by the lower density estimate \eqref{stime densita perimetro lower}: indeed, if $x_h\in K\cap\cl(H\cap\spt\mu_{E_h})$, then by \eqref{stime densita perimetro lower} we can find $r>0$ such that $|\mu_{E_h}|(B_{x_h,s})\ge c_1\, s^{n-1}$ for every $s<r$. In particular, since $B(x_h,r/2)\subset B(x,r)$ for $h$ large enough, by \eqref{compactness 0} one gets
\[
|\mu_E|(\cl(B_{x,r}))\ge\limsup_{h\to\infty}|\mu_{E_h}|(\cl(B_{x,r}))\ge c_1\, s^{n-1}\,,\qquad\forall s<\frac{r}2\,,
\]
so that, necessarily, $x\in \spt\mu_E$, and \eqref{kuratowski1} is proved.
\end{proof}

\subsection{Contact sets of almost-minimizers}\label{section contact sets} In this section we establish some ``weak'' regularity properties of the contact set \(\tr_{\pa H}(E)\) of a \((\Lambda, r_0)\)-minimizer $E$. In Lemma \ref{lem:finiteperimeter} we show that $\tr_{\pa H}(E)$ is of locally finite perimeter in $A\cap\pa H$. In Lemma \ref{lem:densitycontact} we prove lower density estimates for the perimeter of $\tr_{\pa H}(E)$ in $\pa H$ by means of a strong maximum principle discussed in Lemma \ref{lemma strong max}. Finally, in Lemma \ref{lem:normalization}, we set some normalization conventions on almost-minimizers to be used in the rest of the paper.

The following notation shall be used thorough this section. We decompose $\R^n$ as $\R\times\R^{n-1}$, denote by $\h:\R^n\to\R^{n-1}$ the corresponding projection of $\R^n$ onto $\R^{n-1}$, so that \(x=(x_1,\h x)\) for every \(x\in \R^n\), and define the \emph{vertical} disk and the \emph{vertical} cylinder centered at \(0\) as
\begin{equation}\label{disco verticale}
   \D^{\rm v}_r=\Big\{z\in \partial H: |z|<r\Big\}\,, \qquad \C^{\rm v}_r=\Big\{x\in\R^n:x_1\in(0,r)\,,|\h x|<r\Big\}=(0,r)\times \D^{\rm v}_r\,.
\end{equation}
With the usual abuse of notation (see section \ref{section notation}), we denote by \(\pa \D^{\rm v}_r\) the boundary of \(\D^{\rm v}_r\) inside \(\pa H\), i.e. we set \(\pa \D^{\rm v}_r=\big\{z\in \partial H: |z|=r\big\}\).

\begin{lemma}[Contact sets are of locally finite perimeter]\label{lem:finiteperimeter}
  For every \(\l \ge1\) there exists a constant  \(C=C(n,\l)\) with the following property. If $A$ is an open set, $H=\{x_1>0\}$, $\Phi\in\X(A\cap H,\l,\ell)$, and $E$ is a $(\La,r_0)$-minimizer of $\PHI$ in $(A,H)$, then
\[   P(\tr_{\partial H}(E); B_{x,r}\cap\partial H)\le C\, r^{n-2}\,,
  \]for every $x\in \partial H$ and $r<\min\big\{r_0,\dist(x,\pa A),1/2\l \L,1/\ell\big\}\big /4$. In particular \(\tr_{\partial H}(E)\) is a set of locally finite perimeter in \(A\cap\partial H\).
\end{lemma}

\begin{proof}
Up to replace \(E\) and $\Phi$ with \(E^{x,r}\) and \(\Phi^{x,r}\) respectively, we can directly assume that $\Phi\in\X(B_4\cap H,\l,\ell  )$, \(\ell \le 1\),  $E$ is a $(1/8\l ,4)$-minimizer of $\PHI$ in $(B_4,H)$, and prove that
\begin{equation}\label{eq:perimetertrace3}
P(\tr_{\partial H}(E); \D_{1}^{\rm v})\le C\,,
\end{equation}
for a constant $C=C(n,\l)$. Given $s\in(0,1/2)$, let $\vphi_s\in C_c^1((0,2);[0,s])$ be such that $\vphi_s=s$ on $(0,1)$ and $|\vphi_s'|\le 3\,s$ on $(0,2)$, set
\[
G_s=\Big\{x\in\C^{\rm v}_2:x_1\le\vphi_s(|\h x|)\Big\}\,,
\]
and consider the bi-Lipschitz map $f_s:H\setminus G_s\to H$ defined as
\begin{eqnarray*}
&&f_s(x)=\bigg( 1-\Big(\frac{1-x_1}{1-\vphi_s(|\h x|)}\Big),\h x\bigg)\,,\qquad x\in[\D^{\rm v}_2\times(0,1)]\setminus G_s\,,
\\
&&f_s(x)=x\,,\qquad\hspace{4.4cm} x\in H\setminus[\D^{\rm v}_2\times(0,1)]\,,
\end{eqnarray*}
see Figure \ref{fig Gh}. Notice that $f_s(\C^{\rm v}_2\setminus G_s)=\C^{\rm v}_2$, with
\begin{figure}
  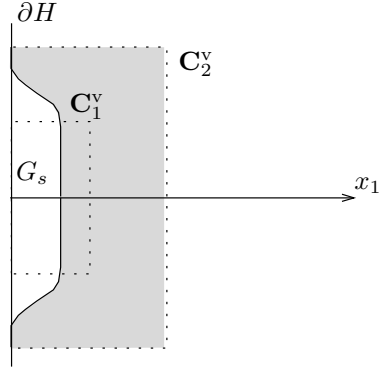\caption{{\small Given $z\in\D_2^{{\rm v}}$, the map $f_s$ stretches each segment $[\vphi_s(|z|),1]\times\{z\}$ into $[0,1]\times\{z\}$, while keeping the point $(1,z)$ fixed.}}\label{fig Gh}
\end{figure}
\begin{equation}\label{eq:gradf}
\sup_{x\in H\setminus G_s}|f_s(x)-x|+\|\nabla f_s(x)-\Id\|\le C\,s\,,
\end{equation}
for a constant $C=C(n)$. If we set \(E_s=f_s(E\setminus G_s)\), then $E_s$ is a set of locally finite perimeter in $B_4$ (as $E\setminus G_s$ is), with $E_s\subset H$ and $E_s\Delta E \subset \C^{\rm v}_2$ with $\diam(\C^{{\rm v}}_2)=\sqrt{20}<8$. We may thus exploit the fact that $E$ is a $(1/8\l ,4)$-minimizer of $\PHI$ in $(B_4,H)$ to deduce that
\begin{equation}\label{comparison}
\PHI(E;\C_2^{\rm v})\le \PHI(f_s(E\setminus G_s);\C_2^{\rm v})+\Lambda|(E\Delta f_s(E))\cap\C^{\rm v}_2|.
\end{equation}
By \cite[Lemma 17.9]{maggiBOOK} and \eqref{eq:gradf}, we have
\begin{equation}
  \label{ins1}
|(E\Delta f_s(E))\cap\C^{\rm v}_2|\le C\,P(E;B_{3})\,s\,,
\end{equation}
for some $C=C(n)$. Moreover, by \eqref{change of variables cofattore} and $f_s(\C_2^{\rm v})=\C_2^{\rm v}$ we find that
\begin{eqnarray}\nonumber
  \PHI(f_s(E\setminus G_s);H)&=&
  \int_{\C_2^{\rm v}\cap\pa^*(E\setminus G_s)}\Phi\Big(f_s(x),\cof(\nabla f_s(x))\nu_{E\setminus G_s}(x)\Big)\,d\H^{n-1}(x)
  \\\label{ins2}
  &\le&\PHI(E; \C_2^{\rm v} \setminus G_s)+C(n,\l)\,s\,P(E;\C_2^{\rm v} \setminus G_s)\,,
\end{eqnarray}
where we have used \eqref{eq:gradf}, \eqref{Phi x ell}, \eqref{Phi nabla 1} and the fact that $\ell \le 1$ (so that  $C$ depend on $n$ and $\l$ only). By \eqref{comparison}, \eqref{ins1}, \eqref{ins2} and  \eqref{stime densita perimetro upper} we thus find
 \[
 \PHI(E;G_s)\le C  P(E, B_3) s\le C\, s
 \]
with $C=C(n,\l)$. Since $(0,s)\times \D_{1}^{\rm v} \subset G_s$, we conclude by \eqref{Phi 1} that
\begin{equation*}
P(E;(0,s)\times \D_{1}^{\rm v})\le C(n,\lambda)\,s\qquad \forall\, s\in (0,1/2)\,.
\end{equation*}
By the coarea formula for rectifiable set, see, e.g. \cite[Equation (18.25)]{maggiBOOK}, we find
\[
\int_0^s P(E_t; \D_{1}^{\rm v})\,dt\le P(E;(0,s)\times \D_{1}^{\rm v})\le C(n,\lambda)\,s\,.
\]
Hence, for every \(s\in(0,1/2)\) we can find \(t_s\in (0,s)\) such that $P(E_{t_s}; \D_{1}^{\rm v})\le C(n,\lambda)$. By taking the limit as \(s\to 0^+\), thanks to \eqref{eq:traceapprox} and by lower semicontinuity of perimeter,
\[
P(\tr_{\partial H}(E);\D_{1}^{\rm v})\le \liminf_{s \to 0^+}P(E_{t_s}; \D_{1}^{\rm v})\le C(n,\lambda)\,,
\]
which is \eqref{eq:perimetertrace3}.
\end{proof}

Our next goal is proving lower density estimates for the perimeter of the contact set. To this end, we shall need a strong maximum principle for local minimizers of regular autonomous elliptic integrands. (This should be compared with \cite{solomonwhite}, where, however, even integrands are considered in order to deal with non-orientable surfaces.) We prove the strong maximum principle in Lemma \ref{lemma strong max}, as a corollary of a comparison lemma, Lemma \ref{lemma comparison}, illustrated in Figure \ref{fig compa}. We premise to these results the following lemma, about the existence of Lipschitz solution of non-parametric problems. (Part two of the statement will be used in section \ref{section singular set}.)

\begin{lemma}\label{nonpar}
Let \(\Phi \in \X_*(\lambda)\), $\l>0$, and set \(\Phi^\#(\xi)=\Phi(\xi,-1)\) for \(\xi\in \R^{n-1}\).

\medskip

\noindent {\em Part one:} If \(\vphi\in C^{1,1}(\cl (\D_r))\), then there exists a unique \(u\in C^2(\D_r)\cap {\rm Lip} (\cl (\D_r))\) such that
\begin{equation}\label{eq:nonparint}
\begin{cases}
\Div (\nabla_{\xi} \Phi^\#(\nabla u))=0\,,\qquad&\textrm{in \(\D_r\)\,,}\\
u=\varphi\,, &\textrm{on \(\pa \D_r\)}\,.
\end{cases}
\end{equation}
In addition, if  \(\varphi\ge 0\), \(\vphi\not\equiv 0\), then \(u(0)>0\).

\medskip

\noindent {\em Part two:} Given $e\in\R^{n-1}$ with $|e|=1$, let $H=\{z\in\R^{n-1}:z\cdot e\ge 0\}$. If \(\vphi\in C^{1,1}(\cl (\D_r\cap H))\) with  \(\vphi=0\) on \(\D_r\cap \pa H\), then there exists a unique \(u\in C^2(\D_r\cap \cl (H))\cap {\rm Lip} (\cl (\D_r\cap H))\) with
\begin{equation}\label{eq:nonparbound}
\begin{cases}
\Div (\nabla_{\xi} \Phi^\#(\nabla u))=0\,, \qquad&\textrm{in \(\D_r\cap H\,,\)}\\
u=\varphi\,, &\textrm{on \(\pa (\D_r\cap H)\)}\,.
\end{cases}
\end{equation}
In addition, if \(\varphi\le M\, (e\cdot z)\) on \(\pa (\D_r\cap H)\) for some \(M\in \R\), then \(|\nabla \vphi(0)|=|\pa_e \vphi(0)|<M\).
\end{lemma}

\begin{proof} The couple  \((\varphi,\D_r)\) satisfies the so called Bounded Slope Condition, see \cite[Theorem 1.1]{Gi}, hence existence of a solution to \eqref{eq:nonparint} follows by \cite[Theorem 1.2]{Gi}. Uniqueness follows by noticing that for every two Lipschitz solutions to \eqref{eq:nonparint} the difference solves a uniformly linear elliptic equation, see for instance \cite[Chapter 10]{gt}. The last statement in part one  is just the  strict maximum principle, \cite[Theorem 3.5]{gt}. To prove the existence and uniqueness of solutions to \eqref{eq:nonparbound}, one can  make an odd reflection around \(\pa H\) and then use part one. For what concerns the last statement of part 2, note that \(w=M z\cdot e-u(z)\ge 0\) on \(\pa (\D_r\cap H)\), \(w(0)=0\) and that \(w\) solves the uniformly elliptic (linear) equation,
\[
0=\Div (\nabla_{\xi} \Phi^\#(\nabla u ))=\sum_{i,j}^{n-1} \nabla^2_{\xi_{i}\xi_{j} }\Phi^\#(\nabla u) \partial_{ij} u=\sum_{i,j}^{n-1} A_{ij}\, \partial_{ij} w\,.
\]
We conclude by Hopf's boundary lemma, \cite[Lemma 3.4]{gt}.
\end{proof}

\begin{lemma}
  [Comparison lemma]\label{lemma comparison}
  If $\l\ge1$,
  \begin{figure}
  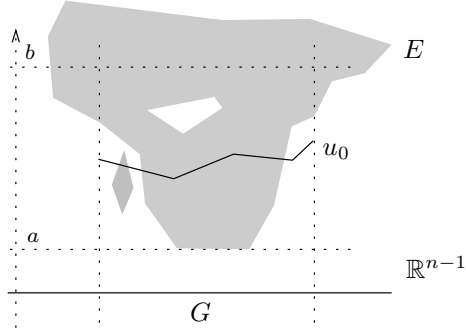\caption{{\small The situation in Lemma \ref{lemma comparison}: the role of \eqref{compa 2} is to ensure that, over the boundary of $G$, $E$ lies above the graph of $u_0$. The figure tries to stress the fact that $E$ does not need to be the epigraph of a function.}}\label{fig compa}
  \end{figure}
  $\Phi\in\X_*(\l)$, $G\subset\R^{n-1}$ is a bounded open set with Lipschitz boundary, $a<b$, $J=G\times(a,b)$, $E\subset\{x_n>a\}$ is a set of finite perimeter in an open neighborhood of $J$ such that
  \begin{eqnarray}
    \label{compa 1}
      \PHI(E;\cl(J))\le\PHI(F;\cl(J))\,,\qquad\mbox{whenever $F\subset E$, $E\setminus F\subset J$}\,,
  \end{eqnarray}
  and $u_0\in C^2(G)\cap\Lip(\cl(G))$ with $a<u_0<b$ on $\cl(G)$ and
  \begin{eqnarray}
    \label{compa 2}
    &&E^{(1)}\cap[(\pa G)\times(a,b)]\subset\Big\{(z,t)\in(\pa G)\times(a,b):t\ge u_0(z)\Big\}\,,
    \\
    \label{compa 3}
    &&\Div(\nabla_\xi\Phi^\#(\nabla u_0(z)))=0\,,\qquad\forall z\in G\,,
  \end{eqnarray}
  where \(\Phi^\#(\xi)=\Phi(\xi,-1)\) for \(\xi\in \R^{n-1}\), then
  \begin{equation}
    \label{compa tesi}
    E\cap J\subset_{\H^n}\Big\{(z,t)\in J:t\ge u_0(z)\Big\}\,.
  \end{equation}
\end{lemma}

\begin{proof}
  Let us begin noticing that the lower bound in \eqref{k1k2} can be written in the form
  \begin{eqnarray}\label{compa 0}
  \k_1\,\frac{|\nu_1-\nu_2|^2}2\le\Big(\nabla\Phi(\nu_2)-\nabla\Phi(\nu_1)\Big)\cdot\nu_2\,,\qquad\forall\nu_1,\nu_2\in \mathbf S^{n-1}\,,
  \end{eqnarray}
  We now consider the open sets
  \[
  F_+=\Big\{(z,t)\in J:t> u_0(z)\Big\}\,,\qquad F_-=\Big\{(z,t)\in J:t<u_0(z)\Big\}\,.
  \]
  By \eqref{compa 2} and since $a<u_0<b$ on $\cl(G)$, we can exploit \eqref{compa 1} with $F=(E\setminus J)\cup(E\cap F_+)$ to find that
  \begin{eqnarray}
    \int_{\pa^*E\cap[F_-\cup(G\times\{a\})]}\Phi(\nu_E)\,d\H^{n-1}
    \le\int_{\pa^*F_+\cap(E^{(1)}\cap J)}\Phi(\nu_{F_+})\,d\H^{n-1}\,.
    \label{compa 4}
  \end{eqnarray}
  Let us now set $T(x)=\nabla\Phi(\nabla u_0(\p x),-1)$ for $x\in G\times\R$, so that $T\in C^1(G\times\R;\R^n)$ with $\Div T(x)=0$ for every $x\in G\times\R$ thanks to \eqref{compa 3}. If we set $\widetilde{E}=(E\cap J)\setminus F_+$, then $\widetilde{E}\subset J$ and $\pa^*\widetilde{E}\cap[(\pa G)\times(a,b)]$ is $\H^{n-1}$-negligible thanks to \eqref{compa 2}. Hence we can apply the divergence theorem to $T$ on $\widetilde{E}$ to find that
  \begin{equation}
    \label{compa5}
      \int_{\pa^*F_+\cap(E^{(1)}\cap J)}T\cdot\nu_{F_+}\,d\H^{n-1}=\int_{\pa^*E\cap(F_-\cup(G\times\{a\}))}T\cdot\nu_E\,d\H^{n-1}\,.
  \end{equation}
  Now, by zero-homogeneity of $\nabla\Phi$ and by
  \begin{equation}
    \label{normale F+}
      \nu_{F_+}(z,u_0(z))=\frac{(\nabla u_0(\p x),-1)}{\sqrt{1+|\nabla u_0(\p x)|^2}}\,,\qquad\forall x\in J\cap\pa F_+\,,
  \end{equation}
  we find that, for $\H^{n-1}$-a.e. on $J\cap\pa F_+$,
  \begin{equation}
    \label{compa6}
      T\cdot\nu_{F_+}=\nabla\Phi(\nu_{F_+})\cdot\nu_{F_+}=\Phi(\nu_{F_+})\,,\qquad\mbox{on $J\cap\pa F_+$}\,,
  \end{equation}
  while, by \eqref{compa 0},
  \begin{equation}
    \label{compa7}
      T\cdot\nu_E=\Phi(\nu_E)-\Big(\nabla\Phi(\nu_E)
      -\nabla\Phi(\nabla u_0,-1)\Big)\cdot\nu_E\le\Phi(\nu_E)-\frac{\k_1}2\,\bigg|\nu_E-\frac{(\nabla u_0,-1)}{\sqrt{1+|\nabla u_0|^2}}\bigg|^2\,.
  \end{equation}
  We may thus combine \eqref{compa 4}, \eqref{compa5}, \eqref{compa6} and \eqref{compa7} to deduce
  \[
  \int_{\pa^*E\cap(F_-\cup(G\times\{a\}))}\bigg|\nu_E-\frac{(\nabla u_0(\p(x)),-1)}{\sqrt{1+|\nabla u_0(\p x)|^2}}\bigg|^2\, d\, \H^{n-1}(x)=0,
  \]
 so that
 \begin{equation}\label{forse?}
 \nu_E=\frac{(\nabla u_0(\p(x)),-1)}{\sqrt{1+|\nabla u_0(\p x)|^2}}\qquad \textrm{\(\H^{n-1}\)-a.e. on \(\pa^*E\cap(F_-\cup(G\times\{a\}))\).}
 \end{equation}
 Folllowing \cite{duzaarsteffencomparison}, we apply  the divergence theorem on the set $\widetilde{E}$ to the vector field $S\in C^1(G\times\R;\R^n)$ defined by $S(x)=\big(\p x, \p x \cdot \nabla u_0 (\p x)\big)$ for every  $x\in G\times\R$.  Since  $S\cdot\nu_{\widetilde{E}}=0$ $\H^{n-1}-a.e.$ on $\pa^*\widetilde{E}$ thanks to \eqref{cap}, \eqref{minus}, \eqref{compa 2}, \eqref{normale F+} and \eqref{forse?}, we conclude that
 \[
 (n-1)|\widetilde{E}|=\int_{\widetilde{E}} \Div S=\int_{\pa^* \widetilde{E}} S\cdot \nu _{\widetilde{E}}\,d\H^{n-1}=0\,,
 \]
 that is, $|(E\cap J)\setminus F_+|=0$. This proves the lemma.
\end{proof}

\begin{lemma}[Strong maximum principle]\label{lemma strong max}
  If $\l\ge1$, $H=\{x_1>0\}$, $\Phi\in\X_*(\l)$, and $E\subset H$ is a set of locally finite perimeter in $B$ such that
  \begin{eqnarray}
  \label{minimissimo}
  &&\PHI(E;W)\le\PHI(F;W)\,,\qquad\quad\mbox{whenever $E\Delta F\cc W\cc B$}\,,
  \end{eqnarray}
  then either $\H^{n-1}(B\cap\tr_{\pa H}(E))=\H^{n-1}(B\cap\pa H)$ or $0\not\in\spt\mu_E$.
\end{lemma}

\begin{proof} Let us assume that
\begin{equation}
\label{cazzata2}
  \H^{n-1}(B\cap\tr_{\pa H}(E))<\H^{n-1}(B\cap\pa H)\,.
\end{equation}
Since \eqref{minimissimo} means that $E$ is a minimizer of $\PHI$ in $(B,\R^n)$, by Lemma \ref{lemma stime di densita} we find
\[
c_1\,|B_{x,r}\cap H|<|E\cap B_{x,r}|<(1-c_1)\,|B_{x,r}\cap H|
\]
for every $x\in B\cap\spt\mu_E$ and $r<\dist(x,\pa B)$. In particular, $B\cap\spt\mu_E\subset B\cap\pa^{\rm e} E$, so that, by \eqref{federer theorem}, $B\cap\spt\mu_E\subset_{\H^{n-1}}B\cap\pa^*E$. Thus, by \eqref{traces equivalenza} and \eqref{cazzata2} we find that
\begin{equation}
  \label{cazzata9}
  \H^{n-1}(B\cap\spt\mu_E)<\H^{n-1}(B\cap\pa H)\,.
\end{equation}
We now define a function $w_E:\D_1^{\rm v}\to[0,\infty]$ by setting
\[
w_E(z)=\inf\Big\{t\in\R:(t,z)\in B\cap\spt\mu_E\Big\}\,,\qquad z\in \D_1^{{\rm v}}\,.
\]
Since \(\spt \mu_E\) is a closed subset of $H$, it turns out that $w_E$ is non-negative and lower semicontinuous on $\D_1^{{\rm v}}$, with the property that
\begin{equation}
  \label{wE contenimento}
  E\cap B\subset_{\H^n}\Big\{x\in \R^n:x_1\ge w_E(\h x)\Big\}\,,
\end{equation}
see Figure \ref{fig wE2}.
\begin{figure}
  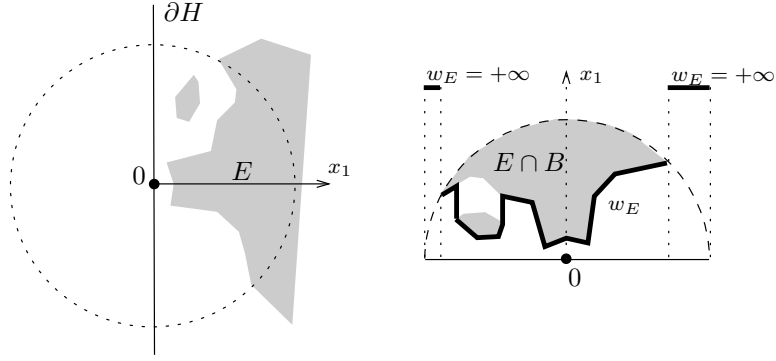\caption{{\small Inclusion \eqref{wE contenimento}. Notice that $w_E$ may take the value $+\infty$.}}\label{fig wE2}
\end{figure}
By the coarea formula, \eqref{cazzata9} and \eqref{wE contenimento}, there exists \(r_*\in (0,1/2\sqrt2)\) such that
\begin{equation}\label{momangio}
 \H^{n-2}( \spt\mu_E \cap \pa \D^{\rm v}_{r_*})<\H^{n-2}(B \cap \pa \D^{\rm v}_{r_*})\,,\qquad\H^{n-1}(\pa^*E\cap\pa\C_{r_*}^{\rm v})=0\,,
\end{equation}
and
\begin{equation}
  \label{momangio2}
  x_1\ge w_E(\h x)\,,\qquad\mbox{for $\H^{n-1}$-a.e. $x\in E^{(1)}\cap B\cap[\R\times\pa\D_{r_*}^{\rm v}]$}\,.
\end{equation}
By definition of $w_E$ and by \eqref{momangio}, $w_E(z)>0$ on a subset of $\pa\D^{\rm v}_{r_*}$ with positive $\H^{n-2}$-measure. Therefore, there exists \(\vphi\in C^\infty(\pa\D^{\rm v}_{r_*})\) such that
\begin{equation}\label{pappero}
\max_{\pa\D_{r_*}^{\rm v}}\vphi>0\,,\qquad 0\le \vphi(z)\le \min\Big\{w_E(z),\frac{r_*}2\Big\}\,,\qquad \forall\, z\in \pa \D_{r_*}^{\rm v}\,.
\end{equation}
By Lemma \ref{nonpar}, part one, there exists \(u\in C^2(\D^{\rm v}_{r_*})\cap\Lip(\cl(\D^{\rm v}_{r_*}))\) such that
\begin{equation}\label{tuttovero}
\begin{cases}
\Div (\nabla_{\xi} \Phi^\#(\nabla u))=0\,,\qquad&\textrm{in \(\D^{\rm v}_{r_*}\)}\,,
\\
u=\varphi\,,&\textrm{on \(\pa \D^{\rm v}_{r_*}\)}\,,
\end{cases}
\end{equation}
where $\Phi^\#(\xi)=\Phi(-1,\xi)$ for $\xi\in\R^{n-1}$; moreover, \(u(0)>0\) by \eqref{pappero}. By \eqref{minimissimo}, \eqref{momangio2}, \eqref{pappero}, and \eqref{tuttovero}, we can apply Lemma \ref{lemma comparison} to infer that
\[
E\cap\C_{r_*}^{\rm v}\subset_{\H^n} \Big\{(z,t)\in\C_{r_*}^{\rm v}: t\ge u(z)\Big\}\,.
\]
Since \(u(0)>0\), this last inclusion implies that $0\not\in\spt\mu_E$, and the lemma is proved.
\end{proof}

\begin{lemma}[Lower perimeter estimate for contact sets]\label{lem:densitycontact} For every \(\l\ge 1\) there exist two positive constants \(\e= \e (n,\l)\) and \(c=c(n,\l)\) with the following property. If $A$ is an open set, $H=\{x_1>0\}$, $\Phi\in\X(A\cap H,\l,\ell)$, $E$ is a $(\La,r_0)$-minimizer of $\PHI$ in $(A,H)$, and $(\Lambda +\ell )r \le \e$, then, for every $x\in \pa H\cap \spt\mu_E$ and $r<\min\{r_0,\dist (x,\partial A)\}/4$, we have
\begin{equation}
  \label{stime densita perimetro lower traccia}
  \mathcal H ^{n-1} (B_{x,r}\cap \tr_{\partial H}(E))\ge c\, r^{n-1}\,.
\end{equation}
In particular, $\H^{n-1}(\spt\mu_E\setminus\pa^*E)=0$.
\end{lemma}

\begin{proof} We start  showing as  \eqref{stime densita perimetro lower traccia} implies the last part of the statement, indeed  \eqref{traces equivalenza} and \eqref{stime densita perimetro lower traccia} imply
\begin{equation}
  \label{powerpuff}
  \liminf_{r\to 0^+}\frac{\H^{n-1}(B_{x,r}\cap\pa^*E)}{r^{n-1}}>0\,,
\end{equation}
for every $x\in A\cap\pa H\cap\spt\mu_E$. Since \eqref{powerpuff} holds true at every $x\in A\cap\cl(H\cap\spt\mu_E)$ by \eqref{stime densita perimetro lower}, we conclude that \eqref{powerpuff} holds true at every $x\in\spt\mu_E$. By differentiation of Hausdorff measures (see, e.g., \cite[Corollary 6.5]{maggiBOOK}), we conclude that \(\H^{n-1}(\spt \mu_E\setminus \pa^*E)=0\).

We now prove \eqref{stime densita perimetro lower traccia}. Up to consider \(E^{x,r}\) and \(\Phi^{x,r}\) in place of $E$ and $\Phi$ we may reduce to prove the following statement: if \(\Phi \in \X(B_{2}\cap H, \l, \ell)\), and \(E\) is a \((\Lambda, 4)\) minimizer of \(\PHI\)  in \((B_{2}, H)\) with \(0\in \spt \mu_E\) and $(\Lambda +\ell) \le \e$, then
\[
\mathcal H ^{n-1} (B\cap\tr_{\partial H} (E))\ge c\,.
\]
We argue by contradiction, and assume that for every $h\in\N$ there exist $\Phi_h\in \X(B_2\cap H,\l,\ell_h r_h )$ and $E_h$ a $(\Lambda_h r_h ,4)$ minimizer of $\PHI_h$ in $(B_{2},H)$, satisfying
\[
  0\in \bigcap_{h\in\N}\spt \mu_{E_h}\,, \qquad \lim_{h\to\infty}(\Lambda_h +\ell_h) r_h=0\,,\qquad\lim_{h\to\infty}\H^{n-1}(B\cap \tr_{\partial H}(E_h))=0\,.
\]
By Theorem \ref{thm compactness minimizers}, there exist $\Phi_\infty \in \X_*(\l)$ and a $(0,4)$-minimizer $E_\infty$ of $\PHI_\infty$ in $(B_2,H)$ such that, up to extracting a not relabeled subsequence, \(E_h\to E_\infty\) in \(L^1_{\rm loc}(B_2)\), and
\begin{equation}\label{cazzata}
\H^{n-1}(B\cap\tr_{\partial H}(E_\infty))=0\,.
 \end{equation}
In fact, $E_\infty$ is a minimizer of $\PHI_\infty$ in $(B,\R^n)$, i.e.
\begin{equation}
  \label{strong minimo infinito}
  \PHI_\infty(E_\infty;B)\le\PHI_\infty(F;B)\,,
\end{equation}
whenever $E_\infty\Delta F\cc B$ (note that this is a  a stronger property than being a $(0,4)$-minimizer of $\PHI_\infty$ in $(B_2,H)$).

To prove \eqref{strong minimo infinito}, pick $F$ such that $E_\infty\Delta F\cc B$: since $E_\infty\Delta (F\cap H)\cc B$ and $E_\infty$ is a $(0,4)$-minimizer $\PHI_\infty$  in $(B_2,H)$ we find
\begin{equation}
  \label{strong minimo infinito2}
  \PHI_\infty(E_\infty;B\cap H)\le  \PHI_\infty(F\cap H;B\cap H)\,.
\end{equation}
The left-hand sides of \eqref{strong minimo infinito} and \eqref{strong minimo infinito2} coincide  by \eqref{subset of H}, \eqref{traces equivalenza} and \eqref{cazzata}; the right-hand side of \eqref{strong minimo infinito2} is instead smaller than the right-hand side of \eqref{strong minimo infinito} since $H\cap B\cap\pa^*(F\cap H)\subset B\cap \pa^*F$ with $\nu_{F\cap H}=\nu_F$ $\H^{n-1}$-a.e. on $H\cap B\cap\pa^*(F\cap H)$ by \eqref{cap}. This proves \eqref{strong minimo infinito}. Since \(E_\infty \subset H\) and \(E_\infty\) satisfies   \eqref{strong minimo infinito} we can apply  Lemma \ref{lemma strong max} to deduce that \eqref{cazzata} implies  $0\not\in\spt \mu_{E_\infty}$.

We now achieve a contradiction, and thus prove the lemma, by showing that $0\in \spt \mu_{E_\infty}$. Indeed, let us set
\[
\delta_h=\max\left\{\H^{n-1}(B\cap\tr_{\partial H} (E_h))\,,\frac 1 h \right\}>0\,,\qquad \varrho_h=\Big(\frac{2\delta_h}{\omega_{n-1}}\Big)^{1/(n-1)}\,,\qquad h\in\N\,.
\]
(Notice that, up to take $h$ large enough, we can assume $\varrho_h<1$ for every $h\in\N$.) Since $\de_h>0$ and $0\in\spt\mu_{E_h}$, we find $|E_h\cap B_{\varrho_h}|>0$ for every $h\in\N$. Similarly, it must be $|(H\cap B_{\varrho_h})\setminus E_h|>0$ for every $h\in\N$, for otherwise, by the locality of the trace, we would have $B_{\varrho_h}\cap\pa H=_{\H^{n-1}}B_{\varrho_h}\cap\tr_{\pa H}(E_h)$ for some $h\in\N$, and correspondingly
\[
\de_h\ge \H^{n-1}(B\cap\tr_{\pa H}(E_h))\ge\H^{n-1}(B_{\varrho_h}\cap\tr_{\pa H}(E_h))=\H^{n-1}(B_{\varrho_h}\cap\pa H)=\om_{n-1}\varrho_h^{n-1}=2\,\de_h\,,
\]
a contradiction to $\de_h>0$. This shows that $|E_h\cap B_{\varrho_h}|\,|(H\setminus E_h)\cap B_{\varrho_h}|>0$ for every $h\in\N$. In particular, for every $h\in\N$ there exists $x_h\in\cl(H\cap\spt\mu_{E_h})\cap B_{\varrho_h}$, and since $\varrho_h\to 0$ as $h\to\infty$, we conclude by  \eqref{kuratowski1} that $0\in\spt\mu_{E_\infty}$. As already noticed, this completes the proof.
\end{proof}

We finally prove the following normalization lemma. Recall that, if $E$ is a set of locally finite perimeter in $A$ with $E\subset H$, then $\spt\mu_E$, $\pa^*E$ and $\tr_{\pa H}(E)$ are defined as subsets of $A$. Moreover we are going to denote by \(\pa_{\pa H}\) the topological boundary of subsets of \(\pa H\), and by \(\pa^*_{\partial H}\) the reduced boundary of sets of locally finite perimeter in \(\pa H\).

\begin{lemma}[Normalization]\label{lem:normalization} If $A$ is an open set, $H=\{x_1>0\}$, $\Phi\in\X(A\cap H,\l,\ell)$, and $E$ is a $(\La,r_0)$-minimizer of $\PHI$ in $(A,H)$, then, up to modify \(E\) on a set of measure zero,
\begin{enumerate}
\item[(i)] \(E\cap A\) is open  and \(A\cap\pa E=\spt \mu_E\);
\item[(ii)] \(\H^{n-1}((A\cap\pa E)\setminus \pa^* E)=0\) and  \(\H^{n-1}\big((A\cap\pa E\cap \pa H)\Delta \tr_{\pa H}(E)\big) =0\);
\item[(iii)] \(\pa E\cap \pa H\) is a set of locally finite perimeter in \(A\cap\pa H\), and
\begin{equation}\label{pizza}
\H^{n-2}\Big(\big[\partial_{\pa H}(\pa E\cap \pa H)\setminus \partial_{\pa H}^*(\pa E\cap \pa H)\big]\cap A\Big)=0\,.
\end{equation}
Moreover,
\begin{equation}\label{pizza2}
\partial_{\pa H}\big(\pa E\cap \pa H\big)\cap A=\cl\big(\pa E \cap H)\cap \pa H\cap A\,.
\end{equation}
\end{enumerate}
\end{lemma}

\begin{remark} In the sequel we will \emph{always} assume that \(E\) is normalized in order to satisfy the conclusions of Lemma \ref{lem:normalization}. In particular, $\pa E$ shall be used in place of $\spt\mu_E$.
\end{remark}

\begin{proof}[Proof of Lemma \ref{lem:normalization}] Let us consider the set
\begin{equation*}
\widetilde E=\big\{x\in A\cap H:  |E\cap B_{x,r}|=|B_{x,r}| \textrm{ for some  \(r>0\)}  \big\}\cup(E\setminus A)\,.
\end{equation*}
Obviously, $\widetilde E\cap A$ is open and, by \eqref{sptmuE},
\begin{equation}
  \label{bartali2}
  A\cap\pa\widetilde{E}=\Big\{x\in A:0<|E\cap B_{x,r}|<|B_{x,r}|\quad\forall r>0\Big\}=\spt\mu_E\,.
\end{equation}
We now claim that $\widetilde{E}$ is equivalent to $E$. Clearly, $\widetilde E \cap A\subset E^{(1)}\cap A$. At the same time, if $x\in (E^{(1)}\cap A\cap H)\setminus \widetilde E$, then there exists $r_*>0$ such that $0<|E\cap B_{x,r}|<|B_{x,r}|$ for every $r<r_*$, that is, $x\in H\cap \spt\,\mu_E$: but then we cannot have $x\in E^{(1)}$ because of the density estimate \eqref{stime densita volume upper}. We have thus proved $H\cap A\cap (E^{(1)}\Delta \widetilde E)=\emptyset$, so that, by Lebesgue's density points theorem, $\widetilde E$ is equivalent to $E$. In particular, $\mu_E=\mu_{\widetilde{E}}$, and thus, by \eqref{bartali2}, we have
\begin{equation}\label{bartali}
  A\cap\pa\widetilde{E}=\spt \mu_{\widetilde E}\subset_{\H^{n-1}}\pa^*\widetilde{E}\,,
\end{equation}
where the last inclusion follows by Lemma \ref{lem:densitycontact} (clearly,  \(\widetilde E\) is a $(\La,r_0)$-minimizer of $\PHI$ in $(A,H)$). By \eqref{bartali},  $\H^{n-1}((A\cap\pa\widetilde{E})\Delta \pa^*\widetilde{E})=0$, and since $\H^{n-1}(\pa^*\widetilde{E}\Delta\tr_{\pa H}(\widetilde{E}))=0$ by \eqref{traces equivalenza}, we have completed the proof of (ii). By (ii) and by Lemma \ref{lem:finiteperimeter}, \(\pa \widetilde E\cap \pa H\) is a set of locally finite perimeter in \(\pa H\cap A\). Let us now prove  \eqref{pizza2}. To this end we notice that, clearly,
\begin{eqnarray}\nonumber
A\cap\partial_{\pa H}\big(\pa \widetilde E\cap \pa H\big)&\subset&\Big\{x\in A\cap\pa H: |E\cap B_{x,r}|\,|(H\setminus E)\cap B_{x,r}|>0
\quad\forall r>0\Big\}
\\\label{nonsopiuchenome}
&\subset& A\cap\pa H\cap \cl\big(H\cap \pa  \widetilde  E\big)\,,
\end{eqnarray}
where the second inclusion follows as $A\cap \pa \widetilde  E=\spt\mu_{\widetilde{E}}$. At the same time, since \(H\cap\spt \mu_{\widetilde E}=H\cap\spt \mu_{H\setminus \widetilde E}\) and $A\cap\pa\widetilde E=\spt\mu_{\widetilde E}$, we have
\[
A\cap\pa H\cap \cl\big(H\cap \pa  \widetilde  E\big)\subset \spt \mu_{\widetilde E}\cap \spt \mu_{H\setminus\widetilde E}\cap \pa H\,,
\]
so that, by Remark \ref{remark complement}, Lemma \ref{lem:densitycontact} (applied to both \(\widetilde E\) and \(H\setminus\widetilde E\)) and \eqref{tracciacomplementare}, one has, for every $x\in A\cap\pa H\cap \cl(H\cap \pa  \widetilde  E)$ and $r>0$ sufficiently small,
\begin{equation}\label{nonsopiuchenome2}
c\, \H^{n-1}(B_{x,r}\cap\pa H)\le \H^{n-1}(B_{x,r}\cap\pa \widetilde E\cap\pa H)\le (1-c) \H^{n-1}( B_{x,r}\cap\pa H)\,,
\end{equation}
where $c=c(n,\l)$. This implies of course $A\cap\pa H\cap \cl(H\cap \pa  \widetilde  E)\subset A\cap\partial_{\pa H}(\pa \widetilde E\cap \pa H)$,
that, together with \eqref{nonsopiuchenome}, implies \eqref{pizza2}. Finally, we  notice that, by \eqref{pizza2},  the relative isoperimetric inequality in \(\pa H \simeq\R^{n-1}\) and \eqref{nonsopiuchenome2} give
\[
\H^{n-2}\Big(\pa^*_{\pa H} (\pa \widetilde E\cap \pa H)\cap B_{x,r} \Big)\ge c(n,\lambda)\, r^{n-2}\,,\qquad \forall\, x\in \pa_{\pa H} (\pa \widetilde E\cap \pa H)\cap A\,,
\]
which implies \eqref{pizza} by differentiation of Hausdorff measures, see \cite[Corollary 6.5]{maggiBOOK}.
\end{proof}

\subsection{Ellipticity, minimality, and affine transformations}\label{section cambio di variabili L} In the proof of the $\e$-regularity theorem will need to look at a given almost-minimizers from different directions at different scales. A very useful trick is then that of using affine transformations in order to always write things in the same system of coordinates. It is thus convenient, for the sake of clarity, to state separately how the considered class of elliptic functionals and almost-minimizers behave under these transformations.

\begin{lemma}\label{lemma PHIL}
  If $A$ is an open set in $\R^n$, $H$ an open half-space, $\Phi\in\X(A\cap H,\l,\ell)$, $E$ is a $(\La,r_0)$-minimizer of $\PHI$ in $(A,H)$, $L$ is an invertible affine map on $\R^n$, and
  \begin{equation}
    \label{PhiL}
      \Phi^L(x,\nu)=\Phi(L^{-1}x,(\cof \nabla L)^{-1}\nu)\,,\qquad (x,\nu)\in\cl(L(A\cap H))\times \R^{n}\,,
  \end{equation}
  then $\Phi^L\in\X(L(A\cap H),\widetilde{\l},\widetilde{\ell})$ and $L(E)$ is a $(\widetilde{\La},\widetilde{r_0})$-minimizer of $\PHI^L$ on $(L(A),L(H))$, where
  \begin{eqnarray*}\label{PHIL tilde1}
  &&\widetilde{\l}=\l\,\max\Big\{\|\nabla L\|\|(\nabla L)^{-1}\|^2,\|\nabla L\|^2\|(\nabla L)^{-1}\|\Big\}^{n-1}\,,
  \\\label{PHIL tilde2}
  &&\widetilde{\ell}= \ell\,\|(\nabla L)^{-1}\|^n\,,\qquad
  \widetilde{\La}=\frac{\La}{|\det\nabla L|}\,,\qquad \widetilde{r_0}=\frac{r_0}{\|(\nabla L)^{-1}\|}\,.
  \end{eqnarray*}
  In particular, there exist positive constants $\e_*=\e_*(n)$ and $C_*=C_*(n)$, such that, if $\|\nabla L-\Id\|<\e_*$, then
  \begin{equation}\label{PHIL continuita parametri}
  \max\Big\{\frac{|\widetilde{\l}-\l|}\l,\frac{|\widetilde{\ell}-\ell|}\ell,\frac{|\widetilde{\La}-\La|}{\La},\frac{|\widetilde{r_0}-r_0|}{r_0}\Big\}
  \le C_*\,\|\nabla L-\Id\|\,.
  \end{equation}
\end{lemma}

\begin{remark}
  {\rm  Definition \eqref{PhiL} is conceived to give the identity
  \begin{equation}\label{change of variable linear}
    \PHI^L(L(E);L(K))=\PHI(E;K)\,,
  \end{equation}
  for every $E$ with locally finite perimeter in $A$ and $K\cc A$. Also, \eqref{PHIL continuita parametri} has to be understood in the sense that if, say, $\Lambda=0$, then $\widetilde{\La}=\La$, and so on.}
\end{remark}

\begin{proof}[Proof of Lemma \ref{lemma PHIL}] Without loss of generality, in order to simplify the notation, we directly assume that $L$ is linear, and write $L$ in place of $\nabla L$ and $L^{-1}$ in place of $(\nabla L)^{-1}$.

\medskip

\noindent {\it Step one}: If $\s_{min}=\s_{min}(L)$ and $\s_{max}=\s_{max}(L)$ denote the square roots of the maximum and minimum eigenvalues of $L^*L$, then we have
  \begin{eqnarray}\label{l-l+}
    &&\s_{min}|z|\le \min\{|Lz|,|L^*z|\}\le\max\{|Lz|,|L^*z|\}\le\s_{max}\,|z|\,,\qquad\forall z\in\R^n\,,
    \\
    \label{l-l+3}
    &&\|L\|=\s_{max}\,,\qquad \|L^{-1}\|=\s_{min}^{-1}\,,\qquad \s_{min}^n\le \det\,L\le\s_{max}^n\,.
  \end{eqnarray}
  On taking into account that $(\det\,L)\,L^{-1}=(\cof L)^*$, we find
  \begin{equation}
    \label{mortacci}
  (\cof L)^{-1}=(\det\,L)^{-1}\,L^*\,,
  \end{equation}
  and thus, by \eqref{l-l+} and \eqref{l-l+3},
  \begin{equation}\label{l-l+2}
    \frac{|z|}{\s_{max}^{n-1}}\le |(\cof L)^{-1}z|\le\frac{|z|}{\s_{min}^{n-1}}\,,\qquad\forall z\in\R^n\,.
  \end{equation}
  By \eqref{Phi 1}, \eqref{PhiL}, and \eqref{l-l+2} we thus find
  \begin{equation}\nonumber
  \frac{1}{\l\,\s_{max}^{n-1}}\le\Phi^L(x,\nu)\le\frac{\lambda}{\s_{min}^{n-1}}\,,\qquad\forall x\in\cl(L(A\cap H))\,,\nu\in \mathbf S^{n-1}\,.
  \end{equation}
  Similarly, setting for the sake of brevity $M=(\cof L)^{-1}$, and by taking into account that
  for every $x\in\cl(L(A\cap H))$, $\nu\in \mathbf S^{n-1}$, and $z,w\in\R^n$, one has
  \begin{eqnarray*}
    \nabla\Phi^L(x,\nu)\cdot z=\nabla\Phi(L^{-1}x,M\nu)\cdot (Mz)\,,
    \\
   \nabla^2\Phi^L(x,\nu)z\cdot w=\nabla^2\Phi(L^{-1}x,M\nu)(Mz)\cdot (M w)\,,
  \end{eqnarray*}
  we find that,
  for every $x,y\in \cl(L(A\cap H))$ and $\nu,\nu'\in \mathbf S^{n-1}$,
  \begin{eqnarray*}
  |\nabla\Phi^L(x,\nu)|&\le&\frac{\lambda}{\s_{min}^{n-1}}\,,
  \\
  \|\nabla^2\Phi^L(x,\nu)\|&\le&\lambda\, \Big(\frac{\s_{max}}{\s_{min}^2}\Big)^{n-1}\,,
  \\
  \|\nabla^2\Phi^L(x,\nu)-\nabla^2\Phi^L(x,\nu')\|&\le& \l\,\Big(\frac{\s_{max}}{\s_{min}^2}\Big)^{n-1}\,|\nu-\nu'|\,,
  \\
  |\Phi^L(x,\nu)-\Phi^L(y,\nu)|&\le& \frac{\ell}{\s_{min}^n}\,|x-y|\,,
  \\
  |\nabla\Phi^L(x,\nu)-\nabla\Phi^L(y,\nu)|&\le&\frac{\ell}{\s_{min}^n}\,|x-y|\,.
  \end{eqnarray*}
  Finally if $\nu\in \mathbf S^{n-1}$ and $e\in\R^n$ then, by  by \eqref{elliptic} and the \((-1)\) homogeneity of \(\nabla^{2}\Phi\),
  \begin{equation*}
  \begin{split}
    \nabla^2\Phi^L(x,\nu)e\cdot e&\ge\frac{1}{\l\,|M\nu|}\Big|Me-\Big(Me\cdot\frac{M\nu}{|M\nu|}\Big)\,\frac{M\nu}{|M\nu|}\Big|^2
    \ge\frac{\s_{min}^{n-1}}\l\,\bigg|M\bigg(e-\Big(Me\cdot\frac{M\nu}{|M\nu|^2}\Big)\,\nu\bigg)\bigg|^2
    \\
    &\ge\frac1\l\,\Big(\frac{\s_{min}}{\s_{max}^2}\Big)^{n-1}\Big|e-\Big(Me\cdot\frac{M\nu}{|M\nu|^2}\Big)\,\nu\Big|^2
    \ge\frac1\l\,\Big(\frac{\s_{min}}{\s_{max}^2}\Big)^{n-1}\,\big|e-(e\cdot\nu)\nu\big|^2\,,
  \end{split}
  \end{equation*}
where in the last inequality we have used that \(t\mapsto |e-t\nu|^2\) is minimized by \(t_*=e\cdot \nu\).
  \medskip

  \noindent {\it Step two}: Clearly, $L(E)\subset L(H)$, and one easily checks from the distributional definition of relative perimeter that $L(E)$ has locally finite perimeter in $L(A)$. Let now $G\subset L(H)$ with $G\Delta L(E)\cc V\cc L(A)$ for some open set $V$ with $\diam(V)<2s_0$. If we set $W=L^{-1}(V)$ and $F=L^{-1}(G)$, then $F\subset H$, $F\Delta E\cc W\cc A$ and $\diam(W)<\|(\nabla L)^{-1}\|\diam(V)<\|(\nabla L)^{-1}\|\,s_0<r_0$ (provided $s_0=r_0/\|(\nabla L)^{-1}\|$), so that
  \[
  \PHI(E;W\cap H)\le \PHI(F;W\cap H)+\La\,|E\Delta F|\,.
  \]
  By \eqref{change of variable linear} we find
  \begin{eqnarray*}
    \PHI^L(L(E);V\cap L(H))&\le& \PHI^L(G;V\cap L(H))+\La\,|L^{-1}(L(E)\Delta G)|
  \\
  &=&\PHI^L(G;V\cap L(H))+\La\,|\det L|^{-1}\,|L(E)\Delta G|\,,
  \end{eqnarray*}
  and this concludes the proof.
\end{proof}

\section{The $\e$-regularity theorem}\label{section eps regularity} This section is devoted to the proof of a boundary regularity criterion (Theorem \ref{thm epsilon}) formulated in terms of the smallness of a quantity known as spherical excess. We thus begin by introducing the relevant notation and definitions needed in the formulation of this criterion. Given an open set $A$ and an open half-space $H$ in $\R^n$, we consider a set $E\subset H$, of locally finite perimeter in an open set $A$, fix $x\in A\cap \pa H$ and $r<\dist(x,\pa A)$, and define the {\it spherical excess of $E$ at the point $x$, at scale $r$, relative to $H$}, by setting
\[
{\bf exc}^H(E,x,r)=\inf\Big\{\frac1{r^{n-1}}\int_{B_{x,r}\cap H\cap\pa^*E}\frac{|\nu_E-\nu|^2}2\,d\H^{n-1}:\nu\in \mathbf S^{n-1}\Big\}\,.
\]
Another useful notion of excess is that of cylindrical excess. Given $\nu\in \mathbf S^{n-1}$, we set $\q_\nu(y)=y\cdot\nu$, $\p_\nu(y)=y-(y\cdot\nu)\,\nu$ for every $y\in\R^n$, and let
\begin{eqnarray*}
\C_\nu(x,r)=\Big\{y\in\R^n:|\p_\nu(y-x)|<r\,,|\q_\nu(y-x)|<r\Big\}\,,
\\
\D_\nu(x,r)=\Big\{y\in\R^n:|\p_\nu(y-x)|<r\,,|\q_\nu(y-x)|=0\Big\}\,.
\end{eqnarray*}
With this notation at hand,
\begin{figure}
  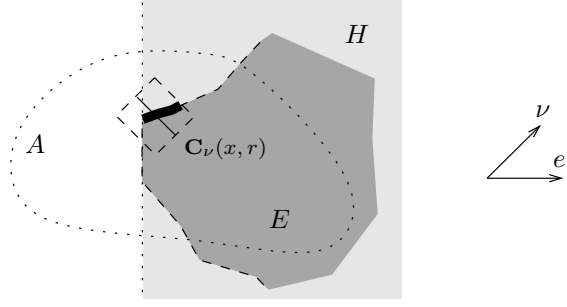\caption{{\small When computing $\exc_\nu^H(E,x,r)$, one considers only the part of the boundary of $E$ that is interior to $H$ (that is depicted as a bold line in this figure).}}\label{fig excess}
\end{figure}
the {\it cylindrical excess of $E$ at $x$, at scale $r$, in the direction $\nu$, relative to $H$}, is defined as
\[
\exc_\nu^H(E,x,r)=\frac1{r^{n-1}}\,\int_{H\cap\C_\nu(x,r)\cap\pa^*E}\frac{|\nu_E-\nu|^2}2\,d\H^{n-1}\,;
\]
see Figure \ref{fig excess}. If $\nu=e_n$, then we shall simply set $\exc_n^H$ in place of $\exc_{e_n}^H$. As usual, the definition is made so that the excess is invariant by scaling, precisely
\begin{equation}
  \label{excess scaling}
  \exc_\nu^H(E,x,r)=\exc_\nu^{H^{z,s}}\Big(E^{z,s},\frac{x-z}s,\frac{r}s\Big)=\exc_{\nu}^{H^{x,r}}(E^{x,r},0,1)\,.
\end{equation}
It is easily seen that, if $x\in\cl(H\cap\spt\mu_E)$ and $r>0$, then
\begin{eqnarray}\label{sferico nullo}
&&{\bf exc}^H(E,x,r)=0\quad\mbox{iff}\quad
\left\{\begin{array}
  {l}
  \mbox{there exist $\nu\in{\bf S}^{n-1}$ and $s\in\R$ such that}
  \\
  \mbox{$E\cap H\cap  B_{x,r}=\{x\cdot \nu<s\}\cap H \cap B_{x,r}$\,,}
\end{array}
\right .
\\\label{cilindrico nullo}
&&\exc_\nu^H(E,x,r)=0\quad\mbox{iff}\quad
\left\{\begin{array}
  {l}
  \mbox{there exists $s\in\R$ such that}
  \\
  \mbox{$E\cap H\cap \C_\nu(x,r)=\{x\cdot \nu<s\}\cap H \cap \C_\nu(x,r)$\,.}
\end{array}
\right .
\end{eqnarray}
 Finally, we recall that the normalized projection ${\bf e_1}(\nu)$ of $\nu\in{\bf S}^{n-1}$ (with $|\nu\cdot e_1|<1$) on $e_1^\perp$ was defined in \eqref{bf e1} by ${\bf e_1}(\nu)=(\nu-(\nu\cdot e_1))/\sqrt{1-(\nu\cdot e_1)^2}$, so that
\[
\nu^\perp=\Big\{x\in\R^n:x\cdot\nu=0\Big\}=\Big\{x\in\R^n:x\cdot{\bf e_1}(\nu)=-\frac{(\nu\cdot e_1)\,x_1}{\sqrt{1-(\nu\cdot e_1)^2}}\Big\}\,.
\]
and that, the normalization of Lemma \ref{lem:normalization} being in force, we have $A\cap\pa E=\spt\mu_E$ on almost-minimizers.

\begin{theorem}[$\e$-regularity theorem]\label{thm epsilon}
  For every $n\ge 2$ and \(\lambda\ge 1\)  there exist positive constants  $\ecc=\ecc(n,\l)$, $\beta_1=\beta_1(n,\l)\le \beta_2=\beta_2(n,\l)$ and $C=C(n,\l)$ with the following properties. If $H=\{x_1>0\}$,
  \begin{eqnarray*}\nonumber
  &&\Phi\in\X(B_{4\,r}\cap H,\lambda,\ell)\,,
  \\\nonumber
  &&\mbox{$E$ is a $(\La,r_0)$-minimizer of $\PHI$ in $(B_{4\,r},H)$ and  \(0<2\,r\le r_0\)}\,,
  \\\nonumber
  &&0\in\cl(H\cap\pa E)\,,
  \\\label{sferico piccolo}
  &&{\bf exc}^H(E,0,2\,r)+(\La+\ell)\,r\le \ecc\,,
  \end{eqnarray*}
  then \(M=\cl(\pa E\cap H)\cap B(0, \beta_1 r)\) is a \(C^{1,1/2}\) manifold with boundary, with
  \[
  M \cap \pa H=\pa_{\pa H} (\pa E \cap \pa H)\cap B(0, \beta_1 r)\,,
  \]
  and such that the anisotropic Young's law holds true on $M\cap\pa H$, i.e.
  \[
  \nabla \Phi\big(x,\nu_E(x))\cdot e_1=0\,,\qquad \forall\, x\in  M\cap \pa H\,.
  \]
  More precisely, there exist $\nu\in \mathbf S^{n-1}$ with
  \[
  \nabla \Phi(0,\nu)\cdot e_1=0\,,\qquad |\nu\cdot e_1|\le 1-\frac1C\,,
  \]
  and \(u\in C^{1,1/2}(\cl(\D_{\mathbf {e_1}\!(\nu)}(0,\beta_2\,r)\cap H))\) with
  \begin{equation*}
  \sup_{x,y\in\D_{\mathbf {e_1}\!(\nu)}(0,\beta_2\,r)\cap H}\frac{|u(x)|}r+|\nabla u(x)|+r^{1/2}\,\frac{|\nabla u(x)-\nabla u(y)|}{|x-y|^{1/2}}
  \le C\,\sqrt{\ecc}\,,
  \end{equation*}
  such that, $M$ is obtained, in a $\beta_2\,r$-neighborhood of $\nu^\perp$, as a perturbation of $\nu^\perp$ by $u$:
  \begin{multline}\label{allard1x}
  \Big\{x\in H\cap\pa E:|\p_{\mathbf {e_1}\!(\nu)}x|<\beta_2\,r\,,\,
  -\beta_2 r-\frac{(\nu\cdot e_1)\,x_1}{\sqrt{1-(\nu\cdot e_1)^2}}<
  \q_{\mathbf {e_1}\!(\nu)}\,x<
  \beta_2 r-\frac{(\nu\cdot e_1)\,x_1}{\sqrt{1-(\nu\cdot e_1)^2}}\Big\}
  \\
  =\Big\{x\in H: |\p_{\mathbf {e_1}\!(\nu)}x|<\beta_2\,r\,,\,\q_{\mathbf {e_1}\!(\nu)}\,x=-\frac{(\nu\cdot e_1)\,x_1}{\sqrt{1-(\nu\cdot e_1)^2}}+u\big(\p_{\mathbf {e_1}\!(\nu)}\,x\big)\Big\}\,;
  \end{multline}
  see
  \begin{figure}
    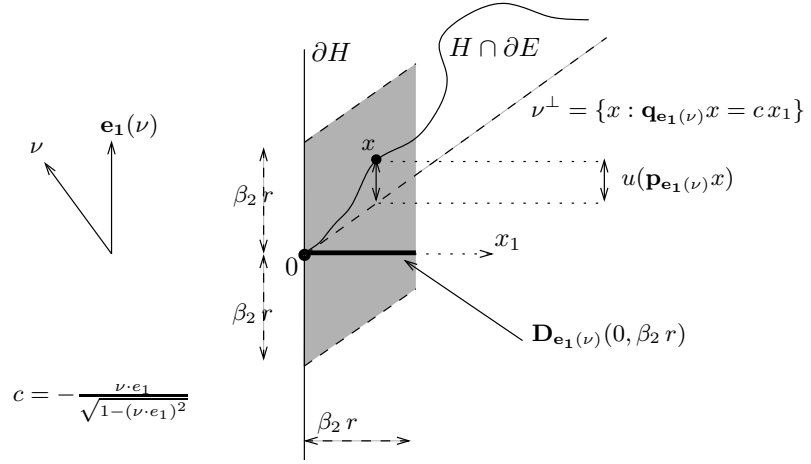\caption{{\small The situation in \eqref{allard1x}. The region (depicted in gray) where the graphicality of $H\cap\pa E$ is proved is obtained as a ``shear-strained'' deformation of $\C_{{\bf e_1}(\nu)}(0,\beta_2\,r)$, where the amount of vertical deformation depends on the coordinate $x_1$ only. Of course, the function $u$ parameterizing $H\cap\pa E$ depends on the full set of variables $\p_{{\bf e_1}(\nu)} x$, not just on $x_1$.}}\label{fig sei}
  \end{figure}
  Figure \ref{fig sei}.
  \end{theorem}

\begin{definition}[Boundary singular set]\label{definition singular set}
  If \(E\) is  a \((\Lambda,r_0)\)-minimizer of \(\PHI\) in \((A,H)\) (normalized as in Lemma \ref{lem:normalization}) and we set $M=A\cap\cl(H\cap\pa E)$, then the {\it boundary singular set $\S_A(E;\pa H)$ of $E$} is defined as the subset of $M\cap \pa H$ such that
  \begin{equation}
    \label{regularset}
  (M\cap\pa H)\setminus \S_A(E,\pa H)=\bigg\{x\in M\cap \pa H:
  \begin{array}{l}
  \textrm{there exists \(r_x>0\) such that \(M\cap B_{x,r_x}\)}
  \\
  \textrm{is a \(C^{1,1/2}\) manifold with boundary}
  \end{array}\bigg\}\,.
  \end{equation}
\end{definition}

\begin{remark}\label{rmk:singular set}
  By Theorem \ref{thm epsilon},
  \[
  \S_A(E;\pa H)=\Big\{x\in M\cap \pa H :\liminf_{r\to 0^+}{\bf exc}^H(E,x,r)>0\Big\}\,.
  \]
  This identity will be the starting point in section \ref{section singular set} to prove that $\H^{n-2}(\S(E;\pa H))=0$.
\end{remark}

The main step in the proof of Theorem \ref{thm epsilon} is proving the validity of the following lemma. (We will do this in section \ref{section proof of lemma eps}.)

\begin{lemma}\label{thm regularity en}
  For every \(\lambda\ge 1\)  there exist positive constants  $\er=\er(n,\l)$ and  \(C=C(n,\l)\) with the following properties. If $H=\{x\cdot e_1>0\}$,
  \begin{eqnarray*}
  &&\Phi\in\X(\C_{128\,r}\cap H,\lambda,\ell)\,,
  \\
  &&\mbox{$E$ is a $(\La,r_0)$-minimizer of $\PHI$ in $(\C_{128\,r},H)$ with \(0<64r\le r_0\)}\,,
  \\
  &&0\in\cl(H\cap\pa E)\,,
  \\
  &&\big|\nabla \Phi(0,e_n)\cdot e_1\big|+ \exc_n^H(E,0,64r)+(\L +\ell) r<\er\,,
  \end{eqnarray*}
  then there exists a function \(u\in C^{1,1/2}(\cl(\D_r\cap H))\) such that
  \begin{gather}
  \nonumber
  \C^+_{r}\cap \pa E=\Big\{x\in H: |\p x|<r\,,\q x=u(\p x)\Big\}\,,
  \\
  \label{allard4}
  \sup_{z,y\in\D_r\cap H}\frac{|u(z)|}r+|\nabla u(z)|+r^{1/2}\frac{|\nabla u(z)-\nabla u(y)|}{|z-y|^{1/2}}\le C\,\sqrt{\er}\,,
  \\
  \label{allard2}
  \nabla \Phi\big((z,u(z)),(-\nabla u (z),1)\big)\cdot e_1=0\,\qquad\forall z\in\D_{r}\cap\pa H\,.
  \end{gather}
 \end{lemma}

We now devote the remaining part of this section to show how to deduce Theorem \ref{thm epsilon} from Lemma \ref{thm regularity en}. The first step consists in showing that the smallness of the spherical excess implies the existence of a direction such that the cylindrical excess is small. Moreover, this direction is close to satisfy the anisotropic Young's law.

\begin{lemma}\label{lemma sferico cilindrico}
  For every $\l\ge 1$ and $\tau>0$ there exists $\esc=\esc(n,\tau)>0$ with the following property. If $H=\{x_1>0\}$, $\Phi\in\X(\C_{4r}\cap H,\l,\ell)$, and $E$ is a $(\La,r_0)$-minimizer of $\PHI$ on $(\C_{4r},H)$ with $0\in\cl(H\cap\pa E)$, $2\,r<r_0$, and
  \[
  {\bf exc}^H(E,0,2r)+(\La+\ell)\,r<\esc\,,
  \]
  then there exists $\nu\in \mathbf S^{n-1}$ with
  \[
  |\nabla\Phi(0,\nu)\cdot e_1|+\exc_{\nu}^H(E,0,r)+(\La+\ell)\,r<\tau\,.
  \]
\end{lemma}

\begin{remark}\label{remark continuity excess}
  The following continuity properties of the cylindrical excess are useful in the proof of Lemma \ref{lemma sferico cilindrico} (as well as in other arguments): if $\{E_h\}_{h\in\N}$ and $E$ are sets of locally finite perimeter in $A$, with $E_h\to E$ in $L^1_{\rm loc}(A)$ and
  \begin{equation}
    \label{compactness 00}
    |\mu_{E_h}|\llcorner H\weak|\mu_E|\llcorner H\,,\quad |\mu_{E_h}|\weak|\mu_E|\,,\quad\mbox{as Radon measures in $A$}\,,
  \end{equation}
  as $h\to\infty$,  then
  \begin{eqnarray}
    \label{compactness 6}
    &&\exc_\nu^H(E,x,r)\le\liminf_{h\to\infty}\exc_\nu^H(E_h,x,r)\,,\qquad\mbox{whenever $\C_\nu(x,r)\cc A$}\,,
    \\
    \label{compactness 7}
    &&\exc_\nu^H(E,x,r)=\lim_{h\to\infty}\exc_\nu^H(E_h,x,r)\,,\qquad
    \hspace{0.1cm}\begin{array}
      {l}
      \mbox{whenever $\C_\nu(x,r)\cc A$}
      \\
      \mbox{with $P(E;\pa\C_\nu(x,r))=0$}\,.
    \end{array}
  \end{eqnarray}
  Indeed, \eqref{compactness 6} follows from \eqref{compactness 7} by monotonicity, while $|\nu_E-\nu|^2/2=1-(\nu_E\cdot\nu)$ gives
  \[
  \exc_\nu^H(E,x,r)=\frac{|\mu_E|(\C_\nu(x,r)\cap H)-\nu\cdot\mu_E(\C_\nu(x,r)\cap H)}{r^{n-1}}\,,
  \]
  by which \eqref{compactness 7} is immediately seen to be consequence of \eqref{compactness 00}.
\end{remark}

\begin{proof}[Proof of Lemma \ref{lemma sferico cilindrico}]
  By Remark \ref{remark blow-up} and \eqref{excess scaling}, up to replace \(E\) and \(\Phi\) with \(E^{x_0,r}\) and \(\Phi^{x_0,r}\) respectively, we can directly assume that \(x_0=0\) and \(r=1\). We then argue by contradiction, and assume the existence  of $\tau_0>0$ and of sequences $\{\Phi_h\}_{h\in\N}\in\X(\C_{4}\cap H,\l,\ell_h)$ and $\{E_h\}_{h\in\N}$ such that $E_h$ is a $(\La_h,2)$-minimizer of $\PHI_h$ on $(\C_{4},H)$ with $0\in\cl(H\cap\pa E_h)$ for every $h\in\N$ and
  \begin{gather}\label{zurigo1}
  \lim_{h\to\infty}  {\bf exc}^H(E_h,0,2)+\La_h+\ell_h=0\,,
  \\\label{zurigo2}
  \inf_{\nu\in \mathbf S^{n-1}}\Big\{
  |\nabla\Psi_h(0,\nu)\cdot e_1|+\exc_{\nu}^H(E_h,0,1)\Big\}\ge\tau_0\,.
  \end{gather}
  By Theorem \ref{thm compactness minimizers} and \eqref{zurigo1}, there exist $\Phi_\infty\in\X_*(\l)$, and a $(0,2)$-minimizer $E_\infty$ of $\PHI_\infty$ on $(\C_{4},H)$ with $0\in\cl(H\cap\pa E_\infty)$, such that $\nabla\Phi_h(0,\nu)\to\nabla\Phi_\infty(\nu)$ uniformly on $\nu\in{\bf S}^{n-1}$, $E_h\to E_\infty $ in $L^1_{\rm loc}(\C_4)$, and $|\mu_{E_h}|$ and $|\mu_{E_h}|\llcorner H$ that converge, respectively, to $|\mu_{E_\infty}|$ and $|\mu_{E_\infty}|\llcorner H$, as Radon measures in $\C_4$ when $h\to\infty$. We can thus apply \eqref{compactness 6} and \eqref{zurigo1} to deduce that ${\bf exc}^H(E_\infty,0,2)=0$. By \eqref{sferico nullo}, there exist $\nu\in \mathbf S^{n-1}$ and $s\in\R$ such that
  \begin{equation}
    \label{sei iperpiano}
      E_\infty\cap B_2\cap H=B_2\cap\{x\cdot\nu<s\}\cap H\,,
  \end{equation}
  so that, in particular $\exc_{\nu}^H(E_\infty,0,1)=0$ and we can apply \eqref{compactness 7} to deduce that
  \[
  \lim_{h\to\infty}\exc_{\nu}^H(E_h,0,1)=0\,.
  \]
  By \eqref{zurigo2}, we conclude that
  \[
  |\nabla\Phi_\infty(\nu)\cdot e_1|\ge\tau_0\,.
  \]
  However, $0\in\cl(H\cap\pa E_h)$ for every $h\in\N$ and \eqref{kuratowski1} imply $0\in\cl(H\cap\pa E_\infty)$, so that $B_2\cap H\cap\pa E_\infty$ is non-empty, in particular \(\nu\ne \pm e_1\). Thanks to \eqref{sei iperpiano} we can apply Proposition \ref{proposition young} to conclude that $\nabla\Phi_\infty(\nu)\cdot e_1=0$. We thus reach a contradiction and complete the proof of the lemma.
  \end{proof}

  The second tool we need to deduce Theorem \ref{thm epsilon} from Lemma \ref{thm regularity en} is the following lemma.

  \begin{lemma}\label{lemma tauuu}
  For every $\tau\in[0,1)$, there exists a positive constant $C=C(\tau)$ with the following property. If $H=\{x_1>0\}$, $\nu\in \mathbf S^{n-1}$ and $|\nu\cdot e_1|\le\tau<1$, then there exists a linear map $L:\R^n\to\R^n$ such that $L(H)=H$,
  \begin{eqnarray}\label{BBQ2 t}
    &&L(\nu^\perp)=e_n^\perp\,,\qquad\mbox{so that}\qquad e_n=\frac{(\cof\nabla L)\nu}{|(\cof\nabla L)\nu|}\,,
    \\\label{cilindri spostati}
    &&L^{-1}(\C\cap H)=\Big\{y\in H:|y-(y\cdot{\bf e_1}(\nu)){\bf e_1}(\nu)|<1\,,
    \\\nonumber
    &&\hspace{4cm}
    -1-\frac{(\nu\cdot e_1)\,y_1}{\sqrt{1-(\nu\cdot e_1)^2}}< y\cdot {\bf e_1}(\nu)<
    1-\frac{(\nu\cdot e_1)\,y_1}{\sqrt{1-(\nu\cdot e_1)^2}}\Big\}\,,
  \end{eqnarray}
  see Figure \ref{fig cinque},
  \begin{figure}
    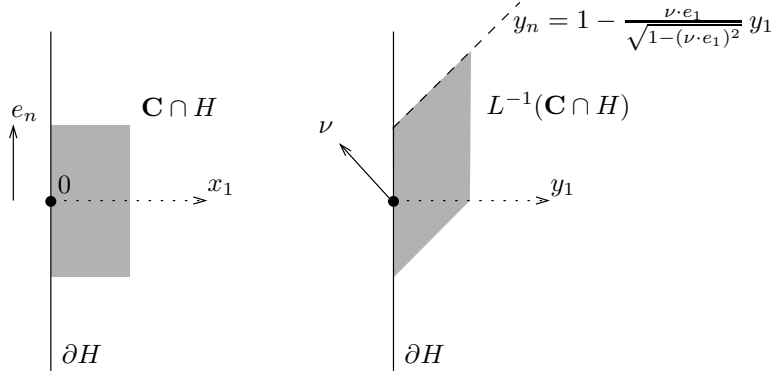\caption{{\small The image of $\C\cap H$ through $L^{-1}$. The picture refers to the situation when ${\bf e_1}(\nu)=e_n$. Notice that, in this case, the projection of $L^{-1}(\C\cap H)$ on $e_n^\perp$ is $\D\cap H$. In the general case, the projection of $L^{-1}(\C\cap H)$ over ${\bf e_1}(\nu)^\perp$ is $\D_{{\bf e_1}(\nu)}\cap H$.}}\label{fig cinque}
  \end{figure}
  and
  \begin{equation}\label{tau conclusione}
  \max\{\|\nabla L\|,\|(\nabla L)^{-1}\|\}\le C\,,\qquad\det\nabla L=1\,,\qquad\nabla\Phi^L(e_n)\cdot e_1=\nabla\Phi(\nu)\cdot e_1\,.
  \end{equation}
  whenever $\Phi$ is an autonomous elliptic integrand and $\Phi^L$ is defined by $\Phi^L(e)=\Phi((\cof\,\nabla L)^{-1}e)$. Moreover,
  \begin{equation}\label{LmenoId continuity}
  \|\nabla L-\Id\|\le C\,|\nu- e_n|\,.
  \end{equation}
  \end{lemma}

  \begin{proof} For some $|\a|\le \eta(\tau)<\pi/2$ we have
  \begin{eqnarray}\label{piccolo nu0 t}
  \nu=\cos\a\,{\bf e_1}(\nu)-\sin\a\,e_1\,, \qquad|\sin\,\a|=|\nu\cdot e_1|\le |\nu-e_n|\,.
  \end{eqnarray}
  We define a linear map $Q$ by setting $Q=\Id$ if ${\bf e_1}(\nu)=e_n$, and by setting $Q=\Id$ on $e_n^\perp\cap{\bf e_1}(\nu)^\perp$ and $Q$ to be the rotation taking ${\bf e_1}(\nu)$ into $e_n$ on ${\rm Span}(e_n,{\bf e_1}(\nu))$ otherwise. Finally, we define a linear map \(L:\R^n\to \R^n\) by setting
  \begin{equation}\label{defL t}
  \nabla L=Q\circ\Big(\Id-\tan\a\,{\bf e_1}(\nu)\otimes e_1\Big)\,.
  \end{equation}
  Trivially, $\det\nabla L=1$ and $\|\nabla L\|\le 1+|\tan\,\a|\le C(\tau)$. If $v\in \nu^\perp$, then $ {\bf e_1}(\nu)\cdot v=\tan\,\a\,(e_1\cdot v)$ and thus
  \[
  (Lv)\cdot e_n=\Big(v-\tan\a\,(e_1\cdot v)\,{\bf e_1}(\nu)\Big)\cdot Q^{-1}e_n=v \cdot {\bf e_1}(\nu)-\tan\,\a\,(e_1\cdot v)=0\,,
  \]
  so that \eqref{BBQ2 t} holds true. By noticing that,
  \[
  \nabla L^{-1}= \Big(\Id+\tan\a\,{\bf e_1}(\nu)\otimes e_1\Big)\circ Q^{-1}\,,
  \]
  we also have $\|(\nabla L)^{-1}\|\le C(\tau)$. By definition of $\Phi^L$, we see that
  \begin{equation}
    \label{mad1 t}
      \nabla\Phi^L(e_n)\cdot e_1=\nabla\Phi((\cof \nabla L)^{-1}e_n)\cdot\Big((\cof \nabla L)^{-1}e_1\Big)\,.
  \end{equation}
  Since $\nabla L^*=(\det \nabla L)\,(\cof \nabla L)^{-1}=(\cof \nabla L)^{-1}$ and $\nabla\Phi$ is zero-homogeneous, by \eqref{BBQ2 t} we find
  \begin{equation}
    \label{mad2 t}
  \nabla\Phi^L(e_n)\cdot e_1=\nabla\Phi(\nu)\cdot(\nabla L^*e_1)=\nabla\Phi(\nu)\cdot e_1\,,
  \end{equation}
  where we have used that $\nabla L^*e_1=e_1$, as it can be seen from \eqref{defL t}. This completes the proof of \eqref{tau conclusione}, while the validity of \eqref{cilindri spostati} is easily checked. To finally prove \eqref{LmenoId continuity}, we notice that $\|Q-\Id\|\le C\,|{\bf e_1}(\nu)-e_n|\le C\,|\nu-e_n|$ for some constant $C$, so that for every $e\in \mathbf S^{n-1}$ one has
  \[
  |\nabla L\,e-e|=|Qe-e-\tan\,a\,(e_1\cdot e)\,e_n|\le \|Q-\Id\|+|\tan\,\a|\le C\,|\nu-e_n|\,.\qedhere
  \]
  \end{proof}

  Finally, we estimate how cylindrical excess changes under transformation by an affine map.

  \begin{lemma}\label{lm:excess comparison} For every $\eta\ge 1$ there exists a constant $C=C(n,\eta)$  with the following property. If $H$ is an open half-space and \(L:\R^n\to \R^n\) is an affine transformation, with $L(H)=H$ and
\begin{equation}
  \label{ki}
  \|\nabla L\|\,,\|\nabla L^{-1}\|\le\eta\,,
\end{equation}
then for every set of finite perimeter \(E\) and every \(\nu\in \mathbf S^{n-1}\),
\begin{equation}\label{eq:excess comparison}
\exc^{H}_{\hat \nu} \bigg(L(E),Lx, \frac{r}{\sqrt{2}\,\eta}\bigg)\le C\,\exc^{H}_{\nu} (E,x,r)\qquad\textrm{where}\qquad
\hat \nu =\frac{(\cof \nabla L) \nu}{|(\cof \nabla L) \nu|}\,.
\end{equation}
\end{lemma}

\begin{proof}
  By \eqref{excess scaling} we can assume that $x=0$ and $r=1$, as well as that $L$ is linear. Correspondingly, we set $L$ in place $\nabla L$. In this way, by arguing as in \eqref{l-l+2}, we have
\begin{equation}
  \label{ki2}
  \frac{|e|}{\|L^{-1}\|^{n-1}}\le\s_{min}^{n-1}|e|\le|\cof L\,e|\le \s_{max}^{n-1}|e|=\|L\|^{n-1}|e|\,,\qquad\forall e\in\R^n\,.
\end{equation}
By \eqref{ki}, we find $\C_{\hat\nu}(0,1/{\sqrt{2}\eta})\subset L(\C)$, so that, if we set $M=\cof\,L$, then
\begin{eqnarray*}\label{tearancia}
\int_{\C_{\hat\nu}(0,1/{\sqrt{2}\eta})\cap L(H)\cap L(\pa^* E)}|\hat \nu-\nu_{L(E)}|^2\,d\H^{n-1}&\le&\int_{L(\C\cap H\cap \pa^* E)}|\hat \nu-\nu_{L(E)}|^2\,d\H^{n-1}
\\
&\le&\int_{\C\cap H\cap \pa^* E}\Big|\frac{M \nu }{|M \nu|}-\frac{M \nu_E }{|M \nu_E|}\Big|^2\,|M\nu_E|\,d\H^{n-1}\,.
\end{eqnarray*}
We thus find \eqref{eq:excess comparison} thanks to the fact that, by \eqref{ki2},
\[
\Big|\frac{M \nu }{|M \nu|}-\frac{M \nu_E }{|M \nu_E|}\Big|^2\,|M\nu_E|\le \Big(\frac{2|M\nu-M\nu_E|}{|M\nu_E|}\Big)^2|M\nu_E|\le 4\,\eta^{3(n-1)}\,|\nu-\nu_E|^2\,.\qedhere
\]
\end{proof}

  We now combine Lemma \ref{thm regularity en} (to be proved in section \ref{section proof of lemma eps}) with Lemma \ref{lemma sferico cilindrico}, Lemma \ref{lemma tauuu} and Lemma \ref{lm:excess comparison} to prove Theorem \ref{thm epsilon}.

  \begin{proof}
  [Proof of Theorem \ref{thm epsilon} (assuming Lemma \ref{thm regularity en})] Correspondingly to $\l\ge 1$, we can find $\e_*=\e_*(\l)$ such that, if we set
  \[
  \tau=\sqrt{1-\frac1{\l^4}}+\e_*\,\l^4\,,
  \]
  then $\tau\in(0,1)$. Let us now consider $\Phi\in\X(B_{4r}\cap H,\lambda,\ell)$, $E$ a $(\La,r_0)$-minimizer of $\PHI$ in $(B_{4r},H)$ with \(0<2\,r\le r_0\), $0\in\cl(H\cap\pa E)$, and
  \[
  {\bf exc}^H(E,0,2\,r)+(\La+\ell)\,r\le \ecc\,.
  \]
  If $\ecc(n,\l)\le\esc(n,\min\{\e_*(\l),\er(n,\l)\})$, then by Lemma \ref{lemma sferico cilindrico} there exists $\nu\in \mathbf S^{n-1}$ such that
  \[
  |\nabla\Phi(0,\nu)\cdot e_1|+\exc_{\nu}^H(E,0,r)+(\La+\ell)\,r<\min\{\e_*(\l),\er(n,\l)\}\,.
  \]
  By $|\nabla\Phi(0,\nu)|\ge1/\l$, we have $|\nabla\Phi(0,\nu)\cdot e_1|\le \l\,\e_*\,|\nabla\Phi(0,\nu)|$, and thus, by \eqref{elliptic boundary eps},
  \begin{equation}
    \label{tau fissato}
      |\nu\cdot e_1|\le \sqrt{1-\frac1{\l^4}}+\e_*\,\l^4=\tau(\l)<1\,.
  \end{equation}
  By Lemma \ref{lemma tauuu}, there exists a linear map $L:\R^n\to\R^n$ such that $L(H)=H$ and \eqref{BBQ2 t} and \eqref{tau conclusione} hold true.
  By Lemma \ref{lemma PHIL} and \eqref{tau conclusione}, if we set
  \[
  \Phi^L(x,\nu)=\Phi(L^{-1}x,(\cof \nabla L)^{-1}\nu)\,,
  \]
  then $\Phi^L\in\X(L(B_{4\,r})\cap H,\widetilde{\l},\widetilde{\ell})$ and $L(E)$ is a $(\La,\widetilde{r_0})$-minimizer of $\PHI^L$ on $(L(B_{4\,r}),H)$, where
  \[
  \widetilde{\l}\le C\,\l\,,\qquad \widetilde{\ell}\le C\,\ell\,,\qquad r_0\le C\,\widetilde{r_0}\,,\qquad B_{r/C}\subset L(B_{4r})\,,
  \]
  for a constant $C=C(\l)$. Moreover, by Lemma \ref{lm:excess comparison} and \eqref{BBQ2 t}, for some $\eta=\eta(\l)\ge 1$ and $C=C(n,\l)$, we have
  \[
  \exc^{H}_n \bigg(L(E),0, \frac{r}{\sqrt{2}\,\eta}\bigg)\le C\,\exc^{H}_{\nu} (E,0,r)\qquad\textrm{since}\qquad
  e_n=\frac{(\cof \nabla L) \nu}{|(\cof \nabla L) \nu|}\,,
  \]
  as well as, again by \eqref{tau conclusione},
  \[
  \nabla\Phi^L(0,e_n)\cdot e_1=\nabla\Phi(0,\nu)\cdot e_1\,.
  \]
  Summarizing, there exist positive constants $C_*=C_*(n,\l)$ and $C_{**}=C_{**}(n,\l)$ such that $\Phi^L\in\X(B_{r/C_*}\cap H,C_*\l,C_*\ell)$, $L(E)$ is a $(\La,r_0/C_*)$-minimizer of $\PHI^L$ on $(B_{r/C_*},H)$ with $0\in\cl(H\cap\pa L(E))$, and
  \[
  |\nabla\Phi^L(0,e_n)\cdot e_1|+\exc^{H}_n \bigg(L(E),0,\frac{r}{2C_*}\bigg)+(\La+C_*\,\ell)\,r\le C_{**}\,\ecc\,.
  \]
  If $C_{**}\,\ecc\le\er(n,C_*\,\l)$, then by Lemma \ref{thm regularity en} there exists a function \(u\in C^{1,1/2}(\cl(\D_{r/128\,C_*}\cap H))\) such that
  \begin{equation}\label{allard1xx}
    \C_{r/128\,C_*}\cap H\cap \pa L(E)=\Big\{x\in H: |\p x|<r\,,\q x=u(\p x)\Big\}\,,
  \end{equation}
  with
    \begin{gather*}
     \sup_{z,y\in\D_{r/128\,C_*}\cap H}\frac{|u(z)|}r+|\nabla u(z)|+r^{1/2}\,\frac{|\nabla u(z)-\nabla u(y)|}{|z-y|^{1/2}}
     \le C\,\sqrt{\er}\,,
  \\
    \nabla \Phi\big((z,u(z)),(-\nabla u (z),1)\big)\cdot e_1=0\,,\qquad\forall z\in\D_{r/128C_*}\cap\pa H\,.
  \end{gather*}
  for some $C=C(n,\l)$. By exploiting \eqref{cilindri spostati} and by applying $L^{-1}$ to the identity \eqref{allard1xx} we complete the proof of Theorem \ref{thm epsilon}.
\end{proof}

\section{Proof of Lemma  \ref{thm regularity en}}\label{section proof of lemma eps}
In this section we prove Lemma \ref{thm regularity en}. The argument is that commonly used in most proofs of $\e$-regularity criterions, and can be very roughly sketched as follows. Based on a {\it height bound} (Lemma \ref{lemma height bound}), one shows that locally at points with small excess it is possible to cover a large portion of the boundary with the graph of a Lipschitz function $u$, that is close to solve the linearized Euler-Lagrange equation of a suitable non-parametric functional (Lemma \ref{thm lip approx}). One then approximates $u$ with a solution of the Euler-Lagrange equation it approximately solves, transfers to $u$ the estimates that the exact solution enjoys by elliptic regularity theory, and then reads these estimates on the boundary of $E$ (Lemma \ref{thm tilt lemma}). An iteration of this scheme leads to prove that, at a sufficiently small scale, $u$ covers {\it all} of the boundary of $E$, and that it is actually of class $C^{1,1/2}$ by a classical integral criterion for h\"olderianity due to Campanato. For ease of presentation, we dedicate a separate section to each step of this long argument. Recall that thorough the proof, the normalization conventions of Lemma \ref{lem:normalization} are in force. In particular, we always have \(A\cap\pa E=\spt \mu_E\).

\subsection{Height bound}\label{section height bound} We start with the  height bound.

\begin{lemma}[Height bound]\label{lemma height bound} For every \(\l\ge 1\) and   $\s\in(0,1/4)$ there exists a positive constant $\eh=\eh(n,\l,\s)$ with the following property. If $H=\{x_1>b\}$ for some $b\in\R$, $x_0\in\cl(H)$, and
  \begin{eqnarray}
    \nonumber
    &&\Phi\in\X(\C_{x_0,4\,r}\cap H,\l,\ell)\,,
    \\
    \nonumber
    &&\mbox{$E$ is a $(\La,r_0)$-minimizer of $\PHI$ in $(\C_{x_0,4\,r},H)$ with \(0<r\le r_0\)}\,,
    \\
    \label{hb1}
    &&x_0\in\cl(H\cap\pa E)\,,
    \\\label{hb2}
    && (2\lambda \Lambda+\ell)\, r\le 1 \,,
    \\\label{hb3}
    &&\exc_n^H(E,x_0,2\,r)<\eh\,,
  \end{eqnarray}
  then
  \begin{eqnarray}
    \label{height bound}
        \sup\bigg\{\frac{|\q (x-x_0)|}{r}:x\in\C_{x_0,r_0}\cap H\cap\pa E \bigg\}\le \s\,,
       \\
   \label{height bound volume 1}
    \Big|\Big\{x\in \C_{x_0,r}\cap H\cap E:\q(x-x_0)>\s\,r\Big\}\Big|=0\,,
    \\\label{height bound volume 2}
    \Big|\Big\{x\in (\C_{x_0,r}\cap H)\setminus E:\q(x-x_0)<-\s\,r\Big\}\Big|=0\,.
  \end{eqnarray}
  Moreover, the identity
  \[
  \zeta(G)=P(E;\p^{-1}(G)\cap\C_{x_0,r}\cap H)-\H^{n-1}(G\cap H)\,,\qquad G\subset\D_{\p x_0,r}\,,
  \]
  defines a finite Radon measure $\zeta$ on $\D_{\p x_0,r}$ concentrated on \(H\cap   \D_{\p x_0, r}\) and such that
  \begin{equation}
    \label{zeta total mass excess}
    \zeta(\D_{\p x_0,r})=r^{n-1}\,\exc_n^H(E,x_0,r)\,.
  \end{equation}
\end{lemma}

\begin{proof}[Proof of Lemma \ref{lemma height bound}]
   The fact that \(\zeta\) is a positive Radon measure and satisfies \eqref{zeta total mass excess} follows by \eqref{height bound}, \eqref{height bound volume 1} \eqref{height bound volume 2} and Lemma \ref{misecc} below. We thus focus on the proof of these three properties. By \eqref{excess scaling} and by Remark \ref{remark blow-up}, up to replace \(E\) and \(\Phi\) with \(E^{x_0,r}\) and \(\Phi^{x_0,r}\) respectively, we may directly assume that \(x_0=0\), \(r=1\) and that \(H=\{x_1>-t\}\) with  \(t\ge 0\). Arguing by contradiction, we thus assume the existence of $\l\ge1$ and $\s\in (0,1/4)$ such that for every $h\in\N$ there exist an half-space $H_{h}=\{x_1> -t_h\}$ ($t_h\ge 0 $), and
   \begin{eqnarray*}
    \nonumber
    &&\Phi_h\in\X(\C_{4}\cap H_{h},\l,\ell_h)\,,
    \\
    \nonumber
    &&\mbox{$E_h$ a $(\La_h,1)$-minimizer of $\PHI_h$ in $(\C_{4},H_{h})$}
    \\
    &&0\in\cl(H_{h}\cap\pa E_h)\,,
    \\
    && 2\lambda \Lambda_h\, +\ell_h \le 1 \,,
   \end{eqnarray*}
   such that $\exc_n^{H_{h}}(E_h,0,2)\to 0$ as $h\to\infty$, and
     \begin{align}
    \label{buffa2}
    \textrm{either}&&\sup\Big\{|\q x|:x\in\C\cap H_{h}\cap\pa E_h\Big\}&>\s\,,
    \\
    \label{buffa2a}
     \textrm{or}&&  \Big|\Big\{x\in \C\cap H_h\cap\pa E_{h}:\q x>\s\,\Big\}\Big|&>0\,,
     \\
     \label{buffa2b}
     \textrm{or}&&  \Big|\Big\{x\in (\C\cap H_h)\setminus  E_h:\q x<-\s\,\Big\}\Big|&>0\,,
  \end{align}
  for infinitely many \(h\in \N\).

 \medskip

 \noindent {\it Step one}: We start  showing  that \eqref{buffa2} cannot hold for infinitely many \(h\in \N\). To this end, we set $H_0=\{x_1>0\}$, and notice that
 \[
 H_0\subset H_h=H_0-t_h\,e_1\,,\qquad\forall h\in\N\,.
 \]
 In order to apply the compactness theorem, Theorem \ref{thm compactness minimizers}, we need to get rid of the moving half-spaces $H_h$. Since $t_h\ge0$ for every $h\in\N$, up to extracting subsequences, we may assume that $t_h\to t_*\in[0,\infty]$ as $h\to\infty$. We then consider two cases separately:

  \medskip

  \noindent {\it Case one}: We have $t_*\in[0,5]$. Set $F_h=E_h+t_h\,e_1$ and $\Psi_h(x,\nu)=\Phi_h(x-t_he_1,\nu)$ so that $\Psi_h\in\X(\C(t_h\,e_1,4)\cap H,\l,1)$. Since $H_0=H_h+t_h\,e_1$, and, for $h$ large enough, $\C(t_*e_1,3)\cc \C(t_he_1,4)$, we find that $F_h$ is a $(1/2\l,1)$-minimizer of $\PSI_h$ on $(\C(t_*e_1,3),H)$ (recall that \(2\l \Lambda_h\le 1\)), with $t_h\,e_1\in\cl(H\cap\pa F_h)$ and (up to extracting a subsequence)
  \begin{gather}
  \nonumber
   \lim_{h\to\infty}\exc_n^H(F_h,t_h\,e_1,2)=0\,,
    \\
    \label{buffa2*}
   \sup\Big\{|\q x|:x\in\C(t_h\,e_1,1)\cap H\cap\pa F_h \Big\}\ge\s\,,\qquad\forall h\in\N\,.
  \end{gather}
  By Theorem \ref{thm compactness minimizers}, we can find $\Psi_\infty \in\X(\C(t_*e_1,3)\cap H,\l,1)$ and $F_\infty \subset H$ such that $F_\infty$ is a $(1/2\l,1)$-minimizer of $\PSI_\infty$ on $(\C(t_*e_1,3),H)$ with $t_*e_1\in\cl(H\cap\pa F_\infty)$. By \eqref{compactness 6},
  \[
  \begin{split}
  \exc_n^H(F_\infty,t_*e_1,3/2)&\le\liminf_{h\to\infty}\exc_n^H(F_h,t_*e_1,3/2)\\
  & \le\Big(\frac{4}{3}\Big)^{n-1}\liminf_{h\to\infty}\exc_n^H\Big(F_h,t_he_1,2\Big)=0\,.
  \end{split}
  \]
  Thus, by \eqref{cilindrico nullo},
  \begin{equation}
    \label{luis}
    F_\infty \cap\C(t_*e_1,3/2)\cap H=\Big\{x\in \C(t_*e_1,3/2)\cap H:\q x<0\Big\}\,.
  \end{equation}
  At the same time, by \eqref{buffa2*}, for every $h\in\N$ we can find $z_h\in \C(t_h\,e_1,1)\cap H\cap\pa F_h$ with $|\q z_h|>\s$. In particular, up to extracting subsequences and by \eqref{kuratowski1}, $z_h\to z_0$ for some $z_0\in \cl(\C(t_*\,e_1,1)\cap H\cap\pa F_\infty)$ with  $|\q z_0|\ge\s$. By \eqref{luis}, it must be $P(F_\infty;B_{z_0,s})=0$ for $s$ small enough: hence, $z_0\not\in\pa F_\infty$, contradiction.

  \medskip

  \noindent {\it Case two}: We have $t_*>5$. In this case the presence of $H_h$ is not detected by the minimality condition of $E_h$ in $\C_4$,   so that $E_h$ turns out to be a $(1/2\l,1)$-minimizer of $\PHI_h$ in $(\C_4,\R^n)$. This time we apply Theorem \ref{thm compactness minimizers} with the degenerate half-space $\R^n$ and we find a contradiction with \eqref{buffa2} by the same argument used in dealing with case one.

  \medskip

  \noindent {\it Step two}: We now prove that neither \eqref{buffa2a} nor \eqref{buffa2b} can hold for infinitely many \(h\). Once again, we argue by contradiction, and assume for example that \eqref{buffa2a} holds true for infinitely many values of $h$. Since we know by step one that
  \[
  \C\cap H_h\cap \pa E_h\subset \D\times [-\sigma, \sigma]\,,\qquad \forall h\in \N\,,
  \]
  this assumption, combined with basic properties of the distributional derivative, implies that
  \[
  \C\cap H_h\cap E_h\cap \{\q x>\sigma \}= \C\cap H_h\cap \{\q x>\sigma \}\,,\qquad \forall h\in\N\,.
  \]
  By exploiting the compactness theorem for almost-minimizers as in step one, we see however that \(E_h\) is converging to \(\{\q x<0\}\) (up to horizontal translations and inside of $\C$), thus reaching a contradiction.
\end{proof}

The following lemma can be proved by a simple application of the divergence theorem to vector fields of the form \(\vphi(\p x)\,e_n\) for \(\vphi\in C^1_c(\D_{x_0,r})\). We refer to \cite[Lemma 22.11]{maggiBOOK}, and leave the details to the reader.

\begin{lemma}\label{misecc} If \(H=\{x_1>b\}\) for some \(b\in \R\) and \(E\subset H\) is a set of finite perimeter with
\begin{gather*}
    |\q x|<\s_0\,,\qquad\forall x\in \C \cap H \cap\spt\mu_E\,,
    \\
    \Big\{x\in\C\cap H:\q x<-\s_0\Big\}\subset E\cap \C \cap H\subset\Big\{x\in\C \cap H:\q x<\s_0\Big\}\,,
  \end{gather*}
   for some \(\s_0\in (0,1)\), then the set function
  \[
  \zeta(G)=P(E;\p^{-1}(G)\cap\C\cap H)-\H^{n-1}(G\cap H)\,,\qquad G\subset\D\,,
  \]
  defines a positive finite Radon measure on \(\D\) concentrated on \(H\) with
  \[
  \zeta(G)=\int_{\p^{-1}(G)\cap\C\cap H\cap\pa^*E}\frac{|\nu_E-e_n|^2}2\,d\H^{n-1}\,,\qquad\forall G\subset\D\,.
  \]
  In particular, $\zeta (\D)=\exc_n^H(E,0,1)$.
\end{lemma}

\subsection{Lipschitz approximation} The next key ingredient in the proof of Lemma \ref{thm regularity en} is the construction of a Lipschitz approximation of \(\pa E\).

\begin{lemma}[Lipschitz approximation]\label{thm lip approx} For every \(\lambda\ge 1\) and   $\s\in(0,1/4)$ there exist positive constants  $\el=\el(n,\l,\s)$,  \(C_2=C_2(n,\l)\), and $\de_1 =\de_1(n,\l)$ with the following property. If $H=\{x_1>b\}$ for some $b\in\R$, $x_0\in\cl(H)$, and
  \begin{eqnarray*}
  &&\Phi\in\X(\C_{x_0,16\,r}\cap H,\lambda,\ell)\,,
  \\
  &&\mbox{$E$ is a $(\La,r_0)$-minimizer of $\PHI$ in $(\C_{x_0,16\,r},H)$ with \(0<4r\le r_0\)}\,,
  \\
  &&\mbox{$x_0\in\cl(H\cap\spt\,\pa E)$ }\,,
  \\
  &&(8\lambda \Lambda+4\ell)\, r\le 1\,,
  \\
  &&\exc_n^H(E,x_0,8r)<\el\,,
  \end{eqnarray*}
  then there exists a Lipschitz function $u:\R^{n-1}\to\R$ such that, on setting,
  \begin{eqnarray*}
  M&=&\C_{x_0,r}\cap H\cap\spt\pa E\,,
  \\
  M_0&=&\Big\{y\in M:\sup_{0<s<4r}\exc_n^H(E,y,s)\le\de_1\Big\}\,,
  \\
  \Gamma&=&\Big\{(z,u(z)):z\in\D_{\p x_0,r}\cap H\Big\}\,,
  \end{eqnarray*}
  we have
  \begin{equation}  \label{lapp4}
    \sup_{\R^{n-1}}\frac{|u-\q x_0|}r\le \s\,,\qquad\Lip(u)\le 1\,,\qquad M_0\subset M\cap\Gamma\,,
  \end{equation}
  and
  \begin{align}
  \label{lapp5}
  \frac{\H^{n-1}(M\Delta \Gamma)}{r^{n-1}}&\le C_2\,\exc_n^H(E,x_0,8\,r)\,,
  \\
  \label{lapp6}
  \frac1{r^{n-1}}\,\int_{\D_{\p x_0,r}\cap H}|\nabla u|^2&\le C_2\,\exc_n^H(E,x_0,8\,r)\,.
  \end{align}
  Finally, if
      \begin{eqnarray}
      \label{boundary case pde}
      \mbox{either}&\qquad&\mbox{$x_0\in\pa H$ and $\nabla\Phi(x_0,e_n)\cdot e_1=0$}\,,
      \\  \label{interior case pde}
      \mbox{or}&\qquad&\dist(x_0,\pa H)>r\,,
  \end{eqnarray}
 then we also have that
  \begin{equation}\label{eq:pde}
    \frac1{r^{n-1}}\int_{\D_{\p x_0,r}\cap H}\Big(\nabla^2\Phi(x_0,e_n)(\nabla u,0)\Big)\cdot(\nabla\vphi,0)\le C_2\,\|\nabla \vphi\|_\infty\,\Big(\exc_n^H(E,x_0,8\,r)+(\L+\ell)\,r\Big)\,,
  \end{equation}
 for every $\vphi\in C^1(\D_{\p x_0,r})$ with $\vphi=0$ on $H\cap\pa\D_{\p x_0,r}$. (Notice that this implies $\vphi=0$ on $\pa\D_{\p x_0,r}$ when \eqref{interior case pde} holds true.)
\end{lemma}

\begin{proof}
  {\it Step one}: By \eqref{excess scaling} and Remark \ref{remark blow-up} we may directly assume that \(x_0=0\), \(r=1\) and \(H=\{x_1>-t\}\) with \(t\ge 0\). With $\eh(n,\l,\s)$ defined as in Lemma \ref{lemma height bound}, we shall assume that
  \begin{equation}
    \label{lapp2}
      \el(n,\l,\s)\le \eh(n,\l,\s)\,.
  \end{equation}
  Since $E$ is a $(\Lambda, 4)$-minimizer of $\PHI$ in $(\C_{16},H)$, with $\Phi\in\X(\C_{16}\cap H,\l,\ell )$, $0\in\cl(H\cap\pa E)$, $\exc_n^{H}(E,0,8)<\el$, and $8\,\lambda \Lambda\, +4\,\ell\, \le 1$, by \eqref{lapp2} we can apply  Lemma \ref{lemma height bound} to get that
  \begin{equation}
    \label{lapp3}
    \sup\Big\{|\q y|:y\in \C_4\cap H\cap\pa E \Big\}\le\s\,;
  \end{equation}
  moreover,
  \[
  \zeta(G)=P(E;\p^{-1}(G)\cap\C_4\cap H)-\H^{n-1}(G\cap \D_4\cap H)\,, \quad G\subset\D_4\,,
  \]
  defines a positive  finite Radon measure $\zeta$ on $\D_4\cap H$. We now notice that if $y\in M_0$, $x\in M$ and $s=\max\{|\p(x-y)|,|\q(x-y)|\}$, then $s< 2$ and by definition of $M_0$, we have
  \[
  \C_{y,4s}\subset\C_{16}\,,\qquad \exc_n^H(E,y,2s)\le\de_1\,,\qquad (2\l\La+\ell)s\le1\,.
  \]
  Up to assume that
  \[
  \de_1(n,\l)<\eh\Big(n,\l,\frac18\Big)\,,
  \]
  and since, by construction, $x\in\cl(\C_{y,s}\cap\pa E)\cap H$, we can thus apply Lemma \ref{lemma height bound} at the point $y$ at scale $s$ to infer
  \[
  |\q y-\q x|\le \frac{s}8\,,
  \]
  which in turn implies
  \[
  |\q x-\q y|\le \frac{|\p x-\p y|}{8}\,.
  \]
 In particular, $\p$ is invertible on $M_0$, so that we can define a function $u:\p(M_0)\to\R$ with the property that $u(\p x)=\q x$ for every $x\in M_0$ (thus $|u(z)|<\s$ for every $z\in\p(M_0)$ by \eqref{lapp3}), and
  \[
  |u(\p y)-u(\p x)|\le \frac{|\p y-\p x|}{8}\,,\qquad \forall x,y\in M_0\,.
  \]
  We may thus extend $u$ as a Lipschitz function on $\R^{n-1}$ with the properties that
  \[
  \sup_{\R^{n-1}}|u|\le\s\,,\qquad\Lip(u)\le \frac18\,,\qquad M_0\subset\Gamma=\Big\{(z,u(z)):z\in \D\cap H\Big\}\,.
  \]
  This proves \eqref{lapp4}. We now prove \eqref{lapp5} by a standard application of Besicovitch's covering theorem (that we sketch in detail just for the sake of completeness). We start noticing that, if $x\in M\setminus M_0$, then there exists $s_x\in(0,4)$ such that
  \[
  \de_1\,s_x^{n-1}\le \int_{H\cap\C_{x,s_x}\cap \pa^*E}\frac{|\nu_E-e_n|^2}2\,d\H^{n-1}\,.
  \]
  If $\xi(n)$ is the Besicovitch covering constant, then (see for instance \cite[Corollary 5.2]{maggiBOOK}) we can find a countable disjoint family of balls $B(x_h,\sqrt{2}\,s_h)$ with $x_h\in M\setminus M_0$, $s_h=s_{x_h}\in(0,4)$,
  \[
  \de_1\,s_h^{n-1}\le \int_{H\cap\C_{x_h,s_h}\cap \pa^*E}\frac{|\nu_E-e_n|^2}2\,d\H^{n-1}\,,\qquad\forall h\in\N\,,
  \]
  and
  \[
  \H^{n-1}(M\setminus M_0)\le\xi(n)\sum_{h\in\N}\H^{n-1}\Big((M\setminus M_0)\cap B(x_h,\sqrt{2}\,s_h)\Big)\,.
  \]
  By \eqref{stime densita perimetro upper}, $\H^{n-1}(M\cap B(x_h,\sqrt{2}\,s_h))\le C(n,\l)\,s_h^{n-1}$ for every $h\in\N$, so that, by combining these last three inequalities,
  \begin{eqnarray*}
  \H^{n-1}(M\setminus M_0)&\le&\frac{C(n,\l)}{\de_1}\,\sum_{h\in\N}\int_{H\cap\C_{x_h,s_h}\cap \pa^*E}\frac{|\nu_E-e_n|^2}2\,d\H^{n-1}
  \\
  &=&\frac{C(n,\l)}{\de_1}\,\int_{H\cap\C_8\cap \pa^*E}\frac{|\nu_E-e_n|^2}2\,d\H^{n-1}\,,
  \end{eqnarray*}
  where in the last identity we have used the fact that the cylinders $\C_{x_h,s_h}$ are disjoint (as they are contained in the disjoint balls $B(x_h,\sqrt{2}\,s_h)$), as well as the fact that their union is contained in $\C_8$. Recalling that  $M\setminus\Gamma\subset M\setminus M_0$, we have proved that
  \[
  \H^{n-1}(M\setminus \Gamma)\le C(n,\lambda) \,\exc_n^{H}(E,0,8)\,.
  \]
  The proof of \eqref{lapp5} is then completed by noticing that, since $\Lip(u)\le 1$,
  \begin{equation*}
  \begin{split}
  \H^{n-1}(\Gamma\setminus M)&\le\sqrt{2}\H^{n-1}(\p(\Gamma\setminus M))
  \\
  &\le\sqrt{2}\H^{n-1}(M\cap\p^{-1}\p(\Gamma\setminus M))\le \sqrt{2} \H^{n-1}(M\setminus \Gamma)\,,
  \end{split}
  \end{equation*}
  where in the last inequality we have used  that
  \[
  0\le \zeta(\p(\Gamma\setminus M))=\H^{n-1}(M\cap\p^{-1}\p(\Gamma\setminus M))-\H^{n-1}(\p(\Gamma\setminus M)),
  \]
  and that $M\cap\p^{-1}\p(\Gamma\setminus M)\subset M\setminus\Gamma$. We finally prove \eqref{lapp6}. We first notice that
  \begin{equation}
    \label{hey}
      \nu_E(x)=\pm\,\frac{(-\nabla u(\p x),1)}{\sqrt{1+|\nabla u(\p x)|^2}}\,,\qquad\mbox{for $\H^{n-1}$-a.e. $x\in M\cap\Gamma$}\,.
  \end{equation}
  Since $|\nu_E-e_n|^2\ge |\p\nu_E|^2$ and $\Lip(u)\le 1$, by the area formula and by \eqref{hey} we get
  \begin{eqnarray*}
   8^{n-1} \exc_n^{H}(E,0,8)&\ge&\frac12\int_{M\cap\Gamma}|\p\nu_E|^2=\frac12\int_{M\cap\Gamma}\frac{|\nabla u(\p x)|^2}{1+|\nabla u(\p x)|^2}\,d\H^{n-1}
    \\
    &=&\frac12\int_{\p(M\cap\Gamma)}\frac{|\nabla u|^2}{\sqrt{1+|\nabla u|^2}}\ge\frac1{2\sqrt{2}}\int_{\p(M\cap\Gamma)}|\nabla u|^2\,,
  \end{eqnarray*}
  as well as $\int_{\p(M\Delta\Gamma)}|\nabla u|^2\le\H^{n-1}(\p(M\Delta\Gamma))\le\H^{n-1}(M\Delta\Gamma)$, so that \eqref{lapp6} follows from  \eqref{lapp5}. We now devote the next two steps of the proof to show the validity of \eqref{eq:pde}.

  \medskip

  \noindent {\it Step two}: We start showing that, setting  \(\Phi_0(\nu)=\Phi(0,\nu)\),
  \begin{equation}
    \label{hey-1}
      \Big|\int_{\C\cap H\cap\pa^*E}\nabla\Phi_0(\nu_E)\cdot(\nabla\vphi(\p x,0),0)\,(\nu_E\cdot e_n)\,d\H^{n-1}\Big|\le C\,\|\nabla\vphi\|_\infty\,(\La+\ell)\,,
  \end{equation}
  whenever $\vphi\in C^1(\D\cap H)$ with \(\vphi=0\) on \(\pa \D\cap H\). In showing this, we can assume without loss of generality that $\|\nabla\vphi\|_\infty=1$, and notice that, correspondingly, $\sup_{\D\cap H_*}|\vphi|\le 1$. Let us now fix $\a\in C^\infty_c([-1,1];[0,1])$ with $\a=1$ on $[0,1/2]$ and $|\a'|\le 3$, and consider the family of maps
  \[
  f_t(x)=x+t\,\a(\q x)\,\vphi(\p x)\,e_n\,,\qquad x\in\R^n\,.
  \]
  For $t$ small enough, $f_t$ is a diffeomorphisms of $\R^n$, with
  \begin{equation*}
  f_t(E)\subset H,\qquad f_t(E)\Delta E\cc \C_2\,,\qquad |f_t(E)\Delta E|\le C(n)\,|t|\,P(E;\C_2) \,,
  \end{equation*}
  (see \cite[Lemma 17.9]{maggiBOOK} for the last inequality). Hence, by minimality of $E$ and by \eqref{stime densita perimetro upper},
  \begin{equation}\label{hey0}
  \begin{split}
  \PHI(E,H\cap\C_2)&\le\PHI(f_t(E),H\cap\C_2)+\La\,|f_t(E)\Delta E|
  \\
  &\le \PHI(f_t(E),H\cap\C_2)+C(n,\l)\,\La\,|t|\,.
  \end{split}
  \end{equation}
  By \eqref{change of variables cofattore}, \eqref{Phi x ell}, \(|f_t(x)-x|\le C\,|t|\), and, again, by \eqref{stime densita perimetro upper}, we find
    \begin{equation}\label{caos}
    \begin{split}
  \PHI(f_t(E),H\cap\C_2)&=\int_{H\cap\C_2\cap \pa^*E}\Phi\big(f_t(x),(\cof \nabla f_t(x))\nu_E(x)\big)d\H^{n-1}\,
  \\
  &\le\int_{H\cap\C_2\cap \pa^*E}\Phi\big(x,(\cof \nabla f_t(x))\nu_E(x)\big)d\H^{n-1}+C(n,\l)\,\ell |t|\,.
  \end{split}
  \end{equation}
  By \eqref{sviluppini}, $(\cof \nabla f_t)\nu_E=\nu_E+t\,((\Div\,T)\,\nu_E-(\nabla T)^*\nu_E)$, with $T(x)=\a(\q x)\,\vphi(\p x)\,e_n$. Since  $\a(\q x)\equiv 1$ for $x$ in a neighborhood of $\pa^*E$, we have
  \[
  (\cof \nabla f_t)\nu_E=\nu_E-t\,(e_n\cdot\nu_E)\,(\nabla\vphi,0)\qquad  \textrm{on $\pa^*E$.}
  \]
 Thus, by \eqref{caos},
  \begin{equation}\label{hey1}
    \begin{split}
    \PHI(f_t(E),H\cap\C_2)&- \PHI(E,H\cap\C_2)\\
    &\le \int_{H\cap\C_2\cap \pa^*E}\big[\Phi(x,(\cof \nabla f_t)\nu_E)-\Phi(x,\nu_E)\Big] d\H^{n-1}+C(n,\l)\,\ell |t|\\
    &\le- t\int_{H\cap\C_2\cap \pa^*E}\nabla \Phi(x,\nu_E)\cdot (\nabla \phi,0)(e_n\cdot\nu_E)d\H^{n-1}+C(n,\l)\Big(\ell\,|t|+t^2\Big)\\
    &\le -t\,\int_{H\cap\C_2\cap\pa^*E}\nabla\Phi_0(\nu_E)\cdot(\nabla\vphi,0)\,(\nu_E\cdot e_n)\,d\H^{n-1}+C(n,\l)\Big(\ell\,|t|+t^2\Big)\,,
\end{split}
  \end{equation}
 where in the second inequality we have used  \eqref{Phi nabla 1} and that \(P(E,\C_2)\le C(n,\l)\)  (by  \eqref{stime densita perimetro upper}) while  in the third one we have used  \eqref{Phi x ell}  (and again that   \(P(E,\C_2)\le C(n,\l)\)). We combine \eqref{hey0} and \eqref{hey1} to find that
  \[
  \,\Big|\int_{H\cap\C_2\cap\pa^*E}\nabla\Phi_0(\nu_E)\cdot(\nabla\vphi,0)\,(\nu_E\cdot e_n)\,d\H^{n-1}\Big|
  \le C(n,\l)\,\Big(|t|+(\La+\ell)\Big)\,.
  \]
  We prove \eqref{hey-1} by  letting choosing \(t\to 0\).
  \medskip

  \noindent {\it Step three}: We conclude the proof of \eqref{eq:pde}. We start by showing that
    \begin{equation}
    \label{hey3}
      \Big|\int_{\D\cap H}\nabla\Phi_0((-\nabla u(z),1))\cdot(\nabla\vphi(z),0)\,dz\Big|\le C\,\|\nabla \vphi\|_{\infty}\,\Big(\exc_n^{H}(E,0,8)+(\La+\ell)\Big)\,,
  \end{equation}
  whenever $\vphi\in C^1(\D\cap H)$ with $\vphi=0$ on $H\cap\pa\D$. Indeed, let us set
  \[
  \Gamma_1=\bigg\{x\in\Gamma:\nu_E(x)=\frac{(-\nabla u(\p x),1)}{\sqrt{1+|\nabla u(\p x)|^2}}\bigg\}\subset \Gamma\,.
  \]
  By \eqref{hey}, if \(x\in \Gamma \setminus \Gamma_1\), then \(|\nu_E-e_n|^2\ge1\). Hence, by \eqref{lapp5},
  \begin{equation}
    \label{hey4}
     \H^{n-1}(M\Delta\Gamma_1)\le8^{n-1}\exc_n^{H}(E,0,8)+\H^{n-1}(M\Delta\Gamma)\le C\,\exc_n^{H}(E,0,8)\,,
  \end{equation}
  and thus, by \eqref{hey-1},
  \begin{equation}
    \label{hey5}
      \Big|\int_{H\cap\Gamma_1}\nabla\Phi_0(\nu_E)\cdot(\nabla\vphi(\p x),0)\,(\nu_E\cdot e_n)\,d\H^{n-1}\Big|\le C\,\|\nabla \vphi\|_ \infty\,\Big(\exc_n^{H}(E,0,8)+(\La+\ell)\Big)\,,
  \end{equation}
  where $C=C(n,\l)$. By taking into account the definition of $\Gamma_1$ and the formula for the area of a graph of a Lipschitz function we thus find
  \[
  \Big|\int_{H\cap\p(\Gamma_1)}\frac{\nabla\Phi_0((-\nabla u,1))\cdot(\nabla\vphi,0)}{\sqrt{1+|\nabla u|^2}}\Big|
  \le C\,\|\nabla \vphi\|_ \infty\,\Big(\exc_n^{H}(E,0,8)+(\La+\ell)\Big)\,.
  \]
  By \eqref{Phi nabla 1}, for every $G\subset\D\cap H$, we have
  \begin{equation*}
    \Big|\int_{G}\frac{\nabla\Phi_0((-\nabla u,1))\cdot(\nabla\vphi,0)}{\sqrt{1+|\nabla u|^2}}-
    \int_{G}\nabla\Phi_0((-\nabla u,1))\cdot(\nabla\vphi,0)\Big|
    \le C(n,\l)\,\|\nabla \vphi\|_ \infty\,\int_G|\nabla u|^2\,,
  \end{equation*}
  and thus by  \eqref{lapp6} we find
  \begin{equation}
    \label{hey6}
      \Big|\int_{H\cap\p(\Gamma_1)}\nabla\Phi_0((-\nabla u,1))\cdot(\nabla\vphi,0) \Big|\le C\,\|\nabla \vphi\|_ \infty\,\Big(\exc_n^{H}(E,0,8)+(\La+\ell)\Big)\,.
  \end{equation}
  At the same time, by \eqref{Phi nabla 1} and again by \eqref{hey4} we have
  \begin{eqnarray*}
          \Big|\int_{H\cap(\D\Delta\p(\Gamma_1))}\nabla\Phi_0((-\nabla u,1))\cdot(\nabla\vphi,0)\Big|
          &\le&
          \l\,\|\nabla \vphi\|_ \infty\,\H^{n-1}(M \Delta\Gamma_1)
          \\
          &\le& C\,\|\nabla \vphi\|_ \infty\,\exc_n^{H}(E,0,8)\,.
  \end{eqnarray*}
  We combine this last inequality with \eqref{hey6} to prove \eqref{hey3}. We now notice that, by \eqref{Phi nabla 1},
  \[
  |\nabla\Phi_0(-\nabla u,1)-\nabla\Phi_0(e_n)-\nabla^2\Phi_0(e_n)\cdot(-\nabla u,0)|\le C(\l)\,|\nabla u|^2\,.
  \]
 This, together with \eqref{lapp6} and \eqref{hey3}, gives
  \begin{equation}
    \label{hey7}
      \Big|\int_{\D\cap H}\Big(\nabla\Phi_0(e_n)+\nabla^2\Phi_0(e_n)(-\nabla u,0)\Big)\cdot(\nabla\vphi,0)\Big|\le C\,\|\nabla \vphi\|_ \infty\,\Big(\exc_n^{H}(E,0,8)+(\La+\ell)\Big)\,.
  \end{equation}
  We finally notice that, by Gauss--Green theorem,
  \begin{eqnarray*}
    \int_{\D\cap H}\nabla\Phi_0(e_n)\cdot(\nabla\vphi,0)
    =\nabla\Phi_0(e_n)\cdot\int_{H\cap\pa\D}\vphi\,\nu_{\D}-(\nabla\Phi_0(e_n)\cdot e_1)\int_{\D\cap\pa H}\vphi\,\,,
  \end{eqnarray*}
  where the first term vanishes as $\vphi=0$ on $H\cap\pa \D$, and the second term vanishes as either \eqref{boundary case pde} is in force (and thus \(\nabla \Phi_0(e_n)\cdot e_1=0\) by assumption) or \eqref{interior case pde} holds true, and then one simply has \(\D\cap \pa  H=\emptyset\). This completes the proof of \eqref{eq:pde}, thus of the lemma.
\end{proof}

\subsection{Caccioppoli inequality} The third tool used in the proof of Lemma \ref{thm regularity en} is the {\it Caccioppoli inequality} of Lemma \ref{thm caccioppoli} below. This result, also known as reverse Poincar\'e inequality, is morally analogous to its well-known counterpart in elliptic regularity theory, and shall be used here to translate decay estimates for the flatness of almost-minimizers into decay estimates for their excess. Here, given an open set $A$ and an open half-space $H$ in $\R^n$, a set $E\subset H$ of locally finite perimeter in $A$, and $x\in A\cap \cl(H)$, $\nu\in \mathbf S^{n-1}$ and $r>0$ such that $\C_\nu(x,r)\cc A$, we define the {\it flatness of $E$ at $x$, at scale $r$, in the direction $\nu$, relative to $H$} as
\[
\fl_\nu^H(E,x,r)=\inf_{c\in\R}\,\frac1{r^{n-1}}\int_{H\cap\C_\nu(x,r)\cap\pa^*E}\frac{|(y-x)\cdot\nu-c|^2}{r^2}\,d\H^{n-1}\,.
\]
As usual, we set $\fl_\nu^H(E,x,r)=\fl_n^H(E,x,r)$ when $\nu=e_n$, and notice that flatness enjoys analogous scaling properties to the one of excess, see \eqref{excess scaling}.

\begin{lemma}[Caccioppoli inequality]\label{thm caccioppoli}
 For every \(\lambda\ge 1\) there exist positive constants  $\ec=\ec(n,\l)$ and  \(C_3=C_3(n,\l)\) with the following property. If $H=\{x_1>b\}$ for some $b\in\R$, $x_0\in\cl(H)$, and
  \begin{eqnarray*}
  &&\Phi\in\X(\C_{x_0,16\,r}\cap H,\lambda,\ell)\,,
  \\
  &&\mbox{$E$ is a $(\La,r_0)$-minimizer of $\PHI$ in $(\C_{x_0,16\,r},H)$ with \(0<8r\le r_0\)}\,,
  \\
  &&\mbox{$x_0\in\cl(H\cap\pa E)$ }\,,
  \\
  && (16\,\lambda \Lambda+8\ell)\, r\le 1\,,
  \\
  &&\exc_n^H(E,x_0,8r)<\ec\,,
 \end{eqnarray*}
 with
    \begin{eqnarray}
      \label{boundary case reverse poincare}
      \mbox{either}&\qquad&\mbox{$x_0\in\pa H$ and $\nabla\Phi(x_0,e_n)\cdot e_1=0$}\,,
      \\  \label{interior case reverse poincare}
      \mbox{or}&\qquad&\dist(x_0,\pa H)>16\,r\,,
  \end{eqnarray}
  then
  \begin{equation}
    \label{rev poncare 0}
  \exc_n^H(E,x_0,r)\le C_3\,\Big(\fl_n^H(E,x_0,4r)+(\Lambda+\ell)\,r\Big)\,.
  \end{equation}
\end{lemma}


The proof is based on the construction of ``interior'' and ``exterior'' competitors in arbitrary cylinders, that is detailed in Lemma \ref{lemma FinFout} below, and originates from \cite{Almgren68}; see also  \cite[Section V]{bombieri} and \cite[Section 4]{DuzaarSteffen}. We shall need the following terminology and notation: first, we shall say that $E\subset\R^n$ is a polyhedron if $E$ is open and $\pa E$ is contained in finitely many hyperplanes (in this case, $\H^{n-2}(\pa E\setminus\pa^*E)=0$, and $\nu_E(x)$ agrees with the elementarily defined outer unit normal to $E$ at every $x\in\pa^*E$); second, given $z\in\R^{n-1}$ and $r>0$, we shall set
\begin{equation}\label{eq:defK}
\K_{z,r}=\Big\{x\in\R^n:|\p x-z|<r\,,|\q x|<1\Big\}\,.
\end{equation}
Referring to Figure \ref{fig finfout} for an illustration of the considered construction,
\begin{figure}
  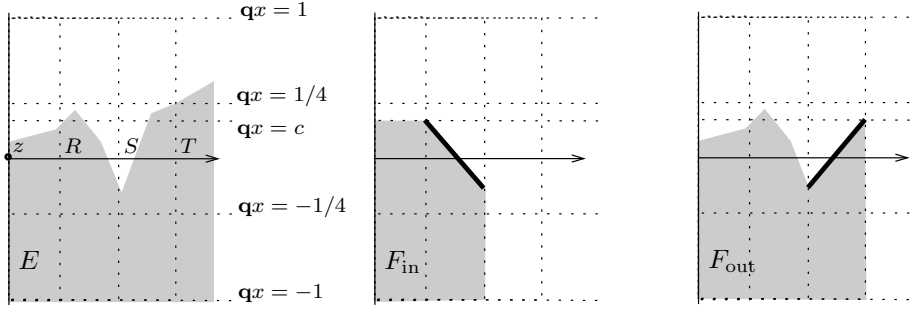\caption{\small{The construction of the the competitors $F_{\rm in}$ and $F_{\rm out}$ in Lemma \ref{lemma FinFout}. The picture is relative to the case when $z\in\pa H$. The bold lines represent the surfaces in the boundaries of $F_{\rm in}$ and $F_{\rm out}$ that are obtained by affine interpolation between $\{\q x=c\}\cap \pa\K_{z,R}$ and $\pa E\cap\pa\K_{z,S}$ in the case of $F_{\rm in}$, and between $\pa E\cap\pa\K_{z,S}$ and  $\{\q x=c\}\cap \pa\K_{z,T}$ in the case of $F_{\rm out}$.}}\label{fig finfout}
\end{figure}
we now state and prove the following lemma.

\begin{lemma}\label{lemma FinFout}
  If $H=\{x_1>b\}$ for some $b\in\R$, $0<R<S<T$, $|c|<1/4$, $z\in \R^{n-1} \cap H$, $E\subset H$ is a polyhedron, and
  \begin{gather}\label{E 1}
    \mbox{$|\nu_E(x)\cdot e_n|<1$ for every $x\in H\cap\K_{z,T}\cap\pa^*E$}\,,
    \\\label{E 2}
    |\q x|<\frac14\,,\qquad\forall x\in\K_{z,T}\cap H\cap\pa E\,,
    \\\label{E 3}
    \Big\{x\in\K_{z,T}\cap H:\q x<-\frac14\Big\}\subset E\cap \K_{z,T}\subset\Big\{x\in\K_{z,T}\cap H:\q x<\frac14\Big\}\,,
  \end{gather}
  and if
  \begin{equation}
    \label{hp su z}
      \mbox{either $z\in\pa H$ or $\dist(z,\pa H)>T$}\,,
  \end{equation}
  then there exist open sets of finite perimeter $F_{\rm in}$ and $F_{\rm out}$ such that $F_{\rm in}\subset \K_{z,S}\cap H$ with
  \begin{eqnarray}\label{Fin 1}
    H\cap\pa\K_{z,S}\cap \cl(F_{\rm in})&=&H\cap\pa\K_{z,S}\cap E\,,
    \\\label{Fin 2}
    H\cap\K_{z,R}\cap F_{\rm in}&=&H\cap\K_{z,R}\cap\{\q x<c\}\,,
  \end{eqnarray}
  \begin{equation}
  \label{Fin 3}
    \Big\{x\in \K_{z,S}\cap H:\q x<-\frac14\Big\}\subset F_{\rm in}\subset\Big\{x\in\K_{z,S}\cap H:\q x<\frac14\Big\}\,,
  \end{equation}
  and
  \begin{eqnarray}\label{perimetro Fin}
    P(F_{\rm in};\K_{z,S}\cap H)&\le& \H^{n-1}(\D_{z,R}\cap H)
    \\\nonumber
    &&+\frac{ S^{n-1}-R^{n-1}}{(n-1)S^{n-2}}\int_{H\cap\pa\K_{z,S}\cap\pa E}
    \sqrt{1+\Big(\frac{\q x-c}{S-R}\Big)^2}\,d\H^{n-2}\,,
  \end{eqnarray}
  while $F_{\rm out}\subset\K_{z,T}\cap H$ with
  \begin{eqnarray}
    \label{Fout 1}
  H\cap\pa\K_{z,T}\cap\cl(F_{out})&=&H\cap\pa\K_{z,T}\cap \{\q x<c\}\,,
    \\
        \label{Fout 2}
    H\cap\K_{z,S}\cap F_{\rm out}&=&H\cap\K_{z,S}\cap E\,,
  \end{eqnarray}
  \begin{equation}
      \label{Fout 3}
    \Big\{x\in \K_{z,T}\cap H:\q x<-\frac14\Big\}\subset F_{\rm out}\subset\Big\{x\in\K_{z,T}\cap H:\q x<\frac14\Big\}\,,
  \end{equation}
  and
  \begin{eqnarray}\label{perimetro Fout}
    P(F_{\rm out};\K_{z,T}\cap H)&\le& P(E;\K_{z,S}\cap H)
    \\\nonumber
    &&+\frac{T^{n-1}- S^{n-1}}{(n-1)\,S^{n-2}}\int_{H\cap\pa\K_{z,S}\cap\pa E}
    \sqrt{1+\Big(\frac{\q x-c}{T-S}\Big)^2}\,d\H^{n-2}\,.
  \end{eqnarray}
\end{lemma}

\begin{proof}
  Without loss of generality we may assume that $z=0$, and write $\K_r$ in place of $\K_{0,r}$ for every $r>0$. Moreover, we shall set $G^+=G\cap H$ for every $G\subset\R^n$. By \eqref{E 1} there exists a partition (modulo $\H^{n-1}$) of $\D_T^+$ by finitely many open Lipschitz sets $\{\Om_i\}_{i=1}^N$,
  \[
  \D_T^+=_{\H^{n-1}} \bigcup_{i=1}^N\Om_i\,,\qquad \H^{n-1}(\Om_i\cap\Om_{i'})=0\quad\mbox{if $1\le i<i'\le N$}\,,
  \]
  and finitely many affine functions $\{f_{i,j}\}_{j=1}^{N(i)}$ and $\{g_{i,j}\}_{j=1}^{N(i)}$ with $f_{i,j}<g_{i,j}<f_{i,j+1}$ for every $i$ and $j$, such that
  \begin{equation}
    \label{KT E}
      \K_T^+\cap E=_{\H^n}\bigcup_{i=1}^N\bigcup_{j=1}^{N(i)}\Big\{x\in \Om_i\times\R: f_{i,j}(\p x)<\q x<g_{i,j}(\p x)\Big\}\,.
  \end{equation}
  By \eqref{E 2} and \eqref{E 3} we have
  \[
  f_{i,1}=-1\,,\qquad g_{i,1}\ge-\frac{1}4\,,\qquad g_{i,N(i)}\le\frac14\,,\qquad\forall i=1,\dots,N\,,
  \]
  and, moreover,
  \[
  \K_T^+\cap\pa E=_{\H^{n-1}}\bigcup_{i=1}^N\bigcup_{j=2}^{N(i)}{\rm graph}(f_{i,j},\Om_i)\cup \bigcup_{i=1}^N\bigcup_{j=1}^{N(i)}{\rm graph}(g_{i,j},\Om_i)\,.
  \]

  \medskip

  \noindent {\it Construction of $F_{\rm in}$.} For every $i=1,\dots,N$, we can define an open set $\S_i$ such that
  \[
  \S_i=_{\H^{n-1}}\Big\{y\in(\D_S\setminus\D_R)^+:S\,\hat{y}\in\Om_i\cap(\pa\D_S)^+\Big\}\,,\quad\mbox{where}\quad \hat{y}=\frac{y}{|y|}\,.
  \]
  Thanks to  \eqref{hp su z},  $\{\S_i\}_{i=1}^N$ is a partition modulo $\H^{n-1}$ of $(\D_S\setminus\D_R)^+$.  We then define functions $f_{i,j}^*$ and $g_{i,j}^*$ on $\S_i$ by joining the values of $f_{i,j}$ and $g_{i,j}$ on $\Om_i\cap(\pa\D_S)^+$ to the constant value $|c|<1/4$ via affine interpolation along radial directions: precisely, we set
  \begin{eqnarray}
    &&f_{i,1}^*(y)=-1\,,
    \\
    &&f_{i,j}^*(y)=\frac{|y|-R}{S-R}\,f_{i,j}(S\,\hat{y})+\frac{S-|y|}{S-R}\,c\,,\qquad j=2,\dots,N(i)\,,
    \\
    &&g_{i,j}^*(y)=\frac{|y|-R}{S-R}\,g_{i,j}(S\,\hat{y})+\frac{S-|y|}{S-R}\,c\,,\qquad j=1,\dots,N(i)\,.
  \end{eqnarray}
  where \(\hat y=y/|y|\). Finally, we define
  \[
  F_{\rm in}=\Big(\K_R^+\cap\{\q x< c\}\Big)\cup
  \bigcup_{i=1}^N\bigcup_{j=1}^{N(i)}\Big\{x\in \S_i\times\R: f_{i,j}^*(\p x)<\q x<g_{i,j}^*(\p x)\Big\}\,.
  \]
  Trivially, $F_{\rm in}\subset\K_S^+$ is an open set of finite perimeter, and \eqref{Fin 1}, \eqref{Fin 2} and \eqref{Fin 3} hold true. In order to prove \eqref{perimetro Fin}, we start noticing that
  \[
  P(F_{\rm in};\K_S^+)=\H^{n-1}(\D_R^+)+\sum_{i=1}^N\sum_{j=2}^{N(i)}\int_{\S_i}\sqrt{1+|\nabla f_{i,j}^*|^2}
  +\sum_{i=1}^N\sum_{j=1}^{N(i)}\int_{\S_i}\sqrt{1+|\nabla g_{i,j}^*|^2}\,.
  \]
  Now, if $\vphi_{i,j}^S:(\pa\D_S)^+\to\R$ is defined as the restriction of $f_{i,j}$ to $(\pa\D_S)^+$, then
  for every $y\in \S_i$,
  \begin{eqnarray*}
  \nabla f_{i,j}^*(y)&=&\frac{|y|-R}{S-R}\,\frac{S}{|y|}\,\Big(\nabla f_{i,j}(S\hat{y})-\Big(\nabla f_{i,j}(S\hat{y})\cdot \hat{y}\Big)\hat{y}\Big)+\frac{f_{i,j}(S\,\hat{y})-c}{S-R}\,\hat{y}
  \\
  &=&\frac{|y|-R}{S-R}\,\frac{S}{|y|}\,\nabla_\tau\vphi^S_{i,j}(S\hat{y})+\frac{\vphi^S_{i,j}(S\,\hat{y})-c}{S-R}\,\hat{y}\,,
  \end{eqnarray*}
  where \(\nabla_\tau\) is the tangential gradient along \((\pa \D_S)^+\). Since  $\nabla_\tau\vphi^S_{i,j}(S\hat{y})\cdot\hat{y}=0$ for every $y\in\S_i$ and
  \[
  0\le \frac{|y|-R}{S-R}\,\frac{S}{|y|}\le 1\qquad \textrm{for \(R\le |y|\le S\),}
  \]
  we obtain that
  \[
  |\nabla f_{i,j}^*(y)|^2\le |\nabla_\tau\vphi^S_{i,j}(S\hat{y})|^2+\Big(\frac{\vphi^S_{i,j}(S\hat y)-c}{S-R}\Big)^2\qquad \forall\, y\in \S_i.
  \]
  Hence by the co area formula and   by the elementary inequality $\sqrt{1+a^2+b^2}\le\sqrt{1+a^2}\sqrt{1+b^2}$ we  find
  \begin{equation*}
  \begin{split}
  \int_{\S_i}\sqrt{1+|\nabla f_{i,j}^*|^2}&=\int_{R}^S\,dr\int_{\S_i\cap(\pa\D_r)^+}\,\sqrt{1+|\nabla f_{i,j}^*|^2}\,d\H^{n-2}
  \\
  &\le \frac1{S^{n-2}}\int_{R}^S\,r^{n-2}\,dr\int_{\S_i\cap(\pa\D_S)^+}\,\sqrt{1+|\nabla_\tau\vphi^S_{i,j}(S\hat{y})|^2+\Big(\frac{\vphi^S_{i,j}(S\hat y)-c}{S-R}\Big)^2}\,d\H^{n-2}
  \\
  &\le\frac{ S^{n-1}-R^{n-1}}{(n-1)S^{n-2}}\int_{\Om_i\cap(\pa\D_S)^+}
  \sqrt{1+|\nabla_\tau\vphi^S_{i,j}|^2}\sqrt{1+\Big(\frac{\vphi^S_{i,j}-c}{S-R}\Big)^2}\,d\H^{n-2}
  \\
  &=\frac{ S^{n-1}-R^{n-1}}{(n-1)S^{n-2}}\,\int_{{\rm graph}(f_{i,j})\cap(\pa\K_S)^+}
  \sqrt{1+\Big(\frac{x-c}{S-R}\Big)^2}\,d\H^{n-2}\,,
 \end{split}
  \end{equation*}
  where in the last step we have used the area formula.  Since similar inequalities apply to $g_{i,j}^*$, by \eqref{KT E} we deduce the validity of \eqref{perimetro Fin}.

  \medskip

  \noindent {\it Construction of $F_{\rm out}$.} In this case we use affine interpolation along radial directions above the annulus $(\D_T\setminus\D_S)^+$. More precisely, this time setting $\Gamma_i=\{y\in(\D_T\setminus\D_S)^+:S\,\hat{y}\in\Om_i\cap(\pa\D_S)^+\}$, we let, for every $y\in\Gamma_i$,
  \begin{eqnarray}
    &&f_{i,1}^{**}(y)=-1\,,
    \\
    &&f_{i,j}^{**}(y)=\frac{T-|y|}{T-S}\,f_{i,j}(S\,\hat{y})+\frac{|y|-S}{T-S}\,c\,,\qquad j=2,\dots,N(i)\,,
    \\
    &&g_{i,j}^{**}(y)=\frac{T-|y|}{T-S}\,g_{i,j}(S\,\hat{y})+\frac{|y|-S}{T-S}\,c\,,\qquad j=1,\dots,N(i)\,,
  \end{eqnarray}
  and correspondingly define
  \[
  F_{\rm out}=\Big(\K_S^+\cap E\Big)\cup
  \bigcup_{i=1}^N\bigcup_{j=1}^{N(i)}\Big\{x\in \Gamma_i\times\R: f_{i,j}^{**}(\p x)<\q x<g_{i,j}^{**}(\p x)\Big\}\,.
  \]
  One checks the validity of \eqref{Fout 1}, \eqref{Fout 2}, \eqref{Fout 3} and \eqref{perimetro Fout} by arguing as above.
\end{proof}

\begin{proof}[Proof of Lemma \ref{thm caccioppoli}]
  By \eqref{excess scaling} and Remark \ref{remark blow-up} we can assume that \(x_0=0\) and \(r=1\). By requiring that $\ec<\eh(n,\l,1/8)$, we can apply Lemma \ref{lemma height bound} to find that
  \begin{gather}\label{fagioli1}
    |\q x|<\frac18\,,\qquad\forall x\in \C_4\cap \pa E\,,
    \\\label{fagioli2}
    \Big\{x\in\C_4\cap H:\q x<-\frac18\Big\}\subset E\cap \C_4\cap H\subset\Big\{x\in\C_4\cap H:\q x<\frac18\Big\}\,,
  \end{gather}
  and to have that
  \begin{equation}\label{qfinito}
  \zeta(G)=P(E;\p^{-1}(G)\cap \C_4\cap H))-\H^{n-1}(G\cap H )\,,\qquad G\subset\D_4\,,
  \end{equation}
  defines a finite positive  Radon measure on $\D_4$, concentrated on $\D_4\cap H$, and such that
  \begin{eqnarray}
  \label{qfinito2}
  &&\zeta(\D_4)=4^{n-1}\,\exc_n^H(E,0,4)\le 2^{n-1}\,\ec\,,
  \\\nonumber
  &&\zeta(G)=\int_{\p^{-1}(G)\cap\C_4\cap H\cap\pa^*E}\frac{|\nu_E-e_n|^2}2\,,\qquad\forall G\subset\D_4\,.
  \end{eqnarray}
  We now divide the argument in two steps, setting $G^+=G\cap H$ for every $G\subset\R^n$.

  \medskip

  \noindent {\it Step one}: We prove that for every $\xi\in (1,2)$ there exist positive constants $C_*=C_*(n,\l,\xi)$ and $\theta_*=\theta_*(\xi)$  such that if $z\in \R^{n-1}$, $s>0$, $\D_{z,\xi\,s}\subset\D_4$ (i.e., $|z|+\xi\,s\le 4$), $|c|<1/4$ and either $z\in\pa H$ or $\dist(z,\pa H)>\xi\,s$, then
  \begin{eqnarray}\label{weak poincare bordox}
  && P(E;\K_{z,s}^+)-\H^{n-1}(\D_{z,s}^+)
  \le C_*\bigg\{\theta\Big(P(E;\K_{z,\xi\,s}^+)-\H^{n-1}(\D_{z,\xi\,s}^+)\Big)
  \\&&\hspace{6cm}+\frac1\theta\int_{\K_{z,\xi\,s}^+\cap\pa^*E}|\q x-c|^2\,d\H^{n-1}\bigg\}+C_*(\La+\ell)\,,
    \nonumber
  \end{eqnarray}
  for every $\theta\in(0,\theta_*)$. This follows by testing the minimality of $E$ against the competitors constructed in Lemma \ref{lemma FinFout} and by exploiting the ellipticity of $\Phi$ to compare these competitors with half-spaces through Proposition \ref{proposition young}.
  Precisely, with the end of exploiting Lemma \ref{lemma FinFout}, let us consider a sequence $\{E_k\}_{k\in\N}$ of open subsets of $H$ with polyhedral boundaries such that
  \begin{gather}
  \label{bojo1}
    \textrm{\(|\nu_{E_k}\cdot e_n|<1\) on $\pa E_k$}\,,
 \\
   \label{bojo2}
    \mbox{$E_k\to E$ in $L^1_{\rm loc}(\C_{16})$} \qquad \textrm{and}  \qquad  |\mu_{E_k}|\llcorner H\weak|\mu_E|\llcorner H\,,
   \\
   \label{bojo3}
    |\q x|<\frac14\,,\qquad\forall x\in \C_4\cap H\cap\pa E_k\,,
    \\
    \label{bojo4}
    \Big\{x\in\C_4\cap H:\q x<-\frac14\Big\}\subset E_k\cap \C_4\cap H\subset\Big\{x\in\C_4\cap H:\q x<\frac14\Big\}\,.
   \end{gather}
 Note that the existence of a sequence $\{E_k\}_{k\in\N}$ satisfying the above properties can be   obtained by  trivial modifications of the classical polyhedral approximation of sets of finite perimeter, see e.g. \cite[Theorem 13.8]{maggiBOOK}.  In  particular, since the normal of a polyhedron takes finitely many values,  \eqref{bojo1} can be  achieved by performing arbitrary small rotations.  Recalling that \(\D_{\xi,s}\subset \D_4\), by \eqref{qfinito},
 \[
 \begin{split}
 P(E;\K_{z, \xi \,s}^+\setminus \K_{z,s}^+ )-\H^{n-1}(\D_{z, \xi \,s}^+\setminus \D_{z,s}^+)&=\zeta(\D_{z, \xi \,s}\setminus \D_{z,s})
 \\
 &\le \zeta(\D_{z, \xi \,s})= P(E;\K_{z, \xi \,s}^+)-\H^{n-1}(\D_{z, \xi \,s}^+)\,;
 \end{split}
 \]
 moreover, by  \eqref{bojo3},  \eqref{bojo4} and Lemma \ref{misecc}, an analogous inequality holds true with \(E_k\) in place of $E$. By a slicing argument based on coarea formula, there exists $\a\in (1,\xi)$ with
  \begin{equation}
    \label{alpha posizione}
      1<\frac{2}{3}+\frac\xi 3<\a<\frac13+\frac{2\xi}{3}<\xi\,,
  \end{equation}
  such that, up to extracting a not relabeled subsequence in $k$, one has
  \begin{eqnarray}
  \label{poin1}
  &&\H^{n-2}((\pa\K_{z,\a s})^+\cap\pa E_k)-\H^{n-2}((\pa\D_{\a s})^+)\le \frac{C}s\,\Big( P(E_k;\K_{z,\xi \,s}^+)-\H^{n-1}(\D_{z,\xi \,s}^+)\Big)\,,\hspace{1.4cm}
  \\
  \label{poin2}
  &&\int_{(\pa\K_{z,\a s})^+\cap\pa E_k}(\q x-c)^2\,d\H^{n-2}\le \frac{C}s\,\int_{\K_{z,\xi\,s}^+\cap\pa E_k} (\q x-c)^2 \,d\H^{n-1}\,,
  \end{eqnarray}
  for every $k\in\N$ and for a suitable constant  $C=C(\xi)$,
  and
    \begin{eqnarray}
    \label{tende a zero poincare}
      &&\lim_{k\to\infty}\H^{n-1}\Big((E_k\Delta E^{(1)})\cap (\pa\K_{z,\a s})^+\Big)=0\,,
      \\
      \label{tende a zero poincare amica}
      &&P(E;(\pa\K_{\a\,s})^+)=0\,.
  \end{eqnarray}
  Next, we choose any $\theta_*=\theta_*(\xi)$ such that
  \[
  1<(1-\theta_*)\Big(\frac23+\frac\xi 3\Big)\,,\qquad \Big(\frac13+\frac{2\xi }3\Big)<\xi(1-\theta_*)\,,
  \]
  so to entail that,
  \[
  s<(1-\theta)\a\,s<\a\,s<\frac{\a s}{(1-\theta)}<\xi s\,,\qquad \forall\, \theta\in (0,\theta_*]\,.
  \]
  Finally, we set
  \begin{equation}\label{defRST}
  R=(1-\theta)\,\a s\,, \qquad S=\a\,s\,,\qquad T=\frac{\a s}{1-\theta},
  \end{equation}
  so that $s\le R< S< T< \xi s$.

  Since we are assuming that either $z\in\pa H$ or $\dist(z,\pa H)>\xi\,s$, by \eqref{bojo1}, \eqref{bojo3}, and \eqref{bojo4} we can apply Lemma \ref{lemma FinFout} to find sequences of sets $\{F^k_{\rm in}\}_{k\in\N}$ and $\{F^k_{\rm out}\}_{k\in\N}$ corresponding to the values of $R$, $S$ and $T$ defined in \eqref{defRST}. Let us now  define
  \[
  \widetilde F_k=(F^k_{\rm in}\cap\K_{z,S}^+)\cup(E\setminus\K_{z,S}^+)\,.
  \]
  By exploiting the minimality of \(E\), by \eqref{cap}, \eqref{cup}, and \eqref{minus} (that shall be repeatedly used in the sequel), and by taking into account  \eqref{Fin 1} and \eqref{Fin 2}, we thus find that for every $k\in\N$,
  \begin{eqnarray}\label{poinmini}
  \PHI(E;\K_{z,S}^+)\le  \PHI(F^k_{\rm in};\K_{z,S}^+)+\l\,\H^{n-1}\Big((E^{(1)}\Delta E_k)\cap(\pa\K_{z,S})^+\Big)+\La\,|E\Delta F^k_{\rm in}|\,,
  \end{eqnarray}
  where \eqref{Phi 1} was also taken into account. Since $|E\Delta F^k_{\rm in}|\le |K_T^+|\le |\C_4|$, and, by \eqref{bojo2}, \eqref{tende a zero poincare amica}, and Reshetnyak continuity theorem \cite[Theorem 1, section 3.4]{GMSbook2} $\PHI(E_k;\K_{z,S}^+)\to\PHI(E;\K_{z,S}^+)$ as $k\to\infty$, we deduce from \eqref{tende a zero poincare} and \eqref{poinmini} that
  \begin{equation}\label{poinmini3}
  \PHI(E_k;\K_{z,S}^+)\le\PHI(F^k_{\rm in};\K_{z,S}^+)+\e_k+C \La\,,
  \end{equation}
  where $C=C(n)$ and $\e_k\to0$ as $k\to\infty$. We now notice that by  \eqref{Fout 2} we can apply Lemma \ref{misecc} to $F^k_{\rm out}$ on the cylinder $\K_{z,T}^+$ to see that
  \[
  \zeta_k(G)=P(F^k_{\rm out};\p^{-1}(G)\cap\K_{z,T}^+)-\H^{n-1}(G\cap\D_{z,T}^+)\,,\qquad G\subset\D_{z,T}\,,
  \]
  defines a positive Radon measure on $\D_{z,T}$ with
  \begin{eqnarray}\nonumber
      P(E_k;\K_{z,S}^+)-\H^{n-1}(\D_{z,S}^+)&=&P(F^k_{\rm out};\K_{z,S}^+)-\H^{n-1}(\D_{z,S}^+)=\zeta_k(\D_{z,S}^+)
      \\\nonumber
      &\le&\zeta_k(\D_{z,T}^+)
      =\int_{\K_{z,T}^+\cap\pa F^k_{\rm out}}\frac{|\nu_{F^k_{\rm out}}-e_n|^2}2\,d\H^{n-1}
      \\\label{navierstokes2}
      &\le&\frac1{\k_1}\Big\{\PHI_0(F^k_{\rm out};\K_{z,T}^+)-\Phi_0(e_n)\,\H^{n-1}(\D_{z,T}^+)\Big\}\,.
  \end{eqnarray}
  where in the last inequality we have applied either \eqref{minimality of half-spaces lambda neumann} or \eqref{minimality of half-spaces lambda interna} (depending on whether $z\in\pa H$ and thus $\nabla\Phi(0,e_n)\cdot e_1=0$, or $\dist(z,\pa H)>\xi\,s>T$) to the autonomous elliptic integrand $\Phi_0(\nu)=\Phi(0,\nu)$. By  \eqref{Fout 2} and \eqref{tende a zero poincare amica} we find (with obvious notations)
  \begin{equation}\label{navierstokes3}
    \PHI_0(F^k_{\rm out};\K_{z,T}^+)=\PHI(E_k;\K_{z,S}^+)+\PHI(F^k_{\rm out};\K_{z,T}^+\setminus\K_{z,S}^+)+(\PHI_0-\PHI)(F^k_{\rm out};\K_{z,T}^+)\,.
  \end{equation}
  Thus, we may combine \eqref{bojo2}, \eqref{poinmini3}, \eqref{navierstokes2}, and \eqref{navierstokes3} to get
  \begin{equation}
  \begin{split}\label{ecci2}
  P(E;\K_{z,S}^+)&-\H^{n-1}(\D_{z,S}^+)
  \\
  &\le P(E_k;\K_{z,S}^+)-\H^{n-1}(\D_{z,S}^+)+\e_k
  \\
  &\le C\Big\{\PHI_0(F^k_{\rm out};\K_{z,T}^+)-\Phi_0(e_n)\,\H^{n-1}(\D_{z,T}^+)\Big\}+\e_k
  \\
  &\le C\Big\{\PHI(E_k;\K_{z,S}^+)+\PHI(F^k_{\rm out};\K_{z,T}^+\setminus\K_{z,S}^+)-\Phi_0(e_n)\,\H^{n-1}(\D_{z,T}^+)
  \\
  &\qquad+(\PHI_0-\PHI)(F^k_{\rm out};\K_{z,T}^+)+\e_k\Big\}
  \\
  &\le C\Big\{\PHI(F^k_{\rm in};\K_{z,S}^+)+\PHI(F^k_{\rm out};\K_{z,T}^+\setminus\K_{z,S}^+)-\Phi_0(e_n)\,\H^{n-1}(\D_{z,T}^+)\\
  &\qquad +(\PHI_0-\PHI)(F^k_{\rm out};\K_{z,T}^+)+\La+\e_k\Big\}
  \\
  &=C\Big\{\PHI(F^k;\K_{z,T}^+)-\Phi_0(e_n)\,\H^{n-1}(\D_{z,T}^+)+(\PHI_0-\PHI)(F^k_{\rm out};\K_{z,T}^+)+\La+\e_k\Big\}\,.
 \end{split}
  \end{equation}
  where \(\e_k\to 0\) as $k\to\infty$,  $C=C(n,\l)$, and  we have  set
  \begin{equation}\label{definizionefk}
  F^k=(F^k_{\rm in}\cap\K_{z,S}^+)\cup(F^k_{\rm out}\cap(\K_{z,T}^+\setminus\K_{z,S}^+))\,.
  \end{equation}
  (Notice that, thanks to \eqref{Fin 1}, \eqref{Fout 1}, and \eqref{tende a zero poincare amica} one has $P(F^k;(\pa\K_{z,S})^+)=0$.) By applying \eqref{minimality of half-spaces lambda neumann} or \eqref{minimality of half-spaces lambda interna}
to the set $F^k$, we see that
\begin{equation}\label{fioparma4}
\begin{split}
 \PHI(F^k;\K_{z,T}^+)&-\Phi_0(e_n)\,\H^{n-1}(\D_{z,T}^+)\\
 &=  \PHI_0(F^k;\K_{z,T}^+)-\Phi_0(e_n)\,\H^{n-1}(\D_{z,T}^+)+(\PHI-\PHI_0)(F^k;\K_{z,T}^+)
 \\
 &\le\kappa_2 \Big(P(F^k,\K_{z,T}^+)-\H^{n-1}(\D_{z,T}^+)\Big)+(\PHI-\PHI_0)(F^k;\K_{z,T}^+)\,.
\end{split}
\end{equation}
By combining \eqref{ecci2} and \eqref{fioparma4} and by recalling the definition of \(F^k\), \eqref{definizionefk}, we then obtain
\begin{equation}\label{sprudel}
\begin{split}
  P(E;\K_{z,S}^+)&-\H^{n-1}(\D_{z,S}^+)
  \\
  &\le C\Big\{\PHI(F^k;\K_{z,T}^+)-\Phi_0(e_n)\,\H^{n-1}(\D_{z,T}^+)+(\PHI_0-\PHI)(F^k_{\rm out};\K_{z,T}^+)+\La+\e_k\Big\}
  \\
  &\le C\Big\{\kappa_2 \big(P(F^k,\K_{z,T}^+)-\H^{n-1}(\D_{z,T}^+)\big)\\
  &\qquad+(\PHI-\PHI_0)(F_{\rm in}^k;\K_{z,S}^+)+(\PHI_0-\PHI)(E_k;\K_{z,S}^+)+\La+\e_k\Big\}\,.
  \end{split}
\end{equation}
Let us now notice that since \(\K_{z,S}\subset \C_4\), thanks to \eqref{Phi x ell}, \eqref{bojo2} and the density estimates \eqref{stime densita perimetro upper} we have
\begin{equation}\label{bonnbis1}
\big|(\PHI-\PHI_0)(E_k;\K_{z,S}^+)\big|\le C\ell P(E_k,\C^+_4)\le  C\ell P(E,\C^+_4)+\e_k\le C\ell +\e_k .
\end{equation}
Moreover,  since  \(\H^{n-1}(\D_{z,S}^+)\le C(n)\) (by  \(S\le \xi s\le 4\)),
\begin{equation}\label{bonnbis2}
\big|(\PHI-\PHI_0)(F_{\rm in}^k;\K_{z,S}^+)\big|\le C\ell P(F_{\rm in}^k,\K_{z,S}^+)\le   C\ell \big\{P(F_{\rm in}^k,\K_{z,S}^+)-\H^{n-1}(\D_{z,S}^+)\big\}+C\ell\,.
\end{equation}
Finally , by \eqref{Fin 1}, \eqref{Fout 1}, and \eqref{tende a zero poincare amica},
\begin{equation}
\label{infinegiunsero}
\begin{split}
  P(F^k,\K_{z,T}^+)-\H^{n-1}(\D_{z,T}^+)&=P(F^k_{\rm in},\K_{z,S}^+)-\H^{n-1}(\D_{z,S}^+)
  \\
  &+P(F^k_{\rm out},\K_{z,T}^+\setminus\K_{z,S}^+)-\H^{n-1}(\D_{z,T}^+\setminus \D_{z,S}^+)\,.
  \end{split}
\end{equation}
Hence, by combining \eqref{sprudel}, \eqref{bonnbis1}, \eqref{bonnbis2} and \eqref{infinegiunsero}  and taking into account that \(\ell \le 1\) and that both terms in the right hand side of \eqref{infinegiunsero} are non-negative (by Lemma \ref{misecc}), we obtain:
\begin{equation}\label{bonnbis3}
\begin{split}
  P(E;\K_{z,S}^+)&-\H^{n-1}(\D_{z,S}^+)
  \\
  &\le C\Big\{ P(F^k_{\rm in},\K_{z,S}^+)-\H^{n-1}(\D_{z,S}^+)\\
  &\qquad+ P(F^k_{\rm out},\K_{z,T}^+\setminus\K_{z,S}^+)-\H^{n-1}(\D_{z,T}^+\setminus \D_{z,S}^+)+(\L+\ell)+\e_k \Big\}.
\end{split}
\end{equation}
We now observe that   both in the case $z\in\pa H$ and in the case $\dist(z,\pa H)>\xi\,s$ we have
\begin{equation}\label{ovvia}
\begin{split}
\H^{n-1}(\D_{z,S}^+\setminus\D_{z,R}^+)&=\frac{ S^{n-1}-R^{n-1}}{(n-1)S^{n-2}}\H^{n-2}\big((\pa\D_{z,S})^+\big)
\\
   \H^{n-1}(\D_{z,T}^+\setminus\D_{z,S}^+)&=\frac{T^{n-1}- S^{n-1}}{(n-1)T^{n-2}}\H^{n-2}\big((\pa\D_{z,S})^+\big)\,.
\end{split}
\end{equation}
By \eqref{perimetro Fin}, by \eqref{ovvia} and since $\sqrt{1+t^2}\le1+t^2$, we find
\begin{equation*}
\begin{split}
   P(F_{\rm in}^k;\K_{z,S}^+)&-\H^{n-1}(\D_{z,S}^+)
    \\
    &\le\frac{S^{n-1}-R^{n-1}}{(n-1)S^{n-2}}\int_{(\pa\K_{z,S})^+\cap\pa E_k}
     1+\Big(\frac{\q x-c}{S-R}\Big)^2 \,d\H^{n-2}-\H^{n-1}(\D_{z,S}^+\setminus\D_{z,R}^+)
    \\
    &=
    \frac{ S^{n-1}-R^{n-1}}{(n-1)S^{n-2}}\bigg\{\int_{(\pa\K_{z,S})^+\cap\pa E_k}
     \Big(\frac{\q x-c}{S-R}\Big)^2 \,d\H^{n-2}\\
     &\qquad+\H^{n-1}((\pa\K_{z,S})^+\cap\pa E_k)-\H^{n-2}((\pa\D_{z,S})^+)\bigg\}
     \\
    &\le
    \frac{C(\xi)}s\,\frac{ S^{n-1}-R^{n-1}}{(n-1)S^{n-2}}\bigg\{\int_{\K_{z,\xi\,s}^+\cap\pa E_k}
     \Big(\frac{\q x-c}{S-R}\Big)^2 \,d\H^{n-1}+P(E_k;\K_{z,\xi\,s}^+)-\H^{n-1}(\D_{z,\xi\,s}^+)\bigg\}\,,
 \end{split}
\end{equation*}
where in the last inequality we have used \eqref{poin1} and \eqref{poin2}. Since $S=\a\,s$ and $R=(1-\theta)\a\,s$ we have that
\[
S-R=\a \theta s\qquad\textrm{and}\qquad \frac{S^{n-1}-R^{n-1}}{(n-1)S^{n-2}}\le C(n)\,\theta\,\a\,s \,.
\]
Thus, by also taking into account that $\a\in(1,2)$ by \eqref{alpha posizione}, we conclude that
\begin{multline}\label{fioparma1}
  P(F_{\rm in}^k;\K_{z,S}^+)-\H^{n-1}(\D_{z,S}^+)
     \\
    \le C\,\bigg\{\frac1\theta\int_{\K_{z,\xi\,s}^+\cap\pa E_k}
     \Big(\frac{\q x-c}{s}\Big)^2 \,d\H^{n-1}+\theta\Big(P(E_k;\K_{z,\xi\,s}^+)-\H^{n-1}(\D_{z,\xi\,s}^+)\Big)\bigg\}\,,
\end{multline}
for some $C=C(n,\xi)$. By an entirely similar argument we exploit \eqref{perimetro Fout}, \eqref{poin1}, and \eqref{poin2} to show that
\begin{multline}
\label{fioparma2}
P(F^k_{\rm out};\K_{z,T}^+\setminus \K_{z,S}^+)-\H^{n-1}(\D_{z,T}^+\setminus\D_{z,S}^+)
\\
 \le C\,\bigg\{\frac1{\theta}\int_{\K_{z,\xi\, s}^+\cap\pa E_k}
    \Big(\frac{\q x-c}{s}\Big)^2\,d\H^{n-1}
+\theta\,\Big(P(E_k; \K_{z,\xi\, s}^+)-\H^{n-1}(\D_{\xi\, s}^+)\Big)\bigg\}\,,
\end{multline}
for some $C=C(n,\xi)$.  By combining \eqref{bonnbis3},  \eqref{fioparma1} and \eqref{fioparma2} and by letting \(k\to \infty\), taking also into account \eqref{bojo2}, we finally get
\begin{multline}\label{fioparma6}
  P(E;\K_{z,S}^+)-\H^{n-1}(\D_{z,S}^+)
     \\
    \qquad \le   C\,\bigg\{\frac1\theta\int_{\K_{z,\xi\,s}^+\cap\pa E}
     \Big(\frac{\q x-c}{s}\Big)^2 \,d\H^{n-1}+\theta\Big(P(E;\K_{z,\xi\,s}^+)-\H^{n-1}(\D_{z,\xi\,s}^+)\Big)+(\La+\ell)\bigg\}\,,
\end{multline}
for some $C=C(n,\l,\xi)$. By \eqref{qfinito} and since $S>s$,
 the left-hand side of \eqref{fioparma6} is an upper bound to $P(E;\K_{z,s}^+)-\H^{n-1}(\D_{z,s}^+)$, so that \eqref{fioparma6} implies \eqref{weak poincare bordox}.
%

\medskip

\noindent {\it Step two}: We finally deduce \eqref{rev poncare 0} from the weaker inequality \eqref{weak poincare bordox} through a covering argument, see \cite{Simon2}. We start noticing that, as a consequence of \eqref{weak poincare bordox}, for every $\xi\in(1,2)$ there exist $C_*=C_*(n,\l,\xi)$ and $\theta_*=\theta_*(\xi)$ such that
 \begin{equation}\label{weak poincare bordo}
  s^2\,\zeta(\D_{z,s})\le C_*\,\Big\{\theta s^2\,\zeta(\D_{z,\xi\,s}) +\frac{h}{\theta}+ (\La+\ell)\Big\}\,,\qquad \forall \theta\in (0,\theta_*]\,,
 \end{equation}
  whenever $z\in \R^{n-1}$, $s>0$, $\D_{z,\xi\,s}\subset \D_{4}$ with either $z\in\pa H$ or $\dist(z,\pa H)>\xi\,s$,  \(\zeta\) is given by \eqref{qfinito} and
    \begin{equation}\label{eq:defh}
  h=\inf_{|c|<1/4}\int_{\C_4\cap H\cap\pa^* E}|\q x-c|^2\,d\H^{n-1}\,.
  \end{equation}
We now conclude the proof of the lemma under the assumption \eqref{boundary case reverse poincare}, that is $0\in\pa H$ and $\nabla\Phi(0,e_n)\cdot e_1=0$ (recall that we have set \(x_0=0\) and \(r=1\)). A simpler, analogous argument covers the case when \eqref{interior case reverse poincare} holds true. We start by showing that if $\xi\in(1,2)$ is sufficiently close to \(1\), then there there exist $C=C(n,\l,\theta)$ such that
\begin{equation}\label{vediamo}
s^2\zeta(\D_{z,s})\le C \Big\{\theta s^2\zeta(\D_{z,4s})+\frac{h}{\theta}+(\Lambda+\ell)\Big\}\,,
\end{equation}
for every $\theta\in (0,\theta_*]$ and \(\D_{z,4s}\subset \D_{4}\). Indeed, let $z\in\R^{n-1}$ and $s>0$ satisfy $\D_{z,4s}\subset \D_{4}$. If $\dist(z,\pa H)>\xi s$ then \eqref{vediamo} follows immediately from \eqref{weak poincare bordo} by the trivial inclusion $\D_{z,\xi\,s}\subset \D_{z,4\,s}$ (recall that $\xi<2$). If, instead, $\dist(z,\pa H)\le\xi s$, then we consider the projection $\bar z$ of $z$ on $\pa H$; since $\D_{z,s}\subset \D_{\bar z,(\xi+1)s}$, we have $\zeta(\D_{z,s})\le\zeta(\D_{\bar z,(\xi+1)\,s})$; at the same time, since $\bar z\in\pa H$ with $\D_{\bar z,\xi(\xi+1)\,s}\subset\D_{z,\xi(\xi+2)s}\subset\D_{z,4s}\subset\D_4$, we can apply \eqref{weak poincare bordox} at $\bar z$ at the scale $(\xi+1)\,s$, to conclude that
\[
s^2\zeta(\D_{z,s})\le
s^2\zeta(\D_{\bar z,(\xi+1)\,s})
\le
  C \Big\{\theta s^2\zeta(\D_{\bar z,(\xi^2+\xi)\,s})+\frac{h}{\theta}+(\Lambda+\ell)\Big\}\,,\qquad\forall \theta\in(0,\theta_*]\,.
\]
If we choose \(\xi\)  such that  $\xi(\xi+2)<4$, then $\D_{\bar z,(\xi^2+\xi)\,s}\subset \D_{z,(\xi^2+2\xi)\,s}\subset \D_{z,4s}$ and we deduce the validity of \eqref{vediamo}. Let us now define
  \[
  Q=\sup\big\{s^2\,\zeta(\D_{z,s}): \D_{z, 4\,s}\subset\D_{4}\big\}\,,
  \]
  so that $Q<\infty$ by \eqref{qfinito2}.  We now  notice  that, if $\D_{z,4\,s}\subset\D_2$, then there exists a family of points $\{z_k\}_{k=1}^{N(n)}\subset\D_{z,s}$ such that $\D_{z,s}\subset\bigcup_{k=1}^{N(n)}\D_{z_k,s/16}$. Since, trivially $\D_{z_k,s}\subset \D_4$, by applying \eqref{vediamo} at each $z_k$ at scale $s/16$ we find that
  \begin{eqnarray*}
    s^2\,\zeta(\D_{z,s})&\le&256\,\sum_{k=1}^{N(n)}\Big(\frac{s}{16}\Big)^2\zeta(\D_{z_k,s/16})
    \\
    &\le& C\,\sum_{k=1}^{N(n)} \bigg\{\theta \Big(\frac{s}{16}\Big)^2\zeta(\D_{z_k,s/4}) +\frac{h}{\theta}+(\La+\ell)\bigg\}
    \\
    &\le& C\Big\{\theta Q+\frac{h}{\theta}+(\La+\ell)\Big\}\,,\qquad \forall \theta\in (0,\theta_*]\,,
  \end{eqnarray*}
  where $C=C(n,\l,\xi)$. In other words,
  \[
  Q\le C\Big\{\theta Q+\frac{h}{\theta}+(\La+\ell)\Big\}\,,\qquad \forall \theta\in (0,\theta_*]\,.
  \]
  Keeping $\xi$ fixed, we choose  \(\theta<\theta_*\) such that $C\,\theta\le 1/2$ in order to conclude that
  \[
  Q\le 2C\,(h+(\La+\ell)\,r)\,.
  \]
  By recalling the definition of \(h\) (i.e., \eqref{eq:defh}), and by noticing that \(\D\) is admissible in the definition of \(Q\), we conclude that
  \begin{equation}
    \label{amen1}
      \exc_n^H(E,0,1)=\zeta (\D)\le  Q\le C\Big\{\inf_{|c|<1/4}\int_{{\C_4\cap H\cap\pa^* E}}|\q x-c|^2\,d\H^{n-1}+(\La+\ell)\Big\}\,.
  \end{equation}
  Finally, if $|c|\ge1/4$, then by \eqref{fagioli1}, \eqref{stime densita perimetro lower} and since \(\H^{n-1}((\C_{16}\cap\pa E)\setminus \pa^* E)=0\) by Lemma \ref{lem:normalization}, we find
  \[
  \int_{\C_4\cap H\cap \pa^* E}|\q x-c|^2\,d\H^{n-1}\ge  \frac{c_1 16^{n-1}}{8^2}\,.
  \]
  Hence, provided $\ec< c_1\,(16)^{n-1}/64$, we find
  \begin{equation}
    \label{amen2}
  \exc_n^H(E,0,1)\le 8^{n-1}\,\exc_n^H(E,0,8)\le  8^{n-1}\,\int_{\C_4\cap H\cap\pa ^* E}|\q x-c|^2\,d\H^{n-1}\,.
  \end{equation}
  We combine \eqref{amen1} and \eqref{amen2} to deduce \eqref{rev poncare 0}.
\end{proof}

\subsection{Tilt lemma} We now combine the results from the previous three sections to obtain the key estimate in the proof of Lemma \ref{thm regularity en}. Indeed, Lemma \ref{thm regularity en} will follow by an iterated application of the following lemma.

\begin{lemma}[Tilt lemma]\label{thm tilt lemma}
  For every \(\lambda\ge 1\) and \(\beta\in (0,1/64)\) there exist positive constants  $\et=\et(n,\l,\beta)$ and  \(C_4=C_4(n,\l)\) with the following properties. If $H=\{x_1>b\}$ for some $b\in\R$, $x_0\in\cl(H)$,
  \begin{eqnarray*}
  &&\Phi\in\X(\C_{x_0,16\,r}\cap H,\lambda,\ell)\,,
  \\
  &&\mbox{$E$ is a $(\La,r_0)$-minimizer of $\PHI$ in $(\C_{x_0,16\,r},H)$ with \(0<8r\le r_0\)}\,,
  \\
  &&\mbox{$x_0\in\cl(H\cap\pa E)$ }\,,
  \\
  &&\exc_n^H(E,x_0,8r)+(\L +\ell) r<\et\,,
  \end{eqnarray*}
  and
  \begin{eqnarray}
   \label{boundary case tilt}
   \mbox{either}&\qquad&\mbox{$x_0\in\pa H$ and $\nabla\Phi(x_0,e_n)\cdot e_1=0$}\,,
      \\  \label{interior case tilt}
      \mbox{or}&\qquad&\dist(x_0,\pa H)>r\,,
  \end{eqnarray}
  then there exists an affine map $L:\R^n\to\R^n$ with \(L x_0=x_0\) and \(L(H)=H\), such that
  \begin{eqnarray*}\label{L vicina id}
  \|\nabla L-\Id\|^2&\le& C_4\,\Big(\,\exc^H_n(E,x_0, 8r)+(\La+\ell)\,r\Big)\,,
  \\
  \label{tilt decay E}
  \exc^{H}_n(L(E),x_0,\beta \,r)&\le& C_4\,\Big(\beta^2\,\exc^H_n(E,x_0, 8r)+\beta\,(\La+\ell) r\Big)\,.
  \end{eqnarray*}
  Moreover, if we set as usual $\Phi^L(x,\nu)=\Phi(L^{-1}(x),(\cof \nabla L)^{-1}\nu)$, then $\Phi^L\in\X(\C_{x_0,2\beta\,r}\cap H,\widetilde{\l},\widetilde{\ell})$ and $L(E)$ is a $(\La,\widetilde{r_0})$-minimizer of $\PHI^L$ on $(\C_{x_0,2\beta r},H)$, where
  \begin{equation*}\label{nonvacosimale}
  \max\Big\{\frac{|\widetilde{\l}-\l|}\l,\frac{|\widetilde{\ell}-\ell|}\ell,\frac{|\widetilde{r_0}-r_0|}{r_0}\Big\}\le
  C_4\,\Big(\,\exc^H_n(E,x_0, 8r)+(\La+\ell)\,r\Big)^{1/2}\,
  \end{equation*}
  Finally,
  \begin{equation*}\label{restaneum}
   \nabla \Phi^L(x_0,e_n)\cdot e_1=0\,,\qquad \textrm{ if \eqref{boundary case tilt} holds.}
 \end{equation*}
 \end{lemma}

We premise the following lemma, usually known as a $A$-harmonic approximation lemma, that in our setting just amounts to a remark in the theory of constant coefficients elliptic PDEs.

\begin{lemma}\label{lemma pde stabilita}
  For every $\l\ge1$ and $\tau>0$ there exists a positive constant $\ea=\ea(\tau,\lambda)$ with the following property. If \(H=\{x_1>b\}\) for some $b\in \R$, $A\in\Sym(n)$ with $\l^{-1}\,\Id\le A\le \l\,\Id$ and $u\in W^{1,2}(\D\cap H)$ is such that
  \[
  \int_{\D\cap H}|\nabla u|^2\le 1\,,\qquad
    \int_{\D\cap H}(A\nabla u)\cdot\nabla\vphi\le \ea \|\nabla \vphi\|_\infty\,,
  \]
  for every $\vphi\in C^1(\D)$ with $\vphi=0$ on $H\cap \pa\D$, then there exists $v\in W^{1,2}(\D\cap H)$ such that
  \[
  \int_{\D\cap H}|u-v|^2\le\tau\,,\quad \int_{\D\cap H}|\nabla v|^2\le 1\,,\quad \int_{\D\cap H}(A\nabla v)\cdot\nabla\vphi=0\,,
  \]
 for every $\vphi\in C^1(\D\cap H)$ with $\vphi=0$ on $H\cap\pa\D$.
\end{lemma}

\begin{proof}[Proof of Lemma \ref{lemma pde stabilita}]
By contradiction; see, for example, \cite[Lemma 2.1]{duzaarmingioneharmonic}.
\end{proof}

\begin{proof}[Proof of Lemma \ref{thm tilt lemma}]
  {\it Step one}: By \eqref{excess scaling} and Remark \ref{remark blow-up}, we may reduce to prove the following statement (where $H=\{x_1\ge-t\}$ for some $t\ge 0$, and $G^+=G\cap H$ for every $G\subset\R^n$). If
  \begin{eqnarray}
  \nonumber
  &&\Phi\in\X(\C_{16}\cap H,\lambda,\ell)\,,
  \\\nonumber
  &&\mbox{$E$ is a $(\La,r_0)$-minimizer of $\PHI$ in $(\C_{16},H)$ with $r_0>8$}\,,
  \\ \nonumber
  &&\mbox{$0\in\cl(H\cap\spt\,\pa E)$ }\,,
  \\\label{gino4}
  &&\exc_n^H(E,0,8)+(\L +\ell) <\et\,,
  \end{eqnarray}
  and
  \begin{eqnarray}
   \label{boundary case tilt ris}
   \mbox{either}&\qquad&\mbox{$0 \in\pa H$ and $\nabla\Phi(0,e_n)\cdot e_1=0$}\,,\quad\hspace{1.5cm}\mbox{(boundary case)}
      \\  \label{interior case tilt ris}
      \mbox{or}&\qquad&\dist(0,\pa H)>1\,,\hspace{4cm}\quad\mbox{(interior case)}
  \end{eqnarray}
  then there exists a linear map $L:\R^n\to\R^n$ with \(L(H)=H\) such that
  \begin{eqnarray}\label{L vicina id ris}
  &&\|\nabla L-\Id\|^2\le C_4\,\Big(\,\exc^H_n(E,0, 8)+(\La+\ell)\Big)\,,
  \\
  \label{tilt decay E ris}
  &&\exc^{H}_n(L(E),0,\beta)\le C_4\,\Big(\beta^2\,\exc^H_n(E,0, 8)+\beta\,(\La+\ell) \Big)\,,
  \\
  \label{proprieta di PhiL}
  &&\Phi^L\in\X(\C_{2\beta}^+,\widetilde{\l},\widetilde{\ell})\,,
  \\
  \label{proprieta di LE}
  &&\mbox{$L(E)$ is a $(\La,\widetilde{r_0})$-minimizer of $\PHI^L$ on $(\C_{2\beta},H)$}\,,
  \\
  \label{nonvacosimale ris}
  &&\max\Big\{\frac{|\widetilde{\l}-\l|}\l,\frac{|\widetilde{\ell}-\ell|}\ell,\frac{|\widetilde{r_0}-r_0|}{r_0}\Big\}\le
  C_4\,\Big(\,\exc^H_n(E,0, 8)+(\La+\ell)\Big)^{1/2}\,,\hspace{1cm}
  \end{eqnarray}
  and
  \begin{equation}\label{restaneum ris}
   \nabla \Phi^L(0,e_n)\cdot e_1=0\qquad \textrm{ if \eqref{boundary case tilt ris} holds}\,.
  \end{equation}

  \medskip

  \noindent {\it Step two}: Given $\s\in(0,1/4)$, let us consider the constant \(\el=\el(n,\l,\sigma)\) determined by Lemma \ref{thm lip approx}. We shall work under the assumption that
  \begin{equation}
    \label{etilt 1}
      \et<\min\Big\{\el(n,\l,\sigma),\frac1{8\,\l}\Big\}\,.
  \end{equation}
  Of course, this will be compatible with $\et=\et(n,\l,\beta)$ as we shall fix (later on in the argument) a definite (sufficiently small) value of $\s$ depending on $n$, $\l$, and $\beta$ only. This said, by \eqref{gino4}, we can apply Lemma \ref{thm lip approx} to find a Lipschitz function $u:\R^{n-1}\to\R$ satisfying
\begin{equation}\label{u}
\begin{gathered}
\sup_{\R^{n-1}}|u|\le \s\,,\qquad\qquad\qquad \Lip(u)\le 1\,,
\\
  \H^{n-1}(M\Delta \Gamma) \le  C_2\,\exc_n^{H}(E,0,8)\,, \qquad \int_{\D^+}|\nabla u|^2 \le  C_2\,\exc_n^{H}(E,0,8)\,,
\end{gathered}
\end{equation}
where $C_2=C_2(n,\l)$ (note in particular that $C_2$ does not depend on $\s$), and
\begin{equation}\label{utile}
M=\C^+\cap\pa E\subset \D^+\times (-\sigma,\sigma)\,,\qquad \Gamma=\big\{(z,u(z)):z\in\D^+\big\}\,.
\end{equation}
Moreover, setting \(\Phi_0(\nu)=\Phi(0,\nu)\), we also know that
\begin{equation}
  \label{quasipde}
    \int_{\D^+}\Big(\nabla^2\Phi_0(e_n)(\nabla u,0)\Big)\cdot(\nabla\vphi,0)\le C_2\,\|\nabla \vphi\|_\infty\,\Big(\exc_n^{H}(E,0,8)+(\L+\ell)\Big)\,,
  \end{equation}
  for every $\vphi\in C^1(\D)$ with $\vphi=0$ on $(\pa\D)^+$. (Notice that $(\pa\D)^+$ is half $\pa\D$ in the boundary case, and it it actually coincides with the whole $\pa\D$ in the interior case.) Let us  set
 \begin{equation}\label{defchi}
   \chi=C_2\,\Big(\exc_n^H(E,0,8)+(\L+\ell)\Big)\le C_2\, \et\,
    \end{equation}
 and let us define
 \begin{equation}\label{defA}
 u_0=\frac{u}{\sqrt{\chi}}\,,\qquad
 A_{ij}=\nabla^2\Phi_0(e_n)e_i \cdot e_j \qquad i\,,j=1\dots, n-1\,.
 \end{equation}
 If we require $\et<1/C_2(n,\l)$, then $\chi<1$, while \(\l^{-1}\Id_{n-1}\le A\le \l \Id_{n-1}\) thanks to \eqref{Phi nabla 1} and \eqref{elliptic}. Moreover, by \eqref{u} and \eqref{quasipde},
  \begin{equation}\label{defchi2}
    \int_{\D^+}|\nabla u_0|^2\le 1\,,\quad \int_{\D^+}A\nabla u_0\cdot \nabla\vphi\le \|\nabla \vphi\|_\infty\,\sqrt{\chi}\,,
  \end{equation}
  for every $\vphi\in C^1(\D)$  with $\vphi=0$ on $(\pa\D)^+$.  Let us now introduce, in addition to $\s$, an additional parameter $\tau>0$ to be fixed later on depending on $n$, $\l$, and $\beta$ only. In this way it makes sense to require that $\et\le \ea(n,\l,\tau)/C_2(n,\l)$. Correspondingly, \eqref{defchi} and \eqref{defchi2} allows us to apply Lemma \ref{lemma pde stabilita} to find $v_0\in W^{1,2}(\D^+)$ with
  \begin{equation}\label{v0 risolve pde}
  \int_{\D^+}|\nabla v_0|^2\le 1\,,\qquad \int_{\D^+}(A\nabla v_0)\cdot\nabla\vphi=0\,,
  \end{equation}
  for every $\vphi\in C^1(\D^+)$ with $\vphi=0$ on $(\pa\D)^+$, and
  \begin{equation}\label{tau}
  \int_{\D^+}|u_0-v_0|^2\le\tau\,.
  \end{equation}
  By elliptic regularity, there exists a constant \( C= C(n,\lambda)>1\) such that if
  \begin{equation}\label{w_0}
  w_0(z)=v_0(0)+\nabla v_0 (0)\cdot z\,,\qquad z\in \D\,,
  \end{equation}
  is the tangent map to \(v_0\) at the origin,  then we have
 \begin{eqnarray}
 \label{bound}
 &&|\nabla w_0|\le C\,,
 \\
 \label{neum}
 &&(A\nabla w_0)\cdot e_1=0\,,\qquad\mbox{in the boundary case \eqref{boundary case tilt ris}}\,,
 \end{eqnarray}
 as well as, for every $s\le 1/2$,
 \begin{eqnarray}
   \label{height}
 |w_0(0)|^2\le \frac{C}{s^{n-1}} \int_{\D_s^+}|v_0|^2\,,\qquad
 \frac1{s^{n-1}}\int_{\D_s^+}\frac{|v_0-w_0|^2}{s^2}\le  C \,s^2\,.
 \end{eqnarray}
 Let us now set
  \begin{equation}\label{defwnu}
 v=\sqrt{\chi}v_0\,,\qquad w=\sqrt\chi\,w_0\,,\qquad \nu=\frac{(-\nabla w,1)}{\sqrt{1+|\nabla w|^2}}\,,\qquad c=\frac{w(0)}{\sqrt{1+|\nabla w|^2}}\,.
  \end{equation}
 By  \eqref{bound},  \(|\nabla w|\le C \sqrt \chi\) and, provided \(\et\) is sufficiently small (with respect to a constant depending on $n$ and $\l$ only), we have
   \begin{equation}\label{nuen}
  |\nu-e_n|\le C(n,\l)|\nabla w|\le   C(n,\l)\,\sqrt \chi\,.
 \end{equation}
 We now claim that, if \(|\nu-e_n|\le 1/4\) and we are in the boundary case \eqref{boundary case tilt ris}, then
  \begin{equation}\label{fondamentale}
  |\nabla\Phi_0(\nu)\cdot e_1|\le  C(n,\l)\,\chi\,,
  \end{equation}
  Indeed, by zero-homogeneity of $\nabla\Phi_0$, one has $\nabla\Phi_0(\nu)=\nabla \Phi_0(-\sqrt \chi\, \nabla w_0,1)$, and then  \eqref{boundary case tilt ris}, \eqref{defA}, \eqref{neum}, and a Taylor expansion (recall \eqref{Phi nabla 1}) imply
  \[
  \begin{split}
  |\nabla\Phi_0(\nu)\cdot e_1|
  &\le \big|\nabla \Phi_0(e_n)\cdot e_1+\big(\nabla^2 \Phi_0(e_n)(\sqrt \chi \,\nabla w_0,0)\big)\cdot e_1|+C|\sqrt \chi \,\nabla w_0|^2\\
  &=C|\sqrt \chi\, \nabla w_0|^2\le  C\, \chi\,,
    \end{split}
  \]
  for $C=C(n,\l)$. Up to further decrease the value of $\et$ depending on $n$ and $\l$ only, \eqref{fondamentale} enables us to apply Lemma \ref{lemma 0} to deduce that, if  we are in the boundary case \eqref{boundary case tilt ris}, then there exists $\nu_0\in \mathbf S^{n-1}$ such that
  \begin{eqnarray}\label{Phi e nu0 boundary case}
  &&\nabla\Phi_0(\nu_0)\cdot e_1=0\,,
  \\\label{Phi e nu0 contro nu}
  &&|\nu_0-\nu|<C(n,\l)\,\chi\,.
  \end{eqnarray}
  In the interior case, \eqref{interior case tilt ris}, we simply set $\nu_0=\nu$, so that \eqref{Phi e nu0 contro nu} holds true in both cases. We now notice that, by \eqref{tau} and \eqref{height}, if $s\le 1/2$, then, for $C=C(n,\l)$,
  \begin{equation}
    \label{eio}
      \frac1{s^{n-1}}\int_{\D_s^+}\frac{|u-w|^2}{s^2}\le
  C\,\Big(\frac1{s^{n-1}}\int_{D_s^+}\frac{|v-w|^2}{s^2}+\frac{\tau\,\chi}{s^{n+1}}\Big)\le C\,\Big(s^2+\frac{\tau}{s^{n+1}}\Big)\,\chi\,.
  \end{equation}
  By taking into account the definition of $c$ in \eqref{defwnu}, and thanks to \eqref{height} and \eqref{tau}, we find
 \begin{equation}\label{stimac}
 |c|^2 \le \chi |w_0(0)|^2\le C\,\chi\,\int_{\D_{1/2}^+}|v_0|^2
 \le C\,\Big(\chi\int_{\D_{1/2}^+}|v_0-u_0|^2+\int_{\D_{1/2}^+}|u|^2\Big\}\le C(\chi+\sigma)\,,
 \end{equation}
 for $C=C(n,\l)$, and where we have also taken into account that  \(u=\sqrt{\chi}u_0\) and  \(|u|\le \sigma\), as well as that $\s, \tau<1$. Moreover, by \eqref{defwnu} and \eqref{bound}, for some $C=C(n,\l)$ we find
 \begin{equation}\label{stimapezzaccio}
\sup_{x\in M\cup \Gamma} |x\cdot \nu|^2\le \sup_{x\in M\cup \Gamma}\big( |\p x|\,|\nabla w|+|\q x|\big)^2\le C(\sqrt \chi+\sigma)^2\le C(\chi+\s)\,,
 \end{equation}
where we have used that $\s,\chi<1$.  Finally, setting \(\K_s^+=\D_s^+\times(-1,1)\subset \C\) and using that we have both $\tau<1$ and $\chi<1$,
 \begin{eqnarray*}
     \frac1{s^{n-1}}\int_{\K_s^+\cap\pa^* E}\frac{|x\cdot\nu_0-c|^2}{s^2}
  &\le&\frac2{s^{n-1}}\int_{\K_s^+\cap\pa^* E}\frac{|x\cdot\nu-c|^2}{s^2}+\frac{4P(E;\C^+)|\nu-\nu_0|^2}{s^{n+1}}
   \\
  \small{\mbox{(by \eqref{stime densita perimetro upper} and \eqref{Phi e nu0 contro nu})}}\qquad
  &\le&\frac2{s^{n-1}}\int_{\K_s^+\cap\pa^* E}\frac{|x\cdot\nu-c|^2}{s^2}+C\,\frac{\chi^2}{s^{n+1}}
  \\
  &\le& \frac2{s^{n-1}}\int_{\K_{s}^+\cap\Gamma}\frac{|x\cdot\nu-c|^2}{s^2}+C\,\frac{\chi^2}{s^{n+1}}
  \\
  &&+ \frac{C}{s^{n+1}}\,\H^{n-1}(M\Delta \Gamma)\,\Big(|c|^2+\sup_{x\in M\cup \Gamma}|x\cdot \nu|^2\Big)
  \\  \small{\textrm{(by \eqref{u},  \eqref{stimac}, and \eqref{stimapezzaccio})}}\qquad
 &\le& \frac2{s^{n-1}}\int_{\K_{s}^+\cap\Gamma}\frac{|x\cdot\nu-c|^2}{s^2}+\frac{C}{s^{n+1}}\,\chi\,(\chi+\s)
 \\
 \small{\mbox{(by \eqref{defwnu})}}\qquad
  &=&
   \frac2{s^{n-1}}\int_{\D_{s}^+}\frac{|u-w|^2\,\sqrt{1+|\nabla u|^2}}{s^2\sqrt{1+|\nabla w|^2}}+\frac{C}{s^{n+1}}\chi\,\big(\chi+\sigma\big)
  \\
  \small{\mbox{(since \(\Lip (u)\le 1\))}}\qquad
  &\le& \frac{2\sqrt 2}{s^{n-1}}\int_{\D_{s}^+}\frac{|u-w|^2}{s^2}+\frac{C}{s^{n+1}}\chi\,\big(\chi+\sigma \big)
  \\
  \small{\mbox{(by \eqref{eio})}}\qquad
 &\le& C\,\chi\,\Big(s^2+\frac{\chi+\sigma +\tau}{s^{n+1}}\Big)\,.
 \end{eqnarray*}
 where $C=C(n,\l)$. We plug the value \( s=32\beta\le 1/2\) into this estimate so that, recalling the definition of \(\chi\) \eqref{defchi}, we get
 \[
 \frac1{\beta^{n-1}}\int_{\K_{32\,\beta}^+\cap\pa^* E}\frac{|x\cdot\nu_0-c|^2}{\beta^2}\le
 C(n,\l)\,\Big(\beta^2+\frac{\chi+\sigma +\tau}{\beta^{n+1}}\Big)\,\Big(\exc_n^H(E,0,8)+(\La+\ell)\Big)\,.
 \]
 If we first  choose $\s=\tau=\beta^{n+3}$ and  then $\et<\beta^{n+3}$, then the above estimate gives
 \begin{equation}\label{questax}
 \frac1{\b^{n-1}}\int_{\K_{32\beta}^+\cap\pa^* E}\frac{|x\cdot\nu_0-c|^2}{\beta^2}\le C(n,\l) \beta^2 \,\Big(\exc_n(E,0,8)+(\L+\ell)\Big)\,.
 \end{equation}
 We notice that, by \eqref{nuen} and \eqref{Phi e nu0 contro nu}, one has
 \begin{equation}
   \label{nu0en vicini}
    |\nu_0-e_n|\le C(n,\l)\,\sqrt{\chi}\,.
 \end{equation}
  We now use  Lemma \ref{lemma tauuu} to construct the map \(L\). More precisely, \eqref{nu0en vicini} ensures that  $|\nu_0\cdot e_1|\le1/2$, provided $\et$ is small enough. Then we can apply Lemma \ref{lemma tauuu} to $\nu_0$ to construct a linear map $L:\R^n\to\R^n$ with $L(H)=H$  and
 \begin{eqnarray}\label{BBQ2}
  &&L(\nu_0^\perp)=e_n^\perp\,,\quad\mbox{so that}\quad e_n=\frac{(\cof\nabla L) \nu_0}{|(\cof \nabla L) \nu_0|}\,,
  \\\label{BBQ 10}
  &&\nabla\Phi^L(0,e_n)\cdot e_1=\nabla\Phi_0(\nu_0)\cdot e_1\,,
  \\
  \label{BBQ1}
  &&\|\nabla L-\Id\|\le C(n,\l)\,\sqrt\chi\,,\qquad\det\,\nabla L=1\,.
  \end{eqnarray}
Thanks to Lemma \ref{lemma PHIL}, $\Phi^L\in\X(L(\C_{16})\cap H,\widetilde{\l},\widetilde{\ell})$, $L(E)$ is a  $(\La,\widetilde{r_0})$-minimizer of $\PHI^L$ on $(L(\C_{16}),H)$ and
  \begin{equation}\label{dai}
  \max\Big\{\frac{|\widetilde{\l}-\l|}\l,\frac{|\widetilde{\ell}-\ell|}{\ell},\frac{|\widetilde{r_0}-r_0|}{r_0}\Big\}\le
  C(n,\l)\,\sqrt{\chi}\,.
  \end{equation}
  Note that \(\widetilde \Lambda=\Lambda\) in the application of  Lemma \ref{lemma PHIL}, since  \(\det\,\nabla L=1\). This proves  \eqref{proprieta di PhiL}, \eqref{proprieta di LE}, and \eqref{nonvacosimale ris}. (Indeed $\C_{2\beta}\subset L(\C_{16})$ as, trivially, $\C_{2\beta}\subset\C$, and as one can make $L$ close enough to the identity to ensure $\C\subset L(\C_{16})$ by \eqref{BBQ1} and up to further tuning the value of $\et$.) Also, \eqref{BBQ 10} implies \eqref{restaneum ris} when we are in the boundary case. We are thus left to prove \eqref{tilt decay E ris}. Up to further decrease $\et$, by \eqref{nu0en vicini} we can entail the inclusion \(\C_{4\beta}\subset L(\K_{32 \beta})\), so that, if
  \[
   \widetilde{c}= \frac{c}{|(\cof \nabla L)\nu_0|}\,,
  \]
  then by the area formula we find that
  \begin{equation}
    \label{BBQ3}
      \int_{\C_{4\beta }^+\cap L(\pa ^* E)} |x \cdot e_n-\widetilde{c}|^2 \,d\H^{n-1}\le
   \int_{\pa^* E\cap \K_{32 \beta }^+} \bigg| Lx \cdot \frac{(\cof  \nabla L) \nu_0}{|(\cof \nabla L) \nu_0|}-\frac{c }{|\cof \nabla L \nu_0|}\bigg|^2|(\cof \nabla L) \nu_E| \,d\H^{n-1}\,.
  \end{equation}
  Now,  \(\det L=1\) so that   $L^*(\cof \nabla L)=\Id$ (see \eqref{mortacci}). Hence
  \[
  Lx \cdot (\cof \nabla L) \nu_0=\,x\cdot \nu_0\,,\qquad\forall x\in\R^n\,,
  \]
  and thus, taking also into account that, by \eqref{BBQ1},
  \[
  \frac{|(\cof \nabla L)\nu_E|}{|(\cof  \nabla L)\nu_0|^2}\le1+C(n)\,\|\nabla L-\Id\|\le C(n,\l)\,,
  \]
  we deduce from \eqref{BBQ3} that
  \begin{equation*}
      \int_{\C_{4 \beta }^+\cap L(\pa ^* E)} |x \cdot e_n-\widetilde{c}|^2 \,d\H^{n-1}
      \le
      C(n,\l)\int_{\pa^* E\cap \K_{32 \beta }^+}|x\cdot \nu_0-c|^2\,d\mathcal H^{n-1}\,.
  \end{equation*}
  Hence  \eqref{questax} implies that
 \begin{equation}\label{pippo}
 \fl_n^H(L(E),0,4\beta )\le C(n,\l)\,\beta^2  \,\Big(\exc_n^H(E,0,8)+(\L+\ell)\Big)\,.
 \end{equation}
 We now want to apply Lemma \ref{thm caccioppoli} to $L(E)$. To this end, we start noticing that, up to decrease the value of $\et$ in order to entail $L^{-1}(\K_{8\beta})\subset\C_8$, and setting $M=\cof\nabla L$ for the sake of brevity, we have
 \begin{eqnarray}\nonumber
  2\,  \exc_n^H(L(E),0,8\beta)
 &=&\frac1{(8\beta)^{n-1}}\int_{\K_{8\beta}^+\cap\pa^*\,L(E)}|\nu_{L(E)}-e_n|^2\,d\H^{n-1}
   \\\nonumber
   &=&\frac1{(8\beta)^{n-1}}\int_{[L^{-1}(\K_{8\beta})]^+\cap\pa^*E}\Big|\frac{M\nu_E}{|M\nu_E|}-\frac{M\nu_0}{|M\nu_0|}\Big|^2\,
   |M\nu_E|\,d\H^{n-1}
   \\\nonumber
  &\le&\frac{1+C(n,\l)\,\|\nabla L-\Id\|}{(8\beta)^{n-1}}\int_{\C_8^+\cap\pa^*E}|\nu_E-\nu_0|^2\,d\H^{n-1}
  \\\nonumber
  {\small \mbox{(by \eqref{BBQ1})}}\quad
  &\le&\frac{C(n,\l)}{\beta^{n-1}}\Big(\int_{\C_8^+\cap\pa^*E}|\nu_E-e_n|^2\,d\H^{n-1}+P(E;\C_8^+)\,|\nu_0-e_n|^2\Big)
  \\\label{johell1}
  {\small \mbox{(by \eqref{stime densita perimetro upper} and \eqref{nu0en vicini})}}\quad
  &\le&\frac{C(n,\l)}{\beta^{n-1}}\,\chi<\ec(n,2\l)\,,
 \end{eqnarray}
 provided $\et$ is small enough. In the same way, we can deduce from \eqref{dai} that
 \begin{equation}\label{fiolazio}
 \widetilde \l \le 2\l\,,\qquad  \widetilde \ell \le 2\ell\,, \qquad \widetilde r_0\ge r_0/2\,,
 \end{equation}
 and $\C_{16\beta}\subset L(\C_{16})$, again provided \(\et\) is small enough. In particular, since $\beta<1/64$, $r_0>8$, and $(\La+\ell)<1/8\l$ by \eqref{etilt 1}, we find that
  \begin{eqnarray*}
  &&\Phi^L\in\X(\C_{16\,\b}^+,2\lambda,2\ell)\,,
  \\
  &&\mbox{$L(E)$ is a $(2\La,r_0/2)$-minimizer of $\PHI^L$ in $(\C_{16\,\beta},H)$ with $8\,\beta<r_0/2$}\,,
  \\
  && 16\,(2\lambda) \Lambda\, \beta+8(2\ell )\, \beta\le 1\,,
 \end{eqnarray*}
 with $\exc_n^H(L(E),0,8\b)<\ec$ (by \eqref{johell1}), $0\in\cl(H\cap\spt\pa E)\cap\pa H$ and, when we are in the boundary case, with $\nabla\Phi^L(0,e_n)\cdot e_1=0$ (by \eqref{restaneum ris}). We can thus apply Lemma \ref{thm caccioppoli} to $L(E)$ at scale $16\beta$ to deduce that
 \[
 \exc_n^H(L(E),0,\beta)\le C_3\Big(\fl_n^H(L(E),0,4\b)+\beta (\La+\ell)\Big)\,.
 \]
 We combine this estimate with \eqref{pippo} to obtain
 \[
 \exc_n^H(L(E),0,\beta)\le C(n,\l)\Big(\beta^2\,\exc_n^H(E,0,8)+\beta\,(\La+\ell)\Big)\,,
 \]
 that is \eqref{tilt decay  E ris}. This completes the proof of the lemma.
\end{proof}

\subsection{Proof of Lemma \ref{thm regularity en}}
  By scaling, see \eqref{excess scaling} and Remark \ref{remark blow-up}, we can directly set $r=1$. With the notation $G^+=G\cap H$ for $G\subset\R^n$, we thus want to prove that, setting $\Phi_0(\nu)=\Phi(0,\nu)$, if
  \begin{eqnarray}\label{pablo1}
  &&\Phi\in\X(\C_{128}^+,\lambda,\ell)\,,
  \\ \nonumber
  &&\mbox{$E$ is a $(\La,r_0)$-minimizer of $\PHI$ in $(\C_{128},H)$ with \( r_0\ge 64\)}\,,
  \\
  \nonumber
  &&0\in\cl(H\cap\pa E )\,,
  \\\label{pablo4}
  &&|\nabla \Phi_0(e_n)\cdot e_1|+ \exc_n^H(E,0,64)+(\L +\ell) <\er\,,
  \end{eqnarray}
   then there exists  \(u\in C^{1,1/2}(\cl(\D^+))\) such that
 \begin{eqnarray}  \label{allard4 scaled}
 &&\sup_{z,y\in\D^+}|u(x)|+|\nabla u(x)|+\frac{|\nabla u(x)-\nabla u(y)|}{|x-y|^{1/2}} \le C(n,\l)\,\sqrt{\er}\,,
 \\
 \label{allard1 scaled}
&& \C^+\cap \pa E=\Big\{x\in H: |\p x|<1\,,\q x=u(\p x)\Big\}\,,
  \\
 \label{allard2 scaled}
  &&\nabla \Phi\big((z,u(z)),(-\nabla u (z),1)\big)\cdot e_1=0\,,\qquad\forall z\in\D\cap\pa H\,.
  \end{eqnarray}
  We divide the proof into four steps.

  \medskip

  \noindent {\it Step one}: We claim that for every \(x\in \cl(H\cap\pa E)\cap \pa H\cap \C_{16}\) there exists an affine map \(L:\R^n\to\R^n\) (depending on $x$) with \(L(x)=x\), \(L(H)=H\),  and
  \begin{equation}\label{exc decay bordo}
  \|\nabla L-\Id\|^2\le C\,\er\,,\qquad \exc_n^H(L(E),x,\varrho )\le C\,\er \, \varrho\,, \qquad  \forall\varrho\le16\,,
  \end{equation}
  where $C=C(n,\l)$. Firsy we notice that it suffices to prove this under the assumption that
  \begin{equation}
    \label{watson}
      \e(x)=|\nabla \Phi(x,e_n)\cdot e_1|+ \exc_n^H(E,x,32)+(\L +\ell) \le\e_0\,,
  \end{equation}
  for a suitably small positive constant $\e_0=\e_0(n,\l)$. Indeed, by \eqref{pablo1}, \eqref{Phi x ell}, and \eqref{pablo4}, if $x\in\C_{16}\cap\pa H$, then
  \begin{eqnarray*}
  &&|\nabla \Phi(x,e_n)\cdot e_1|\le |\nabla \Phi_0(e_n)\cdot e_1|+32\ell \le 32\,\er\,,
  \\
  &&\exc_n^H(E,x,32)\le 2^{n-1} \exc_n^H(E,0, 64)\le 2^{n-1} \er\,,
  \end{eqnarray*}
  so that $\e(x)\le C(n)\,\er$ for every $x\in\C_{16}\cap\pa H$; in particular, we can ensure the validity of \eqref{watson} at every $x\in \cl(H\cap\pa E)\cap \pa H\cap \C_{16}$ provided we pick $\er$ sufficiently small.

  \medskip

  We now prove our claim, setting $\e$ in place of $\e(x)$ for the sake of brevity. By exploiting the convergence of the geometric series, it will suffice to prove the following statement:

  \medskip

  There exist positive constants $\e_*,\beta_*<1$, $K_1$ and $K_2$ (depending on $n$ and $\l$ only) such that, if \(\e\le \e_*\), then for every \(k\in \N\) there exists an affine map \(L_k:\R^n\to\R^n\) with $L_k(x)=x$, \(L_k(H)=H\) and
  \begin{equation}\label{lk}
  \left\{\begin{array}
    {l l}
    \|\nabla L_k-\nabla L_{k-1}\|^2\le K_1\b_*^{k} \e\,,&\qquad\mbox{if $k\ge 1$}\,,
    \\
    \|\nabla L_0-\Id\|^2\le K_1\,\e\,,&
  \end{array}\right .
  \end{equation}
  such that,
  \begin{eqnarray}\label{rob1}
  &&\Phi_k=\Phi^{L_k}\in \X\big(\C_{x,2\b_*^k}^+, \l_k, \ell_k  \big )\,,
  \\\label{rob2}
  &&\mbox{\(E_k=L_k(E)\) is a \((\Lambda  , r_{0,k})\) minimizer of \(\PHI_k\) in \((\C_{x,2\b_*^k}, H)\)}\,,
  \\
  \label{neu}
  &&\nabla\Phi^{L_k}(x,e_n)\cdot e_1=0\,,
  \\ \label{excgrad}
  &&\exc_n^H(E_k,x,\b_* ^k)\le K_2 \b_* ^k \e\,.
  \end{eqnarray}
  where $\lambda_0=2\lambda$, $\ell_0=2\ell$, $r_{0,0}=r_0/2$, and
  \begin{equation}\label{lambdak}
  \max\Big\{
  \frac{|\l_k-\l_{k-1}|}{\l},\frac{|\ell_k-\ell_{k-1}|}{\ell},
  \frac{|r_{0,k}-r_{0,k-1}|}{r_{0,k-1}}\Big\}\le K_1\, \sqrt{\b_*^k\,\e},\qquad\forall k\ge 1\,.
  \end{equation}
  We prove this statement by induction.

  \smallskip

  \noindent {\it Base case}: If $\e_*$ is small enough, then by $|\nabla\Phi(x,e_n)\cdot e_1|<\e$ we can apply Lemma \ref{lemma 0} to find \(\nu_0\in \mathbf S^{n-1}\) such that
  \begin{equation}
  \label{exploit1}
  |\nu_0-e_n|\le C(n,\l)\,\e\,,\qquad \nabla \Phi(x,\nu_0)\cdot e_1=0\,.
  \end{equation}
  Up to further decrease $\e_*$ so to entail $|e_1\cdot\nu_0|\le1/2$, we can apply Lemma \ref{lemma tauuu} to $\nu_0$ to find an affine map $L_0:\R^n\to\R^n$ with $L_0(H)=H$, $L_0(x)=x$, $\|\nabla L_0-\Id\|\le C(n,\l)\,|e_n-\nu_0|$, and $\nabla\Phi^{L_0}(x,e_n)\cdot e_1=\nabla\Phi(x,\nu_0)\cdot e_1$ (so that \eqref{lk} and \eqref{neu} hold true by \eqref{exploit1}). The validity of \eqref{rob1} and \eqref{rob2} is easily checked thanks to Lemma \ref{lemma PHIL} and \eqref{lk}, up to further decrease the value of $\e_*$. Finally, by exploiting \eqref{lk} (with $k=0$) and \eqref{exploit1} as in the proof of \eqref{johell1}, we see that, if $\e_*$ is small enough (also to entail that $L_0^{-1}(\C_x)\subset\C_{x,64}$), then we have
  \[
  \exc_n^H(L_0(E),x,1)\le C(n,\l)\,(32)^{n-1}\exc_n^H(E,x,32)\le  C(n,\l)\,\e\,.
  \]
  This proves the case $k=0$ of our claim.

  \smallskip

  \noindent {\it Choice of $\e_*$, $\beta_*$, $K_1$ and $K_2$}: Since $\e_*$, $\beta_*$, $K_1$ and $K_2$ have to be chosen in a careful order, it seems useful to fix their choice before entering into the inductive step. We shall pick $\beta_*=\beta_*(n,\l)$ so that
  \begin{equation}
    \label{scelta betastar}
    \beta_*<\min\Big\{\frac1{512}\,,\frac1{64\,C_4(n,3\l)}\Big\}\,,
  \end{equation}
  where $C_4(n,3\l)$ is defined by means of Lemma \ref{thm caccioppoli}. By \eqref{scelta betastar}, it is possible to choose $K_2=K_2(n,\l)$ so that
  \begin{equation}
    \label{scelta K2}
    K_2\ge \frac{3\,C_4(n,3\l)}{1-64\,C_4(n,3\l)\,\beta_*}\,.
  \end{equation}
  Finally, we choose $K_1=K_1(n,\l)$ so that
  \begin{equation}
    \label{scelta K1}
    K_1\ge \frac{3\,C_4(n,3\l)\,\sqrt{K_2+3}}{\sqrt{\beta_*}}\,,
  \end{equation}
  and in such a way that the case $k=0$ of \eqref{lk} and \eqref{lambdak} holds true. Finally, $\e_*$ shall be chosen to be small enough with respect to other constants determined by $n$, $\l$, $\beta_*$, $K_1$ and $K_2$.

\smallskip

  \noindent {\it Inductive step}: Let us  assume our claim holds true for  $j\le k$ and let us prove its validity for $j=k+1$. To this end,
  we notice that, by exploiting \eqref{lambdak}, and provided $\e_*$ is small enough, we can certainly ensure that
  \begin{equation}\label{3}
  \lambda_k\le 3\l\,,\qquad \ell_k \le 3\ell\,,\qquad r_{0,k}\ge \frac{r_0}3\,.
  \end{equation}
Let us set  $\beta=8\,\b_*\in(0,1/64)$, so that we can consider the constant \(\et(n,3\l,8\b_*)\) determined by Lemma \ref{thm tilt lemma}. By the inductive step on \eqref{excgrad}, by \eqref{3} and by definition of $\e$, we see that
  \[
  \exc_n^H(E_k,x,\beta_*^k)+(\L+\ell_k)\,\beta^k_*\le K_2\,\e+3\,(\Lambda+\ell)\le (K_2+3)\,\e\,,
  \]
  so that, by \eqref{rob1}, \eqref{rob2}, \eqref{neu} and provided we assume that
  \[
  (K_2+3) \e_*\le \et(n,3\l,8\,\b_*)\,,
  \]
  we can apply Lemma \ref{thm tilt lemma} with $x$, $\Phi_k$, \(E_k\), $2\beta_*^k$, $\l_k\le 3\l$, $\ell_k$  and $r_{0,k}$ in place of $x_0$, $\Phi$, $E$, $16r$, $\l$ and $r_0$ respectively. (Notice that we have $r_{0,k}\ge 16\beta_*^k$ thanks to \eqref{3}, $r_0\ge64$, and $\beta_*<1/512$.) Hence, there exists an affine map $\widetilde L:\R^n\to\R^n$ with $\widetilde{L}(x)=x$, $\widetilde L(H)=H$, and constants $\widetilde \lambda\ge1$, \(\tilde \ell\ge0\),  and \(\widetilde r_0>0\) such that
  \begin{eqnarray}\label{boia2}
    &&\Phi_k^{\widetilde L}\in \X(\C_{x,\beta \beta_*^k/4}^+,\widetilde\l,\widetilde\ell)\,,
    \\\label{boia3}
    &&\mbox{$\widetilde{L}(E_k)$ is a $( \Lambda, \widetilde r_0)$-minimizer of \(\PHI_k^{\widetilde L}\) in \((\C_{x,\beta  \beta_*^k/4}, H)\)}\,,
    \\
    \label{talk}
    &&\nabla \Phi_k^{\widetilde L} (x,e_n)\cdot e_1=0\,,
    \\
    \label{tilt decay E a}
    &&\exc^{H}_n\bigg(\widetilde L(E_k),x,\beta \frac{\beta_*^{k}}8\bigg)\le C_4(n,3\l)\,\Big(\beta^2\,\exc^H_n(E_k,x,\beta_*^k)
    +\beta\,\frac{\beta_*^k}8\,(\La+\ell_k) \Big)\,,
  \end{eqnarray}
  \begin{eqnarray}\label{boia}
  &&\max\Big\{\|\widetilde{L}-\Id\|\,,\frac{|\widetilde{\l}-\l_k|}{\l_k},\frac{|\widetilde{\ell}-\ell_k|}{\ell_k},
  \frac{|\widetilde{r_0}-(r_{0,k})|}{r_{0,k}}\Big\}
  \\\nonumber
  &&\hspace{5cm}\le\,C_4(n,3\l)\,\Big(\exc^H_n(E_k,x,\beta_*^k)+(\La+\ell_k)\,\beta_*^k\Big)^{1/2}\,.
  \end{eqnarray}
  We claim that by setting
  \begin{equation}
    \label{boia4}
      L_{k+1}=\widetilde{L}\circ L_k\,,\qquad \l_{k+1}=\widetilde{\l}\,,\qquad \ell_{k+1}=\widetilde\ell\,,\qquad r_{0,k+1}=\widetilde{r_0}\,,
  \end{equation}
  the proof of the inductive step is completed. First, by \eqref{boia4} and since $\beta\,\beta_*^k/4=2\beta_*^{k+1}$ and $\Phi_k^{\widetilde L}=\Phi^{L_{k+1}}$, we see that \eqref{boia2}, \eqref{boia3} and \eqref{neu} immediately imply \eqref{rob1}, \eqref{rob2}, and \eqref{talk} with $k+1$ in place of $k$ respectively. Next we notice that, by \eqref{tilt decay E a} and by $\beta=8\,\beta_*$,
  \begin{eqnarray}\nonumber
    \exc^{H}_n(L_{k+1}(E),x,\beta_*^{k+1})
    &\le& C_4(n,3\l)\,\Big(64\,\beta_*^2\,\exc^H_n(E_k,x,\beta_*^k)+\beta_*^{k+1}\,(\La+\ell_k)\Big)
    \\\nonumber
    {\small \mbox{(by \eqref{excgrad} and by \eqref{3})}}
    \qquad&\le& C_4(n,3\l)\,\Big(64\,K_2\,\beta_*^{k+2}\,\e+3\,\beta_*^{k+1}\,(\La+\ell)\Big)
    \\\label{boia5}
    &\le&C_4(n,3\l)(64\,K_2\,\beta_*+3)\,\beta_*^{k+1}\,\e\,,
  \end{eqnarray}
  where in the last inequality we have used $\La+\ell<\e$. By the choice \eqref{scelta K2} of $K_2$, \eqref{boia5} implies the validity of \eqref{excgrad} with $k+1$ in place of $k$. Similarly, we notice that, by \eqref{excgrad}, by $\eqref{3}$ and by definition of $\e$
  \begin{eqnarray}\nonumber
    C_4(n,3\l)\,\Big(\,\exc^H_n(E_k,x,\beta_*^k)+(\La+\ell_k)\,\beta_*^k\Big)^{1/2}
    &\le&
    C_4(n,3\l)\,\Big(K_2\,\beta_*^k\,\e+3\,\e\,\beta_*^k\Big)^{1/2}
    \\
    &=&\label{boia6}
    \frac{C_4(n,3\l)\,\sqrt{K_2+3}}{\sqrt{\beta_*}}\,\sqrt{\beta_*^{k+1}\e}\,.\hspace{0.7cm}
  \end{eqnarray}
  By \eqref{boia}, \eqref{boia4}, \eqref{boia6} and \eqref{scelta K1} we deduce that \eqref{lambdak} holds true with $k+1$ in place of $k$. Finally,
  by exploiting the validity of \eqref{lk} for  $j\le k$, we see that
  \[
  \|\nabla L_k\|\le 1+\|\nabla L_0-\Id\|+\sum_{j=0}^{k-1}\|\nabla L_{j+1}-\nabla L_j\|\le 1+\Big(1+\frac{\sqrt{\e_*}}{1-\sqrt{\beta_*}}\Big)\,\sqrt{K_1}\le 3\,,
  \]
  provided $\e_*$ is small enough. Hence, by \eqref{boia} and \eqref{boia6}, we find
  \begin{eqnarray*}
  \|\nabla L_{k+1}-\nabla L_k\|&\le& \|\nabla L_k\|\|\nabla \widetilde  L-\Id\|
  \\
  &\le&\frac{3\,C_4(n,3\l)\,\sqrt{K_2+3}}{\sqrt{\beta_*}}\,\sqrt{\beta_*^{k+1}\e}\le K_1\,\sqrt{\beta_*^{k+1}\e}\,,
  \end{eqnarray*}
  once again thanks to \eqref{scelta K1}. This completes the proof of step one.

  \medskip

  \noindent {\it Step two}: The argument used in step one, where now the interior case of Lemma \ref{thm tilt lemma} is used in place of the boundary case at each step of the iteration, allows us to prove the following statement: There exists  $\e_{**}=\e_{**}(n,\l)$ such that, if \(x\in \cl(H\cap\pa E)\cap H\cap \C_{16}\) with
  \begin{equation}\label{dusseldorf}
  \e=\exc_n^H\big(E,x,2\,\dist(x,\pa H)\big)+2\dist(x,\pa H)(\Lambda+\l) \le\e_{**}\,,
  \end{equation}
  then there exists an affine map \(L:\R^n\to\R^n\) (depending on $x$) with \(L(x)=x\), \(L(H)=H\),  and
  \begin{equation}\label{exc decay interno lontano}
  \|\nabla L-\Id\|^2\le C(n,\l)\,\e\,,\qquad \exc_n^H(L(E),x,\varrho )\le C(n,\l)\,\e \, \varrho\,, \qquad  \forall\,\varrho\le\dist(x,\pa H)\,.
  \end{equation}
  This statement is an ``interior'' analogous to the ``boundary'' statement proved in step one, with \eqref{watson} playing the role of \eqref{dusseldorf}. The only difference is that \eqref{watson} follows directly from \eqref{pablo4}, while \eqref{dusseldorf} cannot be so immediately deduced from it. Showing the validity of \eqref{dusseldorf} at every \(x\in \C_{16}^+\cap \pa E\) is, essentially, the content of the next step of the proof.

  \medskip

  \noindent {\it Step three}: We now prove that for every \(x\in \cl(H\cap\pa E)\cap  \C\) there exists an affine map \(L\) such that \(L(H)=H\) and
  \begin{eqnarray}
  \label{exc decay dentro 0}
  \label{exc decay dentro 1}
  &&\|\nabla L-\Id\|^2\le C(n,\l)\,\er\,,
  \\
  \label{exc decay dentro 2}
  &&\exc_n^H(L(E),L(x),\varrho)\le C(n,\l)\,\er \, \varrho\,, \qquad \forall\,  \varrho\, \le 8\,.
  \end{eqnarray}
 We start with the following simple observation: if \(\er\) is sufficiently small with respect to  \(\eh(n, \l, 1/32)\), then by applying Lemma \ref{lemma height bound} to \(E\) in \(\C_8\) we have
   \begin{equation}\label{pranzo}
   \begin{gathered}
  \sup\Big\{|\q y| :y\in\C_4^+\cap\pa E\Big\}\le \frac1{32}\,,
  \\
    \Big|\Big\{y\in \C_{4}^+ \cap E:\q y> \frac1{32}\Big\}\Big|=0\,,
    \\
    \Big|\Big\{y\in \big(\C_{4}^+\setminus E\big): \q<-\frac1{32}\Big\}\Big|=0\,.
  \end{gathered}
  \end{equation}
  From this it follows that for every \(y \in \pa H\cap \C_2\) there exists a point \(y'\in \pa H\)  such that
  \begin{equation}\label{uffa}
y'\in    \cl (\C_2\cap \pa E)\cap \pa H\qquad \textrm{and} \qquad \p y' =\p y\,.
  \end{equation}
  Indeed, thanks to \eqref{pranzo}, for every \(s\in(0,2)\)
  \[
 | \K_{\p y, s} \cap E|>0\qquad | \K_{\p y, s}\setminus E|>0,
  \]
  where \(\K_{\p y,s}\) is defined as in \eqref{eq:defK}. This gives  that \(\spt\mu_E\cap  \K_{\p y,s}\ne \emptyset \), thus \eqref{uffa}, since \(\pa E=\spt \mu_E\)  by  Lemma \ref{lem:normalization}.
  Let now \(x\in \cl(H\cap\pa E)\cap  \C\) and \(\bar x\in \pa H\cap \C_2\) be such that
  \begin{equation}\label{distanze}
  |\p x-\p \bar x|=|x-\bar x|=\dist(x,\pa H)\,,
  \end{equation}
   Let \(x' \in \cl (\C_2\cap \pa E)\cap \pa H\) be the corresponding point satisfying \eqref{uffa}. By step one, there exists an affine map \(L_1:\R^n\to\R^n\) with \(L_1(H)=H\), $L_1(x')=x'$, and
  \begin{eqnarray}
  \label{exc decay dentro x1}
  &&\|\nabla L_1-\Id\|^2\le C(n,\l)\,\er\,,
  \\
  \label{exc decay dentro x2}
  &&\exc_n^H(L_1(E),x',\varrho)\le C(n,\l)\,\er \, \varrho\,, \qquad \forall  \varrho \le 16\,.
  \end{eqnarray}
 Since $L_1(x')=x'$ and $L_1$ is affine, by \eqref{exc decay dentro x1},
  \begin{equation}
    \label{L1x}
      |L_1(x)-x|=|\nabla L_1(x-x')-(x-x')|\le C(n,\l)\sqrt{\er}\,|x-x'|\\\
  \end{equation}
  and
  \begin{equation}
    \label{L1xbis}
     \big(1-C\sqrt{\er}\big)|L_1(x)-x'|\le  |x-x'|\le \big(1+C\sqrt{\er}\big)|L_1(x)-x'|.
  \end{equation}
  In particular we can choose \(\er\) sufficiently small to ensure that \(L_1(x)\in \C_2\). We now claim that, provided \(\er\) is sufficiently small,
  \begin{equation}\label{distanza2}
  |L_1(x)-x'|\le 2\,\dist(L_1(x),\pa H)\,.
  \end{equation}
  First notice that  thanks to \eqref{L1x} and \eqref{L1xbis},
  \begin{equation}\label{primopezzetto}
  \begin{split}
  \dist(x,\partial H)&\le \dist (L_1(x),\partial H)+C\sqrt{\er}\,|x-x'|\\
  &\le \dist (L_1(x),\partial H)+C\sqrt{\er}\,|L_1(x)-x'|.
  \end{split}
  \end{equation}
Moreover,  thanks to  Lemma \ref{lemma 0},
  \begin{equation}\label{ugo2}
  \begin{aligned}
    &\Phi^{L_1}\in\X(\C_{127}^+,2\l,2\ell)\,,
    \\
    &\mbox{$L_1(E)$ is a $(\La,r_0/2)$-minimizer of $\PHI^{L_1}$ on $(\C_{127},H)$ with $r_0/32\ge 32$}\,.
  \end{aligned}
  \end{equation}
  By \eqref{exc decay dentro x2} and \eqref{L1x}, if $\er$ is small enough with respect to $\eh(n,2\l,1/8)$, we can thus apply Lemma \ref{lemma height bound} to $L_1(E)$ on the cylinder $\C(x',4\,|x'-L_1(x)|)$, to deduce that
  \[
|\q L_1(x)-\q x'|\le | L_1(x)- x'|/8.
  \]
 By this, \eqref{L1x}, \eqref{L1xbis},  \eqref{distanze} and recalling that \(\p x'=\p \bar x\), we obtain
 \[
 \begin{split}
\frac{7}{8} |L_1(x)-x'|&\le |\p L_1(x)-\p x'|\le |\p L_1(x)-  \p x|+|\p x- \p x'|\\
&\le  |L_1(x)-  x|+|\p x- \p \bar x|\\
&\le C\sqrt{\er}|L_1(x)-x'|+ \dist(x,\pa H)\\
&\le C\sqrt{\er}|L_1(x)-x'|+ \dist(L(x),\pa H),
 \end{split}
  \]
 where in the last inequality we have used \eqref{primopezzetto}. Choosing \(\er\) suitably small we obtain   \eqref{distanza2}.
By  \eqref{distanza2}, if $\varrho\ge2\,\dist(L_1(x),\pa H)$, then
\[
\C(L_1(x),\varrho)\subset \C\big( x',\varrho+|L_1(x)-x'|\big)\subset\C( x',2\varrho)\,,
\]
and thus
  \[
  \exc_n^H(L_1(E),L_1(x),\varrho)\le 2^{n-1}\,\exc_n^H(L_1(E),x',2\varrho)\,.
  \]
Hence,  \eqref{exc decay dentro x2} implies
  \[
  \exc_n^H(L_1(E),L_1(x),\varrho)\le C\er \, \varrho\,, \qquad \forall  \varrho\in\big(2\,\dist(L_1(x),\pa H),8\big)\,,
  \]
 for a constant \(C\) depending on \(n\) and \(\l\) only.  Of course up to suitably  increase the constant \(C\)  we also have
  \begin{equation}\label{exc decay dentro xx}
  \exc_n^H(L_1(E),L_1(x),\varrho)\le C\er \, \varrho\,, \qquad \forall  \varrho\in\big(\dist(L_1(x),\pa H),8\big)\,.
  \end{equation}
  We thus find
  \begin{equation}\label{exc decay dentro xx bis}
  \exc_n^H\big(L_1(E),L_1(x),\dist(L_1(x),\pa H)\big) \le C \er \,\dist (L_1(x),\pa H)\,.
 \end{equation}
  Since, by \eqref{pablo4},
 \[
 \dist(L_1(x),\pa H)\,(\L+\ell)\le \dist(L_1(x),\pa H)\,\er\,,
 \]
 by choosing \(\er\) sufficiently small we can  exploit \eqref{ugo2}  to apply step two to $L_1(E)$ at $L_1(x)$, and deduce the existence of an affine map $L_2:\R^n\to\R^n$ such that $L_2(L_1(x))=L_1(x)$, $L_2(H)=H$ and
  \begin{gather}
  \label{exc decay interno lontano x 1}
 \|\nabla L_2-\Id\|^2\le C \er \,\dist (L_1(x),\pa H) \,,
  \\
  \label{exc decay interno lontano x 2}
 \exc_n^H(L_2(L_1(E)),L_1(x),\varrho )\le C\er \,\dist (L_1(x),\pa H)  \, \varrho\,, \quad  \forall\, \varrho\le\dist(L_1(x),\pa H)\,.
  \end{gather}
 We now claim that the map  $L=L_2\circ L_1$ satisfies   \eqref{exc decay dentro 1} and   \eqref{exc decay dentro 2}. Indeed, clearly
\(L(H)=H\) while \eqref{exc decay dentro 1} follows from  \eqref{exc decay dentro x1} and \eqref{exc decay interno lontano x 1}. Let us now prove that
\begin{equation}\label{mais}
\exc_n^H(L(E),L(x),\varrho )\le C\er   \, \varrho\,, \quad  \forall\, \varrho\in \big(\dist(L_1(x),\pa H),8)\,.
\end{equation}
For, let us set \(M_2=\cof \nabla L_2\), so that
\[
\nu_{L(E)}= \frac{M_2\,\nu_{L_1(E)}}{|M_2\nu_{L_1(E)}|}\,,
\]
and consider $\hat\nu\in \mathbf S^{n-1}$ such that
\[
e_n=\frac{ M_2 \hat \nu}{|M_2 \hat \nu|}\,.
\]
Since \(L_2(L_1(x))=L_1(x)\) we can choose  \(\er\) suitably small to ensure that  \((L_2)^{-1}(\C(L_1(x),\varrho))\subset \C(L_1(x),2\varrho)\), hence  we  get (compare with \eqref{johell1})
\begin{equation}
\begin{split}\label{mopiove}
 2\,\exc_n^H(L(E),L(x),\varrho )&=\frac{1}{\varrho^{n-1}}\int_{L(\pa^* E)\cap \C(L_1(x),\varrho)\cap H}|\nu_{L(E)}-e_n|^2\\
 &\le \frac{1}{\varrho^{n-1}}\int_{L_1(\pa^* E)\cap \C(L_1(x),2\varrho)\cap H}\left| \frac{M_2\,\nu_{L_1(E)}}{|M_2\nu_{L_1(E)}|}-\frac{ M_2 \hat \nu}{|M_2 \hat \nu|} \right|^2|M_2 \nu_{L_1(E)}|\\
 &\le  \frac{C}{\varrho^{n-1}}\int_{L_1(\pa^* E)\cap \C(L_1(x),2\varrho)\cap H}\left| \nu_{L_1(E)}-\hat \nu \right|^2\\
&\le C\, \exc_n^H(L_1(E),L_1(x),2\varrho)+\frac{C  P\big(L_1(E);\C(L_1(x),2\varrho)\cap H \big)}{\varrho^{n-1}}\left | \hat \nu  -e_n\right|^2\\
&\le C\er \varrho+C\er \dist(L_1(x),\pa H)\,.
\end{split}
\end{equation}
Where in the last inequality we have used \eqref{stime densita perimetro upper}, as well as the fact that
\[
|\hat \nu-e_n|^2=\left|\frac{ M_2 \hat \nu }{|M_2 \hat \nu|}-\hat \nu\right|^2\le C\er \dist(L_1(x),\pa H)\,,
\]
since  \(\|M_2-\Id\|^{2}\le C \er \dist(L_1(x),\pa H)\) by \eqref{exc decay interno lontano x 1}. Since \eqref{mopiove} immediately implies \eqref{mais}, and \eqref{mais} together with \eqref{exc decay interno lontano x 2} gives \eqref{exc decay dentro 2}, the proof of this step is complete.

  \medskip

  \noindent {\it Step four}: We finally prove \eqref{allard4 scaled}, \eqref{allard1 scaled} and  \eqref{allard2 scaled}. For \(x\in \cl(\pa E \cap H)\cap \C\) let us define
  \begin{equation}\label{nux}
\nu(x)=\frac{\cof (\nabla L^{-1})\,e_n}{|\cof (\nabla L^{-1})\,e_n|}\,.
  \end{equation}
  where \(L\) is the affine map appearing in  \eqref{exc decay dentro 2} (which, of course, depends on \(x\)).  In this way, provided \(\er\) is sufficiently small to ensure that \(\C(x,\varrho)\subset L^{-1}(\C(L(x),2\varrho))\), the same computations done in \eqref{mopiove} give
 \begin{equation}\label{vedolaluce}
 \frac{1}{\varrho^{n-1}}\int_{\pa ^*E\cap \C^+_{x,\varrho}}\frac{|\nu_E-\nu(x)|^2}{2}\le C\, \exc_n^H(L(E),L(x),2\varrho)\le C\er \varrho\,, \qquad \forall\, \varrho \le 4\,.
 \end{equation}
Moreover thanks to \eqref{exc decay dentro 1} and the definition of \(\nu(x)\), \eqref{nux}, \(|\nu(x)-e_n|^2\le C\er\). By exploiting the upper density estimates \eqref{stime densita perimetro upper} we get
\begin{equation}\label{vedolaluce1}
\begin{split}
\frac{1}{\varrho^{n-1}}\int_{\pa ^*E\cap \C^+_{x,\varrho}}\frac{|\nu_E-e_n|^2}{2}&\le \frac{1}{\varrho^{n-1}}\int_{\pa ^* E\cap \C^+_{x,\varrho}}|\nu_E-\nu(x)|^2+\frac{P(E;\C^+_{x,\varrho})|\nu(x)-e_n|^2}{\varrho^{n-1}}\\
&\le C\er\varrho+C\er\le C\er\,, \qquad \forall \varrho\le 4\,.
\end{split}
\end{equation}
Now, \eqref{vedolaluce} and \eqref{vedolaluce1} imply  \eqref{allard4 scaled} and \eqref{allard1 scaled} by a classical argument. For the sake of completeness we give  below a sketch of the proof.  First, if we  choose \(\er\) small enough, then we can apply Lemma \ref{thm lip approx} to find a Lipschitz function $u:\R^{n-1}\to\R$ such that, if we set
  \[
  M_0=\Big\{x\in \C^+\cap\pa E:\sup_{0<s<4}\exc_n^H(E,x,s)\le\de_1(n,\l)\Big\}\,,
  \]
  and $\Gamma=\{(z,u(z)):z\in\D\}$, then $M_0\subset\Gamma$. By \eqref{vedolaluce1}, up to further decrease the value of $\er$, we have that $M_0=\C^+\cap\pa E \subset \Gamma$. This easily implies, see for instance \cite[Theorem 23.1]{maggiBOOK}, that
    \[
\C^+\cap \pa E =\Big\{x\in H: |\p x|<1 \,,\q x=u(\p x)\Big\}\,,
\]
and this proves \eqref{allard1 scaled}.  We now notice that \eqref{allard1 scaled} and  \(0\in \pa E\) imply \(u(0)=0\), while \eqref{lapp6} gives
\[
\int_{\D} |\nabla u|^2\le C \er\,,
\]
so that \eqref{allard4 scaled} will follow by interpolation up to bound (in terms of a constant depending on $n$ and $\l$ only) the $C^{0,1/2}$ semi-norm of $\nabla u$ on $\D^+$. To this end, let us set
\[
v(x)=-\frac{\p \nu(x)}{\q \nu(x)}\,,\qquad x\in\cl(H\cap \pa E)\cap\C\,,
\]
(which is well-defined since \(|\nu(x)-e_n|<1\)). Since the map  \(\psi(\xi)=(-\xi,1) /(1+|\xi|^2)^{1/2}\) ($\xi\in\R^{n-1}$) satisfies $\Lip(\psi)\le 1$, by \eqref{vedolaluce} we get
\begin{equation*}\label{campanato}
\inf_{v\in \R^{n-1}}\frac{1}{\varrho^{n-1}}\int_{\D_{z,\varrho}}|\nabla u-v|^2\le \frac{1}{\varrho^{n-1}}\int_{\D_{z,\varrho}}|\nabla u-v((z,u(z)))|^2\le C\,\er \varrho\,,
\end{equation*}
for every $z\in\D^+$ and $\varrho\le 4$. By Campanato's criterion, see for instance \cite[Theorem 2.9]{Gi}, the $C^{0,1/2}$ semi-norm of $\nabla u$ on $\D^+$ is bounded by some $C=C(n,\l)$.  Finally \eqref{allard2 scaled} can be obtained by a simple first variation argument since the map \(u\) satisfies
\begin{equation*}
\int_{\D^+} \Phi\big(z,u(z),(-\nabla u(z),1)\big)\,dz\le \int_{\D^+} \Phi\big(z,w(z),(-\nabla w(z),1)\big)\,dz+C\Lambda\,\int_{\D^+}|w-u|
\end{equation*}
for every  \(w\in \Lip(\D)\) such that  \(w=u \) on \((\pa \D)^+\), where \(C=C(n,\l)\). This completes the proof of Lemma \ref{thm regularity en}.

\section{On the size of the boundary singular set}\label{section singular set} In this section we estimate the size of the set where Theorem \ref{thm epsilon} does not apply. More precisely, let us recall from Remark \ref{rmk:singular set} that, if $E$ is a $(\La,r_0)$-minimizer of $\PHI$ in $(A,H)$ for some $\Phi\in\X(A\cap H,\l,\ell)$, then the boundary singular set $\S(E;\pa H)$ (i.e., the set of those $x\in \pa_{\pa H}(\pa E\cap\pa H)\cap A$ such that $A\cap\cl(H\cap\pa E)$ is not a $C^{1,1/2}$ manifold with boundary at $x$) is characterized in the terms of the spherical excess of $E$ at $x$ as
 \begin{equation}\label{char}
  \S_A(E;\pa H)=\Big\{x\in\pa_{\pa H}(\pa E\cap\pa H)\cap A:\liminf_{r\to 0^+}{\bf exc}^H(E,x,r)>0\Big\}\,.
 \end{equation}
This identity provides a particularly useful starting point in the study of $\S(E;\pa H)$ undertaken in this section, and leading to the following result.

\begin{theorem}
  \label{thm singular set} If $A$ and $H$ are an open set and an open half-space in $\R^n$, $\Phi\in\X(A\cap H,\l,\ell)$, and $E$ is a $(\Lambda,r_0)$-minimizer of $\PHI$ on $(A,H)$, then for every $x\in \pa_{\pa H}^*(\pa E\cap\pa H)$ we have
  \[
  \lim_{r\to 0}{\bf exc}^H(E,x,r)=0\,.
  \]
  In particular,
  \begin{equation}\label{singularreduced}
  \S_A(E;\pa H)=(\pa_{\pa H}(\pa E\cap\pa H)\cap A)\setminus \pa_{\pa H}^*(\pa E\cap\pa H)\,,
  \end{equation}
  and thus    $\H^{n-2}(\S_A(E;\pa H))=0$.
\end{theorem}

The proof of Theorem  \ref{thm singular set} is based on the study of blow-ups of  $E$ at points $x_0\in \pa_{\pa H}(\pa E\cap\pa H)$. We first show that such blow-ups always exist and are non-trivial (i.e., they are neither empty nor equal to $H$), and that, if $x\in \pa_{\pa H}^*(\pa E\cap\pa H)$, then there exists an half-space inside $\pa H$ that is the trace on $\pa H$ of every such blow-up; see Lemma \ref{lem:blowup}. Then, we show that if a blow-up \(F\) of $E$ has the same trace on $\pa H$ as that left by an half-space, then \(H\cap\pa F\) is actually contained into a ``wedge'' of universal amplitude; see Lemma \ref{lem:wedge}.  At this point we follow some ideas of Hardt \cite{hardt} to show that this wedge property forces a blow-up $G$ of $F$ (at the origin) to coincide with an half-space also {\it inside} of $H$; see Lemma \ref{lem:blowup2}. Since $G$ is also a blow-up of $E$ at $x$, Theorem \ref{thm epsilon} now implies that $H\cap\pa E$ is a $C^{1,1/2}$ manifold with boundary locally at $x$, and thus that $x\not\in\S(E;\pa H)$. We premise to the proof of these lemmas the following useful definition: \\
Given a set \(E\) of locally finite perimeter in $A$ and \(x_0\in A\), we denote by \(\B_{x_0}(E)\) the family of blow-ups of \(E\) at \(x_0\), that is
\begin{equation}\label{def:blowup}
\B_{x_0}(E)=\bigg\{F\subset\R^n:
\begin{array}
  {l}
  \mbox{there exists $r_h\to 0$ as $h\to\infty$ such that}
  \\
  \mbox{$E^{x_0,r_h}\to F$ in \(L^1_{\rm loc}(\R^n)\) as $h\to\infty$}
\end{array}\bigg\}\,.
\end{equation}
By a diagonal argument, one immediately checks that \(\B_{x_0}(E)\) is closed in \(L^{1}_{\rm loc}(\R^n)\), and that
\begin{equation}\label{contenimento blowup}
\B_{0}(F)\subset \B_{x_0}(E)\,,\qquad \forall\, F\in \B_{x_0}(E)\,.
\end{equation}
We now start to implement the strategy described above.

\begin{lemma}\label{lem:blowup} If $A$ is an open set, $H=\{x_1>0\}$, $\Phi\in\X(A\cap H,\l,\ell)$, and $E$ is a $(\La,r_0)$-minimizer of $\PHI$ in $(A,H)$, then for every \(x_0\in \pa_{\pa H}  (\pa E\cap \pa H)\cap A\)
\begin{equation}
  \label{blowup non banali}
\emptyset\,, H\notin \B_{x_0}(E)\setminus\{\emptyset,H\}\,,
\end{equation}
and every \(F\in\B_{x_0}(E)\) is a minimizer of \(\PHI_{x_0}\) in \((\R^{n}, H)\) for \(\Phi_{x_0}=\Phi(x_0,\cdot)\in \X_*(\lambda)\). Moreover, for every \(x_0\in \pa_{\pa H}^*(\pa E\cap\pa H)\) there exists \(e_{x_0}\in \mathbf S^{n-1}\cap e_1^\perp\) such that
\begin{equation}\label{tracehyper}
\H^{n-1}\Big((\pa F \cap \pa H)\Delta \big\{ x\in \pa H: x\cdot e_{x_0}\le 0\big\}\Big)=0\,,\qquad \forall\, F\in \B_{x_0}(E)\,.
\end{equation}
\end{lemma}

\begin{proof} Let $x_0\in \pa_{\pa H}  (\pa E\cap \pa H)\cap A$. Given $r_h\to 0$ as $h\to\infty$, by Remark \ref{remark blow-up} \(E^{x_0,r_h}\) is a $(\La r_h,r_0/r_h)$-minimizer of $\PHI^{x_0,r_h}$ in $(A^{x_0,r_h},H)$, with $\Phi^{x_0,r_h}\in\X(A^{x_0,r_h}\cap H,\l,r_h\,\ell)$ (note that since \(x_0\in \pa H\), \(H^{x_0,r}=H\)). By Theorem \ref{thm compactness minimizers}, up to extracting a not relabeled subsequence, $E^{x_0,r_h}\to F$ in $L^1_{\rm loc}(\R^n)$ as $h\to\infty$, where $F$ is a minimizer of $\PHI_{x_0}$ on $(\R^n,H)$. Moreover, by \eqref{compactness 4}, as $h\to\infty$,
\begin{equation}
  \label{thanksgiving1}
  \tr_{\pa H}(E^{x_0,r_h})\to\tr_{\pa H}(F)\,,\qquad\mbox{in $L^1_{\rm loc}(\pa H)$}\,.
\end{equation}
Since, by \eqref{pizza2}, $\pa_{\pa H}(\pa E\cap \pa H)\cap A=\cl(\pa E\cap H)\cap \pa H\cap A$, we can apply both \eqref{stime densita volume lower} and \eqref{stime densita volume upper} to $E$ at $x_0$, to find that
\[
c_1\,|B_1\cap H|\le |E^{x_0,r_h}\cap H\cap B_1|\le (1-c_1)\,|B_1\cap H|\,,
\]
where \(c_1=c_1(n,\l)\in(0,1)\). By letting $h\to\infty$ in these inequalities, we thus find that $|F|\,|H\setminus F|>0$, and prove \eqref{blowup non banali}. Let us now assume that $x_0\in\pa^*_{\pa H}(\pa E\cap\pa H)$. By De Giorgi's rectifiability  theorem  (applied to the set of finite perimeter $\pa E\cap\pa H$ at the point $x_0$), there exists $e_{x_0}\in \mathbf S^{n-1}\cap e_1^\perp$ such that
\begin{equation}
  \label{thanksgiving2}
(\pa E\cap\pa H)^{x_0,r}\to \{x\in\pa H:x\cdot e_{x_0}\le 0\}\,,\qquad\mbox{in $L^1_{{\rm loc}}(\pa H)$ as $r\to 0^+$}\,.
\end{equation}
Now, for every $r>0$, $(\pa E\cap\pa H)^{x_0,r}=\pa (E^{x_0,r})\cap\pa H$, where $\pa (E^{x_0,r})\cap\pa H=_{\H^{n-1}}\tr_{\pa H}(E^{x_0,r})$ and $\pa F\cap\pa H=_{\H^{n-1}}\tr_{\pa H}(F)$ by statement (ii) in Lemma \ref{lem:normalization}. Therefore \eqref{tracehyper} follows by \eqref{thanksgiving1} and \eqref{thanksgiving2}.
\end{proof}

We now prove a (universal) wedge property for global minimizers with half-spaces as traces.

\begin{lemma}[Wedge property] \label{lem:wedge}
  For ever \(\l \ge1\) there exists a positive constant $L=L(n,\l)$ with the following property. If \(H=\{x_1>0\}\), \(\Phi\in \X_*(\l)\), \(E\) is a minimizer of \(\PHI\) in \((\R^n,H)\) and, for  some \(e\in \mathbf S^{n-1}\cap e_1^\perp\),
\begin{equation*}\label{tracehyperE}
\H^{n-1}\Big((\pa E \cap \pa H)\Delta \big\{ x\in \pa H: x\cdot e\le 0\big\}\Big)=0\,,
\end{equation*}
then
\begin{equation*}\label{eq:wedge}
\sup\,\left\{ \frac{|  x\cdot e |}{x\cdot e_1}:  x\in \pa E\cap H \right\}  \le L\,.
\end{equation*}
\end{lemma}

\begin{proof} We argue by contradiction, and thus assume that for every $h\in\N$ there exist \(\Phi_h\in \X_*(\l)\) and a minimizer $E_h$ of $\PHI_h$ in $(\R^n,H)$ with
\begin{equation}\label{traceh}
\H^{n-1}\Big((\pa E_h \cap \pa H)\Delta \big\{ x\in \pa H: x\cdot e_n  \le 0\big\}\Big)=0\,,
\end{equation}
and \(x_h\in H\cap\pa E_h\) such that, up to a rotation (keeping $e_1$ fixed),
\begin{equation}\label{suph}
\lim_{h\to\infty}\frac{ |x_h\cdot e_n|}{ x_h\cdot e_1}=+\infty\,.
\end{equation}
Up to translate each set along a suitable direction in \(e_1^\perp\cap e_n^\perp\) (note that both \eqref{traceh} and \eqref{suph} are unaffected by such an operation), we can assume that \(x_h=( x_h\cdot e_1,0,\dots,0, x_h\cdot e_n)\). Furthermore, up to change \(E_h\) with \(H\setminus E_h\), and to reflect along  \(\{x_n=0\}\), we can assume that \(x_h\cdot e_n>0\) for every \(h\in \N\). We now look at the sets \(F_h=E_h^{0,x_h\cdot e_n}\). By Remark \ref{remark blow-up}, \(F_h\) is a minimizer of \(\PHI_h\) in $(\R^n,H)$, with
\begin{gather}\label{traceh x}
\H^{n-1}\Big((\pa F_h \cap \pa H)\Delta \big\{ x\in \pa H: x\cdot e_n  \le 0\big\}\Big)=0\,,
\\
\label{puntoh}
\Big( \frac{ x_h\cdot e_1}{ x_h\cdot e_n}\,,0\,,\dots\,,0\,,1\Big)\in \pa F_h\,,
\end{gather}
thanks to $x_h\cdot e_n>0$. By Theorem \ref{thm compactness minimizers}, up to extracting a not relabeled subsequence, $F_h\to F_\infty$ in $L^1_{{\rm loc}}(H)$ as $h\to\infty$, where $F_\infty$ is a minimizer of $\PHI_\infty$ in $(\R^n,H)$ and \(\Phi_\infty\in \X_*(\l)\). By \eqref{compactness 4} and \eqref{traceh x}, we have
\begin{equation}\label{traceinf}
 \H^{n-1}\Big((\pa F_\infty \cap \pa H)\Delta \big\{ x\in \pa H: x\cdot e_n  \le 0\big\}\Big)=0\,.
\end{equation}
However, \eqref{kuratowski2}, \eqref{suph} and \eqref{puntoh} imply that \(p=(0,\ldots,0,1)\in \pa F_\infty\cap \pa H\), so that, by Lemma \ref{lem:densitycontact},  \(\H^{n-1}(\pa F_\infty\cap\pa H\cap B_{p,1/2})>0\). This is a contradiction to \eqref{traceinf}, and the lemma is proved.
\end{proof}

The following lemma, which is the analogous of \cite[Lemma 4.5]{hardt}, shows that for every point in the reduced boundary of the trace of a minimizer it is possible to find a blow-up given by the intersection of \(H\) with an half-space.

\begin{lemma}\label{lem:blowup2}
If $A$ is an open set, $H=\{x_1>0\}$, $\Phi\in\X(A\cap H,\l,\ell)$, and $E$ is a $(\La,r_0)$-minimizer of $\PHI$ in $(A,H)$, then for every \(x_0\in \partial_{\pa H}^*(\pa E\cap \pa H)\) there exists  \(\nu\in \mathbf S^{n-1}\) with
\begin{equation}\label{tesiblowup}
H\cap \{\nu \cdot x\le 0 \}\in \B_{x_0}(E)\,,\qquad \nabla \Phi(x_0,\nu)\cdot e_1=0\,.
\end{equation}
\end{lemma}

\begin{proof} Without loss of generality we take $x_0=0$, and then set \(\Phi_0(\nu)=\Phi(0,\nu)\), so that $\Phi_0\in \X_*(\l)$. As usual,  we shall set  $G^+=G\cap H$ for every $G\subset\R^n$. By Lemma \ref{lem:blowup}, \(\B_{0}(E)\) is a non-empty family of  minimizers of \(\PHI_0\) in \((\R^n,H)\). By Theorem \ref{thm compactness minimizers}, $\B_0(E)$ is also a compact subset of $L^1_{{\rm loc}}(\R^n)$. By \eqref{tracehyper}, there exists a vector \(e\in \mathbf S^{n-1}\cap e_1^\perp\)  such that
\begin{equation}
  \label{mytrace}
  \H^{n-1}\Big((\pa F \cap \pa H)\Delta \big\{ x\in \pa H: x\cdot e\le 0\big\}\Big)=0\,,\qquad\forall F\in\B_0(E)\,.
\end{equation}
Up to a rotation, we can assume that \(e=e_{n}\), so that Lemma \ref{lem:wedge} ensures that
\begin{equation*}
\sup_{(\pa F)^+}\,\frac{ |x\cdot e_n|}{x \cdot e_1}\le L\,,\qquad\forall F\in \B_{0}(E)\,,
\end{equation*}
where $L=L(n,\l)$. Let us now define $\beta_1:\B_0(E)\to[-L,L]$ by setting
\begin{eqnarray}\label{defbeta1}
\beta_1(F)=\sup_{(\pa F)^+}\,\frac {x\cdot e_n}{x \cdot e_1}\,,&&\qquad F\in\B_0(E)\,.
\end{eqnarray}
We notice that $\beta_1$ is lower semicontinuous on $\B_0(E)$ with respect to the $L^1_{{\rm loc}}(\R^n)$ convergence. Indeed, if \(F_{h},F\subset \B_{0}(E)\) and $F_h\to F$ in $L^1_{\rm loc}(\R^n)$, then, by \eqref{kuratowski2}, for every \(x\in H\cap \pa F\) there exist \(x_h\in H\cap\pa F_h\), $h\in\N$, such that \(x_h \to x\) as $h\to\infty$. Hence,
\[
\frac{x\cdot e_n}{x\cdot e_1}=\lim_{h\to\infty} \frac{x_h\cdot e_n}{ x_h \cdot e_1}\le \liminf_{h\to\infty} \beta_1 (F_h)\,,
\]
as claimed. Since $\beta_1$ is lower semicontinuous and $\B_0(E)$ is a compact subset of  $L^1_{{\rm loc}}(\R^n)$, we can find $F_1\in\B_0(E)$ such that
\begin{equation}\label{minF1}
\beta_1(F_1)\le\beta_1(F)\,,\qquad\forall F\in \B_0 (E)\,.
\end{equation}
Correspondingly, we define \(\alpha_1 \in (-\pi/2,\pi /2)\) so that $\tan \alpha_1=\beta_1(F_1)$ and set
\[
\nu_1=\cos \alpha_1\, e_n-\sin \alpha_1 \,e_1\in \mathbf S^{n-1}\,,\qquad H_1=\Big\{x\in H:x\cdot\nu_1\le0\Big\}\,.
\]
We now claim that
\begin{eqnarray}\label{hypercont}
(\pa H_1)^+ \subset (\pa F_1)^+\,.
\end{eqnarray}
To prove \eqref{hypercont}, we first take into account the definition of $\beta_1$ to find
\begin{equation}
  \label{meme1}
  (\pa F_1)^+\subset  H_1\,.
\end{equation}
By  \eqref{mytrace} and \eqref{meme1}, the upper semicontinuous function $w_{F_1}:\R^{n-1}_+\to [-\infty,+\infty)$ defined by setting
\[
w_{F_1}(z)=\sup\Big\{t\in\R:  (z,t)\in \pa F_1\Big\}\,,\qquad z\in\R^{n-1}_+\,,
\]
(here $\R^{n-1}_+=\{z\in\R^{n-1}:z_1>0\}$) satisfies
\begin{eqnarray}
  \label{wF1 1}
  F_1\subset\Big\{x\in H:\q x\le w_{F_1}(\p x)\Big\}\,,&&
  \\
  \label{wF1 2}
  w_{F_1}(z)\le \beta_1(F_1)\,z\cdot e_1\,,&&\qquad\forall z\in\R^{n-1}_+\,.
\end{eqnarray}
Now, if \eqref{hypercont} fails, see Figure \ref{fig fail},
\begin{figure}
  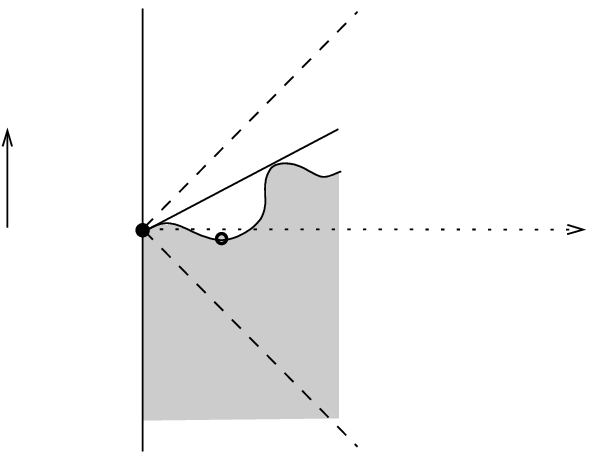\caption{\small{Failure of \eqref{hypercont}.}}\label{fig fail}
\end{figure}
then there exists $x_*\in (\pa F_1)^+$ such that
\begin{equation}
  \label{wF1 3}
  w_{F_1}(\p x_*)< \beta_1(F_1)\,x_*\cdot e_1\,.
\end{equation}
By \eqref{wF1 2} and \eqref{wF1 3}, if we set \(r_*=|\p x_*|\) and $z_*=\p x_*$, then we can find \(\vphi\in C^{1,1}(\pa(\D_{r_*}^+))\) such that
\begin{eqnarray}\label{w0 prop}
w_{F_1}(z)\le \vphi (z)\le \beta_1(F_1)\,z\cdot e_1\,,&&\qquad\forall z\in \pa (\D_{r_*}\cap H)
\\\label{diomio}
\vphi(z_*)<\beta_1(F_1)\,z_*\cdot e_1\,.
\end{eqnarray}
In particular, $\vphi=0 $ on $\D_{r_*}\cap\pa H$. By part two of Lemma \ref{nonpar}, there exists \(u_0\in C^{1,1}(\D_{r_*}^+)\cap\Lip(\cl(\D_{r_*}^+))\) such that, if we set $\Phi_0^\#(\xi)=\Phi_0(\xi,-1)$ for $\xi\in\R^{n-1}$, then
\begin{equation*}
\begin{cases}
\Div (\nabla_{\xi} \Phi_0^\#(\nabla u))=0\,,\qquad&\textrm{in \(\D_{r_*}^+\)}\,,
\\
u=\varphi\,,&\textrm{on \(\pa (\D_{r_*}^+)\)}\,,
\end{cases}
\end{equation*}
with
\begin{equation}\label{koln}
|\nabla u(0)|=|\nabla u(0)\cdot e_1 |<\beta_1(F)\,.
\end{equation}
By \eqref{wF1 1} and \eqref{w0 prop} we can apply Lemma \ref{lemma comparison} to infer that
\begin{equation}
  \label{w01x}
  F_1\cap(\D_{r_*}^+\times\R)\subset_{\H^n}\Big\{(z,t)\in\D_{r_*}^+\times\R:t\le u(z)\Big\}\,.
\end{equation}
If we now pick a sequence $\{s_h\}_{h\in\N}$ such that $s_h\to 0$ as $h\to\infty$ and $(F_1)^{0,s_h}\to\widetilde{F_1}$ in $L^1_{\rm loc}(\R^n)$, then, by \eqref{w01x} and $u(0)=0$, we find that
\[
\widetilde{F_1}\subset\Big\{(z,t):t\le (\nabla u_0(0)\cdot e_1)\,(z\cdot e_1)\Big\}\,,
\]
so that, thanks to \eqref{koln}, $\beta_1(\widetilde{F_1})<\beta_1(F_1)$. Since $\widetilde{F_1}\in \B_{0}(F_1)\subset \B_{0}(E)$, this contradicts \eqref{minF1}, and completes the proof of \eqref{hypercont}. By \eqref{mytrace}, \eqref{hypercont}, and \eqref{meme1},
\begin{equation}
  \label{F1 subset H1}
  \pa H_1\subset_{\H^{n-1}} \pa F_1\,\qquad\mbox{and}\qquad F_1\subset H_1\,.
\end{equation}
Since $F_1$ is a  minimizer of $\PHI_0$ on $(\R^n,H)$, by \eqref{F1 subset H1} and by Proposition \ref{proposition subsuper} we find that $H_1$ is a superminimizer of $\PHI_0$ on $(\R^n,H)$. Hence, by Proposition  \ref{proposition young},
\begin{equation}\label{supersub1}
\nabla\Phi_0(\nu_1)\cdot e_1\ge 0\,.
\end{equation}
In order to prove the lemma, we now take a further blow-up of $F_1$ at $0$. Precisely, we consider $\beta_2:\B_0(F_1)\to[-L,L]$ to be defined as
\begin{eqnarray}
\beta_2(F)=\inf_{(\pa F)^+}\,\frac {x\cdot e_n}{x \cdot e_1}\,,&&\qquad F\in\B_0(F_1)\,.
\end{eqnarray}
Since $\beta_2$ is upper semicontinuous and $\B_0(F_1)$ is a compact subset of  $L^1_{{\rm loc}}(\R^n)$, we can find $F_2\in\B_0(F_1)$ such that
\begin{equation}
  \label{maxF2}
  \beta_2(F)\le\beta_2(F_2)\,,\qquad\forall F\in\B_0(F_1)\,.
\end{equation}
If we now define \(\alpha_2 \in (-\pi/2,\pi /2)\) so that $\tan \alpha_2=\beta_2(F_2)$, and set
\[
\nu_2=\cos \alpha_2\, e_n-\sin \alpha_2 \,e_1\in \mathbf S^{n-1}\,,\qquad H_2=\Big\{x\in H:x\cdot\nu_2\le0\Big\}\,,
\]
then, by arguing as in the proof of \eqref{hypercont} we find that $\pa H_2\subset_{\H^{n-1}} \pa F_2$ and $H_2\subset F_2$. By Proposition \ref{proposition subsuper} \(H_2\) is a subminimizer and hence Proposition \ref{proposition young} implies
\begin{equation}\label{supersub2}
\nabla\Phi_0(\nu_2)\cdot e_1\le 0\,.
\end{equation}
Note now that the second inclusion in \eqref{F1 subset H1} implies \(F\subset H_1\) for every \(F\in \B_0 (F_1)\). In particular, \(F_2\subset H_1\) and thus \(\beta_2(F_2)\le \beta_1(F_2)=\beta_1(F_1)\), that is, $\a_2\le\a_1$. If we set, as in Lemma \ref{lemma 0},
\[
f(\alpha)=\nabla \Phi_0(\cos \alpha\, e_n-\sin \alpha \,e_1)\cdot e_1\,,\qquad \alpha\in [-\pi/2,\pi/2]\,,
\]
then \eqref{supersub1} and \eqref{supersub2} give $f(\alpha_1)\ge 0$ and $f(\alpha_2)\le 0$, and since $f'(\a)<0$ by \eqref{fprimo}, we must conclude that $\a_1=\a_2$. In particular, $H_2=F_2\in\B_0(E)$ and $\nabla\Phi_0(\nu_2)\cdot e_1=0$, as required.
\end{proof}

\begin{proof}
  [Proof of Theorem \ref{thm singular set}] By \eqref{regularset} it is clear that if $x_0\in(A\cap\pa_{\pa H}(\pa H\cap\pa E))\setminus \S_A(E;\pa H)$, then $x_0\in\pa^*_{\pa H}(\pa E\cap\pa H)$. This proves the inclusion $\supset$ in \eqref{singularreduced}. To complete the proof of \eqref{singularreduced} we are going to show that if $x_0\in\pa^*_{\pa H}(\pa E\cap\pa H)$, then
  \begin{equation}
    \label{thanksgiving3}
      \liminf_{r\to 0^+}\exc^H(E,x_0,r)=0\,.
  \end{equation}
  To this end, we exploit Lemma \ref{lem:blowup2} to find a sequence $\{r_h\}_{h\in\N}$ with $r_h\to 0$ as $h\to\infty$, such that
  \[
  E^{x_0,r_h}\to F=H\cap \{\nu \cdot x\le 0 \}\qquad \textrm{ in $L^1_{\rm }(\R^n)$}\,,
  \]
  as $h\to\infty$. By scale invariance of $\exc^H$ and by arguing as in Remark \ref{remark continuity excess} we thus find
  \[
  \lim_{h\to\infty}{\bf exc}^H(E,x_0,r_h)=\lim_{h\to\infty}{\bf exc}^H(E^{x_0,r_h},0,1)={\bf exc}^H(H\cap \{\nu \cdot x\le 0 \},0,1)=0\,,
  \]
  that is \eqref{thanksgiving3}. This proves \eqref{singularreduced}, and then $\H^{n-2}(\S_A(E,\pa H))=0$ follows by \eqref{pizza}.
\end{proof}

\section{Proofs of Theorems \ref{thm capillari} and \ref{thm main}}\label{section examples} In this section Theorem  \ref{thm capillari}, Corollary \ref{corollario capillari}, and Theorem \ref{thm main} are finally deduced from Theorems \ref{thm epsilon} and \ref{thm singular set}. The key step is of course getting rid of the relative adhesion coefficient $\s$, as we do in the following lemma.

\begin{lemma}\label{lemma da 1 a 2}
Given $\l\ge1$ and $\La,\ell,L\ge0$, there exist constants \( \ell_0, \Lambda_0\ge 0\) depending on $n$, $\l$, $\La$, $\ell$ and $L$ only, with the following property. If \(A\) is an open set in $\R^n$, $H=\{x_1>0\}$, $\Phi\in \X(A\cap H,\l, \ell)$, \(\sigma\in \Lip( A\cap \pa H)\) with $\Lip(\s)\le L$ and
\begin{equation}\label{sigmalim22}
-\Phi(z,e_1)<\sigma(z)<\Phi(z,-e_1)\,,\qquad \forall\, z\in A\cap \pa H\,,
\end{equation}
\(E\) is a \((\L,r_0)\)-minimizer of $(\PHI,\s)$ in $(A,H)$ and \(x_0\in  A\cap\pa H\), then there exist $\l_*\ge 1$ and $\varrho_*>0$ and \(\Psi\in   \X(B_{x_0,\varrho_*} \cap H, \l_*,  \ell_0)\) such that \(E\) is a \(( \L_0,r_0)\)-minimizer of \(\PSI\) in $(B_{x_0,\varrho_*},H)$. Moreover, for every $x\in B_{x_0,\varrho}\cap\pa H$ and $\nu\in{\bf S}^{n-1}$ one has
\begin{equation}\label{young forever}
\nabla\Psi(x,\nu)\cdot e_1=0\qquad\textrm{if and only if}\qquad \nabla \Phi(x,\nu)\cdot (-e_1)=\sigma(x)\,.
\end{equation}
\end{lemma}

\begin{proof} Let us assume, without loss of generality, that \(x_0=0\). We set $\Phi_0(\nu)=\Phi(0,\nu)$ and  $G^+=G\cap H$ for every $G\subset\R^n$. We want to prove the existence of ${\varrho_*}>0$, $\l_*\ge1$, and \(\Psi\in   \X(B_{\varrho_*}^+, \l_*,  \ell_0)\) such that
\begin{equation}
  \label{certo}
  \PSI(E;W^+)\le\PSI(F;W^+)+\Lambda_0\,|E\Delta F|\,,
\end{equation}
whenever $F\subset H$ with $E\Delta F\cc W$ and $W$ is an open set with $W\cc B_{{\varrho_*}}$ and $\diam(W)<2\,r_0$. To this end, let us fix such a competitor $F$, with the requirement that ${\varrho_*}<\dist(0,\pa A)$), and, in correspondence to $F$, let $W_0$ be an open set with smooth boundary such that $E\Delta F\cc W_0\cc B_{{\varrho_*}}$, $\diam(W_0)<2\,r_0$, and
\begin{equation}
  \label{W primo}
  \H^{n-1}(\pa W_0\cap \pa^* E)=\H^{n-1}(\pa W_0\cap \pa^* F)=0\,.
\end{equation}
Since $E$ is a $(\La,r_0)$-minimizer of $(\PHI,\s)$ in $(A,H)$, we certainly have
\begin{equation}
  \label{certo certo}
  \PHI(E;W_0^+)+\int_{W_0\cap\pa H\cap\pa^*E}\s\,d\H^{n-1}\le\PHI(F;W_0^+)
  +\int_{W_0\cap\pa H\cap\pa^*F}\s\,d\H^{n-1}+\Lambda\,|E\Delta F|\,.
\end{equation}
Next we define a Lipschitz vector field $T:B_{\varrho_*}\to\R^n$ by setting, for every $x\in B_{\varrho_*}$,
\begin{eqnarray*}
  T(x)=-\frac{\s(\h x)}{\Phi_0(e_1)}\,\nabla\Phi_0(e_1)\,,\qquad\mbox{if $\s(0)\le 0$}\,,
  \\
  T(x)=\frac{\s(\h x)}{\Phi_0(-e_1)}\,\nabla\Phi_0(-e_1)\,,\qquad\mbox{if $\s(0)>0$}\,,
\end{eqnarray*}
where $\h:\R^n\to\pa H=\{x_1=0\}$ denotes the projection over $\pa H$. (The definition is well-posed since $0\in\pa H$, and thus $\h x\in A\cap\pa H$ whenever $x\in B_{\varrho_*}\subset A$.) Notice that, in both cases, since $\nabla\Phi_0(e)\cdot e=\Phi_0(e)$ for every $e\in{\bf S}^{n-1}$, one has
\begin{equation}
  \label{che bello T}
  -T(x)\cdot e_1=\s(x)\,,\qquad\forall x\in B_{\varrho_*}\cap\pa H\,.
\end{equation}
Since $E\subset H$, by \eqref{subset of H}, by \eqref{che bello T} and by applying the divergence theorem to $T$ over $E\cap W_0$,
\[
\int_{E\cap W_0}\Div T=\int_{W^+_0\cap\pa^*E}T\cdot\nu_E\,d\H^{n-1}+\int_{W_0\cap \pa H\cap\pa^*E}\s\,d\H^{n-1}+\int_{E^{(1)}\cap\pa W_0}T\cdot\nu_{W_0}\,d\H^{n-1}\,,
\]
and an analogous relation holds true with $F$ in place of $E$. By plugging these relations into \eqref{certo certo}, and taking also into account that $E^{(1)}\cap\pa W_0=F^{(1)}\cap\pa W_0$ since $E\Delta F\cc W_0$, one finds
\begin{eqnarray}
  \label{certo certo 2}
  &&\PHI(E;W_0^+)-\int_{W_0\cap\pa H\cap\pa^*E}T\cdot\nu_E\,d\H^{n-1}
  \\\nonumber
  &\le&
  \PHI(F;W_0^+)
  -\int_{W_0\cap\pa H\cap\pa^*F}T\cdot\nu_F\,d\H^{n-1}+\Big(\Lambda+\sup_{B_{\varrho_*}}|\Div T|\Big)\,|E\Delta F|\,.
\end{eqnarray}
Thus, if we set
\[
\Psi(x,\nu)=\Phi(x,\nu)-T(x)\cdot\nu\,,\qquad (x,\nu)\in B_{\varrho_*}\times\R^n\,,
\]
then $E$ is a $(\La_0,r_0)$-minimizer of $\PSI$ in $(B_{\varrho_*},H)$, with $\La_0=\La+n\,L\,\l^2$ (as $|\nabla T|\le L\,\l^2$) provided we can check that $\Psi\in\X(B_{\varrho_*}^+,\l_*,\ell_0)$ for suitable values of $\l_*$ and $\ell_0$. A quick inspection of Definition \ref{def:integrand} shows that indeed \eqref{Phi nabla 1}, \eqref{Phi x ell}, \eqref{elliptic} and the upper bound in \eqref{Phi 1} hold true with suitable values of $\l_*$ and $\ell_0$ depending on $\l$, $\ell$ and $L$ only. (In checking this, it is useful notice that $|\s|\le\l$ on $A\cap\pa H$ by \eqref{sigmalim22}.) One has to be more careful in the verification of the lower bound in \eqref{Phi 1}, and indeed this is the place where the values of $\l_*$ and ${\varrho_*}$ has to be chosen in dependence of the positivity of
\[
1+\frac{\s(0)}{\Phi_0(e_1)}\,,\qquad\mbox{if $\s(0)\le 0$}\,,
\]
or in dependence of the positivity of
\[
1-\frac{\s(0)}{\Phi_0(-e_1)}\,,\qquad\mbox{if $\s(0)>0$}\,.
\]
(Of course, both positivity properties descend from \eqref{sigmalim22}.) Precisely, let us first consider the case when $\s(0)\le 0$, and notice that by \eqref{Phi 1} and \eqref{Phi nabla 1} (applied to $\Phi$) one has
\begin{equation}
  \label{indiana}
  \Psi(x,\nu)\ge\Phi_0(\nu)+\frac{\s(0)}{\Phi_0(e_1)}\,\nabla\Phi_0(e_1)\cdot\nu-(\ell+L\l^2)\,{\varrho_*}\,,
\end{equation}
for every $x\in B_{\varrho_*}$ and $\nu\in {\bf S}^{n-1}$. Let us now introduce a parameter $\tau_0>0$ and let us consider the following two cases:
\begin{itemize}
\item[(a)] $\nabla\Phi_0(\nu)\cdot e_1\le\tau_0$,
\item[(b)] $\nabla\Phi_0(\nu)\cdot e_1\ge\tau_0$.
\end{itemize}

\smallskip
\noindent
{\em Case (a)}
 We notice that, by \eqref{indiana}, \eqref{Phi 1} ( applied to $\Phi$), and $|\s(0)|\le\l$,  then
\begin{equation}
  \label{indiana2}
  \Psi(x,\nu)\ge\frac1\l-\frac{|\s(0)|}{\Phi_0(e_1)}\,\tau_0-(\ell+L\l^2)\,{\varrho_*}\ge\frac1\l-\l^2\,\tau_0-(\ell+L\l^2)\,{\varrho_*}\ge\frac1{2\l}\,,
\end{equation}
provided $\tau_0$ and ${\varrho_*}$ are small enough with respect to $\l$, $L$ and $\ell$.

\smallskip
\noindent
{\em Case (b)}
By  convexity and   one-homogeneity $\Phi_0(\nu)\ge\nabla\Phi_0(e_1)\cdot\nu$, hence \eqref{indiana} and the positivity of $1+(\s(0)/\Phi_0(e_1))$ implies  that
\begin{eqnarray}
  \nonumber
  \Psi(x,\nu)&\ge&\Big(1+\frac{\s(0)}{\Phi_0(e_1)}\Big)\,\nabla\Phi_0(e_1)\cdot\nu-(L\l^2+\ell)\,{\varrho_*}
  \ge \Big(1+\frac{\s(0)}{\Phi_0(e_1)}\Big)\,\tau_0-(L\l^2+\ell)\,{\varrho_*}
  \\\label{indiana3}
  &\ge&\Big(1+\frac{\s(0)}{\Phi_0(e_1)}\Big)\,\frac{\tau_0}2\,,
\end{eqnarray}
provided ${\varrho_*}$ is small enough depending on the size of $1+(\s(0)/\Phi_0(e_1))$, $\l$, $L$, $\ell$ and on the value of $\tau_0$ chosen to ensure the validity of \eqref{indiana2}. By combining \eqref{indiana2} and \eqref{indiana3}, we find that $\Psi$ satisfies the lower bound in $\eqref{Phi 1}$ for some value of $\l_*$ depending on $\l$, $L$, $\ell$, and the size of $1+(\s(0)/\Phi_0(e_1))$. In the case that $\s(0)>0$ one can check the validity for $\Psi$ of the lower bound in \eqref{Phi 1} by an entirely analogous argument. This proves that $\Psi\in\X(B_{\varrho_*}^+,\l_*,\ell_0)$, while the validity of \eqref{young forever} is immediate from the definition of $\Psi$.
\end{proof}


\begin{proof}[Proof of Theorem \ref{thm main}]
 According to Lemma \ref{lemma da 1 a 2}, we can cover \(A\cap \pa H\) with countably many balls \(\{B_h\}_{h\in \N}\) with the property that, for every $h\in\N$, $B_h\subset A$, \(E\) is a \((\L_0,r_0)\)-minimizer of \(\PHI_h\) in \((B_h,H)\) for some \(\Phi_h\in\X(B_h\cap H,\l_h,\ell_0)\) such that, if $x\in B_h\cap\pa H$ and $\nu\in S^{n-1}$, then
\begin{equation}
  \label{toro1}
  \nabla\Phi_h(x,\nu)\cdot \nu_H=0\qquad\mbox{if and only if}\qquad \nabla\Phi(x,\nu)\cdot\nu_H=\s(x)\,.
\end{equation}
Setting \(M=\cl (H\cap \pa E)\), by Lemma \ref{lem:normalization}, for every $h\in\N$,
\[
\begin{array}
  {l}
 \textrm{$E\cap B_h$ is an open set}\,,
  \\
  \textrm{$\pa E\cap\pa H$ is of locally finite perimeter in $B_h\cap\pa H$}\,,
  \\
  \textrm{$B_h\cap\pa_{\pa H}(\pa E\cap\pa H)=B_h\cap M\cap\pa H$}\,.
 \end{array}
\]
Let us  define \(A'=\cup_{h} B_h\) so that
\[
A'\cap \pa H=A\cap \pa H\,,\qquad  M\cap \pa H\cap A=M\cap \pa H\cap A'\qquad \textrm{and}\qquad  \pa E\cap\pa H\cap A=\pa E\cap\pa H\cap A'\,.
\]
Since \(B_h\) covers \(A'\) we see by the previous properties that \(E\cap A'\) is (equivalent to) an open set,   $\pa E\cap\pa H$ is of locally finite perimeter in $A'\cap\pa H=A\cap\pa H$, and $\pa_{\pa H}(\pa E\cap\pa H)=M\cap\pa H$.  Moreover
\[
\Sigma_{A'}(E;\pa H)\cap B_h=\Sigma_{A}(E;\pa H)\cap B_h=\Sigma_{B_h}(E;\pa H),
\]
so that by Theorem \ref{thm epsilon} and Theorem \ref{thm singular set},  we find that $\H^{n-2}(\S_A(E;\pa H))=0$, as well as that
\[
\begin{array}{l}
  \textrm{$ M$ is a $C^{1,1/2}$-manifold with boundary in a neighborhood of $x$}
  \\
  \textrm{with $\nabla\Phi_h(x,\nu_E(x))\cdot\nu_H=0$\,,}
\end{array}
\]
for every $x\in \pa_{\pa H}(\pa E\cap\pa H)\setminus\S_A(E;\pa H)$. This complete the proof of the theorem.
\end{proof}

\begin{proof}[Proof of Theorem \ref{thm capillari}]  The existence of a minimizer \(E\) of \eqref{variational problem} follows by applying the  direct methods of the calculus of variation, see for instance   \cite[Section 19.1]{maggiBOOK} for the case \(\Phi(x,\nu)=|\nu|\). By a``volume-fixing variation'' argument, \cite[Example 21.3]{maggiBOOK}, we see that $E$ satisfies the volume-constraint-free minimality property
\begin{equation}
\label{min}
\PHI(E;\Om)+\int_{\pa^* E\cap \pa \Omega} \sigma\, d\H^{n-1}
\le \PHI(F;\Om)+\int_{\pa^* F\cap \pa \Omega} \sigma\, d\H^{n-1}+\L\,|E \Delta F|\,,
\end{equation}
whenever \(F\subset \Omega\) and  \(\diam(E\Delta F)\le r_0\), where $r_0$ and $\La$ are constants depending on $E$, $\Omega$, and \(\|g\|_{L^\infty(\Omega)}\). Let us now fix \(x_0\in \pa \Omega\): by assumption, there exist $r>0$, an open neighborhood $A$ of the origin and a $C^{1,1}$-diffeomorphism $f$ between $B_{x_0,r}\cap\Omega$ and $A\cap H$, and between $B_{x_0,r}\cap\pa\Om$ and $A\cap\pa H$, where $H=\{x_1>0\}$. If we set \(\Lambda^f=\Lambda\,\|\det \nabla f\|_{L^\infty(B(x_0,r))}\), and, for $x\in A\cap H$ and $\nu\in S^{n-1}$,
\begin{eqnarray*}
\Phi^f(x,\nu)&=&\Phi(f^{-1}(x),\cof(\nabla f^{-1}(x))\,\nu)\,,
\\
\sigma^f (x)&=&\sigma(f^{-1}(x))\,\Big|\cof(\nabla f^{-1}(x))\,e_1\Big|\,,
\end{eqnarray*}
then we can find $r_*>0$, $\l_*\ge 1$, and $\ell_*>0$ such that, by \eqref{min}, \eqref{change of variables cofattore}, and by arguing as in Lemma \ref{lemma PHIL},
\[
\PHI^f(f(E);H)+\int_{\pa^* f(E)\cap \pa H} \sigma^f\, d\H^{n-1}
\le \PHI^f(G;H)+\int_{\pa^* G\cap \pa H} \sigma^f\, d\H^{n-1}+\L^f |f(E) \Delta G|\,,
\]
whenever \(G\subset H\),  \(\diam(f(E)\Delta G)\le 2\,r_*\) and $f(E)\Delta G\cc A$, with $\Phi^f\in\X(A\cap H,\l_*,\ell_*)$. In particular, $f(E)$ is a $(\L^f,r_*)$-minimizer of $(\Phi^f,\s^f)$ in $(A,H)$, while
\[
\nu_\Omega(f^{-1}(x))=\frac{\cof(\nabla f^{-1}(x))\,(-e_1)}{|\cof (\nabla f^{-1}(x))e_1|}\,,\qquad\forall x\in A\cap\pa H\,,
\]
and \eqref{sigma constraint stretto} imply that
\[
-\Phi^f(x,e_1)<\sigma^f (x)<\Phi^f(x,-e_1)\,,\qquad \forall\,x\in A\cap \pa H\,.
\]
Hence, we can apply Theorem \ref{thm main} to discuss the boundary regularity of \(f(E)\) in \(A\), and conclude the proof of the theorem by a covering argument and by a change of variables.
\end{proof}

\begin{proof}[Proof of Corollary \ref{corollario capillari}]
  {\it Step one}: We start showing that $\S\cap\pa\Om=\emptyset$ if $n=3$. We argue by contradiction, and assume the existence of \(x_0\in\Sigma\cap \pa  \Omega \). Since $\Om$ has boundary of class $C^{1,1}$, we can find $r>0$, an open neighborhood $A$ of the origin, and a \(C^{1,1}\) diffeomorphism \(f\) between $B_{x,r}\cap\Om$ and $A\cap H$, and between $B_{x,r}\cap\pa\Om$ and $A\cap\pa H$ such that \(f(x_0)=0\) and \(\nabla f (x_0)=\Id\). In particular, in the notation used in the proof of Theorem \ref{thm capillari}, we have
  \[
  \Phi^f(0,\nu)=|\nu|\,,\qquad \sigma^f(0)=\sigma(x_0)\,.
  \]
  By arguing as in Lemma \ref{lem:blowup} we thus see that every blow-up $E_2$ of $E_1=f(E)$ at $0$ satisfies the minimality inequality
  \[
  P(E_2;H)+\sigma(x_0)\,P(E_2;\pa H)\le P(F;H)+\sigma(x_0)\,P(F;\pa H)\,,
  \]
  whenever $F\subset H$ and $E_2\Delta F\cc \R^n$. Given $r>0$, if we plug into this inequality the cone-like comparison set $F_r$ defined by
  \[
  F_r=(E_2\setminus B_r)\cup\Big\{t\,x\in B_r:0\le t\le 1\,,x\in  H\cap E_2^{(1)}\cap\pa B_r\Big\}\,,
  \]
  then, by arguing for example as in \cite[Theorem 28.4]{maggiBOOK}, we find that the function
  \[
  \alpha(r)=\frac{P(E_2;H\cap B_r)+\sigma(x_0)P(E_2;\pa H\cap B_r)}{r^{2}}\,,\qquad r>0\,,
  \]
  is increasing on $(0,\infty)$, with \(\alpha(r)={\rm const}\) if and only if \(E_2\) is a cone: in particular, every blow-up \(E_3\) of \(E_2\) at the origin is a cone. (Alternatively, we could have directly shown $E_2$ to be a cone by using almost-monotonicity formulas.) By interior regularity theory, \(\pa E_3\cap H\) is a smooth surface in $\R^3$ with zero mean curvature. Since this surface is also a cone, and \(\pa E_3\cap \pa B_1\cap H\) must be a finite union of non-intersecting geodesics, we conclude that \(\pa E_3 \cap H\) is a finite union of planes meeting along a common line \(\g\subset \pa H\) with $0\in\g$. Since $0\in\S(E_3;\pa H)$ and $E_3$ is a cone, it must be $\g\subset\S(E_3;\pa H)$, and thus
  \[
  \H^1(\S(E_3;\pa H))=+\infty\,.
  \]
  However $\H^1(\S(E_3;\pa H))=0$ by Theorem \ref{thm main}, and we have thus reached a contradiction.

  \medskip

  \noindent {\it Step two}: By combining step one with the classical dimension reduction argument by Federer (see, \cite[Appendix A]{SimonLN} or  \cite[Sections 28.4-28.5]{maggiBOOK}), one shows that $\S(E;\pa\Om)$ is discrete if $n=4$, and that $\H^s(\S(E;\pa\Om))=0$ for every $s>n-4$ if $n\ge5$.
\end{proof}

%


%
\bibliographystyle{alpha}
\bibliography{references}

\end{document}

%% file: container.pstex_t
\begin{picture}(0,0)%
\includegraphics{container.eps}%
\end{picture}%
\setlength{\unitlength}{3947sp}%
\begingroup\makeatletter\ifx\SetFigFont\undefined%
\gdef\SetFigFont#1#2#3#4#5{%
  \reset@font\fontsize{#1}{#2pt}%
  \fontfamily{#3}\fontseries{#4}\fontshape{#5}%
  \selectfont}%
\fi\endgroup%
\begin{picture}(2099,1379)(590,-1097)
\put(2674, 52){\makebox(0,0)[lb]{\smash{{\SetFigFont{10}{12.0}{\rmdefault}{\mddefault}{\updefault}{\color[rgb]{0,0,0}$\Om$}%
}}}}
\put(605,-643){\makebox(0,0)[lb]{\smash{{\SetFigFont{8}{9.6}{\rmdefault}{\mddefault}{\updefault}{\color[rgb]{0,0,0}$\nu_\Om(x)$}%
}}}}
\put(1040,-432){\makebox(0,0)[lb]{\smash{{\SetFigFont{8}{9.6}{\rmdefault}{\mddefault}{\updefault}{\color[rgb]{0,0,0}$x$}%
}}}}
\put(1510,-165){\makebox(0,0)[lb]{\smash{{\SetFigFont{8}{9.6}{\rmdefault}{\mddefault}{\updefault}{\color[rgb]{0,0,0}$\nu_E(x)$}%
}}}}
\put(1607,-553){\makebox(0,0)[lb]{\smash{{\SetFigFont{9}{10.8}{\rmdefault}{\mddefault}{\updefault}{\color[rgb]{0,0,0}$M$}%
}}}}
\put(2192,-776){\makebox(0,0)[lb]{\smash{{\SetFigFont{10}{12.0}{\rmdefault}{\mddefault}{\updefault}{\color[rgb]{0,0,0}$E$}%
}}}}
\end{picture}%

%% file: cheeger.pstex_t
\begin{picture}(0,0)%
\includegraphics{cheeger.eps}%
\end{picture}%
\setlength{\unitlength}{3947sp}%
\begingroup\makeatletter\ifx\SetFigFont\undefined%
\gdef\SetFigFont#1#2#3#4#5{%
  \reset@font\fontsize{#1}{#2pt}%
  \fontfamily{#3}\fontseries{#4}\fontshape{#5}%
  \selectfont}%
\fi\endgroup%
\begin{picture}(3982,2160)(937,-1747)
\put(4642,-1566){\makebox(0,0)[lb]{\smash{{\SetFigFont{10}{12.0}{\rmdefault}{\mddefault}{\updefault}{\color[rgb]{0,0,0}$\pa\Om$}%
}}}}
\put(1601,-553){\makebox(0,0)[lb]{\smash{{\SetFigFont{10}{12.0}{\rmdefault}{\mddefault}{\updefault}{\color[rgb]{0,0,0}$E_\e$}%
}}}}
\end{picture}%

%% file: minimi.pstex_t
\begin{picture}(0,0)%
\includegraphics{minimi.eps}%
\end{picture}%
\setlength{\unitlength}{3947sp}%
\begingroup\makeatletter\ifx\SetFigFont\undefined%
\gdef\SetFigFont#1#2#3#4#5{%
  \reset@font\fontsize{#1}{#2pt}%
  \fontfamily{#3}\fontseries{#4}\fontshape{#5}%
  \selectfont}%
\fi\endgroup%
\begin{picture}(2143,1993)(507,-1626)
\put(1784,-1439){\makebox(0,0)[lb]{\smash{{\SetFigFont{10}{12.0}{\rmdefault}{\mddefault}{\updefault}{\color[rgb]{0,0,0}$E_1$}%
}}}}
\put(670,-559){\makebox(0,0)[lb]{\smash{{\SetFigFont{10}{12.0}{\rmdefault}{\mddefault}{\updefault}{\color[rgb]{0,0,0}$A$}%
}}}}
\put(1324,221){\makebox(0,0)[lb]{\smash{{\SetFigFont{10}{12.0}{\rmdefault}{\mddefault}{\updefault}{\color[rgb]{0,0,0}$H$}%
}}}}
\put(713,-1329){\makebox(0,0)[lb]{\smash{{\SetFigFont{8}{9.6}{\rmdefault}{\mddefault}{\updefault}{\color[rgb]{0,0,0}$\nu_H$}%
}}}}
\end{picture}%

%% file: superminimi.pstex_t
\begin{picture}(0,0)%
\includegraphics{superminimi.eps}%
\end{picture}%
\setlength{\unitlength}{3947sp}%
\begingroup\makeatletter\ifx\SetFigFont\undefined%
\gdef\SetFigFont#1#2#3#4#5{%
  \reset@font\fontsize{#1}{#2pt}%
  \fontfamily{#3}\fontseries{#4}\fontshape{#5}%
  \selectfont}%
\fi\endgroup%
\begin{picture}(1644,1561)(505,-1165)
\put(1485,-1019){\makebox(0,0)[lb]{\smash{{\SetFigFont{10}{12.0}{\rmdefault}{\mddefault}{\updefault}{\color[rgb]{0,0,0}$E_1$}%
}}}}
\put(633,-346){\makebox(0,0)[lb]{\smash{{\SetFigFont{10}{12.0}{\rmdefault}{\mddefault}{\updefault}{\color[rgb]{0,0,0}$A$}%
}}}}
\put(1133,250){\makebox(0,0)[lb]{\smash{{\SetFigFont{10}{12.0}{\rmdefault}{\mddefault}{\updefault}{\color[rgb]{0,0,0}$H$}%
}}}}
\put(666,-935){\makebox(0,0)[lb]{\smash{{\SetFigFont{8}{9.6}{\rmdefault}{\mddefault}{\updefault}{\color[rgb]{0,0,0}$\nu_H$}%
}}}}
\end{picture}%

%% file: young.pstex_t
\begin{picture}(0,0)%
\includegraphics{young.eps}%
\end{picture}%
\setlength{\unitlength}{3947sp}%
\begingroup\makeatletter\ifx\SetFigFont\undefined%
\gdef\SetFigFont#1#2#3#4#5{%
  \reset@font\fontsize{#1}{#2pt}%
  \fontfamily{#3}\fontseries{#4}\fontshape{#5}%
  \selectfont}%
\fi\endgroup%
\begin{picture}(3464,1889)(1057,-1516)
\put(4334,-368){\makebox(0,0)[lb]{\smash{{\SetFigFont{8}{9.6}{\rmdefault}{\mddefault}{\updefault}{\color[rgb]{0,0,0}$\nu$}%
}}}}
\put(4444,-685){\makebox(0,0)[lb]{\smash{{\SetFigFont{8}{9.6}{\rmdefault}{\mddefault}{\updefault}{\color[rgb]{0,0,0}$e_1$}%
}}}}
\put(3151, 89){\makebox(0,0)[lb]{\smash{{\SetFigFont{10}{12.0}{\rmdefault}{\mddefault}{\updefault}{\color[rgb]{0,0,0}$H$}%
}}}}
\put(2063,-1361){\makebox(0,0)[lb]{\smash{{\SetFigFont{10}{12.0}{\rmdefault}{\mddefault}{\updefault}{\color[rgb]{0,0,0}$E$}%
}}}}
\put(2561,-311){\makebox(0,0)[lb]{\smash{{\SetFigFont{10}{12.0}{\rmdefault}{\mddefault}{\updefault}{\color[rgb]{0,0,0}$A$}%
}}}}
\put(1072,163){\makebox(0,0)[lb]{\smash{{\SetFigFont{8}{9.6}{\rmdefault}{\mddefault}{\updefault}{\color[rgb]{0,0,0}$\{x\cdot\nu=c\}$}%
}}}}
\end{picture}%

%% file: gh.pstex_t
\begin{picture}(0,0)%
\includegraphics{gh.eps}%
\end{picture}%
\setlength{\unitlength}{3947sp}%
\begingroup\makeatletter\ifx\SetFigFont\undefined%
\gdef\SetFigFont#1#2#3#4#5{%
  \reset@font\fontsize{#1}{#2pt}%
  \fontfamily{#3}\fontseries{#4}\fontshape{#5}%
  \selectfont}%
\fi\endgroup%
\begin{picture}(2175,2320)(919,-2118)
\put(1977,-246){\makebox(0,0)[lb]{\smash{{\SetFigFont{10}{12.0}{\rmdefault}{\mddefault}{\updefault}{\color[rgb]{0,0,0}$\C^{{\rm v}}_2$}%
}}}}
\put(3072,-992){\makebox(0,0)[lb]{\smash{{\SetFigFont{10}{12.0}{\rmdefault}{\mddefault}{\updefault}{\color[rgb]{0,0,0}$x_1$}%
}}}}
\put(976, 56){\makebox(0,0)[lb]{\smash{{\SetFigFont{10}{12.0}{\rmdefault}{\mddefault}{\updefault}{\color[rgb]{0,0,0}$\pa H$}%
}}}}
\put(969,-940){\makebox(0,0)[lb]{\smash{{\SetFigFont{10}{12.0}{\rmdefault}{\mddefault}{\updefault}{\color[rgb]{0,0,0}$G_s$}%
}}}}
\put(1300,-510){\makebox(0,0)[lb]{\smash{{\SetFigFont{10}{12.0}{\rmdefault}{\mddefault}{\updefault}{\color[rgb]{0,0,0}$\C^{{\rm v}}_1$}%
}}}}
\end{picture}%

%% file: compa.pstex_t
\begin{picture}(0,0)%
\includegraphics{compa.eps}%
\end{picture}%
\setlength{\unitlength}{3947sp}%
\begingroup\makeatletter\ifx\SetFigFont\undefined%
\gdef\SetFigFont#1#2#3#4#5{%
  \reset@font\fontsize{#1}{#2pt}%
  \fontfamily{#3}\fontseries{#4}\fontshape{#5}%
  \selectfont}%
\fi\endgroup%
\begin{picture}(2503,2085)(1407,-2015)
\put(3875,-290){\makebox(0,0)[lb]{\smash{{\SetFigFont{10}{12.0}{\rmdefault}{\mddefault}{\updefault}{\color[rgb]{0,0,0}$E$}%
}}}}
\put(3895,-1669){\makebox(0,0)[lb]{\smash{{\SetFigFont{10}{12.0}{\rmdefault}{\mddefault}{\updefault}{\color[rgb]{0,0,0}$\R^{n-1}$}%
}}}}
\put(1527,-296){\makebox(0,0)[lb]{\smash{{\SetFigFont{8}{9.6}{\rmdefault}{\mddefault}{\updefault}{\color[rgb]{0,0,0}$b$}%
}}}}
\put(1540,-1445){\makebox(0,0)[lb]{\smash{{\SetFigFont{8}{9.6}{\rmdefault}{\mddefault}{\updefault}{\color[rgb]{0,0,0}$a$}%
}}}}
\put(2554,-1938){\makebox(0,0)[lb]{\smash{{\SetFigFont{10}{12.0}{\rmdefault}{\mddefault}{\updefault}{\color[rgb]{0,0,0}$G$}%
}}}}
\put(3376,-902){\makebox(0,0)[lb]{\smash{{\SetFigFont{10}{12.0}{\rmdefault}{\mddefault}{\updefault}{\color[rgb]{0,0,0}$u_0$}%
}}}}
\end{picture}%

%% file: we2.pstex_t
\begin{picture}(0,0)%
\includegraphics{we2.eps}%
\end{picture}%
\setlength{\unitlength}{3947sp}%
\begingroup\makeatletter\ifx\SetFigFont\undefined%
\gdef\SetFigFont#1#2#3#4#5{%
  \reset@font\fontsize{#1}{#2pt}%
  \fontfamily{#3}\fontseries{#4}\fontshape{#5}%
  \selectfont}%
\fi\endgroup%
\begin{picture}(4385,2245)(319,-1655)
\put(1278,444){\makebox(0,0)[lb]{\smash{{\SetFigFont{10}{12.0}{\rmdefault}{\mddefault}{\updefault}{\color[rgb]{0,0,0}$\pa H$}%
}}}}
\put(1085,-573){\makebox(0,0)[lb]{\smash{{\SetFigFont{10}{12.0}{\rmdefault}{\mddefault}{\updefault}{\color[rgb]{0,0,0}$0$}%
}}}}
\put(3790,-1212){\makebox(0,0)[lb]{\smash{{\SetFigFont{10}{12.0}{\rmdefault}{\mddefault}{\updefault}{\color[rgb]{0,0,0}$0$}%
}}}}
\put(4433, 68){\makebox(0,0)[lb]{\smash{{\SetFigFont{8}{9.6}{\rmdefault}{\mddefault}{\updefault}{\color[rgb]{0,0,0}$w_E=+\infty$}%
}}}}
\put(4040,-739){\makebox(0,0)[lb]{\smash{{\SetFigFont{8}{9.6}{\rmdefault}{\mddefault}{\updefault}{\color[rgb]{0,0,0}$w_E$}%
}}}}
\put(3324,-489){\makebox(0,0)[lb]{\smash{{\SetFigFont{10}{12.0}{\rmdefault}{\mddefault}{\updefault}{\color[rgb]{0,0,0}$E\cap B$}%
}}}}
\put(2911, 79){\makebox(0,0)[lb]{\smash{{\SetFigFont{8}{9.6}{\rmdefault}{\mddefault}{\updefault}{\color[rgb]{0,0,0}$w_E=+\infty$}%
}}}}
\put(3865, 79){\makebox(0,0)[lb]{\smash{{\SetFigFont{8}{9.6}{\rmdefault}{\mddefault}{\updefault}{\color[rgb]{0,0,0}$x_1$}%
}}}}
\put(1701,-549){\makebox(0,0)[lb]{\smash{{\SetFigFont{10}{12.0}{\rmdefault}{\mddefault}{\updefault}{\color[rgb]{0,0,0}$E$}%
}}}}
\put(2301,-508){\makebox(0,0)[lb]{\smash{{\SetFigFont{8}{9.6}{\rmdefault}{\mddefault}{\updefault}{\color[rgb]{0,0,0}$x_1$}%
}}}}
\end{picture}%

%% file: excess.pstex_t
\begin{picture}(0,0)%
\includegraphics{excess.eps}%
\end{picture}%
\setlength{\unitlength}{3947sp}%
\begingroup\makeatletter\ifx\SetFigFont\undefined%
\gdef\SetFigFont#1#2#3#4#5{%
  \reset@font\fontsize{#1}{#2pt}%
  \fontfamily{#3}\fontseries{#4}\fontshape{#5}%
  \selectfont}%
\fi\endgroup%
\begin{picture}(3459,1885)(1062,-1512)
\put(4444,-685){\makebox(0,0)[lb]{\smash{{\SetFigFont{10}{12.0}{\rmdefault}{\mddefault}{\updefault}{\color[rgb]{0,0,0}$e$}%
}}}}
\put(2677,-1066){\makebox(0,0)[lb]{\smash{{\SetFigFont{10}{12.0}{\rmdefault}{\mddefault}{\updefault}{\color[rgb]{0,0,0}$E$}%
}}}}
\put(4334,-368){\makebox(0,0)[lb]{\smash{{\SetFigFont{10}{12.0}{\rmdefault}{\mddefault}{\updefault}{\color[rgb]{0,0,0}$\nu$}%
}}}}
\put(3151, 89){\makebox(0,0)[lb]{\smash{{\SetFigFont{10}{12.0}{\rmdefault}{\mddefault}{\updefault}{\color[rgb]{0,0,0}$H$}%
}}}}
\put(2150,-602){\makebox(0,0)[lb]{\smash{{\SetFigFont{8}{9.6}{\rmdefault}{\mddefault}{\updefault}{\color[rgb]{0,0,0}$\C_\nu(x,r)$}%
}}}}
\put(1163,-595){\makebox(0,0)[lb]{\smash{{\SetFigFont{10}{12.0}{\rmdefault}{\mddefault}{\updefault}{\color[rgb]{0,0,0}$A$}%
}}}}
\end{picture}%

%% file: sei.pstex_t
\begin{picture}(0,0)%
\includegraphics{sei.eps}%
\end{picture}%
\setlength{\unitlength}{3947sp}%
\begingroup\makeatletter\ifx\SetFigFont\undefined%
\gdef\SetFigFont#1#2#3#4#5{%
  \reset@font\fontsize{#1}{#2pt}%
  \fontfamily{#3}\fontseries{#4}\fontshape{#5}%
  \selectfont}%
\fi\endgroup%
\begin{picture}(3811,2885)(1497,-2355)
\put(3402,-2152){\makebox(0,0)[lb]{\smash{{\SetFigFont{9}{10.8}{\rmdefault}{\mddefault}{\updefault}{\color[rgb]{0,0,0}$\beta_2\,r$}%
}}}}
\put(1512,-1978){\makebox(0,0)[lb]{\smash{{\SetFigFont{9}{10.8}{\rmdefault}{\mddefault}{\updefault}{\color[rgb]{0,0,0}$c=-\frac{\nu\cdot e_1}{\sqrt{1-(\nu\cdot e_1)^2}}$}%
}}}}
\put(2880,-727){\makebox(0,0)[lb]{\smash{{\SetFigFont{9}{10.8}{\rmdefault}{\mddefault}{\updefault}{\color[rgb]{0,0,0}$\beta_2\,r$}%
}}}}
\put(3204,-1186){\makebox(0,0)[lb]{\smash{{\SetFigFont{10}{12.0}{\rmdefault}{\mddefault}{\updefault}{\color[rgb]{0,0,0}$0$}%
}}}}
\put(3367,155){\makebox(0,0)[lb]{\smash{{\SetFigFont{10}{12.0}{\rmdefault}{\mddefault}{\updefault}{\color[rgb]{0,0,0}$\pa H$}%
}}}}
\put(1617,-424){\makebox(0,0)[lb]{\smash{{\SetFigFont{9}{10.8}{\rmdefault}{\mddefault}{\updefault}{\color[rgb]{0,0,0}$\nu$}%
}}}}
\put(2055,-290){\makebox(0,0)[lb]{\smash{{\SetFigFont{9}{10.8}{\rmdefault}{\mddefault}{\updefault}{\color[rgb]{0,0,0}${\bf e_1}(\nu)$}%
}}}}
\put(4731,-194){\makebox(0,0)[lb]{\smash{{\SetFigFont{9}{10.8}{\rmdefault}{\mddefault}{\updefault}{\color[rgb]{0,0,0}$\nu^\perp=\{x:\q_{{\bf e_1}(\nu)}x=c\,x_1\}$}%
}}}}
\put(2873,-1468){\makebox(0,0)[lb]{\smash{{\SetFigFont{9}{10.8}{\rmdefault}{\mddefault}{\updefault}{\color[rgb]{0,0,0}$\beta_2\,r$}%
}}}}
\put(4503,-1003){\makebox(0,0)[lb]{\smash{{\SetFigFont{10}{12.0}{\rmdefault}{\mddefault}{\updefault}{\color[rgb]{0,0,0}$x_1$}%
}}}}
\put(5293,-622){\makebox(0,0)[lb]{\smash{{\SetFigFont{9}{10.8}{\rmdefault}{\mddefault}{\updefault}{\color[rgb]{0,0,0}$u(\p_{{\bf e_1}(\nu)}x)$}%
}}}}
\put(4736,-1595){\makebox(0,0)[lb]{\smash{{\SetFigFont{9}{10.8}{\rmdefault}{\mddefault}{\updefault}{\color[rgb]{0,0,0}$\D_{{\bf e_1}(\nu)}(0,\beta_2\,r)$}%
}}}}
\put(4221,187){\makebox(0,0)[lb]{\smash{{\SetFigFont{10}{12.0}{\rmdefault}{\mddefault}{\updefault}{\color[rgb]{0,0,0}$H\cap\pa E$}%
}}}}
\put(3678,-400){\makebox(0,0)[lb]{\smash{{\SetFigFont{9}{10.8}{\rmdefault}{\mddefault}{\updefault}{\color[rgb]{0,0,0}$x$}%
}}}}
\end{picture}%

%% file: cinque.pstex_t
\begin{picture}(0,0)%
\includegraphics{cinque.eps}%
\end{picture}%
\setlength{\unitlength}{3947sp}%
\begingroup\makeatletter\ifx\SetFigFont\undefined%
\gdef\SetFigFont#1#2#3#4#5{%
  \reset@font\fontsize{#1}{#2pt}%
  \fontfamily{#3}\fontseries{#4}\fontshape{#5}%
  \selectfont}%
\fi\endgroup%
\begin{picture}(3386,2337)(911,-1895)
\put(4051,267){\makebox(0,0)[lb]{\smash{{\SetFigFont{10}{12.0}{\rmdefault}{\mddefault}{\updefault}{\color[rgb]{0,0,0}$y_n=1-\frac{\nu\cdot e_1}{\sqrt{1-(\nu\cdot e_1)^2}}\,y_1$}%
}}}}
\put(1247,-1832){\makebox(0,0)[lb]{\smash{{\SetFigFont{10}{12.0}{\rmdefault}{\mddefault}{\updefault}{\color[rgb]{0,0,0}$\pa H$}%
}}}}
\put(2153,-756){\makebox(0,0)[lb]{\smash{{\SetFigFont{10}{12.0}{\rmdefault}{\mddefault}{\updefault}{\color[rgb]{0,0,0}$x_1$}%
}}}}
\put(3376,-1832){\makebox(0,0)[lb]{\smash{{\SetFigFont{10}{12.0}{\rmdefault}{\mddefault}{\updefault}{\color[rgb]{0,0,0}$\pa H$}%
}}}}
\put(4282,-756){\makebox(0,0)[lb]{\smash{{\SetFigFont{10}{12.0}{\rmdefault}{\mddefault}{\updefault}{\color[rgb]{0,0,0}$y_1$}%
}}}}
\put(1741,-283){\makebox(0,0)[lb]{\smash{{\SetFigFont{10}{12.0}{\rmdefault}{\mddefault}{\updefault}{\color[rgb]{0,0,0}$\C\cap H$}%
}}}}
\put(3887,-270){\makebox(0,0)[lb]{\smash{{\SetFigFont{10}{12.0}{\rmdefault}{\mddefault}{\updefault}{\color[rgb]{0,0,0}$L^{-1}(\C\cap H)$}%
}}}}
\put(933,-290){\makebox(0,0)[lb]{\smash{{\SetFigFont{10}{12.0}{\rmdefault}{\mddefault}{\updefault}{\color[rgb]{0,0,0}$e_n$}%
}}}}
\put(1222,-763){\makebox(0,0)[lb]{\smash{{\SetFigFont{10}{12.0}{\rmdefault}{\mddefault}{\updefault}{\color[rgb]{0,0,0}$0$}%
}}}}
\put(2842,-387){\makebox(0,0)[lb]{\smash{{\SetFigFont{10}{12.0}{\rmdefault}{\mddefault}{\updefault}{\color[rgb]{0,0,0}$\nu$}%
}}}}
\end{picture}%

%% file: finfout.pstex_t
\begin{picture}(0,0)%
\includegraphics{finfout.eps}%
\end{picture}%
\setlength{\unitlength}{3947sp}%
\begingroup\makeatletter\ifx\SetFigFont\undefined%
\gdef\SetFigFont#1#2#3#4#5{%
  \reset@font\fontsize{#1}{#2pt}%
  \fontfamily{#3}\fontseries{#4}\fontshape{#5}%
  \selectfont}%
\fi\endgroup%
\begin{picture}(5742,1954)(676,-1746)
\put(1777,-782){\makebox(0,0)[lb]{\smash{{\SetFigFont{8}{9.6}{\rmdefault}{\mddefault}{\updefault}{\color[rgb]{0,0,0}$T$}%
}}}}
\put(2135,-1689){\makebox(0,0)[lb]{\smash{{\SetFigFont{8}{9.6}{\rmdefault}{\mddefault}{\updefault}{\color[rgb]{0,0,0}$\q x=-1$}%
}}}}
\put(2135,-1150){\makebox(0,0)[lb]{\smash{{\SetFigFont{8}{9.6}{\rmdefault}{\mddefault}{\updefault}{\color[rgb]{0,0,0}$\q x=-1/4$}%
}}}}
\put(1055,-782){\makebox(0,0)[lb]{\smash{{\SetFigFont{8}{9.6}{\rmdefault}{\mddefault}{\updefault}{\color[rgb]{0,0,0}$R$}%
}}}}
\put(1426,-778){\makebox(0,0)[lb]{\smash{{\SetFigFont{8}{9.6}{\rmdefault}{\mddefault}{\updefault}{\color[rgb]{0,0,0}$S$}%
}}}}
\put(2130,-659){\makebox(0,0)[lb]{\smash{{\SetFigFont{8}{9.6}{\rmdefault}{\mddefault}{\updefault}{\color[rgb]{0,0,0}$\q x=c$}%
}}}}
\put(2148, 81){\makebox(0,0)[lb]{\smash{{\SetFigFont{8}{9.6}{\rmdefault}{\mddefault}{\updefault}{\color[rgb]{0,0,0}$\q x=1$}%
}}}}
\put(2126,-458){\makebox(0,0)[lb]{\smash{{\SetFigFont{8}{9.6}{\rmdefault}{\mddefault}{\updefault}{\color[rgb]{0,0,0}$\q x=1/4$}%
}}}}
\put(5057,-1502){\makebox(0,0)[lb]{\smash{{\SetFigFont{10}{12.0}{\rmdefault}{\mddefault}{\updefault}{\color[rgb]{0,0,0}$F_{\rm out}$}%
}}}}
\put(772,-1508){\makebox(0,0)[lb]{\smash{{\SetFigFont{10}{12.0}{\rmdefault}{\mddefault}{\updefault}{\color[rgb]{0,0,0}$E$}%
}}}}
\put(3046,-1507){\makebox(0,0)[lb]{\smash{{\SetFigFont{10}{12.0}{\rmdefault}{\mddefault}{\updefault}{\color[rgb]{0,0,0}$F_{\rm in}$}%
}}}}
\put(735,-776){\makebox(0,0)[lb]{\smash{{\SetFigFont{8}{9.6}{\rmdefault}{\mddefault}{\updefault}{\color[rgb]{0,0,0}$z$}%
}}}}
\end{picture}%

%% file: fail.pstex_t
\begin{picture}(0,0)%
\includegraphics{fail.eps}%
\end{picture}%
\setlength{\unitlength}{3947sp}%
\begingroup\makeatletter\ifx\SetFigFont\undefined%
\gdef\SetFigFont#1#2#3#4#5{%
  \reset@font\fontsize{#1}{#2pt}%
  \fontfamily{#3}\fontseries{#4}\fontshape{#5}%
  \selectfont}%
\fi\endgroup%
\begin{picture}(2824,2215)(495,-1891)
\put(1799,-1276){\makebox(0,0)[lb]{\smash{{\SetFigFont{10}{12.0}{\rmdefault}{\mddefault}{\updefault}{\color[rgb]{0,0,0}$F_1$}%
}}}}
\put(537,-283){\makebox(0,0)[lb]{\smash{{\SetFigFont{8}{9.6}{\rmdefault}{\mddefault}{\updefault}{\color[rgb]{0,0,0}$e_n$}%
}}}}
\put(996,-784){\makebox(0,0)[lb]{\smash{{\SetFigFont{8}{9.6}{\rmdefault}{\mddefault}{\updefault}{\color[rgb]{0,0,0}$0$}%
}}}}
\put(2238,-1828){\makebox(0,0)[lb]{\smash{{\SetFigFont{8}{9.6}{\rmdefault}{\mddefault}{\updefault}{\color[rgb]{0,0,0}$\q x=-L\,x\cdot e_1$}%
}}}}
\put(1225,164){\makebox(0,0)[lb]{\smash{{\SetFigFont{10}{12.0}{\rmdefault}{\mddefault}{\updefault}{\color[rgb]{0,0,0}$\pa H$}%
}}}}
\put(3177,-718){\makebox(0,0)[lb]{\smash{{\SetFigFont{8}{9.6}{\rmdefault}{\mddefault}{\updefault}{\color[rgb]{0,0,0}$x_1$}%
}}}}
\put(2280,197){\makebox(0,0)[lb]{\smash{{\SetFigFont{8}{9.6}{\rmdefault}{\mddefault}{\updefault}{\color[rgb]{0,0,0}$\q x=L\,x\cdot e_1$}%
}}}}
\put(1625,-983){\makebox(0,0)[lb]{\smash{{\SetFigFont{8}{9.6}{\rmdefault}{\mddefault}{\updefault}{\color[rgb]{0,0,0}$x_*$}%
}}}}
\put(2131,-289){\makebox(0,0)[lb]{\smash{{\SetFigFont{10}{12.0}{\rmdefault}{\mddefault}{\updefault}{\color[rgb]{0,0,0}$\pa H_1$}%
}}}}
\end{picture}%

%% file: epsilon-dpm2.bbl
\def\cprime{$'$}
\begin{thebibliography}{Gr{\"u}87b}

\bibitem[AC81]{altcaffarelli}
H.~W. Alt and L.~A. Caffarelli.
\newblock Existence and regularity for a minimum problem with free boundary.
\newblock {\em J. Reine Angew. Math.}, 325:105--144, 1981.

\bibitem[ACF84]{altcaffarellifriedman}
H.W. Alt, L.~A. Caffarelli, and A.~Friedman.
\newblock A free boundary problem for quasilinear elliptic equations.
\newblock {\em Ann. Scuola Norm. Sup. Pisa Cl. Sci. (4)}, 11(1):1--44, 1984.

\bibitem[AFP00]{AFP}
L.~Ambrosio, N.~Fusco, and D.~Pallara.
\newblock {\em Functions of bounded variation and free discontinuity problems}.
\newblock Oxford Mathematical Monographs. The Clarendon Press, Oxford
  University Press, New York, 2000.

\bibitem[Alm68]{Almgren68}
F.~J. Almgren.
\newblock Existence and regularity almost everywhere of solutions to elliptic
  variational problems among surfaces of varying topological type and
  singularity structure.
\newblock {\em Ann. Math.}, 87:321--391, 1968.

\bibitem[Alm76]{Almgren76}
F.~J. Almgren.
\newblock Existence and regularity almost everywhere of solutions to elliptic
  variational problems with constraints.
\newblock {\em Mem. Amer. Math. Soc.}, 4(165):viii+199 pp, 1976.

\bibitem[Bae12]{baer}
E.~Baer.
\newblock Minimizers of anisotropic surface tensions under gravity: higher
  dimensions via symmetrization.
\newblock 2012.

\bibitem[Bom82]{bombieri}
E.~Bombieri.
\newblock Regularity theory for almost minimal currents.
\newblock {\em Arch. Ration. Mech. Anal.}, 7(7):99--130, 1982.

\bibitem[CF85]{caffarellifriedman85}
L.~A. Caffarelli and A.~Friedman.
\newblock Regularity of the boundary of a capillary drop on an inhomogeneous
  plane and related variational problems.
\newblock {\em Rev. Mat. Iberoamericana}, 1(1):61--84, 1985.

\bibitem[CM07]{caffarellimellet}
L.~A. Caffarelli and A.~Mellet.
\newblock Capillary drops: contact angle hysteresis and sticking drops.
\newblock {\em Calc. Var. Partial Differential Equations}, 29(2):141--160,
  2007.

\bibitem[DG60]{DeGiorgiREG}
E.~De~Giorgi.
\newblock {\em Frontiere orientate di misura minima}.
\newblock Seminario di Matematica della Scuola Normale Superiore di Pisa.
  Editrice Tecnico Scientifica, Pisa, 1960.

\bibitem[DM09]{duzaarmingioneharmonic}
F.~Duzaar and G.~Mingione.
\newblock Harmonic type approximation lemmas.
\newblock {\em J. Math. Anal. Appl.}, 352(1):301--335, 2009.

\bibitem[DS94]{duzaarsteffencomparison}
F.~Duzaar and K.~Steffen.
\newblock Comparison principles for hypersurfaces of prescribed mean curvature.
\newblock {\em J. Reine Angew. Math.}, 457:71--83, 1994.

\bibitem[DS02]{DuzaarSteffen}
F.~Duzaar and K.~Steffen.
\newblock Optimal interior and boundary regularity for almost minimizers to
  elliptic variational integrals.
\newblock {\em J. Reine Angew. Math.}, 564:73--138, 2002.

\bibitem[Fed69]{FedererBOOK}
H.~Federer.
\newblock {\em Geometric measure theory}, volume 153 of {\em Die Grundlehren
  der mathematischen Wissenschaften}.
\newblock Springer-Verlag New York Inc., New York, 1969.

\bibitem[Fin86]{Finn}
R.~Finn.
\newblock {\em Equilibrium Capillary Surfaces}, volume 284 of {\em Die
  Grundlehren der mathematischen Wissenschaften}.
\newblock Springer-Verlag New York Inc., New York, 1986.

\bibitem[Gau30]{gauss1830}
C.~F. Gauss.
\newblock Principia generalia theoriae figurae fluidorum.
\newblock {\em Comment. Soc. Regiae Scient. Gottingensis Rec.}, 7, 1830.

\bibitem[Giu84]{GiustiMinimalSurfacesBOOK}
E.~Giusti.
\newblock {\em Minimal surfaces and functions of bounded variation}, volume~80
  of {\em Monographs in {M}athematics}.
\newblock Birkh{\"a}user {V}erlag, Basel, 1984.

\bibitem[Giu03]{Gi}
E.~Giusti.
\newblock {\em Direct {M}ethods in the {C}alculus of {V}ariations}, volume~2 of
  {\em Lecture Notes. {S}cuola {N}ormale {S}uperiore di {P}isa (New Series)}.
\newblock {W}orld {S}cientific {P}ublishing {C}o. {I}nc., River {E}dge, NJ,
  2003.

\bibitem[GJ86]{gruterjost}
M.~Gr{\"u}ter and J.~Jost.
\newblock Allard type regularity results for varifolds with free boundaries.
\newblock {\em Ann. Scuola Norm. Sup. Pisa Cl. Sci. (4)}, 13(1):129--169, 1986.

\bibitem[GMS98]{GMSbook2}
M.~Giaquinta, G.~Modica, and J.~Soucek.
\newblock {\em Cartesian currents in the {C}alculus of {V}ariations. II.
  Variational integrals}, volume~38 of {\em Ergebnisse der {M}athematik und
  ihrer {G}renzgebiete. 3. Folge. A Series of Modern Surveys in {M}athematics}.
\newblock Springer-Verlag, Berlin, 1998.

\bibitem[Gon77]{GonzalezREGOLARITAGOCCIA}
E.~Gonzalez.
\newblock Regolarit\`a per il problema della goccia appoggiata.
\newblock {\em Rend. Sem. Mat. Univ. Padova}, 58:25--33, 1977.

\bibitem[Gr{\"u}87a]{gruter}
M.~Gr{\"u}ter.
\newblock Boundary regularity for solutions of a partitioning problem.
\newblock {\em Arch. Rational Mech. Anal.}, 97(3):261--270, 1987.

\bibitem[Gr{\"u}87b]{gruter2}
M.~Gr{\"u}ter.
\newblock Optimal regularity for codimension one minimal surfaces with a free
  boundary.
\newblock {\em Manuscripta Math.}, 58(3):295--343, 1987.

\bibitem[Gr{\"u}87c]{gruter3}
M.~Gr{\"u}ter.
\newblock Regularity results for minimizing currents with a free boundary.
\newblock {\em J. Reine Angew. Math.}, 375/376:307--325, 1987.

\bibitem[GT98]{gt}
D.~Gilbarg and N.~S. Trudinger.
\newblock {\em Elliptic partial differential equations of second order}.
\newblock Springer, Berlin; New York, 1998.

\bibitem[Har77]{hardt}
R.~Hardt.
\newblock On boundary regularity for integral currents or flat chains modulo
  two minimizing the integral of an elliptic integrand.
\newblock {\em Comm. Part. Diff. Equ.}, 2:1163--1232, 1977.

\bibitem[Mag12]{maggiBOOK}
F.~Maggi.
\newblock {\em Sets of finite perimeter and geometric variational problems},
  volume 135 of {\em Cambridge Studies in Advanced Mathematics}.
\newblock Cambridge University Press, Cambridge, 2012.
\newblock An introduction to {G}eometric {M}easure {T}heory.

\bibitem[Rei60]{reifenberg1}
E.~R. Reifenberg.
\newblock Solution of the {P}lateau problem for $m$-dimensional surfaces of
  varying topological type.
\newblock {\em Acta Math.}, 104:1--92, 1960.

\bibitem[Rei64a]{reifenberg2}
E.~R. Reifenberg.
\newblock An epiperimetric inequality related to the analyticity of minimal
  surfaces.
\newblock {\em Ann. of Math.}, 80(2):1--14, 1964.

\bibitem[Rei64b]{reifenberg3}
E.~R. Reifenberg.
\newblock On the analyticity of minimal surfaces.
\newblock {\em Ann. of Math.}, 80(2):15--21, 1964.

\bibitem[Sim83]{SimonLN}
L.~Simon.
\newblock {\em Lectures on geometric measure theory}, volume~3 of {\em
  Proceedings of the Centre for Mathematical Analysis}.
\newblock Australian National University, Centre for Mathematical Analysis,
  Canberra, 1983.

\bibitem[Sim96]{Simon2}
L.~Simon.
\newblock {\em Theorems on regularity and singularity of energy minimizing
  maps}.
\newblock Birkha\"user-Verlag, Basel-Boston-Berlin, 1996.

\bibitem[SSA77]{schoensimonalmgren}
R.~Schoen, L.~Simon, and F.~J.~Jr. Almgren.
\newblock Regularity and singularity estimates on hypersurfaces minimizing
  parametric elliptic variational integrals. {I}, {II}.
\newblock {\em Acta Math.}, 139(3-4):217--265, 1977.

\bibitem[SW89]{solomonwhite}
B.~Solomon and B.~White.
\newblock A strong maximum principle for varifolds that are stationary with
  respect to even parametric elliptic functionals.
\newblock {\em Indiana Univ. Math. J.}, 38(3):683--691, 1989.

\bibitem[Tam84]{tamanini}
I.~Tamanini.
\newblock {\em Regularity results for almost minimal oriented hypersurfaces in
  $\mathbb{R}^N$}.
\newblock Quaderni del Dipartimento di Matematica dell'Universit\`a di Lecce.
  Università di Lecce, 1984.
\newblock available for download at http://cvgmt.sns.it/paper/1807/.

\bibitem[Tay76]{taylor76}
J.~E. Taylor.
\newblock The structure of singularities in soap-bubble-like and soap-film-like
  minimal surfaces.
\newblock {\em Ann. of Math. (2)}, 103(3):489--539, 1976.

\bibitem[Tay77]{taylor77}
J.~E. Taylor.
\newblock Boundary regularity for solutions to various capillarity and free
  boundary problems.
\newblock {\em Comm. Partial Differential Equations}, 2(4):323--357, 1977.

\bibitem[You05]{young1805}
T.~Young.
\newblock An essay on the cohesion of fluids.
\newblock {\em Philos. Trans. Roy. Soc. London}, 95:65--87, 1805.

\end{thebibliography}
